%% file: mainHeitmann.tex
\documentclass[11pt,a4paper,twoside]{article}
\usepackage{titling}
\usepackage{titletoc}
\usepackage{url}

\usepackage[utf8]{inputenc}
\usepackage[T1]{fontenc}
\usepackage{amsmath}
\usepackage{amsfonts}
\usepackage{amssymb}
\usepackage{amsthm}

\usepackage{mathrsfs}
\usepackage{mathtools}
\usepackage{stmaryrd}
\usepackage[all]{xy}
\usepackage{makeidx}
\usepackage{natbib}
\bibpunct{(}{)}{,}{a}{}{,}
\usepackage[french,english]{babel}
\usepackage{xspace}
\usepackage{etoolbox}

\makeatletter\def\th@plain{\slshape}\makeatother
\makeatletter\patchcmd{\th@remark}{\itshape}{\slshape}{}{}\makeatother

\newcounter{bidon}

\newcommand{\rdb}{\refstepcounter{bidon}}

\input{MathMacrosHeitmann.tex}

\newcommand \sibrouillon[1]{}

\newcommand \hum[1] {\sibrouillon{\noindent {\sf hum: #1}}}

\usepackage[bookmarksopen=false,breaklinks=true,%
      backref=page,pagebackref=true,plainpages=false,%
      hyperindex=true,pdfstartview=FitH,colorlinks=true,%
      pdfpagelabels=true,linkcolor=blue,%
      citecolor=red,urlcolor=red,
      ]%
   {hyperref}

\begin{document}
\selectlanguage{english}

\pagenumbering{roman}
\thispagestyle{empty}
~ 
\vspace{1cm}

\noindent In this file you find the English version starting on the page  numbered \pageref{beginenglish}.

\medskip \noindent  {\Large \bf Heitmann dimension of distributive lattices and commutative rings}

\begin{abstract}
This paper is the English translation of the first 4 sections of the article
\citealt*{jCLQ2006}, after some corrections. 

Sections 5--7 of the original article are treated a bit more simply in \citealt*{jCACM}.   

We study the notion of dimension introduced by Heitmann in
his remarkable article \citealt*{jHei84}, as well as a related notion, only implicit in his proofs. We first develop this within the general framework of the theory of distributive lattices and spectral spaces. We then apply these ideas in the
framework of commutative algebra.
\end{abstract}

\smallskip \noindent   {\large \bf Authors}

\smallskip \noindent Thierry Coquand

\noindent 
Chalmers, University of G\"oteborg, Sweden,
email: coquand@cs.chalmers.se, \\
url: \url{https://www.cse.chalmers.se/~coquand/},

\smallskip 
\noindent Henri Lombardi

\noindent Laboratoire de Mathématiques de Besançon (UMR 6623), 
Université Marie et Louis Pasteur, 25030 Besançon cedex, France,
email: henri.lombardi@univ-fcomte.fr, \\
url: \url{http://hlombardi.free.fr},

\smallskip \noindent 
Claude Quitté

\noindent Laboratoire de Mathématiques,
SP2MI, Boulevard 3, Teleport 2, BP 179,
86960 FUTUROSCOPE Cedex, FRANCE,
email: claude.quitte@orange.fr

\selectlanguage{french}  
\bigskip \noindent
Then the French version begins on the page numbered  \pageref{beginfrench}. 

\medskip\noindent   {\Large \bf Dimension de Heitmann des treillis distributifs et des anneaux commutatifs}

\smallskip \noindent Le lecteur ou la lectrice sera sans doute surprise de l'alternance des sexes ainsi que de l'orthographe du mot corolaire, avec d'autres innovations auxquelles elle n'est pas habituée. En fait, nous avons essayé de suivre au plus près les préconisations de l'orthographe nouvelle recommandée, telle qu'elle est enseignée aujourd'hui dans les écoles en France.

\selectlanguage{english}

\bigskip  \noindent {\bf Historical note.}\\
The trigger of the invention of the Heitmann dimension, which we denote by $\Hdim$, is constituted by the remarkable article \citealt*[Generating non-Noetherian modules efficiently]{jHei84}. We review here this question and more generally the question of constructive definitions for various notions of dimension in commutative algebra.

\smallskip 
Constructive definitions of the Krull dimension of a commutative ring have been developed in several papers. 

The first is the  note \citealt*{jJoy76} developed by Luis Español in his thesis (\url{https://dialnet.unirioja.es/descarga/tesis/1402.pdf}) and in his articles \citealt*{jEsp82,jEsp83,jEsp86,jEsp88,jEsp2010}.
 
Joyal’s definition in \citealt*{jBJ1981} has been analysed in \citealt*{jCC00} using the notion of entailment relation (\url{https://www.cse.chalmers.se/~coquand/lattice.ps}).

The second is the paper \citealt*{jLom02} in which an explicit characterization of dimension in purely algebraic terms is proposed. The articles \citealt*{jCL2003,jCL2001-2018}  explain the equivalence of the two notions in constructive mathematics and develop for this a third equivalent constructive characterization (characterization \textsl{2c} in Definition \ref{defDK0}). 

Finally, a characterization in terms of boundaries is given in the article \citealt*{jCLR05} (\url{http://hlombardi.free.fr/publis/lebord.pdf}).
This last definition is recursive and it is this that has made it possible to give constructive versions of many classical results thereafter. In particular Serre's Splitting Off, Forster's theorem on the number of generators of a finitely generated module and Bass's cancellation theorem, in the case where the hypothesis is a bound on the Krull dimension.

The non-Noetherian version of these theorems is due to  \citet*[Generating non-Noetherian modules efficiently]{jHei84}. The new recursive characterization of the Krull dimension has made it possible to translate Heitmann's proofs into algorithms, described in the article \citealt*[Generating non-{N}oetherian modules constructively]{jCLQ2004} (\url{https://www.cse.chalmers.se/~coquand/fs.ps}).

\smallskip 
In the previously cited article, Heitmann examines the variants of these theorems in which the dimension hypothesis is improved (in particular by Swan) to the form of the dimension of the maximal spectrum. He notices that the maximal spectrum is no longer necessarily a spectral space when the ring is not assumed to be Noetherian. Starting from a consideration that he qualifies as philosophical, he then proposes to replace this maximal spectrum by a spectral space, the one the maximal spectrum generates inside the Zariski spectrum. This leads him to introduce a new dimension, which we have denoted by $\Jdim$. Unfortunately he does not obtain for $\Jdim$ all the theorems he wishes for and he leaves the question open.

It turns out that the recursive constructive definition of $\Jdim$ is too difficult to handle and this has led  \citealt*{jCLQ2004} to invent  $\Hdim$, whose recursive definition is easier to manipulate in their constructive proofs. They obtain all the results desired by Heitmann.

\smallskip Another unsuspected outcome of the dimension defined in \citealt*{jLom02} has been updated by  
 \citealt*[Valuative dimension and monomial orders]{jKY2020} (\url{https://arxiv.org/abs/1906.12067}). This article gives a constructive definition of the valuative dimension of a commutative ring based on a very slight variant of the definition of Krull dimension given in \citealt*{jLom02}.

Ihsen Yengui has also used the constructive definition of  Krull dimension in several papers, some of which establish results previously unknown in classical mathematics.

\selectlanguage{french}  

\bigskip  \noindent {\bf Note historique.}  \\ 
Le déclic qui a déclenché l’invention de la dimension de Heitmann, que nous notons $\Hdim$, est constitué par le remarquable article     
\citealt*[Generating non-Noetherian modules efficiently]{jHei84}.
Nous faisons ici le point sur cette question et plus généralement sur la question des définitions constructives pour diverses notions de dimension en algèbre commutative. 

\smallskip 
Les définitions constructives de la dimension de Krull d’un anneau commutatif ont été mises au point dans plusieurs articles. 

Le premier est la note \citealt*{jJoy76} développée par Espa\~nol dans sa thèse (\url{https://dialnet.unirioja.es/descarga/tesis/1402.pdf}) et ses articles \citealt*{jEsp82,jEsp83,jEsp86,jEsp88,jEsp2010}. 

La définition de Joyal évoquée dans \citealt*{jBJ1981}  a été analysée par \citealt*{jCC00} en utilisant la notion de relation implicative (\url{https://www.cse.chalmers.se/~coquand/lattice.ps}).

Le deuxième est l’article  \citealt*{jLom02} dans lequel une caractérisation explicite de la dimension en termes purement algébriques est démontrée. 

Les articles \citealt*{jCL2003,jCL2001-2018}  donnent 
l’explication de l’équivalence des deux notions en mathématiques constructives et développent pour cela une troisième caractérisation constructive équivalente (caractérisation \textsl{2c} dans la définition \ref{fdefDK0}). 

Enfin une caractérisation en termes de bords est donnée dans l’article \citealt*{jCLR05} (\url{http://hlombardi.free.fr/publis/lebord.pdf}). Cette dernière définition est récursive et c’est elle qui a permis de donner des versions constructives de nombreux résultats classiques par la suite. Notamment le Splitting Off de Serre, le théorème de Forster sur le nombre de générateurs d’un module de type fini et le théorème de simplification de Bass, dans le cas où l’hypothèse est une borne sur la dimension de Krull. 

La version non noethérienne de ces théorèmes est due à     
\citet*[Generating non-Noetherian modules efficiently]{jHei84}.
La nouvelle caractérisation récursive de la dimension de Krull~a permis de traduire les démonstrations de Heitmann en des algorithmes, décrits dans l’article \citealt*[Generating non-{N}oetherian modules constructively]{jCLQ2004} (\url{https://www.cse.chalmers.se/~coquand/fs.ps}).   

\smallskip Dans son article déjà cité, Heitmann examine les variantes de ces théorèmes dans lesquelles l’hypothèse de dimension est améliorée (notamment par Swan) sous la forme de la dimension du spectre maximal. Il remarque que le spectre maximal n’est plus nécessairement un espace spectral lorsque l’anneau n’est pas supposé Noethérien. Partant d’une considération qu’il qualifie de philosophique, il propose alors de remplacer ce spectre maximal par un espace spectral, celui que le spectre maximal engendre à l’intérieur du spectre de Zariski. Cela l’amène à introduire une nouvelle dimension, que nous avons notée $\Jdim$. Malheureusement il n’obtient pas pour la~$\Jdim$ tous les théorèmes qu’il souhaite et il laisse la question ouverte. 

Il s’avère que la définition constructive récursive de la $\Jdim$ est trop difficile à manipuler et cela a amené \citealt*{jCLQ2004}  à inventer la $\Hdim$, dont la définition récursive est plus facile à manipuler dans leurs démonstrations constructives. Ils obtiennent tous les résultats souhaités par Heitmann.

\smallskip Une autre issue insoupçonnée de la dimension définie dans \citealt*{jLom02} a été mise à jour  dans l’article \citealt*[Valuative dimension and monomial orders]{jKY2020} (\url{https://arxiv.org/abs/1906.12067}). Cet article donne une définition constructive de la dimension valuative d’un anneau commutatif basée sur une très légère variante de la définition de la dimension de Krull donnée dans \citealt*{jLom02}.

 Ihsen Yengui a par ailleurs utilisé la  définition constructive de la  
dimension de Krull dans plusieurs articles, dont certains établissent des résultats inconnus auparavant en mathématiques classiques. 

\selectlanguage{english}  

\bibliographystyle{plainnat}

\newpage
\pagenumbering{arabic}


\pagestyle{headings}
\patchcmd{\sectionmark}{\MakeUppercase}{}{}{}
\setcounter{page}{0}\renewcommand\thepage{E\arabic{page}}
\input{englishHeitmann.tex}

\clearpage

\newpage
~
\thispagestyle{empty}


\clearpage
\newpage

\setcounter{page}{1}
\renewcommand\thepage{F\arabic{page}}\renewcommand\theHsection{F\arabic{section}}
\input{frenchHeitmann.tex}

\end{document}

%% file: MathMacrosHeitmann.tex
\newcommand \eoe {\hbox{}\nobreak\hfill
\vrule width .5em height .5em depth 0mm \par \smallskip}

\renewcommand \leq{\leqslant}
\renewcommand \preceq{\preccurlyeq}

\renewcommand \geq{\geqslant}

\newcommand \pref[1]{$(\ref{#1})$}

\newcommand \junk[1]{}

\renewcommand \grave{\mathaccent18}    
\newcommand \noi {\noindent}

\newcommand \sni {\smallskip\noi}
\newcommand \ms {\medskip}
\newcommand \mni {\ms\noi}

\newcommand \Ex {{\exists }}
\newcommand \Tt {{\forall }}

\newcommand \mt {\mapsto }

\renewcommand \cir {{^\circ}}
\newcommand\equidef{\buildrel{{\rm 
def}}\over{\quad\Longleftrightarrow\quad}}
\newcommand\eqdefi{\buildrel{\rm def}\over {\;=\;}}

\newcommand\vers[1]{\buildrel{#1}\over \longrightarrow }

\newcommand\gen[1]{\left\langle{#1}\right\rangle}
\newcommand\so[1]{\{{#1}\}}

\newcommand \sotq[2]{\so{\,#1\,\vert\,#2\,}}

\newcommand\bidule[1]{\left\lceil{#1}\right\rceil}

\newcommand \aqo[2]{#1\!\left/\gen{#2}\right.}

\newcommand \ov[1] {\overline{#1}}
\newcommand \ovs[1] {\overline{\{#1\}}}

\newcommand \wi[1] {\widetilde{#1} }

\newcommand \NN{\mathbb {N}}
\newcommand \ZZ{\mathbb {Z}}

\newcommand \gk{{\bf k}}
\newcommand \gA{\mathbf{A}}
\newcommand \gB{\mathbf{B}}

\newcommand \gH{\mathbf{H}}
\newcommand \gK{\mathbf{K}}
\newcommand \gL{\mathbf{L}}

\newcommand \gT{\mathbf{T}}

\newcommand \cI {{\cal I}}

\newcommand \cF {{\cal F}}

\newcommand \cP {{\cal P}}

\newcommand \cM {{\cal M}}

\newcommand \rI {\mathrm{I}}
\newcommand \rD {\mathrm{D}}
\newcommand \rF {\mathrm{F}}
\newcommand \rH {\mathrm{H}}
\newcommand \rJ {\mathrm{J}}
\newcommand \rK {\mathrm{K}}

\newcommand \rS {\mathrm{S}}
\newcommand \rV {\mathrm{V}}
\newcommand \IZ {\mathrm{IZ}}
\newcommand \FZ {\mathrm{FZ}}
\newcommand \IZA {\IZ_\gA}
\newcommand \DA {\rD_\gA}

\newcommand \JA {\rJ_\gA}
\newcommand \JT {\rJ_\gT}
\newcommand \Fmin {\rF_{\mathrm{min}}}
\newcommand \Imax {\rI_{\mathrm{max}}}

\newcommand\fa{\mathfrak{a}}
\newcommand\fb{\mathfrak{b}}
\newcommand\fc{\mathfrak{c}}
\newcommand\fA{\mathfrak{A}}
\newcommand\fB{\mathfrak{B}}
\newcommand\fD{\mathfrak{D}}
\newcommand \DT {\fD_\gT}
\newcommand \VT {\fV_\gT}
\newcommand \DTo {\fD_{\gT\cir}}
\newcommand \VTo {\fV_{\gT\cir}}

\newcommand\fII{\mathfrak{I}}
\newcommand\fj{\mathfrak{j}}
\newcommand\fJ{\mathfrak{J}}
\newcommand\fF{\mathfrak{F}}
\newcommand\ff{\mathfrak{f}}
\newcommand\ffg{\mathfrak{g}}

\newcommand\fm{\mathfrak{m}}

\newcommand\fp{\mathfrak{p}}
\newcommand\fP{\mathfrak{P}}
\newcommand\fq{\mathfrak{q}}
\newcommand\fV{\mathfrak{V}}

\newcommand \vu {\vee} 
\newcommand \vi {\wedge} 
\newcommand \Vu {\bigvee}
\newcommand \Vi {\bigwedge}
\newcommand \im {\rightarrow} 
\newcommand \dar { \,\downarrow\! }
\newcommand \uar { \,\uparrow\! }
\newcommand \bal[1] {^\rK_{#1}}
\newcommand \ul[1] {_\rK^{#1}}

\newcommand \Un {{\bf 1}}
\newcommand \Deux {{\bf 2}}
\newcommand \Trois {{\bf 3}}

\newcommand \etl{^*}

\newcommand \etoz{$^*$}

\newcommand \Pf {{{\cal P}_{\mathrm{f}}}}

\newcommand \Bd {\mathrm{Bd}}

\newcommand \He {\mathrm{He}}
\newcommand \Heit {\mathrm{Heit}}
\newcommand \HeA {{\Heit\,\gA}}

\newcommand \Jspec {\mathsf{Jspec}}
\newcommand \jspec {\mathsf{jspec}}

\newcommand \Hdim {\mathsf{Hdim}}
\newcommand \Kdim {\mathsf{Kdim}}
\newcommand \Jdim {\mathsf{Jdim}}
\newcommand \jdim {\mathsf{jdim}}

\newcommand \Id {\mathrm{Id}}

\newcommand \Idl {\cI}

\newcommand \Max {\mathsf{Max}}
\newcommand \Min {\mathsf{Min}}
\newcommand \Patch {\mathsf{Patch}}
\newcommand \Spec {\mathsf{Spec}}
\newcommand \Zar {\,\mathsf{Zar}}
\newcommand \ZarA {{\Zar\,\gA}}

\newcommand \OQC {\mathsf{Oqc}}
\renewcommand \mod {\,\mathrm{mod}}

\newdimen\xyrowsp
\newcommand{\SCO}[6]{
\xymatrix @R = \xyrowsp {
                                  &1 \ar@{-}[dl] \ar@{-}[dr] \\
#3 \ar@{-}[ddr]                   &   & #6 \ar@{-}[ddl] \\
                                  &\bullet\ar@{-}[d] \\
                                  &\bullet   \\
#2 \ar@{-}[ddr] \ar@{-}[uur]      &   & #5 \ar@{-}[ddl] \ar@{-}[uul] \\
                                  &\bullet \ar@{-}[d] \\
                                  &\bullet  \\
#1 \ar@{-}[uur]                   &   & #4 \ar@{-}[uul] \\
                                  & 0 \ar@{-}[ul] \ar@{-}[ur] \\
}
}

%% file: englishHeitmann.tex

\begingroup

\selectlanguage{english}

\input{EnglishTheoremsHeitmann.tex}

\input{EnglishMacrosHeitmann.tex}
\setcounter{page}{1}

\title{Heitmann dimension of distributive lattices\\ and \coris}

\def\proofname{\textsl{Proof}}

\author{Thierry Coquand, Henri Lombardi, Claude Quitté}

\date{\today}

\maketitle

\startcontents[english]

\rdb
\label{beginenglish}

\begin{abstract}
This paper is the English translation of the first 4 sections of the article
\citealt*{CLQ2006}, after some corrections. 

Sections 5--7 of the original article are treated a bit more simply in \citealt*{CACM}.   

We study the notion of dimension introduced by Heitmann in
his remarkable article \citealt*{Hei84}, as well as a related notion, only implicit in his proofs. We first develop this within the general framework of the theory of distributive lattices and spectral spaces.
\end{abstract}



\setcounter{tocdepth}{4}
\markboth{Contents}{Contents}
\small

\printcontents[english]{}{1}{}
\normalsize


\section*{Introduction}
\addcontentsline{toc}{section}{Introduction}
\markboth{Introduction}{Introduction} 
We study the notion of dimension introduced in \citealt*{Hei84}, as well as a related notion,
only implicit in his proofs. We develop this first
in the general framework of the theory of distributive lattices and
spectral spaces. We then apply this issue in the
framework of commutative algebra. 

In the duality between distributive lattices and spectral spaces, the Zariski spectrum of a \cori corresponds (as indicated by André Joyal in \cite{Joy76}) to the lattice of ideals which are nilradicals of \itfs.
We show that the spectral space defined by Heitmann for his notion of dimension corresponds to the lattice formed by the ideals which are Jacobson radicals of finitely generated ideals. This allows us to obtain an elementary constructive definition of the dimension defined by Heitmann (which we denote by $\Jdim$).

We introduce another dimension, which we call Heitmann dimension (and which we denote by $\Hdim$), which is ``better'' in the sense that $\Hdim\leq \Jdim$ and that it allows for natural proofs by induction.

As consequences, one finds in \citealt*{CACM} constructive versions of certain important classical theorems, in their non-Noetherian version (often due to Heitmann).

Constructive versions of these theorems ultimately turn out to be simpler, and sometimes more general, than the corresponding classical abstract versions.

In particular this gives the non-Noetherian versions of Swan's and Serre's  (Splitting Off) theorems first obtained in \citealt*{CLQ2004} and in \citealt*{Duc2006}.

\smallskip
Naturally, the main advantage that we see in our treatment is its very elementary character. In particular, we do not use \gui{unnecessary} assumptions such as the axiom of choice and the law of excluded middle (LEM) which are unavoidable to make  classical proofs work.

Finally, the fact of having got rid of all Noetherian assumptions
 is also not negligible, and allows to see better
the essence of things.

\medskip In conclusion, this article can be seen essentially as a constructive development of the theory of spectral spaces via distributive lattices, with particular emphasis on the little-known Heitmann dimension, that leads to striking \cov results in non-\noe commutative algebra.

In the following text, the theorems, propositions and lemmas demonstrated in classical mathematics are marked with a star. This indicates that the proof uses non-constructive principles. In general, a constructive proof is in this case impossible because the result in the form indicated implies a non-constructive principle (almost always a use of LEM). For example in classical mathematics one can always recover the points of a spectral space from the distributive lattice formed by its quasi-compact open sets, but this is not always possible from a constructive point of view.

\medskip \rem We solved a terminology problem that has arised while writing this article in the following way. The word \gui{duality} appears a priori in the context of distributive lattices with two different meanings. On the one hand, there is the duality which corresponds to reversing the order relation in a lattice. On the other hand there is the duality between distributive lattices and spectral spaces, which corresponds to an antiequivalence of categories. We have decided to reserve \gui{duality} for this last use. The term \gui{dual lattice} has therefore been systematically replaced with  \gui{opposite lattice}. Similarly, \gui{the dual notion} has been replaced with  \gui{the reverse notion} or with  \gui{the opposite notion}, and \gui{by duality} with  \gui{by reversing the order}.
\eoe

\section{Distributive lattices}
\label{secTRDI}

The axioms of distributive lattices can be formulated with universal equalities concerning only the two laws $\vi$ and $\vu$ and the two constants $0_\gT$ (the minimum element of the distributive lattice~$\gT$) and $1 _\gT$ (the maximum). The order relation is then defined by $a\leq_\gT b\;\Leftrightarrow\;a\vi b=a$. We thus obtain a purely equational theory, with all the related facilities.
For example we can define a distributive lattice by generators and
relations, the category of distributive lattice has general colimits (which can be defined by generators and relations) and general (projective) limits (which have as underlying sets the corresponding set-based limits).

A totally ordered set is a distributive lattice if it has a
maximum and a minimum. We denote by ${\bf n}$ a totally ordered set with $n$ elements. It is a distributive lattice if $n\neq 0$. The lattice $\Deux$ is the free distributive lattice with 0 generator, and $\Trois$ the one with 1 generator.

For any distributive lattice $\gT$, if we replace the order relation $x\leq_\gT y$ by the reverse relation $y\leq_\gT x$, we obtain the
\textsl{opposite lattice} $\gT\cir$ by exchanging $\vi$ and $\vu$.

\subsection{Ideals, filters}

If $\varphi :\gT\rightarrow \gT'$ is a distributive lattice morphism,
$\varphi^{-1}(0)$ is called an \textsl{ideal of $\gT$}. An ideal $\fII $ of~$\gT$ is a subset of $\gT$ subject to the following constraints:
\begin{equation} \label{eqIdeal}
\left.
\begin{array}{rcl}
   & &  0 \in \fII    \\
x,y\in \fII & \Longrightarrow   &  x\vu y \in \fII    \\
x\in \fII ,\; z\in \gT& \Longrightarrow   &  x\vi z \in \fII    \\
\end{array}
\right\}
\end{equation}
(the last may be rewritten $(x\in \fII ,\;y\leq x)\Rightarrow y\in
\fII$). A \textsl{principal ideal} is an ideal generated by a single element $a$: it is equal to
\begin{equation} \label{eqda}
\,\dar a=\sotq{x\in \gT}{x\leq a}.
\end{equation}
The ideal $\dar a$, endowed with the laws $\vi$ and $\vu$ of $\gT$, is a distributive lattice in which the maximum element is~$a$. The canonical injection $\dar a\rightarrow \gT$ \textsl{is not} a distributive lattice morphism because the image of $a$ is not equal to $1_\gT$. However
the surjective map $\gT\rightarrow \,\dar a,\;x\mapsto x\vi a$ is a surjective morphism, which therefore endows $\dar a$ with a quotient structure.

\smallskip The opposite notion to that of ideal is the notion of {\sl
filter}. The principal filter generated by~$a$ is denoted by $\,\uar a$.

\smallskip The \textsl{ideal generated} by a subset $J$ of $\gT$ is
$\cI_\gT(J)=\sotq{x\in\gT}{\Ex J_0\in \Pf(J),\,x\leq \Vu J_0}$.
Consequently \textsl{every finitely generated ideal is principal}.

If $A$ and $B$ are two subsets of $\gT$ we denote
\begin{equation} \label{eqvuvi}
A\vu B=\sotq{a\vu b}{a\in A,\,b\in B} \quad \mathrm{and}\quad A\vi
B=\sotq{a\vi b}{a\in A,\,b\in B}.
\end{equation}

Then the ideal generated by two ideals $\fa$ and $\fb$ is equal to
\begin{equation} \label{eqSupId}
\cI_\gT(\fa\cup \fb) = \fa\vu\fb =\sotq{z}{\exists
x\in\fa,\,\exists
y\in\fb,\,z\leq x\vu y}\,.
\end{equation}

The set $\Idl(\gT)$ of ideals of $\gT$\footnote{In fact, you have to introduce a restriction to really get a set, in order to have a well-defined process for constructing the \ids concerned.
For example, we can consider the set of ideals obtained from the principal ideals by iterating certain predefined operations, such as countable unions and intersections.} itself forms a distributive lattice for inclusion, with for the inf of $\fa$ and $\fb$ the \id:
\begin{equation} \label{eqInfId}
\fa\cap \fb=\fa\vi\fb.
\end{equation}
We will denote by $\cF_\gT(S)$ the filter of $\gT$ generated by the subset $S$.
When we consider the lattice of the filters we must
pay attention to what reversing the order relation produces: $\ff\cap\ffg=\ff\vu\ffg$ is the inf of $\ff$ and $\ffg$,
while their sup is equal to $\cF_\gT(\ff\cup \ffg)=\ff\vi
\ffg$.

\medskip The \textsl{quotient lattice of $\gT$ by the ideal $\fJ$}, denoted
by $\gT/(\fJ=0)$, is defined as the distributive lattice generated by the elements of
$\gT$ with the following relations: the true relations in $\gT$ on the one hand,
and the relations $x=0$ for the $x\in \fJ$ on the other hand. It can
also be defined by the preorder relation
\[a\preceq b\quad\Longleftrightarrow\quad a\leq_{\gT/(\fJ=0)}b 
\equidef \quad
\exists x\in \fJ \;\;a\leq  x\vu b
\]
This gives
\[ a\equiv b\;\;\mod\;(\fJ=0)\quad \Longleftrightarrow \quad \exists
x\in \fJ
\;\;a\vu x=b\vu x
\]
and in the case of the quotient by a principal ideal $\,\dar a$ we obtain
$\gT/(a=0)\simeq\,\uar a$ with the morphism $y\mapsto y\vu a$ of $\gT$
onto
$\,\uar a$.

\subsubsection*{Conductor, difference}
\addcontentsline{toc}{subsubsection}{Conductor, difference}

By analogy with commutative algebra, if $\fb$ is an ideal and $A$
a subset of
$\gT$, we will denote
\begin{equation} \label{eqTrans}
\fb:A\eqdefi\sotq{x\in\gT}{\Tt a\in A\quad a\vi x\in \fb}
\end{equation}
If $\fa$ is the ideal generated by $A$ we have $\fb:A=\fb:\fa$;
it is called the \textsl{conductor of $\fa$ in $\fb$.}

We also denote by
$(b:a)$ the \id
$ (\dar b):(\dar a)=\sotq{x\in\gT}{x\in\gT\;|\;x\vi a\leq b}$.

The opposite notion is called the \textsl{difference filter of two
filters}
\begin{equation} \label{eqDiff}
\ff\setminus\ff'\eqdefi\sotq{x\in\gT}{\Tt a\in \ff'\quad a\vu
x\in\ff}
\end{equation}
We also denote by
$(b\setminus a)$ the filter
$ (\uar b)\setminus(\uar a)=\sotq{x\in\gT}{b\leq x\vu a}$.

\subsubsection*{Jacobson radical}
\addcontentsline{toc}{subsubsection}{Jacobson radical}

An ideal $\fm$ of a nontrivial distributive lattice $\gT$ (i.e.\ distinct from $\Un$) is
said to be \textsl{maximal} \hbox{if $\gT/(\fm=0)\,=\,\Deux$}, i.e.\ if $1\notin\fm$ and
$\Tt x\in\gT\;(x\in\fm$ or $\Ex y\in \fm \;x\vu y=1)$.

It amounts to the same thing to say that it is  an ideal \gui{maximal
among strict ideals}.

In \clama we have the following lemma.
\begin{lemmac}
\label{lemHspec1}
In a distributive lattice $\gT\neq\Un$ the intersection of the \idemas is equal to the \id
\[ \sotq{a\in\gT}{\forall x\in\gT \;(a\vu x = 1 \Rightarrow
x=1)}.
\]
It is called the \emph{Jacobson radical of $\gT$}. We denote it by
$\JT(0)$. \\
More generally the intersection of the \idemas containing an ideal
 $\fJ\neq \gT$
is equal to the \id
\begin{equation} \label{eqRJJ}
\JT(\fJ)\,=\, \sotq{a\in\gT}{\forall x\in\gT \;(a\vu x = 1
\Rightarrow \Ex
z\in \fJ \;\;z\vu x=1)}
\end{equation}
It is called the \emph{Jacobson radical of the ideal $\fJ$}. In
particular:
\begin{equation} \label{eqRJb}
\JT(\dar b)=\sotq{a\in\gT}{\forall x\in\gT \;(\,a\vu x = 1\;
\Rightarrow
\;\;b\vu x=1\,)}
\end{equation}
\end{lemmac}
\begin{proof}
The second statement follows from the first by passing to the
quotient lattice $\gT/(\fJ=0)$. Let's see the first. We show that
$a$ is outside at least one \idema \ssi  $\Ex x\neq 1$ such that
$a\vu x=1$. If so, a \idema that contains $x$ (there are
such ideals since $x\neq 1$) cannot contain $a$ because it would contain $a\vu x$. Conversely, if $\fm$ is a \idema not containing $a$, the \id
generated by $\fm$ and~$a$ contains $1$. But this ideal is the set of
elements bounded by at least one $a\vu x$ where $x\in\fm$.
\end{proof}

In \clama a distributive lattice is called a \textsl{Jacobson lattice} if any
\idep is equal to its Jacobson radical. As every ideal is the
intersection of the \ideps that contain it, this implies that any \id
is equal to its Jacobson radical.

In \coma we adopt the following definitions.

\penalty-2500 
\begin{definitions}[Jacobson radical, weakly Jacobson lattice]
\label{defJac} ~
\begin{enumerate}
\item If $\fJ$ is an \id of the \trdi $\gT$  its \textsl{Jacobson radical} is defined by the \egt~$(\ref{eqRJJ})$ (we don't assume  $\gT\neq\Un$).
We will denote $\JT(\dar a)$ by $\JT(a)$.
\item A \trdi is called a \textsl{weakly Jacobson lattice} if each 
\idp is equal to its Jacobson radical, \cade if
\begin{equation} \label{eqTJac}
\Tt a,b\in\gT\;\;[\,(\forall x\in\gT\, (a\vu 
x=1 \;\Rightarrow\;b\vu 
x=1))\;\Rightarrow
\;a\leq b\,]
\end{equation}
\end{enumerate}
\end{definitions}

It is easily seen that $\JT(\fJ)$ is an \id and that
$1\in\JT(\fJ)\Leftrightarrow 1\in\fJ$.

\subsection{Quotient lattices}

A \textsl{quotient distributive lattice $\gT'$ of $\gT$} is obtained through a binary relation
$\cdot\preceq\cdot$ on $\gT$ satisfying the following \prts:
\begin{equation} \label{eqPreceq}
\left.
\begin{array}{rcl}
a\leq b&  \Longrightarrow  & a\preceq b   \\
a\preceq b,\,b\preceq c&  \Longrightarrow  & a\preceq c   \\
a\preceq b,\,a\preceq c&  \Longrightarrow  & a\preceq b\vi c   \\
b\preceq a,\,c\preceq a&  \Longrightarrow  & b\vu c\preceq a
\end{array}
\right\}
\end{equation}
The equality $a=_{\gT'}b$ (or $a\approx b$) is defined as $a\approx b \iff a\preceq b \hbox{ and } b\preceq a$.

\begin{proposition}[special kinds of quotients]
\label{propIdealFiltre} Let $\gT$ be a distributive lattice and
$(J,U)$ a pair of subsets of $\gT$.
We consider the quotient $\gT'$ of $\gT$ defined by the
relations $x=0$ for all $x\in J$ and $y=1$ for all $y\in U$. Then we have the following equivalence:\\ $a\leq_{\gT'}b$ \ssi 
there exist a finite subset $J_0$ of $J$ and  a finite subset $U_0$ of
$U$ such that
\begin{equation} \label{eqpropIdealFiltre}
a \vi \Vi U_0 \; \leq_\gT\; b \vu \Vu J_0
\end{equation}
We shall denote by $\gT/(J=0,U=1)$ this quotient lattice $\gT'.$
\end{proposition}

\subsubsection*{Ideals in a quotient lattice}
\addcontentsline{toc}{subsubsection}{Ideals in a quotient lattice}

The following fact follows from the equalities (\ref{eqIdeal}), (\ref{eqvuvi}),
(\ref{eqSupId}), and (\ref{eqInfId}).
\begin{fact}
\label{factIdDansQuo}
Let $\pi\colon \gT\to\gL$ be a quotient lattice. Then $\fa\mt\pi(\fa)$ gives a map from $\Idl(\gT)$ to~$\Idl(\gL)$ and $\fb\mt\pi^{-1}(\fb)$ gives a map from $\Idl(\gL)$ to~$\Idl(\gT)$.
\begin{itemize}\itemsep.2em
\item The map $\fb\mt\pi^{-1}(\fb)$ is
 a morphism for $\vu$ and~$\vi$ (but not necessarily
for $\so0$).
\item The map $\fa\mt\pi(\fa)$ gives a surjective lattice \homo from $\Idl(\gT)$ to $\Idl(\gL)$.
\item An ideal of $\gT$ has the form $\pi^{-1}(\fb)$ \ssi   it is
saturated for
the relation~$\approx$.
\item Similar results hold for filters.
\end{itemize}
\end{fact}

Note that in commutative algebra, the morphism of passing to the
quotient by one
ideal does not behave so well for ideals 
since we can very well have
$\fa+(\fb\cap\fc)\varsubsetneq
(\fa+\fb)\cap(\fa+\fc)$.

   The following lemma gives some additional information.
for the quotients by an ideal and by~a filter.

\begin{lemma}
\label{lemIQT}
Let $\fa$ be an ideal and $\ff$ a filter of $\gT$.
\begin{enumerate}
\item If $\gL=\gT/(\fa=0)$ then the canonical projection
$\pi\colon \gT\to\gL$
establishes a nondecreasing bijection between the ideals of $\gT$ containing
$\fa$ and
the ideals of $\gL$. The inverse bijection is provided by
$\fj\mapsto\pi^{-
1}(\fj)$. Moreover if $\fj$ is an ideal of $\gL$, we get~\hbox{$\pi^{-1}(\rJ_\gL(\fj)) = \JT(\pi^{-1 }(\fj))$}.
\item If $\gL=\gT/(\ff=1)$ then the canonical projection
$\pi\colon \gT\to\gL$
establishes a nondecreasing bijection between the ideals $\fJ$ of $\gT$
satisfying
\gui{$\Tt f \in \ff,\;\;\fJ:f=\fJ$} and the ideals of~$\gL$.
\end{enumerate}
\end{lemma}

\rem Note that $\gT\mapsto\JT(0)$ is not 
functorial. The second assertion of Item \textsl{1} of the previous lemma, which admits a direct constructive proof, is easily explained in \clama by the fact that, in the very particular  case of the quotient by an \id, the \idemas of $\gL$ containing $\fj$ correspond by $\pi^{-1}$ to the \idemas of~$\gT$ containing~$\pi^{-1}(\fj)$.
\eoe

\subsubsection*{Gluing quotient lattices}
\addcontentsline{toc}{subsubsection}{Gluing quotient lattices}

In commutative algebra, if $\fa$ and $\fb$ are two ideals of a ring
$\gA$
we have an \gui{exact sequence} of \Amos (with $j$ and $p$ being \homos
of rings)
\[0\to\gA/(\fa\cap\fb)\vers{j}(\gA/\fa) \times 
(\gA/\fb)\vers{p}\gA/(\fa+\fb)\to 0
\]
which can be read in everyday language as follows: the system of congruences
$x\equiv
a\;\mod\;\fa$, $x\equiv b\;\mod\;\fb$ admits a solution \ssi  $a\equiv
b\;\mod\;\fa+\fb$ and in this case the solution is unique modulo
$\fa\cap\fb$.
It is remarkable that this \gui{Chinese remainder theorem}
generalizes to  \textsl{any} system of congruences \ssi  the ring is
\textsl{arithmetic} \cite[Theorem \hbox{XII-1.6}]{ACMC}, i.e.\ if the lattice of ideals is distributive.
The \gui{contemporary} Chinese remainder theorem concerns the particular case of a family of two-by-two co\idemas, and it works without
hypothesis on the base ring.

Other epimorphisms in the category of \coris are the localizations. 
And there is a gluing principle analogous to the above Chinese  theorem
for localizations which is extremely fruitful (the local-global  principle).

\smallskip In the same way we can recover a distributive lattice
from a finite number of its quotients,
if the information they contain is \gui{sufficient}. We can
see this as a gluing procedure (going from the local to the
global), or as a version of the Chinese remainder theorem for distributive lattices. Let's see these
things more precisely.

\begin{definition}
\label{defRecolTD}
Let $\gT$ be a distributive lattice, $(\fa_i)_{i=1,\ldots, n}$ (resp.\
$(\ff_i)_{i=1,\ldots, n}$)
a finite family of ideals (resp.\ of filters) of $\gT$. We say that
the ideals
$\fa_i$ \textsl{cover $\gT$} if $\bigcap_i\fa_i=\so{0}$. Dually we say
that the filters $\ff_i$ \textsl{cover $\gT$} if
$\bigcap_i\ff_i=\so{1}$.
\end{definition}

For an ideal $\fb$ we write $x\equiv y\;\mod\;\fb$ as an
abbreviation
for $x\equiv y\;\mod$ \hbox{$(\fb=0)$}.
\begin{fact}
\label{factRecolTD}
Let $\gT$ be a distributive lattice, $(\fa_i)_{i=1,\ldots, n}$ a finite family of principal ideals ($\fa_i=\dar s_i$) of
$\gT$ and $\fa=\bigcap_i\fa_i$.
\begin{enumerate}\itemsep.2em
\item If $(x_i)$ is a family of elements of $\gT$ such that for
each $i,j$ we have $x_i\equiv x_j\;\mod\;\fa_i\vu\fa_j$, then there is a matching $x$, unique modulo
$\fa$: $x\equiv x_i\;\mod\;\fa_i\;(i=1,\ldots ,n)$.
\item Denote $\gT_i=\gT/(\fa_i=0)$,
$\gT_{ij}=\gT_{ji}=\gT/(\fa_i\vu\fa_j=0)$,
$\pi_i\colon \gT\to\gT_i$ and $\pi_{ij}:\gT_i\to\gT_{ij}$ the canonical projections.
If the $\fa_i$'s cover $\gT$, then
$(\gT,(\pi_i)_{i=1,\ldots, n})$ is
the projective limit of the diagram  
\[((\gT_i)_{1\leq i\leq n},(\gT_{ij})_{1\leq i<j\leq n};(\pi_{ij})_{1\leq i\neq j\leq n})
\]  
(see the figure below).
\item Now let $(\ff_i)_{i=1,\ldots, n}$ be a finite family of
principal filters,
denote $\gT_i=\gT/(\ff_i=1)$,
$\gT_{ij}=\gT_{ji}=\gT/(\ff_i\vi\ff_j=1)$,
$\pi_i\colon \gT\to\gT_i$ and $\pi_{ij}:\gT_i\to\gT_{ij}$ the canonical projections.
If the~$\ff_i$'s cover $\gT$, then $(\gT,(\pi_i)_{i=1,\ldots, n})$ is
the projective limit of the diagram 
\[((\gT_i)_{1\leq i\leq n},(\gT_{ij})_{1\leq i<j\leq n};(\pi_{ij})_{1\leq i\neq j\leq n}).\]
\end{enumerate}
\end{fact}
 
  {\hspace*{10em}{
\xymatrix @R=2em @C=7em{
           & \gT \ar[rd]^{\pi _{k}}\ar[d]^{\pi _{j}}\ar[ld]_{\pi _{i}}\\
  \gT _i\ar[d]_{\pi _{ij}}\ar@/-0.75cm/[dr]^{\pi _{ik}} &
      \gT _j\ar@/-.8cm/[dl]_{\pi _{ji}}\ar@/-.8cm/[dr]^{\pi _{jk}} &
         \gT _k\ar@/-0.75cm/[dl]_{\pi _{ki}}\ar[d]^{\pi _{kj}} &
\\
  \gT _{ij} &
     \gT _{ik} &
       \gT _{jk}
}
}}

\begin{proof}
\textsl{1}. Just prove it with $\fa=0$, which is Item \textsl{2}.

\smallskip \noindent \textsl{2}. Let $(\gH,(\psi_i)_{i\in I})$ be the limit of the diagram.
We have a unique morphism \[\varphi\colon \gT\to \gH\] such that $\varphi\circ \psi_i=\pi_i$ for each~$i$. 
But~$\varphi$ is injective by hypothesis: $\varphi(x)=\varphi(y)$ implies $\varphi(x)\equiv\varphi(y) \mod\, (\fa_i=0)$ for each $i$ and we have $\bigcap_i\fa_i=0$.
We have to show that it is surjective. \hbox{Let $x=(x_i)_{i\in I}$} be an element of $\gH$: \hbox{we have $x_i\in \gT_i$} for each~$i$ and~\hbox {$\pi_{ij}(x_i)=\pi_{ji}(x_j)$} \hbox{for $i\neq j$}. \hbox{If $x_i=\pi_i(y_i)$} we therefore have in $\gT$ the congruence
\[
y_i\equiv y_j \mod \dar (s_i\vu s_j).
\]
The injectivity of $\varphi$ means that $\Vi_{i=1}^ns_i=0$. We have $\pi_i(y_i)=\pi_i(y_i\vu s_i)$ so we can assume \hbox{that $y_i\geq s_i$}.
The equalities \hbox{$\pi_{ij}(x_i)=\pi_{ji}(x_j)$} are written
\[
y_i\equiv y_j \mod \dar (s_i\vu s_j).
\]
i.e.\ $y_i\vu s_j=y_j\vu s_i$. Let's set $y=\Vi_{i=1}^ny_i$.
\\
So, with for example $j=1$, we get
\[ 
y\vu s_1=y_1\vu\Vi\nolimits_{i=2}^n(y_i\vu s_1)=y_1\vi\Vi\nolimits_{i=2}^n(y_1\vu s_i)=y_1
\]
(because $a\vi(a\vu b)=a$). Thus $\pi_j(y)=x_j$ for each $j$.
And $\varphi$ is indeed surjective.
\end{proof}

There is also an actual gluing procedure.
To prove it we need the following lemma.

Recall that for $s\in\gT$ the quotient $\gT/(s=0)$ is isomorphic to the principal filter $\uar s$ which we see as a distributive lattice whose zero element  is $s$.
\begin{lemma}[in a distributive lattice, principal quotients are \gui{split}] \label{lemquoprinctrdi} ~
\\
Let $\pi\colon \gT \to \gT'$ be a distributive lattice morphism and $s\in \gT$.
\Propeq
\begin{enumerate}\itemsep.2em
\item $\pi$ is a morphism from the quotient of $\gT$ by the principal ideal $\fa=\dar s$.
\item There exists a morphism $\varphi\colon \gT'\to\,\uar s$ such that
$\pi\circ \varphi=\Id_{\gT'}$.
\end{enumerate}
In this case $\varphi$ is uniquely determined by $\pi$ and $s$.\\
Naturally, the analogous reversed statement is valid for a quotient by a principal filter.
\end{lemma}
%
\begin{proof}
\textsl{1} $\Rightarrow$ \textsl{2.} Let $y\in \gT'$. We have $y=\pi(x)$ for an $x\in \gT $.
\\
We want to define $\varphi\colon\gT'\to\uar s$ by the equality $\varphi(y)=x\vu s$. First of all it is well defined: if $\pi(x)=\pi(x')$, then $x\vu s=x'\vu s$ according to the previous reminder. Then it is immediate that~$\varphi$ is a distributive lattice morphism and that $\pi\circ \varphi=\Id_{\gT'}$.

\smallskip \noindent
\textsl{2} $\Rightarrow$ \textsl{1.} The equality $\pi\circ \varphi=\Id_{\gT'}$ implies
that $\pi$ is surjective and that $\varphi$ is an \iso from~$\gT'$ onto $\uar s$
with the restriction of $\pi$ as inverse \iso.
This shows that $\varphi$ is uniquely determined by $\pi$ and $s$.
We must show the \eqvc
\[
{\pi(x_1)=\pi(x_2)\; \Leftrightarrow \;x_1\vu s=x_2\vu s.}
\]
As $\varphi(0)=s$, we have $\pi(s)=0$, and $x_1\vu s=x_2\vu s$ implies $\pi(x_1)=\pi(x_2).$
\\
If $\pi(x_1)=\pi(x_2)$ then $\pi(x_1\vu s)=\pi(x_2\vu s)$, and since the restriction of $\pi$ to $\uar s$ is injective, it implies
$x_1\vu s=x_2\vu s$.
\end{proof}
\rem We have used in the title of the lemma the expression \gui{the principal quotients are split} by analogy with the split surjections
between \Amos,
considering the equality $\pi\circ \varphi=\Id_{\gT'}$, but the analogy is only partial. Here the \gui{section} $\varphi$ of $\pi$ is unique
(an important difference),
and it's \gui{not truly} a morphism from~$\gT'$ into~$\gT$ (another important difference).
\eoe

\begin{proposition}[gluing \trdis]
\label{propRecolTD} 
Let $I$ be a linearly ordered finite set and, in the category of distributive lattices, a diagram
\[
\big((\gT_i)_{i\in I},(\gT_{ij})_{i<j\in I},(\gT_{ijk})_{i<j<k\in I} ;
(\pi_{ij})_{i\neq j},(\pi_{ijk})_{i< j, j\neq k\neq i}\big)
\]
as in the figure below, as well as a family of elements
\[
(s_{ij})_{i\neq j\in I}\in \prod\nolimits_{i\neq j\in I}\gT_{i}
\]
satisfying the following conditions:
\begin{itemize}\itemsep.2em
\item the diagram is commutative ($\pi_{ijk}\circ \pi_{ij}=\pi_{ikj}\circ \pi_{ik}$ for all pairwise distinct $i$, $j$, $k$),
\item for $i\neq j$, $\pi_{ij}$ is a quotient morphism by the ideal $\dar s_{ij}$,
\item for pairwise distinct $i$, $j$, $k$, $\pi_{ij}(s_{ik})=\pi_{ji}(s_{jk})$ and $\pi_{ijk}$ is a quotient morphism by \hbox{the ideal $\dar\pi_{ij}(s_{ik})$}.
\end{itemize}

\smallskip {\hspace*{10em}
\xymatrix @R=2em @C=7em{
  \gT_i\ar[d]_{\pi _{ij}}\ar@/-0.75cm/[dr]^{\pi _{ik}} &
      \gT_j\ar@/-.8cm/[dl]_{\pi _{ji}}\ar@/-.8cm/[dr]^{\pi _{jk}} &
         \gT_k\ar@/-0.75cm/[dl]_{\pi _{ki}}\ar[d]^{\pi _{kj}} &
\\
  ~\gT_{ij}~ \ar[rd]_{\pi _{ijk}} &
     ~\gT_{ik}~ \ar[d]^{\pi _{ikj}} &
       ~\gT_{jk}~ \ar[ld]^{\pi _{jki}}
\\
    & ~\gT_{ijk}~
\\
}
}

\smallskip \noindent Then if $\big(\gT\,;\,(\pi_i)_{i\in I}\big)$ is the projective limit of the diagram, the~$\pi_i$'s form a covering of~$\gT$ by principal quotients, and the diagram is isomorphic to the one obtained in Fact~\ref{factRecolTD}.
More precisely, there exist $s_i$'s $\in\gT$ such that each~$\pi_i$ is a quotient morphism by the ideal $\dar s_i$ and $\pi_i(s_j)=s_{ij}$ for all $i\neq j$.

\noindent The analogous result is valid for quotients by principal filters.
\end{proposition}

\begin{proof}
We set $s_{ii}=0$, $\gT_{ii}=\gT_i$, $\varphi_{ii}=\pi_{ii}=\Id_{\gT_i}$. \\
Lemma \ref{lemquoprinctrdi} gives us \gui{sections} $\varphi_{ij}\colon \gT_{ij}\to \gT_i$ and~$\varphi_{ijk}\colon \gT_{ijk}\to \gT_ {ij}$.
\\
The required conditions imply that the
ideals~\hbox{$\dar\pi_{jk}(s_{ji})$} \hbox{and $\dar\pi_{kj}(s_{ki})$} are equal, \hbox{i.e.\ $\pi_{jk}(s_{ji})=\pi_{kj}(s_{ki})$}.
\\
For $i\in I$, we define $s_i\in \prod_k\gT_k$ by
${s_i=(s_{ji})_{j\in I}}$, so that $\pi_j(s_i)=s_{ji}$. 
The coordinates of $s_i$ are compatible (i.e.\ $s\in \gT$) because
$\pi_{jk}(s_{ji})=\pi_{kj}(s_{ki})$.
\\
We then define a morphism $\varphi_i\colon \gT_i\to \prod_k\gT_k$ by letting
\[
{\varphi_i(x)=y=(y_j)_{j\in I} \hbox{ where } y_j=\varphi_{ji}(x_j)=\varphi_{ji}\big(\pi_{ij}(x )\big).}
\]
Let us show that the coordinates of $y$ are compatible (i.e.\ $y\in \gT$). Indeed
\[{y_j=s_{ji}\vu y_j, \hbox{ so }
\pi_{jk}(y_j)= \pi_{jk}(s_{ji}\vu y_j)=\pi_{jk}(s_{ji})\vu\pi_{jk}(y_j),
}
\]
likewise $\pi_{kj}(y_k)=\pi_{kj}(s_{ki})\vu\pi_{kj}(y_k)$.
And since $\pi_{jki}$ is a quotient morphism through the ideal $\dar\pi_{jk}(s_{ji})=\dar\pi_{kj}(s_{ki})$, the \egt
\[
{\pi_{jk}(y_j)=\pi_{kj}(y_k)}
\]
can be tested by taking the images by $\pi_{jki}$. \\
However, since $\pi_{ji}(y_j)=\pi_{ji}\big(\varphi_{ji}(x_j\big)=x_j=\pi_{ij}(x)$, we obtain by using the commutativity of the diagram
\[
{\pi_{jki}\big(\pi_{jk}(y_j)\big)=\pi_{ijk}\big(\pi_{ji}(y_j)\big)=\pi_{ijk}\big(\ pi_{ij}(x)\big).}
\]
Similarly $\pi_{kji}\big(\pi_{kj}(y_k)\big)=\pi_{ikj}\big(\pi_{ik}(x)\big)$. And we conclude by using the commutativity of the diagram a second time.
\\
Once established that $\varphi_i$ is indeed a morphism $\gT_i\to \gT$, we easily see that $\pi_i\circ \varphi_i=\Id_{\gT_i}$, that the image of
$\varphi_i$ is the filter $\uar s_i$ of $\gT$ and that~$\varphi_i$ is a morphism of distributive lattices from $\gT_i$ onto the filter $\uar s_i$. So, by Lemma \ref{lemquoprinctrdi}, $\pi_i$ is a morphism of passing to the quotient by $\dar s_i$.
\end{proof}

\subsubsection*{Heitmann lattice}
\addcontentsline{toc}{subsubsection}{Heitmann lattice}

An interesting quotient of any \trdi, which is neither a quotient by an ideal nor a
quotient
by a filter, is the Heitmann lattice.

\begin{lemma}
\label{lemHeT}
On an arbitrary distributive lattice $\gT$ the relation $\JT(a)\subseteq\JT(b)$
is a  preorder relation $a\preceq b$ which defines a quotient
of $\gT$.
We also have:
\begin{equation} \label{eqJaJb}
a\preceq b \quad \Longleftrightarrow\quad a\in \JT(b)
\quad \Longleftrightarrow\quad \forall x\in\gT \; (a\vu x = 1 \Rightarrow
b\vu x=1)
\end{equation}
\end{lemma}
\begin{proof}
The equivalences  
\[a\in \JT(b) \;\Leftrightarrow\;
\JT(a)\subseteq\JT(b)\;\Leftrightarrow\;\forall x\in\gT \;(a\vu x = 1
\Rightarrow b\vu x=1)
\]  result from what what has been said on page
\pageref{eqRJb}
concerning the Jacobson radical of an ideal (see
equality~(\ref{eqRJb})).\\
Moreover, it is easy to verify the relations
(\ref{eqPreceq}) which are 
necessary  for defining a quotient lattice.
\end{proof}

\begin{definition}
\label{defHeT}
We call \textsl{Heitmann lattice of $\gT$} and we denote by $\He(\gT)$
the
quotient lattice of~$\gT$ obtained by replacing on $\gT$ the
order relation $\leq_\gT $ with the preorder relation
$\preceq_{\He(\gT)}$
defined as follows

\vspace{-1em}
\begin{equation} \label{eqdefHeT}
\begin{array}{rcl}\qquad
a\preceq_{\He(\gT)} b & \equidef & \JT(a)\subseteq\JT(b) \quad \hbox{(cf. definition \ref{defJac})}
   \end{array}
   \end{equation}
This quotient lattice can be identified with the set of
ideals $\JT(a)$,
with the canonical projection
\[ \gT\longrightarrow \He(\gT),\quad a\longmapsto \JT(a)\]

\end{definition}

Note that with the previous identification we have the equalities:
\begin{equation} \label{eqHeT2}
\JT(a\vi b)=\JT(a)\vi_{\He(\gT)}\JT(b),\quad
\JT(a\vu b)=\JT(a)\vu_{\He(\gT)}\JT(b)
\end{equation}

To say that the lattice $\gT$ is weakly Jacobson is equivalent to saying
that
$\gT=\He(\gT)$.

\smallskip The following lemma is a precision (and generalization) of the
first equality above. We will use this result later.
\begin{lemma}
\label{lemJacInter}
If $\fa$ and $\fb$ are two ideals of $\gT$, we have
$\JT(\fa\cap\fb)=\JT(\fa)\cap\JT(\fb)$.
\end{lemma}
\begin{proof}
It suffices to show that if $z\in\JT(\fa)\cap\JT(\fb)$ then
$z\in\JT(\fa\cap\fb)$. Let $t\in\gT$ such that $z\vu t=1$, we
seek
$c\in\fa\cap\fb$ such that $c\vu t=1$.
We have an $a\in\fa$ such that $a\vu t=1$ and~a~$b\in\fb$ such that
that $b\vu t=1$. So just take $c=a\vi b$.
\end{proof}

Note that the proof would not work for an intersection
of infinitely many ideals.

\begin{fact}
\label{factHeHe} Let $\gT$ be a distributive lattice, $\gT'=\gT/(\JT(0)=0)$,
$x\in\gT$ and $\fa$
an ideal.
\begin{enumerate}\itemsep.2em
\item $x=_{\He(\gT)}1\;\Longleftrightarrow\; x=1$.
\item $x=_{\He(\gT)}0\;\Longleftrightarrow\; x\in\JT(0)$.
\item $\He(\He(\gT))\;=\;\He(\gT')\;=\;\He(\gT)$.
\item If $\gL=\gT/(\fa=0)$,
$\He(\gL)$ identifies with $\He(\gT)/(\JT(\fa)=0)$.
\end{enumerate}
\end{fact}

\rem Note however that $\He(\bullet)$ does not define a functor.
\eoe

\begin{proof}
Items \textsl{1} and \textsl{2} are immediate. \\
Item \textsl{4} is left to \llec. It implies
$\He(\gT')\;=\;\He(\gT)$.\\
In Item \textsl{3} we see the lattices $\He(\He(\gT))$ and $\He(\gT')$ as
quotients of $\gT$. Let's show the equality $\He(\He(\gT))=\He(\gT)$,
i.e.\ for all $a,b\in\gT$, $a\preceq_{\He(\He(\gT))}b\Rightarrow
a\preceq_{\He(\gT)}b$. By definition the hypothesis means
\[\Tt x \in
\gT\;\big(\,a\vu x=_{\He(\gT)}1\;\Rightarrow \;b\vu x=_{\He(\gT)}1\,\big). 
\]
But
from Item~\textsl{1} this means $\Tt x \in \gT\;(a\vu
x=1\,\Rightarrow \,b\vu x=1)$, i.e.\ $a\preceq_{\He(\gT)}b$.
\end{proof}
\subsection{\agHs,  Brouwer algebras, \agBs}\label{subsecAgHagB}

\subsubsection*{\agHs}
\addcontentsline{toc}{subsubsection}{\agHs}

A distributive lattice $\gT$ is called an {\sl implicative lattice} (\citealt*{Cur63}) or a {\sl \agH} (\citealt*{Joh1986}) when there is a binary operation $\im $ satisfying for all $a,\,b,\,c$:
\begin{equation} \label{eqAgHey}
a\vi b \leq c \;\;\Longleftrightarrow \;\; a \leq (b\im c).
\end{equation}
This means that for all $b, c\in\gT$, the conductor ideal $(c:b)$ is principal,
its generator being denoted by $b\im c$.
So if it exists, the $\im$ operation is uniquely determined by the lattice structure.
We then define  $\neg x := x\im 0$.
The \agH  structure can be defined as purely equational by giving appropriate axioms. Precisely a lattice $\gT$ (not assumed to be distributive) endowed with a law $\im$ is a \agH if, and only if, the following axioms hold  (\citealt*{Joh1986}):
\[\begin{array}{rcl}
a\im a&= &1 \\
a\vi(a\im b)&= &a\vi b \\
b\vi(a\im b)&= & b \\
a\im(b\vi c)&= &(a\im b)\vi(a\im c)
\end{array}\]
Note also the following important facts:

\[\begin{array}{rcl}
(a\vu b)\im c &=& (a\im c)\vi(b\im c) \\
\neg(a\vu b)&= & \neg a\vi \lnot b \\
  a&\leq &\neg\neg a \\
\neg a\vu b&\leq & a\im b \\
a\leq b&\Leftrightarrow& a\im b =1
\end{array}\]

Every finite distributive lattice is a \agH, because every \itf is principal.

\smallskip An important special case of \agH is a \textsl{\agB}:
it is a distributive lattice in which every element $x$ has \textsl{a complement}, i.e.\ an element $y$ satisfying $y\vi x=0$ and $y\vu x=1$ ($y$ is denoted by $\lnot x$ and we have $a\im b=\lnot a\vu b$).

\smallskip A \textsl{\homo of \agHs} is a \homo $\varphi\colon \gT\to\gT'$
of distributive lattices which satisfies $\varphi(a\im b)=\varphi(a)\im\varphi(b)$ for all $a,b\in\gT$.

\smallskip The following fact is immediate.

\begin{fact}
\label{factQuoAgH}
Let $\pi\colon \gT\to\gT'$ be a \homo of distributive lattices. Suppose that $\gT$ and~$\gT'$ are two \agHs and denote $\varphi(a)\leq_{\gT'}\varphi(b)$ by $a\preceq b$. Then $\pi$ is a \homo of \agHs if, and only if, we have for all $a,a',b,b'\in\gT$:
\[
a\preceq a'\Rightarrow (a'\im b)\preceq(a\im b) \qquad
\mathrm{and}\qquad
b\preceq b'\Rightarrow (a\im b)\preceq(a\im b')
\]
\end{fact}

We also have:

\begin{fact}
\label{factQuoAgH2}
If $\gT$ is a \agH then any quotient $\gT/(y=0)$ (i.e.\ any quotient by a \idp) is also a \agH.
\end{fact}
\begin{proof}
Let $\pi\colon \gT\to\gT'=\gT/(y=0)$ be the canonical projection. We have
\[\pi(x)\vi\pi(a)\,\leq_{\gT'}\, \pi(b)\;\Leftrightarrow\; \pi(x \vi
a)\,\leq_{\gT'}\, \pi(b)\;\Leftrightarrow\; x\vi a\,\leq\, b\vu
y\;\Leftrightarrow\; x\,\leq\, a\im(b\vu y).
\] 
But $y\,\leq\, b\vu
y\,\leq\,
a\im(b\vu y)$, so 
\[\pi(x)\vi\pi(a)\,\leq_{\gT'}\,
\pi(b)\;\Leftrightarrow\;
x\,\leq\, (a\im(b\vu y))\vu y ,\hbox{ i.e.\  }\pi(x)\leq_{\gT'}\pi(a\im(b\vu
y)),
\] 
which shows that $\pi(a\im(b\vu y))$ holds for $\pi(a)\im\pi(b)$ in
$\gT'$.
\end{proof}

\medskip \rem
The notion of \agH is reminiscent of the notion of coherent ring in commutative algebra. Indeed a coherent ring can be characterized as follows: the intersection of two \itfs is a \itf and the conductor of a \itf in a \itf is a \itf. If we \gui{reread} this for a distributive lattice remembering that every \itf is principal we get a \agH.
\eoe

\medskip \rem 
Any distributive lattice $\gT$ generates a \agH in a natural way. In other words, we can formally add a generator for any ideal $(b:c)$. But if we start from a distributive lattice which happens to be a \agH, the \agH that it generates is strictly greater. Let us take for example the lattice $\Trois$ (finite hence a \agH), which is the free  distributive lattice with one generator.
Thus the \agH that it generates is the free \agH with one generator. But this \agH is infinite \cite[section 4.11]{Joh1986}. By contrast, the Boolean lattice generated by $\gT$ (\citealt*{CC00}), \cite[Theorem XI-1.8]{CACM} remains equal to $\gT$ when the latter is Boolean.\eoe

\subsubsection*{Lattices with negation}
\addcontentsline{toc}{subsubsection}{Lattices with negation}

A distributive lattice \textsl{has a negation} if for all $x$ the conductor ideal
$(0:x)$ is principal, generated by an element denoted by $\lnot x$.
The following rules are immediate.

\[\begin{array}{rclcrcl}
x\vi y =0&\Leftrightarrow& y\leq \lnot x&,&a\leq b& \Rightarrow & \lnot b \leq \lnot a \\
a& \leq & \lnot\lnot a &, & \lnot a &=& \lnot\lnot\lnot a \\
\lnot(a\vu b)& = & \lnot a\vi\lnot b & , & \lnot a\vu\lnot b
&\leq&
\lnot(a\vi b)\\
\lnot(x\vu\lnot x)&=&0&,&\lnot\lnot(x\vu\lnot x)&=&1
\end{array}\]

If for all $a$, $\lnot\lnot a=a$, the lattice is a \agB because then
$x\vu\lnot x=1$.

\begin{fact}
\label{factSpecMin}
If $\gT$ has a negation, let $\Fmin(\gT)=\ff$ be the filter
generated by all \elts $x\vu\lnot x$. Then
$\He(\gT\cir)=(\gT/(\Fmin(\gT)=1))\cir$, and this
lattice is a \agB.
\end{fact}
\begin{proof}
It is clear that $\lnot x$ is a complement of $x$ in $\gT/(\ff=1)$; this
lattice is therefore a \agB. In the presence of negation, the relation
$a\leq_{\He(\gT\cir)}b$ is equivalent to $\lnot a\leq \lnot b$ and this is easily shown to be equivalent to $b\leq a\;\mod\;(\ff= 1)$.
\end{proof}
\begin{fact}
\label{factWJavecneg}
If $\gT$ is a lattice with negation, the lattice $\gT\cir$ is
weakly
Jacobson if and only if $\gT$ is~a \agB.
\end{fact}
\begin{proof}
In the presence of negation, the \eqvcs (\ref{eqJaJb}) and
(\ref{eqdefHeT})
give that $\,a\leq_{\He(\gT\cir)}b\,$ is \eqv to
$\lnot b\leq
\lnot a$.
The lattice $\gT\cir$ is therefore weakly Jacobson if, and only if, $\lnot b\leq \lnot a$ implies $a\leq b$. 
In particular we get $b=\lnot\lnot b$
by taking $a=\lnot\lnot b$.\end{proof}

\subsubsection*{Brouwer \algs}
\addcontentsline{toc}{subsubsection}{Brouwer algebras}

A \trdi whose opposite lattice is a \agH is called a
   \textsl{Brouwer \alg}. It is a \trdi in which all the difference filters
   $(c\setminus b)$ are principal (see \pref{eqDiff}). We then denote by $c-b$ the generator of
   $(c\setminus b)$.

Passing to the opposite lattice the following fact says the same thing
as Fact~\ref{factSpecMin}.

\begin{fact}
\label{factSpecMax}
We say that \emph{the lattice $\gT$ has a Brouwer complement} when for all $x$ the filter $(1\setminus x)$ is principal. It is then generated by a single element denoted by $1-x$. 
In this case, let $\Imax(\gT)$ be the ideal generated by all \elts $x\vi(1-x)$. Then $\He(\gT)=\gT/(\Imax(\gT)=0)$ and this lattice is a \agB.
\end{fact}

We leave it to \llec to translate Fact
\ref{factWJavecneg} when reversing the order relation.

\subsection{Noetherian  \trdis}

In classical mathematics, for a \trdi $\gT$, \propeq
\begin{enumerate}\itemsep.2em
\item  Every ideal of $\gT$ is principal.
\item  Any nondecreasing sequence of \elts of $\gT$ is stationary.
\item  Any nondecreasing sequence of ideals of $\gT$ is stationary.
\end{enumerate}

Such a lattice is called \textsl{Noetherian} (by analogy with commutative \alg, one could also call it a \textsl{principal lattice}). In classical mathematics it is clearly a \agH.

Every sublattice and every quotient lattice of a Noetherian lattice is Noetherian.

In constructive mathematics the notion is more delicate. No nontrivial lattice satisfies Item~\textsl{2} (which is a priori the weakest formulation). One could define a Noetherian \trdi as a lattice satisfying an \gui{ACC constructive condition}: any nondecreasing sequence admits two equal consecutive terms. This condition is equivalent to Item \textsl{2} in classical mathematics. But there are a priori several interesting variants.

In practice, one is generally interested in the fact that some well-defined ideals are principal, as in the case of \agHs. Now the fact that a lattice is a \agH does not result \cot from the ACC constructive condition (in the same way, in commutative \alg, coherence, which is often more important than Noetherianity, does not result from any known  constructive variant of Noetherianity). See on this subject Proposition~\ref{propZarHeyt}.

\medskip \rem
Let us show in classical mathematics that if $\gT$ and $\gT\cir$ are Noetherian then~$\gT$ is finite. The \idemas are $\dar x$ where $x$ is an immediate predecessor of $1$. And the maximal spectrum is finite, because if $(\fm_n)=(\dar x_n)$ is an infinite sequence of \idemas, the sequence $(\Vi_{i\leq n}x_i)$ is strictly decreasing. We can then apply the result to each quotient lattice by a \idema. We end with König's lemma. Making this proof constructive with a sufficiently strong constructive definition of Noetherianity is an interesting challenge.
\eoe

\section{Spectral spaces}
\label{secESSP}

\subsection{General facts}

\subsubsection*{In \clama}
\addcontentsline{toc}{subsubsection}{In \clama}

A {\sl \idep} $\fp$ of a lattice $\gT\neq \Un$ is an ideal whose complement $\ff$ is a filter (which is then a {\sl prime filter}).
We then have $\gT/(\fp=0,\ff=1)\simeq\Deux$.
It is the same to give a \idep of $\gT$ or a morphism of \trdis $\gT\rightarrow \Deux$.

In this section, we will denote by $\theta_\fp:\gT\to\Deux$ the \homo associated with the \idep~$\fp$.

It is easy to verify that if $S$ is a generating subset of the \trdi $\gT$, a \idep~$\fp$ of $\gT$ is completely characterized by its trace on $S$ (\citealt*{CC00}).

A \textsl{\idema} (resp.\ \textsl{minimal prime}) is a maximal ideal among strict ideals (resp. minimal among \ideps). It amounts to the same thing to say that $\fm$ is maximal or that \hbox{$\gT/(\fm=0)\simeq\Deux$}, the \idemas are therefore prime. It is the same to say that $\fp$ is a \idemi or that its complement is a maximal filter.

In classical mathematics every strict ideal is contained in a \idema and (by reversing the order) any strict filter is contained in a maximal filter.

\smallskip The \textsl{spectrum of a \trdi $\gT$} is the set $\Spec\,\gT$ of its \ideps, endowed with the following topology: a basis of open sets is given by the
\[
\DT(a)\eqdefi\sotq{\fp\in\Spec
\,\gT}{a\notin\fp},\quad a\in \gT.
\]
We check that
\begin{equation} \label{eqDa}
\left.\begin{array}{rclcrcl}
   \DT(a\vi b) & = & \DT(a)\cap \DT(b) ,&\quad & \DT(0) & = &
\emptyset,\\
   \DT(a\vu b) & = & \DT(a)\cup \DT(b) ,&& \DT(1) & = &
\Spec\,\gT.
   \end{array}
\right\}
\end{equation}

The complement of $\DT(a)$ is a closed subset denoted by $\VT(a)$.

We extend the notation $\VT(a)$ as follows: if $I\subseteq\gT$, we set $\VT(I)\eqdefi\bigcap_{x\in I}\VT(x)$. If $\cI_\gT(I)=\fII$, we have $\VT(I)=\VT(\fII)$. It is sometimes said that $\VT(I)$ is \textsl{the variety defined by~$I$}.

\medskip\noindent
{\bf Definition.} 
A topological space homeomorphic to a space $\Spec(\gT)$
is called a \textsl{spectral space}. The spectral spaces come from  \citealt*{Sto37}.

\medskip Johnstone calls them \textsl{coherent spaces} (\citealt*{Joh1986}). They were baptized \gui{spectral spaces} by \citealt*{Hoc1969}.

With classical logic and the axiom of choice, the space $\Spec \,\gT$
has \gui{enough points}: we can recover the lattice $\gT$
from its spectrum. Here's how.
 First of all we have the following.

\medskip\noindent
{\bf Krull's \tho$\!\!\etl$} (in \clama).\\ 
{\sl Suppose that $\fJ$ is an \id, $\fF$ a filter and
$\fJ\cap\fF=\emptyset$.
Then there exists a \idep $\fP$ such \hbox{that $\fJ\subseteq\fP$} and
$\fP\cap\fF=\emptyset$.
  }

\medskip
We deduce the following facts.
\begin{itemize}\itemsep.2em
\item The map $a\in\gT\,\mapsto\,\DT(a)\in\cP(\Spec\,\gT)$ is injective: it identifies $\gT$ with a lattice of sets
(\textsl{Birkhoff representation theorem}).
\item If $\varphi\colon  \gT\to\gT'$ is an injective \homo the dual map
$\varphi^\star:\Spec\,\gT'\to\Spec\,\gT$  is onto.
\item Any ideal of $\gT$ is the intersection of  \ideps 
containing it.
\item Mapping $\fII\mapsto \VT(\fII)$, from ideals of $\gT$ to
closed sets of $\Spec\,\gT$ is an \iso of ordered sets (for inclusion and
reverse inclusion).
\end{itemize}

One also shows that the \oqcs of $\Spec \,\gT$ are exactly the subsets $\DT(a)$. According to the equalities (\ref{eqDa}), the \oqcs of $\Spec \,\gT$ form a \trdi of subsets of~$\Spec \,\gT$, isomorphic to $\gT$.

In a spectral space $X$ we can consider the \trdi $\OQC(X)$ formed by its \oqcs. Since for any \trdi $\gT$, $\OQC(\Spec(\gT))$ is canonically isomorphic to $\gT$, for any spectral space $X$, $\Spec(\OQC(X))$ is canonically homeomorphic to~$X$.

\smallskip Any \trdi \homo $\varphi\colon \gT\rightarrow \gT'$  provides by duality a continuous map $\varphi^\star:\Spec\,\gT'\rightarrow \Spec \,\gT$, which is called a \textsl{spectral map}. For a continuous map between spectral spaces to be spectral, it is necessary and sufficient that the inverse image of every \oqc be a \oqc.

The seminal paper \citealt*{Sto37} essentially demonstrates that the spectral category thus defined is antiequivalent to the category of  \trdis \cite[{II-3.3}, coherent locales]{Joh1986}.
More precisely, this statement which here seems tautological becomes nontrivial when we give a definition of spectral spaces in purely topological terms, as in the following remark.
For more details on this antiequivalence, one can refer to Krull's theorem, to \citealt*[\hbox{section V-8}]{BW74},
to \citealt*{CL2001-2018} and to the survey paper \citealt*{Lom2020}.

\medskip 
\rem
A purely topological definition of spectral spaces
is the following~(\citealt*{Sto37}).
\begin{itemize}\itemsep.2em
\item The space is Kolmogorov (i.e.\ of type $\mathrm{T}_0$):
given two points there exists a neighbourhood of one of the two which does not contain the other.
\item The space is \qc.
\item The intersection of two \oqcs is a \oqc.
\item Any open subset is a union of \oqcs.
\item For all closed subsets $F$ and for all sets $S$ of \oqcs such
that
\[\textstyle F\cap
\bigcap_{U\in S'} U\neq \emptyset\,\hbox{ for any finite subset }\,S'
\,\hbox{ of }\,S
\]
we also have
$F\cap \bigcap_{U\in S} U\neq \emptyset$.
\end{itemize}
In the presence of the first four properties the last one
can be
rephrased as follows (\citealt*{Hoc1969}).
\begin{itemize}\itemsep.2em
\item Any irreducible closed set\footnote{A closed set which is not the union of two strictly smaller closed sets} admits a generic point.
\end{itemize}
\vspace{-2em}\eoe


\subsubsection*{Generic points, order relation}
\addcontentsline{toc}{subsubsection}{Generic points, order relation}

We say that a point $x \in X$ of a spectral space is the \textsl{generic point of the closed set $F$} \hbox{if $F=\ovs{x }$}. This point (when it exists) is necessarily unique because the spectral spaces are Kolmogorov spaces. The closed sets $\ovs{x }$ are exactly all the irreducible closed sets of $X$. The order relation $y\in\ovs{x}$ will be denoted by $x\leq_X y$.

When $X=\Spec\,\gT$ the relation $\fp\leq_X \fq$ is simply the usual inclusion relation between \ideps of the \trdi $\gT$.
The closed points of $\Spec\,\gT$ are the \idemas of $\gT$.

\medskip
We call \textsl{Stone space}\footnote{The terminology does not seem to be clearly established. \cite{BW74} call Stone space a topological space which is very near to a spectral space. Their goal is a category of topological spaces antiequivalent to that of \gui{unbounded} \trdis, i.e.\ without $0$ and $1$.} a spectral space whose lattice of \oqcs is a \agB. It is well known that Stone spaces can be characterized as totally discontinuous compact spaces.
\subsubsection*{In \coma}
\addcontentsline{toc}{subsubsection}{In \coma}

Constructively, $\gT$ is a \gui{point-free} version of $\Spec\,\gT$. In other words, failing to have access to the points of $\Spec\,\gT$, we can content ourselves with its \oqcs, which are directly visible (without resorting to the axiom of choice or the law of excluded middle). The version without points is easier to understand. On the contrary, the points of $\Spec\,\gT$ are not in general  accessible without resorting to non-constructive principles.

In constructive mathematics there are a priori several possibilities to define the spectrum of~a \trdi (all equivalent in classical mathematics).
The most reasonable seems to define $\Spec\,\gT$ as the set of prime filters of $\gT$, i.e.\ the filters for which we have 
\[
x\vi y\in\fF\quad \Longrightarrow \ \quad x\in\fF \;\;\mathrm{or}\;\;
y\in\fF 
\]
with an explicit \gui{or}. But such spaces $\Spec\,\gT$ do not always have enough points{\footnote{We can for example define an explicit countable infinite \trdi which does not have recursive \ideps. For such a \trdi, there cannot be a \prco that $\Spec\,\gT$ is non-empty.}} and one cannot
assert \cot that the two categories are antiequivalent, at least if one defines morphisms between spectral spaces as maps, since maps require points.

A satisfactory alternative solution (but a little confusing at first glance) is to consider $\Spec\,\gT$ as a \gui{topological space without points}, i.e.\ a topological space defined only through its basis of open sets $ \DT(a)$ (with all $a\in\gT$). The morphisms are then defined in a purely formal way as given by the morphisms of the corresponding lattices, reversing the direction of the arrows. From this point of view the antiequivalence of the spectral category and the category of \trdis becomes a pure definitional tautology.

In any case, although the spectral category remains useful for intuition, all the work is done in the category of \trdis. The advantage is naturally that one obtains constructive theorems.

In this article the spectra will be studied only from the point of view of classical mathematics, as an important source of inspiration for good notions concerning \trdis.

\subsubsection*{Noetherian  spectral spaces}
\addcontentsline{toc}{subsubsection}{Noetherian  spectral spaces}

A topological space $X$ is said to be \textsl{Noetherian} if any nondecreasing sequence of open set is stationary. It is the same to say
that every open set is \qc. For a spectral space, it is equivalent to say that the lattice $\OQC(X)$ is Noetherian. In a Noetherian spectral space, every open set is a $\DT(a)$ and every closed set is a $\VT(b)$.

\subsubsection*{Two alternative topologies on 
$\Spec\,\gT$}
\addcontentsline{toc}{subsubsection}{Two alternative topologies on 
$\Spec\,\gT$}

In classical mathematics we have a canonical bijection between the sets underlying the spaces $\Spec\,\gT$ and $\Spec\,\gT\cir$: to a \idep of $\gT$ we associate the complementary prime filter, which is a \idep of $\gT\cir$. This makes it possible to identify these two sets, even if sometimes the effect is not very happy.
Once the underlying sets are identified, the topology is not the same. The basic open sets of $\Spec\,\gT\cir$ are $\DTo(a)=\VT(a)$.
Modulo this identification, for $X=\Spec\,\gT$ and $X'=\Spec\,\gT\cir$, the order relation $\leq_{X'}$ is the relation opposite~to~$\leq_X$ (the order is reversed), but what happens for the topology is more complicated.

\smallskip
We must also consider the \textsl{constructible topology} (or patch topology) whose basic open sets are the $\DT(a) \cap \VT(b)$.
This gives a compact space naturally homeomorphic to $\Spec\,\gT^{\rm bool}$ where $\gT^{\rm bool}$ is the Boolean lattice generated by $\gT$.
In classical mathematics we obtain $\gT^{\rm bool}$ as the sub-\agB of the set of subsets of $\Spec\,\gT$ generated by the $\DT(a)$. This lattice can also be described \cot as follows (\citealt*{CC00}). We consider a disjoint copy of $\gT$, denoted by $\dot{\gT}$. Then $\gT^{\rm bool}$ is a \trdi defined by generators and relations.
The generators are the elements of the set $T_1=\gT\cup\dot{\gT}$ and the relations are obtained as follows: if $A,F,B,E$ are four finite subsets of $\gT $ we have
\[
   \Vi A \vi \Vi E\leq_\gT\Vu B\vu \Vu F
\quad \Longrightarrow \quad
\Vi A \vi \Vi \dot{F} \leq_{T_1} \Vu B \vu \Vu \dot{E}
\]
It is shown that $\gT$ and $\dot{\gT}$ inject naturally into $\gT^{\rm bool}$ and that the above implication is in fact an equivalence.
We obtain by duality two one-to-one  spectral maps $\Spec\,\gT^{\rm bool}\to\Spec\,\gT$ and $\Spec\,\gT^{\rm bool}\to\Spec\, \gT\cir$.

\subsubsection*{Finite spectral spaces}
\addcontentsline{toc}{subsubsection}{Finite spectral spaces}
 
In \clama, the dual spaces of \textsl{finite} \trdis  are the finite spectral spaces, which are nothing other than the finite ordered sets (because it suffices to know the closure of the points to know the topology), where the  $\dar a$'s give a basis of open subsets.
  The open subsets are all \qc, these are the initial subsets, and the closed subsets are the final subsets. Finally, a map between finite spectral spaces is spectral if, and only if, it is nondecreasing (for the associated order relations).

The notion of spectral space thus appears as a relevant generalization to the infinite case of the notion of finite ordered set. See \citet*[Theorem XI-5.6, duality between finite ordered sets and finite \trdis]{ACMC}.

In the finite case, if we identify the sets underlying $\Spec\,\gT$ and $\Spec\,\gT\cir$, the two spectra are almost the same: they are the same ordered set up to reversing the order relation. Furthermore open subsets and  closed subsets are simply interchanged.

\subsection{Quotient lattices versus spectral subspaces}\label{secSESP}

\subsubsection*{Characterization of spectral subspaces}
\addcontentsline{toc}{subsubsection}{Characterization of spectral subspaces}

By using the antiequivalence of categories, one could directly define the notion of \textsl{spectral subspace} as the notion dual to the notion of quotient lattice. Theorem \ref{propSESP} explains this topic in detail.

We start with an easy lemma, which characterizes the points of $\Spec\,\gT$ which \gui{are elements of $\Spec\,\gT'$} when $\gT'$ is a quotient of 
$\gT$.
\begin{lemmac}
\label{lemSESP}
Let $\gT'$ be a quotient lattice of $\gT$ and $\pi\colon \gT\to\gT'$ the
canonical projection. Denote $X=\Spec\,\gT'$, $Y=\Spec\,\gT$ and
$\pi^\star\colon X\to Y$ the dual injection of $\pi$. Remember that for a
\idep $\fp$ of $\gT$ we denote by $\theta_\fp:\gT\to\Deux$ the \homo
with kernel $\fp$.
\Propeq
\begin{itemize}\itemsep.1em
\item $\fp\in\pi^\star(\Spec\,\gT').$
\item $\theta_\fp$ is factorized by $\gT'.$
\item $\Tt a,b\in\gT\;((a\preceq b,\,b\in\fp)\Rightarrow a\in\fp).$
\end{itemize}
This can be rephrased as follows when the quotient lattice $\gT'$
is defined
by a system $R$ of relations $x_i=y_i$. \Propeq
\begin{itemize}\itemsep.1em
\item $\fp\in\pi^\star(\Spec\,\gT').$
\item $\theta_\fp$ \gui{gives a model of $R$}, i.e.\ $\Tt i\;\;
\theta_\fp(x_i)=\theta_\fp(y_i).$
\item $\Tt i\;\; (x_i\in\fp\;\Leftrightarrow\; y_i\in\fp).$
\end{itemize}
\end{lemmac}

In the following theorem we identify $\Spec\,\gT'$ with a subset of
$\Spec\,\gT$ using the injection  $\pi^\star$.
Similar results stated in a slightly different language can be found in \citealt*[Section~3]{Esc2001}.\footnote{Escard{\'o} writes his article in the language of locales.  If $Y=\Spec\,\gT$, he denotes by $\Patch\, Y$ the Stone space $\Spec\,\gT^{\rm bool}$.}

\begin{theoremc}[\Dfn and \carn of spectral subspaces] \label{propSESP}~
\begin{enumerate}\itemsep.1em
\item With the notation of lemma \ref{lemSESP}, $X$ is a
\textsl{topological subspace} of $Y$. \\
Also $\OQC(X)=\sotq{U\cap X}{U\in\OQC(Y)}$.
We say that
\emph{$X$ is a spectral subspace of~$Y$.}
\item For a subset $X$ of a spectral space $Y$ to be a spectral subspace
it is necessary and sufficient that the following conditions are verified: \\
-- The topology induced by $Y$ makes $X$ a spectral space, and
\\
-- $\OQC(X)=\sotq{U\cap X}{U\in\OQC(Y)}$.
\item A subset $X$ of a spectral space $Y$ is a spectral subspace
 if, and only if, it is closed for the patch topology.
\item If $Z$ is an arbitrary subset of a spectral space $Y=\Spec\,\gT$, its closure for the patch topology is equal to $X=\Spec\,\gT'$ where $\gT'$ is the quotient lattice of~$\gT$ defined by the preorder relation 
$\preceq$:
\begin{equation} \label{eqSSES}
a\preceq b\quad \Longleftrightarrow\quad (\DT(a)\cap Z)\subseteq
(\DT(b)\cap Z).
\end{equation}
Furthermore, $X$ is the smallest spectral subspace of $Y$ containing
$Z$.
\end{enumerate}
\end{theoremc}
\begin{proof}
Item \textsl{1} is easy, and defines the notion of \ssps. Item \textsl{2} follows. Item~\textsl{3} results from Items \textsl{2} and \textsl{4}. 
\\
Let's show Item \textsl{4}. 
Note first of all that the relation (\ref{eqSSES}) defines a quotient lattice $\gT'$ because the relations~(\ref{eqPreceq}) are trivially verified if we take into account the relations (\ref{eqDa}).\\
Let's show that $X=\Spec\,\gT'$ is the smallest \ssps of $Y$ containing $Z$.\\
First $Z\subseteq X$: let $\fp\in Z$, we want to show that if
$b\in\fp$ and $a\preceq b$ then $a\in\fp$. If $\DT(a)\cap Z \subseteq
\DT(b)$ and $b\in\fp$ then $\fp\notin \DT(b)$ therefore $\fp\notin \DT(a)\cap Z$ therefore $\fp\notin\DT(a)$, i.e.\ $a\in\fp$.

\noindent Also $X$ is minimal. Indeed  if
$X_1=\Spec\,\gT_1$ is a \ssps of $Y$ containing~$Z$, we have $a\leq_{\gT_1}b\;\Leftrightarrow\;(\DT(a)\cap X_1 )\subseteq (\DT(b)\cap X_1)$ 
which implies $(\DT(a)\cap Z)\subseteq (\DT(b)\cap Z)$ and therefore $a\leq_{\gT'}b$, hence $X\subseteq X_1$. \\
   It remains to show that $X$ is the closure of $Z$ for the
patch topology. Let $\wi Z$ denote this closure.
We therefore want to prove for all $\fp\in\Spec\,\gT$ the equivalence of the
two following properties:\\
(1) $\fp\in\wi Z$, \cad: $\Tt a,b\in\gT ,\;(\fp\in \DT(a)\cap
\VT(b)\,\Rightarrow\, \DT(a)\cap \VT(b)\cap Z\neq \emptyset) $,\\
(2) $\fp\in\Spec\,\gT'$.\\
But (2) is successively equivalent to

\vspace{-1.3em}
\[\begin{array}{lcr}
\Tt a,b\in\gT\;((a\preceq b,\,b\in\fp)\Rightarrow a\in\fp) &\quad
& (3)
\\[1mm]
\Tt a,b\in\gT\;\;(a\preceq b,\,b\in\fp,\, a\notin\fp) \;\mathrm{are\;
incompatible} &\quad & (4) 
\\[1mm]
\Tt a,b\in\gT\;\;\DT(a)\cap Z\subseteq \DT(b)
\;\,\mathrm{and}\,\;\fp\in
\DT(a)\cap \VT(b) \,\;\mathrm{are\; incompatible} &\quad & (5)
\\[1mm]
\Tt a,b\in\gT\;\;\DT(a)\cap \VT(b)\cap Z=\emptyset
\;\,\mathrm{and}\,\;\fp\in
\DT(a)\cap \VT(b) \,\;\mathrm{are\; incompatible} &\quad & (6)
\end{array}\]

\vspace{-.4em}
\noindent and (6) is clearly equivalent to (1).
\end{proof}

\begin{corollaryc}
\label{corpropSESP}
Any finite union and all intersections of \sspss 
of $X=\Spec\,\gT$ are \sspss.
\begin{itemize}
\item If each $X_i=\Spec\,\gT_i$ with a surjective morphism $\pi_i\colon \gT\to\gT_i$ then $\bigcap_iX_i$ corresponds to the quotient  by all the relations $\pi_i(x)=\pi_i(y)$.
\item  Let $(\pi_i)_{i \in I}$ be a finite family, then $\bigcup_iX_i$ corresponds to the quotient by the relation $\&_i\big(\pi_i(x)=\pi_i(y)\big)$.
%
\end{itemize}
\end{corollaryc}

\begin{propositionc}[basic open and basic closed subsets]
\label{propositionOFBSES} Let $\gT$ be a \trdi and $X=\Spec\,\gT$.
\begin{enumerate}
\item $\DT(a)$ is a \ssps of $X$ canonically 
\homeoc to $
\Spec(\gT/(a=1))$.
\item  $\VT(b)$is a \ssps of $X$ canonically 
\homeoc to  $
\Spec(\gT/(b=0))$.
\end{enumerate}
\end{propositionc}
\begin{proof}
Let $x\preceq y$ be the preorder  corresponding to the \ssps
$\DT(a)$. We have
\[x\preceq y\,\Leftrightarrow\,\DT(x) \cap \DT(a)\subseteq \DT(y) \cap
\DT(a)\,\Leftrightarrow\,\DT(x\vi a)\subseteq \DT(y\vi 
a)\,\Leftrightarrow\,x\vi
a\leq y\vi a.
\]
This is indeed the preorder relation corresponding to the quotient
$\Spec(\gT/(a=1)).$\\
Let $x\preceq' y$ be the preorder  corresponding to the \ssps 
 $\VT(b)$. We have
\[\begin{array}{rcccl}
x\preceq' y& \Longleftrightarrow  &\DT(x) \cap \VT(b)\subseteq \DT(y) 
\cap
\VT(b)   & \Longleftrightarrow   &  \DT(x) \cup \DT(b)\subseteq 
\DT(y) \cup
\DT(b) \\[1mm]
& \Longleftrightarrow  &  \DT(x\vu b)\subseteq \DT(y\vu b) & 
\Longleftrightarrow
&   x\vu b\leq y\vu b.
\end{array}
\]
This is indeed the preorder relation corresponding to the quotient
$\Spec(\gT/(b=0)).$
\end{proof}
\subsubsection*{Closed subsets of $\Spec\,\gT$}
\addcontentsline{toc}{subsubsection}{Closed subsets of $\Spec\,\gT$}

In this paragraph $\gT$ is a \trdi and $X=\Spec\,\gT$.
If $Z\subseteq X$ we denote by $\ov{Z}$ the closure of $Z$ for the usual topology of $X$.

\begin{propositionc}[closed subsets of $\Spec\,\gT$]
\label{propositionFSES} ~
\begin{enumerate}
\item An arbitrary closed subset of $\Spec\,\gT$ is
$\VT(\fJ)=\bigcap_{x\in \fJ}\VT(x)$ where $\fJ$ is an arbitrary ideal of $\gT$.
This \ssps corresponds to the quotient lattice 
$\gT/(\fJ=0)$.
\item
The intersection of a family of closed subsets corresponds to the upper bound of the corresponding family of \ids. The union of two closed subsets corresponds to 
the intersection of the two corresponding \ids.
\item\label{enumtra}
The lattice $\gT/((a:b)=0)$ is the quotient lattice corresponding to 
$\ov{\VT(a)\cap
\DT(b)}$.
\item
So $\gT$ is a \agH \ssi  $X$ satisfies the following \prt: for any
\oqcs $U_1$ and $U_2$, the closure of $\,U_1\setminus U_2$ is the complement of a \oqc.
\item  The quotient lattice  
$\gT/((0:x)=0)$  corresponds to the closure of $\DT(x)$.
\item The boundary of $\DT(x)$ corresponds to the quotient lattice
$\gT\ul x=\gT/(\rK_\gT^x=0)$, where
\begin{equation} \label{eqboundarysup}
\rK_\gT^x\,=\,\dar x \,\vu\, (0:x).
\end{equation}
We call the lattice $\gT\ul x$ the \emph{upper (Krull-)boundary of $x$
in $\gT$}. We say also that $\rK_\gT^x$ is the \emph{Krull boundary \id of $x$
in $\gT$.}\\
When $\gT$ is a \agH, $\rK_\gT^x\,=\,\dar (x \,\vu\, \lnot x)$ 
and
$\gT\ul x\simeq \uar (x \,\vu\, \lnot x)$ with the surjective \homo 
$\pi\ul x:
\left|
\begin{array}{rcl}
\gT& \to  & \uar (x \,\vu\, \lnot x)  \\
y&  \mapsto  & y \,\vu\, x \,\vu\, \lnot x
\end{array}
  \right.
$.
\end{enumerate}
\end{propositionc}
\begin{proof}
The only delicate Item is \textsl{\ref{enumtra}}. Since
$(a:b)=\sotq{x}{x\vi b\leq a}$, the corresponding \sps 
$\VT(a:b)$ is the intersection of all $\VT(x)$ such that $\VT(a)\subseteq 
\VT(x)\cup
\VT(b)$, \cade such that  $\VT(a)\cap \DT(b)\subseteq \VT(x)$. But any closed subset of
$\Spec\, \gT$ is an intersection of basic closed subsets $\VT(x)$, so 
we have got the closure of $\VT(a)\cap \DT(b)$.
\end{proof}

\rems

\noindent 1) In general an arbitrary open subset of $X$ is not a \ssps.

\noindent 2) The \dfn we have given for the Krull boundary lattice 
$\gT\ul x$ is clearly constructive. Our translation of the boundary of a \oqc in terms of quotient lattices in Item~\textsl{6} agrees with the (non-\cov) classical \dfn of the boundary of a \ssps.
The proof that this translation is appropriate  needs \clama since it uses the fact the \sps $\Spec\,\gT$ have enough points.
\eoe

\medskip The \flw lemma allows us to best understand the   
Krull boundary \id of $x$ in~$\gT$.

\begin{lemma}
\label{lemBKReg}
For all $x\in\gT$ the Krull boundary \id of $x$ in $\gT$ is
\emph{\ndz}, i.e.\ 
$0:\fj=0$.
\end{lemma}
\begin{proof}
Let $u\in(0:\fj)$. Since $\fj=\dar x\vu (0:x)$ we have $u\vi x=0$ and, 
for all
$z\in(0:x)$, $u\vi z=0$. In particular $u\vi u=0$.
\end{proof}

For a \agH this is nothing but the law discovered by Brouwer: $\lnot(x \,\vu\, \lnot x)=0$.

\smallskip The \flw proposition  is a dual, constructive, 
\gui{point free} version of a well-known topological fact: if $A$ and $B$  are closed, the  union of the boundaries of $A\cup B$  and of  $A\cap B$ equals the union of the boundaries of $A$ and $B$.
\begin{proposition}
\label{propBordKUnion}
For all $x,y\in\gT$ we have
$\rK_\gT^{x}\cap \rK_\gT^{y}= \rK_\gT^{x\vu y}\cap 
\rK_\gT^{x\vi y}.$
\end{proposition}
\begin{proof}
Let $z\in\rK_\gT^{x}\cap \rK_\gT^{y}$, in other words there exist $u$ 
and $v$ such that $z\leq x\vu u$ and $u\vi x=0$, $z\leq y\vu v$ and $v\vi y=0$.
So $z\leq (x\vu (u\vu v))\vi (y\vu (u\vu v)) = (x\vi y)\vu(u\vu 
v)$ with
$(u\vu v)\vi(x\vi y)=(u\vi(x\vi y))\vu (v\vi(x\vi y))=0\vu 0=0$. Hence
$z\in\rK_\gT^{x\vi y}$. \\
Similarly $z\leq (x\vu y)\vu(u\vi v)$ with $(u\vi v)\vi(x\vu y)=0$, 
so
$z\in\rK_\gT^{x\vu y}$.\\
Finally assume  $z\in\rK_\gT^{x\vu y}\cap \rK_\gT^{x\vi y}$, 
in other words there exist $u$ and $v$ such that $z\leq x\vu y\vu u$ and $u\vi (x\vu y)=0$,
$z\leq (x\vi y)\vu v$ and $v\vi x\vi y=0$.
Let $u_1=(y\vu u)\vi v$. We have  $z\leq (x\vi y)\vu v\leq x\vu v$ and 
$z\leq x\vu
(y\vu u)$, hence $z\leq x\vu u_1$. Furthermore $x\vi u_1=x\vi (y\vu u) 
\vi v\leq
x\vi y\vi v =0$, hence $z\in\rK_\gT^{x}$.
\end{proof}

\subsubsection*{Closed subsets of $\Spec\,\gT\cir$}
\addcontentsline{toc}{subsubsection}{Closed subsets of $\Spec\,\gT\cir$}

It is natural to denote $\DT(a)$  by $\VTo(a)$.
So we give the \flw notation, for any subset $F\subseteq \gT$:
$\VTo(F)=\bigcap_{a\in\fF} \VTo(a)$. If $\fF$ is the filter generated by $F$,  we have $ \VTo(F)=\VTo(\fF)$.

The \flw proposition is a consequence of Proposition 
\ref{propositionFSES} by reversing the order (we only rewrite Items \textsl{1} and \textsl{6}) 
when identifying underlying subsets of $\Spec\,\gT$ and
$\Spec\,\gT\cir$.

Recall that the notion opposite to the \id $(a:b)$ is the filter 
$a\setminus
b\eqdefi \sotq{z}{z\vu b\geq a}$. 

\begin{proposition}[Closed subsets of $\Spec\,\gT\cir$]
\label{FdSES} ~
\begin{enumerate}
\item An arbitrary closed subset of $\Spec\,\gT\cir$ is equal to
$\bigcap_{x\in\fF} \VTo(x)$ where $\fF$ is an arbitrary filter of~$\gT$.
This is the \ssps corresponding to the quotient lattice 
$\gT/(\fF=1)$.

\item We define the quotient
$\gT\bal x=\gT/(\rK^\gT_x=1)$, where $\rK^\gT_x$ is the filter
\begin{equation} \label{eqboundaryinf}
\rK^\gT_x\,=\,\uar x \,\vi\, (1\setminus x)
\end{equation}
The lattice $\gT\bal x$ is called the \emph{lower (Krull) boundary of
$x$ in $\gT$}.
We say also that $\rK^\gT_x$ is \emph{the boundary (Krull) filter of 
$x$}.\\
When $\gT$ is a Brouwer \alg, $\rK^\gT_x\,=\,
\uar (x \,\vi\, (1- x))$ and
$\gT\bal x\simeq \dar (x \,\vi\, (1- x))$ with the surjective \homo 
$\pi\bal x:
\left|
\begin{array}{rcl}
\gT&\to&\dar (x \,\vi\, (1- x))\\
  y&\mapsto& y \,\vi\, x \,\vi\, (1- x)
\end{array}
  \right.
$.\\
The quotient $\gT\bal x$ corresponds to the 
boundary of $\VTo(x)$ for the topology of 
$\Spec\,\gT\cir$.\footnote{This is the opposite notion to the notion of boundary. The intersection of the closures of $\VTo(x)=\DT(x)$ and $\DTo(x)=\VT(x)$ is replaced with the union of their interiors. In $\Spec\,\gT$, this is the complement of the boundary of~$\DT(x)$.}
\end{enumerate}
\end{proposition}

\subsubsection*{Gluing spectral spaces}
\addcontentsline{toc}{subsubsection}{Gluing spectral spaces}

Since \trdis have a purely equational \dfn, the category has arbitrary colimits and limits. Limits and filtered colimits commute with the forgetful functor to the category of sets. Dual \prts are valid in the antiequivalent category of \spss and spectral morphisms. In some cases these limits or colimits correspond to the ones in the category  of topological spaces and continuous maps.

Here is the dual of Proposition~\ref{propRecolTD}
(we leave to \llec the translation of Fact~\ref{factRecolTD}).

\begin{propositionc}[gluing  finitely many \sspss along \oqcs]
\label{propRecolSpec}\hspace{0pt}
\vspace{-1.5em}
\begin{enumerate}
\item Let $(X_i)_{1\leq i\leq n}$ be a finite family of \spss, and for each  $i\neq  j$ a \oqc $X_{ij}$ of $X_i$ with  an \iso
$\varphi_{ij}\colon  X_{ij}\to X_{ji}$. Assume that $\varphi_{ij}= \varphi_{ji}^{-1}$ for all $i,j$ and that the natural relations of compatibility \gui{three by three}
are satisfied: if $x=\varphi_{ji}(y)=\varphi_{ki}(z)$ then $\varphi_{jk}(y)=z$. In this case the inductive limit of the diagram in the 
category of \spss is a space $X$ where each $X_i$ is identified with a 
 \oqc via the morphism $X_i\to X$. 
\item The analogous result also holds when replacing \gui{\oqc} with \gui{basic closed subset} .
\end{enumerate}
\end{propositionc}

\comm Let us note that $X$ is also the filtered colimit,
in the category  of sets and  and in the category of topological spaces, of the  diagram constituted by the spaces $X_i$, the inclusions $f_{ij}\colon X_{ij}\to X_i$ and the isomorphisms $\varphi_{ij}$. In other words the topological space~$X$ is the gluing of the spaces $X_i$ through the $X_{ij}$, when we identify $x\in X_{ij}$ with  $\varphi_{ij}(x)\in X_{ji}$.

In \clama  Proposition~\ref{propRecolTD}  is an easy consequence of Proposition~\ref{propRecolSpec}. Nevertheless this smart shortcut does not give a \prco of Proposition~\ref{propRecolTD}. 

In Item \textsl{2} of Proposition~\ref{propRecolSpec}, if we take for the $X_{ij}$'s arbitrary closed sets (which are \spss) instead of basic closed sets, the gluing will work as topological spaces but would not necessarily provide a spectral space. In Item \textsl{1} if we take  infinitely many basic open sets, the gluing will work as topological spaces but would not necessarily provide a spectral space.
\eoe

\subsection{Maximal spectrum versus Heitmann spectrum}

In the remarkable article \cite[\textsl{Generating non-Noetherian modules efficiently}]{Hei84} Raymond Heitmann explains that the usual notion of j-spectrum for a commutative ring is not the right one in the non-Noetherian case because it does not correspond to a spectral space in Stone's sense. He introduces the following modification of the usual definition: instead of considering the set of prime ideals which are intersections of maximal ideals, he proposes to consider the closure of the maximal spectrum in the prime spectrum, closure to be taken in the sense of the constructible topology (the patch topology).

\begin{definitions}
\label{defHspec1}
Let $\gT$ be a \trdi.
\begin{enumerate}
\item  We denote by $\Max\,\gT$ the topological subspace of $\Spec\,\gT$ formed by the maximal ideals of~$\gT$.
It is called the \textsl{maximal spectrum of $\gT$}.
\item  We denote by $\jspec\,\gT$ the topological subspace of $\Spec\,\gT$ formed by the $\fp$'s which verify the equality \hbox{$\JT(\fp)=\fp$}, i.e.\ the prime ideals $\fp$ which are intersections of maximal ideals (it is the \gui{usual} j-spectrum.
\item We call \textsl{$\rJ$-Heitmann spectrum of $\gT$} denoted by
$\Jspec\,\gT$ the closure of the maximal spectrum in $\Spec\,\gT$, 
 closure to be taken in the sense of the patch topology. This set is equipped with the topology induced by $\Spec\,\gT$.
\item  We denote by $\Min\,\gT$ the topological subspace of $\Spec\,\gT$ formed by the minimal prime ideals of $\gT$. We call it \textsl{the minimal spectrum of $\gT$}.
\end{enumerate}
\end{definitions}

Note that despite their names, the topological spaces $\Max\,\gT$,  $\jspec\,\gT$ and $\Min\,\gT$ are not in general spectral spaces.

\begin{theoremc}
\label{thDK3}
Let $\gT$ be a \trdi. 
The space $\Jspec\,\gT$ is a spectral subspace of $\Spec\,\gT$ canonically homeomorphic to $\Spec(\He(\gT))$. More precisely, if $M=\Max\,\gT$, we have for $a,b\in\gT$:
\begin{equation} \label{eqthDK3}
\DT(a)\cap M\,\subseteq\, \DT(b)\cap M \quad \Longleftrightarrow\quad
a\preceq_{\He(\gT)}b.
\end{equation}
\end{theoremc}
\begin{proof}
The lattice $\He(\gT)$ is defined on page \pageref{defHeT}.
\\
The second statement implies the first. Indeed $\Jspec\,\gT$, according to \ref{propSESP}\,(4), is the spectrum of the quotient lattice $\gT'$ corresponding to the preorder relation $a\leq _{\gT'}b$ defined by $\DT(a)\cap M\,\subseteq\, \DT(b)\cap M$.
\\
For the second statement we notice that the following properties are equivalent:
\[\begin{array}{lcl}
\DT(a)\cap M\subseteq \DT(b)\cap M& \;  &  (1)  \\ [1mm]
\DT(a)\cap \VT(b)\cap M =\emptyset  &&  (2)  \\[1mm]
\Tt \fm\in M\; (b\notin\fm \;\mathrm{or}\; a\in\fm)  &&  (3)  \\[1mm]
  \Tt \fm\in M\; (b\in\fm\Rightarrow  a\in\fm) &&  (4)
\end{array}\]
And assertion (4) amounts to saying that, seen in the quotient lattice $\gT/(b=0)$, $a$ belongs to the Jacobson radical. This means $a\in \JT(b)$, i.e. $a\preceq_{\He(\gT)}b$.
\end{proof}

A few points of comparison.

\begin{factc}
\label{factSpec=Jspec} ~
\begin{enumerate}
\item $\Spec\,\gT=\Jspec\,\gT$  \ssi $\gT=\He(\gT)$, that is, if $\gT$ is weakly Jacobson.
\item  $\Max\,\gT=\Jspec\,\gT$ \ssi $\He(\gT)$ is a Boolean algebra.
\item If $\gT$ has a Brouwer complement, $\Max\,\gT$ is a closed subset of $\Jspec\,\gT$, corresponding to the ideal $\Imax(\gT)$. It is a Stone space; it is equal to $\Jspec\,\gT$.
\item $\Min\,\gT=\Jspec\,\gT\cir$ \ssi  $\He(\gT\cir)$ is a \agB.
\item If $\gT$ has a negation, $\Min\,\gT$ is a closed subset of $\Spec\,\gT\cir$, corresponding to the filter $\Fmin(\gT)$. It is a Stone space; it is equal to $\Jspec\,\gT\cir$.
\end{enumerate}
\end{factc}
\begin{proof}
Item \textsl{1} follows from Theorem \ref{thDK3}. For Item \textsl{2}, (we are in classical mathematics) we notice that a distributive lattice is a Boolean algebra if, and only if, its prime ideals are all maximal. Item \textsl{3} follows from Item \textsl{2} and Fact \ref{factSpecMax}. Items \textsl{4} and \textsl{5} are obtained from Items \textsl{2} and \textsl{3} passing to the opposite lattice.\end{proof}

The following proposition is pointed out in \cite{Hei84}. The assumption in Item \textsl{2} is that the space $M=\Max\,\gT$ is Noetherian, that is to say that every open subset is quasi-compact. As the topology of $M$ is induced by that of $\Spec\,\gT$, the open subsets $\fD_\gT(a)\cap M$ form a basis of the topology. Moreover we have $\fD_\gT(a_1 \vu\cdots\vu a_n)=\fD_\gT(a_1)\cup\cdots\cup\fD_\gT(a_n)$. So, when $M$ is Noetherian, every open subset of $M$ is of the form $\fD_\gT(a)\cap M$ and every closed subset is a basic closed subset $\fV_\gT(a)\cap M$.
\begin{propositionc}[comparison of $\Jspec$ and  $\jspec$]
\label{propJspecjspec} ~
\begin{enumerate}
\item $\jspec\,\gT\subseteq \Jspec\,\gT$.
\item If $M=\Max\,\gT$ is \noe, a fortiori if $\,\Spec\,\gT$ is \noe, we have 
$\jspec\,\gT=\Jspec\,\gT$.
\end{enumerate}
\end{propositionc}
\begin{proof}
Let  $\fp\in\Spec\,\gT$.
\Propeq
\[\begin{array}{rcl}
\fp\in\jspec\,\gT&   &     \\[1mm]
\Tt a\in \gT\quad [\;\fp\in \DT(a)&  \Rightarrow  &
\Ex \fm\in (M\,\cap \DT(a)),\, \fp\subseteq \fm \;]  \\[1mm]
\Tt a\in \gT\quad [\;\fp\in \DT(a)&  \Rightarrow  &
\Ex \fm\in (M\,\cap  \DT(a)),\, \Tt b\in\gT\,(\fm\in \DT(b)
\Rightarrow \fp\in \DT(b))\;] \\[1mm]
\Tt a\in \gT\quad [\;\fp\in \DT(a)&  \Rightarrow  &
\Ex \fm\in (M\,\cap \DT(a)),\, \Tt b\in\gT\,(\fp\in \VT(b)
\Rightarrow \fm\in \VT(b))\;]\\[1mm]
\Tt a\in \gT\quad [\;\fp\in \DT(a)&  \Rightarrow  &
\Ex \fm\in M,\, \Tt b\in\gT\,(\fp\in \VT(b)
\Rightarrow \fm\in  M\cap \DT(a) \cap \VT(b) )\;]\qquad (*)
\end{array}\]
We have also \eqvcs
\[\begin{array}{rcl}
\fp\in\Jspec\,\gT&   &     \\[1mm]
\Tt a,b\in \gT\quad [\;\fp\in (\DT(a)\cap \VT(b))&  \Rightarrow  &
\Ex \fm\in M\cap \DT(a) \cap \VT(b)\;]
\qquad\qquad  (**)
\end{array}\]
\textsl{1}. So we see that $(*)$ is  \gnlt stronger than $(**)$, since in $(**)$,
$\fm$ can depend on  $a$ and $b$ while in $(*)$ it depends only on $a$.\\
\textsl{2}. If $M=\Max\,\gT$ is \noe the closed set 
$\bigcap\limits_{\VT(b)\ni\fp}
(M\cap \VT(b))$ is equal to a basic closed set $M\cap \VT(b_0)$ and  $(**)$ with this $b_0$ gives $(*)$.
\end{proof}

\comm
As Heitmann points out, known theorems using 
$\jspec$ always need the \gui{$\Max$ \noe} assumption. It is therefore likely that  $\Jspec$ is the only really interesting notion.
 Note that $\jspec\,\gT$ is  \ssps of $\Spec\,\gT$ only 
when it is equal to  $\Jspec\,\gT$.\eoe

\section{Krull and Heitmann dimensions of \trdis}
\label{secHtrdi}

We arrive in this section at the heart of the article.
We return to the point of view of constructive mathematics.
The only nonconstructive proofs are those which make the link between a classical notion and its constructive reformulation.
In classical mathematics the  \textsl{Krull  dimension of a \trdi} is defined as in commutative algebra: it is the upper bound of the lengths of strictly increasing chains of prime ideals.

\subsection{Krull boundaries and Krull dimension}

We recall here an \elr \cov version of the Krull dimension (\citealt*{CL2003,CLR05}). 
We rely on the following intuition: a variety has dimension $\leq k$ \ssi  the boundary of any subvariety has dimension $\leq k-1$. Similar \cov \eqv ideas are in \cite{Esp78,Esp82,Esp83}.

The \flw \tho in \clama gives an intuitive meaning for the
Krull dimension of a \trdi.

\begin{theoremc}
\label{thDK1} Let  $\gT$ be a \trdi generated by a subset $S$ and let 
$\ell$ be a nonnegative integer. \Propeq
\begin{enumerate}
\item  The lattice $\gT$ has dimension $\leq \ell.$
\item  For any $x\in S$ the boundary
 $\gT\ul{x}$ has dimension $\leq \ell-1$.
\item  For any $x\in S$ the boundary  $\gT\bal{x}$ has dimension $\leq \ell-1$.
\item 
For any  $x_0,\ldots,x_\ell\in S$
there exist $a_0,\ldots,  a_\ell\in \gT$ such that 
\begin{equation}
     a_0 \vi x_0  \leq  0\,,\;\;\;
     a_1 \vi x_1 \leq   a_0 \vu x_0\,,\;\;\; \dots\;\;,\;\;\;
     a_{\ell} \vi x_{\ell} \leq     a_{\ell-1} \vu x_{\ell-1}\,,\;\;\;
     1  \leq  a_{\ell} \vu x_{\ell}.
\end{equation}
\end{enumerate}
When $\gT$ is a \agH the previous conditions are also \eqv to
\begin{enumerate}\setcounter{enumi}{4}
\item   For any  $x_0,\dots,x_{\ell}$  in  $S$  we have the equality
\begin{equation}
     1 = x_{\ell}\vu (x_{\ell}\im( \cdots (x_1 \vu (x_1 \im (x_0\vu
\neg x_0)))\cdots))
\end{equation}
\end{enumerate}
When $\gT$ is a Brouwer \alg the previous conditions are also \eqv to
\begin{enumerate}\setcounter{enumi}{5}
\item   For any  $x_0,\dots,x_{\ell}$  in  $S$  we have the equality
\begin{equation}
     0 = x_{0}\vi (x_{0}-(x_1 \vi (x_1 - ( \cdots (x_\ell\vi
(1- x_\ell))))\cdots))
\end{equation}
\end{enumerate}
\end{theoremc}

In particular a \trdi has dimension $\leq 0$ \ssi it is a \agB.

Equivalence of Items \textsl{1}, \textsl{4}, and \textsl{5} is stated, without using boundaries, in  \citealt*{CL2003}, by \flw  \citealt*{Joy71,Joy76} and  \citealt*{Esp78,Esp82,Esp83,Esp86,Esp88}.
See also the more recent paper \citealt*{Esp2010}.

Note that the theory of \agHs of dimension $\leq k$ is purely equational.

\begin{proof}[Proof of \thref{thDK1}]
First we notice that the quotient $\gT\ul{x}=\gT/\rK_\gT^x$  
can also be seen as the ordered set obtained from the preorder relation $\leq ^x$ on $\gT$ defined as
\begin{equation} \label{eqBordSup}
a\leq ^x b \qquad \Longleftrightarrow\qquad \exists y\in 
\gT\;\;(\,x\vi
y=0\;\;\& \;\;a\leq  x\vu y \vu b\,)\,.
\end{equation}
\textsl{1} $\Leftrightarrow$ \textsl{2}. First we show that any maximal filter  $\ff $ of $\gT$ becomes trivial in~$\gT\ul{x}$, i.e.\ it contains $0$.
In other words we have to find an $a$ in $\ff $ such \hbox{that $a\leq ^x0$}.
If $x\in \ff $ we take $a=x$ and $y=0$ in \pref{eqBordSup}.
If $x\notin \ff $ there exists $z\in \ff $ such \hbox{that $x\vi z= 0$} (since the filter generated by $\ff $ and $x$ is trivial in $\gT$) and we take~\hbox{$a=y=z$} in (\ref{eqBordSup}).
So the dimension of  $\gT\ul{x}$ decreases by at least one
w.r.t.\ the dimension of~$\gT$ (assumed to be finite).\\
Now we show that if we have two prime filters $\ff '\subset
\ff $, with $\ff $ maximal and $x\in \ff \setminus \ff '$ then~$\ff '$ does not become trivial in  $\gT\ul{x}$ (this shows that the dimension of  
$\gT\ul{x}$ decreases by only one if $x$ is accurately chosen). 
Indeed, in the opposite case, we should have \hbox{a $z\in \ff'$} such \hbox{that $z\vi x=0$},
but since $z$ and $x\in \ff$ we should get $0\in \ff$, which is
falsum.\\
Finally we note that if $\ff'\subset \ff $ are distinct prime filters and if $S$ generates $\gT$ we can find an $x\in S$ such that
$x\in \ff \setminus \ff'$.

\noindent \textsl{1} $\Leftrightarrow$ \textsl{3}: consequence of \textsl{1} $\Leftrightarrow$ \textsl{2} by reversing the order.

\noindent \textsl{2} $\Leftrightarrow$ \textsl{4}:  by induction on $\ell$ 
(\dfn of the boundary).

\noindent \textsl{2} $\Leftrightarrow$ \textsl{5}:  by induction on $\ell$ (\dfn of the boundary in a \agH).

\noindent \textsl{3} $\Leftrightarrow$ \textsl{6}:  by induction on $\ell$ (\dfn of the boundary in a Brouwer \alg).
\end{proof}

The \flw \tho \ref{thDK2} gives a \carn, in \clama and in \cov terms, of the \ddk of $\Spec\, \gT$, i.e.\ the maximum length of closed \ird subsets.
Moreover we have shown (Item \textsl{6} of \thref{propositionFSES}) that if $X$ is the spectrum of a \trdi~$\gT$ and $x\in\gT$, then the  boundary of the \oqc $\DT(x)$ of~$X$ is 
canonically \isoc to $\Spec(\gT\ul{x})$.
So we get as corollary of Theorems \ref{fpropositionFSES} and \ref{fthDK1} a \carn (in \clama and in \cov terms) of the \ddk of \spss. Let us recall that the empty space is the unique \sps of dimension~$-1$. It corresponds to the  trivial \trdi $\Un$.

\begin{theoremc}
\label{thDK2}
Let $k$ be a nonnegative integer. A spectral space $X$ has dimension 
$\leq  k$
\ssi  any \oqc of $X$ has a boundary of dimension~\hbox{$\leq  
k-1$}.
\end{theoremc}

Concerning \ddk, we choose in \coma the \flw \dfn.
\begin{definition}[constructive \dfn of the Krull dimension]
\label{defDK0}~\\ 
The Krull dimension (denoted by $\Kdim$) of \trdis is defined in the following way.
\begin{enumerate}
\item $\Kdim(\gT)=-1$ \ssi  $1=_\gT0$ (i.e. the lattice is reduced  
to a point).
\item For $\ell\geq 0$ we define  $\Kdim(\gT)\leq \ell$ by the following \eqv conditions:
\begin{enumerate}
\item \label{l1}$\forall x\in \gT,\; \Kdim(\gT\ul x)\leq \ell-1$
\item \label{l2}$\forall x\in \gT,\;\Kdim(\gT\bal{x})\leq \ell-1$
\item \label{l3}$\forall x_0,\ldots,x_\ell\in \gT$
$\Ex a_0,\ldots,  a_\ell\in \gT$ satisfying:
\[    a_0 \vi x_0  \leq  0\,,\;\;\;
     a_1 \vi x_1 \leq   a_0 \vu x_0\,,\dots\,,\;\;\;
     a_{\ell} \vi x_{\ell} \leq     a_{\ell-1} \vu x_{\ell-1}\,,\;\;\;
     1  \leq  a_{\ell} \vu x_{\ell}.
\]
\end{enumerate}
\end{enumerate}
\end{definition}

Note that there are actually three possible definitions above for $\Kdim(\gT) \leq  l$. The definitions based on \textsl{2a} and \textsl{2b} are inductive, while the definition based on \textsl{2c} is global. The equivalence of the definitions based on \textsl{2a} and \textsl{2c} is immediate by induction (same thing for \textsl{2b} and \textsl{2c}).

For example, for $\ell = 2$ the inequalities in Item \textsl{2c} correspond to the following drawing in $\gT$.
\[\SCO{x_0}{x_1}{x_2}{a_0}{a_1}{a_2}\]

The fact that the definition works equally well with the upper boundary as with the lower boundary gives the following observation.

\begin{fact}
\label{corTTO}
A distributive lattice and its opposite lattice have the same dimension.\end{fact}

In classical mathematics, we can  directly understand this fact by considering the chains of prime ideals which have chains of prime filters as complements (and vice versa): if we identify the sets underlying $X=\Spec\,\gT$ and $X'=\Spec\,\gT\cir$ the order relations $\leq_X$ and $\leq_{X'}$ are opposite.

\medskip \rem 
We can illustrate Item \textsl{2c} in Definition \ref{defDK0}. We introduce the \gui{iterated Krull boundary ideal}. For $x_1,\ldots ,x_n\in\gT$ we denote by
\[
\gT_\rK[x_0]=\gT\ul{x_0},\,\gT_\rK[x_0,x_1]=(\gT\ul{x_0})\ul{x_1}
,\,\gT_\rK[x_0,x_1,x_2]=((\gT\ul{x_0})\ul{x_1})\ul{x_2}, 
\hbox{ etc}. \,\,
\]  
the successive boundary quotient lattices, and $\rK_\gT[x_0,\ldots ,x_k]$ denotes the kernel of the canonical projection $\gT\to \gT_\rK[x_0,\ldots ,x_k]$.
Then we have $y\in\rK_\gT[x_0,\ldots ,x_\ell]$
\ssi there exist $ a_0,\ldots,a_\ell\in \gT$ such that
\[    a_0 \vi x_0  \leq  0\,,\;\;\;
     a_1 \vi x_1 \leq   a_0 \vu x_0\,,\dots\,,\;\;\;
     a_{\ell} \vi x_{\ell} \leq     a_{\ell-1} \vu x_{\ell-1}\,,\;\;\;
     y  \leq  a_{\ell} \vu x_{\ell}.
\]
If $\gT$ is a \agH we have:
\[
\rK_\gT[x_0,\ldots ,x_\ell]=\dar(x_{\ell}\vu (x_{\ell}\im( \cdots (x_1 
\vu (x_1 \im (x_0\vu
\neg x_0)))\cdots)))
\]
The \ddk of the lattice is $\leq \ell$ \ssi  
$1\in\rK_\gT[x_0,\ldots ,x_\ell]$
for all $x_0,\ldots ,x_\ell$.
\eoe

\medskip In \coma the Krull dimension of $\gT$ is not a priori a well-defined \elt of $\NN\cup\so{-1}\cup\so{\infty}$. In \clama this \elt is defined as the lower bound of integers $\ell$ such that  $\Kdim(\gT)\leq \ell$.
In \coma we use the \flw \textsl{notation}{\footnote{~In fact 
if we see  $\Kdim(\gT)$ as the set of \elts  $\ell$ such that 
$\Kdim(\gT)\leq
\ell$, we can use final subsets of
$\NN\cup\so{-1}$ with reverse inclusion as order relation, the upper bound is then the intersection and the lower bound the 
union.}} in order to mimic the language of \clama:

\begin{notation}
\label{notaKdiminf}
Let $\gT$, $\gL$, $\gT_i$ be \trdis.
\begin{itemize}
\item  $\Kdim\,\gL\leq \Kdim\,\gT$ means $\Tt\ell\geq -1\; 
(\Kdim\,\gT\leq
\ell\;\Rightarrow \Kdim\,\gL\leq \ell)$.
\item  $\Kdim\,\gL= \Kdim\,\gT$ means  $\Kdim\,\gL\leq  
\Kdim\,\gT$ and
$\Kdim\,\gL\geq  \Kdim\,\gT$.
\item  $\Kdim\,\gT\leq  \sup_i\Kdim\,\gT_i$ means $\Tt\ell\geq 
-1\; (\&_i
(\Kdim\,\gT_i\leq \ell)\;\Rightarrow\Kdim\,\gT\leq \ell) $.
\item  $\Kdim\,\gT=  \sup_i\Kdim\,\gT_i$ means $\Tt\ell\geq -1\; 
(\&_i
(\Kdim\,\gT_i\leq \ell)\;\Leftrightarrow\Kdim\,\gT\leq \ell) $.
\end{itemize}
\end{notation}

\smallskip Let us denote by $\Bd(V,X)$ the boundary of $V$ in $X$ ($X$ is a topological space and  $V$  is a subset of $X$). 
Then if $Y$ is a subspace of
$X$ we have $\Bd(V\cap Y,Y)\subseteq \Bd(V,X)\cap Y$, with  equality 
if $Y$ is open. In the case of a \sps the \flw proposition gives a \cdpfv of this statement.
\begin{proposition}[Krull boundary and quotient lattice]
\label{propTquoBord} ~\\
Let $\gL$ be a quotient lattice of a \trdi $\gT$. By abuse, we 
denote by $x$ the image of $x\in\gT$ in $\gL$. Then  $\gL\ul x$ is a quotient  of  $\gT\ul x$ and  $\gL\bal x$ is a quotient of  $\gT\bal x$. Moreover if $\gL$ is the quotient of $\gT$ by a filter $\ff$,   $\gL\ul x$ is the quotient of  $\gT\ul x$  by the filter image of $\ff$ in  $\gT\ul x$.
\end{proposition}
\begin{proof}
Let $a,b,x\in\gT$; if $a\leq_{\gT\ul x}b$ there exists $z\in\gT$ such that
$x\vi z\leq_{\gT} 0$ and $a\leq _{\gT}x\vu z\vu b$. Since  $\gL$ is 
a quotient of $\gT$, we have a fortiori $x\vi z\leq_{\gL}0$ and 
$a\leq_{\gL}x\vu z\vu b$, hence $a\leq_{\gL\ul x}b$.
Let us see the second point. Let us denote by  $\pi\colon \gT\to\gL$, $\pi\ul 
x:\gT\to\gT\ul x$ and $\theta\colon \gT\ul x\to\gL\ul x$ the natural projections. It is clear that 
$\theta(\pi\ul
x(\ff))=\so{1}$,  so we have a factorization of $\theta$ via 
$\gT\ul x/(\pi\ul x(\ff)=1)$. Reciprocally let $a,b\in\gT$ be
such that $a\leq _{\gL\ul x} b$. We have to show that 
$a\leq _{\gT\ul x/(\pi\ul x(\ff)=1)} b$. By hypothesis there exists $z\in\gT$ such that
$a \leq_{\gL}x\vu z\vu b$ and $x\vi z\leq_{\gL}0$. This means that there exist $f_1$ and $f_2\in\ff$  such that $a\vi f_1\leq_{\gT} b\vu  x\vu z$ and 
$x\vi z\vi f_2\leq_{\gT} 0$.
Let us take $f=f_1\vi f_2$ and $z'=z\vi f_2$, we get 
$a\vi f\leq_{\gT} b\vu x\vu z'$ and $x\vi z'\leq_{\gT}0$, 
i.e.\  $a\vi f\leq_{\gT\ul x}  b$.
\end{proof}

The \flw corollary gives a \cdpfv of the statement that the dimension of a \ssps is always  
less  than or equal to the one of the total space.
\begin{corollary}
\label{corpropTquoBord}
If $\gL$ is a quotient lattice  of $\gT$ then $\Kdim\,\gL\leq 
\Kdim\,\gT$.
\end{corollary}

In the \flw proposition, Item \textsl{2} is a \cdpfv of the statement that the notion of boundary is a local notion. 
In the sequel we mainly use Item \textsl{1}. 

Furthermore, by reversing the order relation we get the analog statement for the other boundary.

\begin{proposition}[local character of Krull boundary]
\label{propLocBord} ~
\begin{enumerate}
\item Let $(\fa_i)_{1\leq i\leq m}$ be a finite family of \ids of 
$\gT$, with
$\bigcap_{i=1}^m\fa_i=\so{0}$.
If $x\in\gT$ let us denote by~$x$ its image in $\gT_i=\gT/(\fa_i=0)$.
The boundary ${\gT\!_i}\ul x$ can be seen as the quotient of  $\gT \ul x$
by an \id $\fb_i$ and we have: $\bigcap_{i=1}^m\fb_i=\so{0}$.
\item Let $(\ff_i)_{1\leq i\leq m}$ be a finite family of filters of 
$\gT$, with
$\bigcap_{i=1}^m\ff_i=\so{1}$.
If $x\in\gT$ let us denote by~$x$ its image in $\gT_i=\gT/(\ff_i=1)$.
The boundary ${\gT\!_i}\ul x$ can be seen as the quotient of  $\gT  \ul x$
by a filter $\ffg_i$ and we have: $\bigcap_{i=1}^m\ffg_i=\so{1}$.
\end{enumerate}
\end{proposition}
\begin{proof}
Let us see Item \textsl{1}. Consider the projection $\pi_i\colon \gT\to\gT_i$  
and the projection \hbox{$\pi'_i\colon \gT_i\to{\gT\!_i} \ul x$}.
The composite map $\gT\to{\gT\!_i} \ul x$ shows that ${\gT\!_i}\ul x\simeq
\gT/(\pi_i^{-1}(\rK_{\gT_i}^x)=0)$, and the \id $\pi_i^{-1}(\rK_{\gT_i}^x)$ contains $\rK_\gT^x$. This shows the first statement.
Let now $y$ be an \elt of $\gT$ such that for each~$i$, 
$\pi_i(y)\in\rK_{\gT_i}^x$.
It is the same thing to say there exist $b_i$'s such that
$\pi_i(y)\leq \pi_i(x)\vu\pi_i(b_i)$ and 
$\pi_i(b_i)\vi\pi_i(x)=\pi_i(0)$,
i.e.\ for some $a_i\in\fa_i$: $y\leq x\vu a_i\vu b_i$ and $x\vi 
b_i\in \fa_i$. Taking $c_i=a_i\vu b_i$ this gives $y\leq x\vu c_i$ and 
$x\vi c_i\in\fa_i$. Finally with  $c=c_1\vi\cdots \vi c_m$ we get $y\leq 
x\vu c$ with 
$c\vi x=0$. So $y\in\rK_\gT^x$, and this shows that a $z\in\gT \ul x$
belonging to all $\fb_i$'s is zero (it comes from a~$y$).\\
Item \textsl{2} is an immediate consequence of Proposition~\ref{propTquoBord} and from Fact \ref{factIdDansQuo} (from which we see in particular that if a finite  intersection of filters equals $\so 1$ this remains true in quotient).
\end{proof}

\begin{corollary}[local character of Krull dimension]
\label{corpropLocBord}~
\begin{enumerate}
\item Let $(\fa_i)_{1\leq i\leq m}$ be a finite family of \ids of 
$\gT$ and
$\fa=\bigcap_{i=1}^m\fa_i$. \\ Then
$\Kdim(\gT/(\fa=0))=\sup_i\Kdim(\gT/(\fa_i=0))$.
\item Let $(\ff_i)_{1\leq i\leq m}$ be a finite family of filters of 
$\gT$ and
$\ff=\bigcap_{i=1}^m\ff_i$. \\ Then
$\Kdim(\gT/(\ff=1))=\sup_i\Kdim(\gT/(\ff_i=1))$.
\end{enumerate}
\end{corollary}
\begin{proof}
It is sufficient to prove Item \textsl{1}. By replacing $\gT$ with 
$\gT/(\fa=0)$ we have to deal with the case \hbox{$\fa=0$}. The result is clear for $\Kdim=-1$. 
The induction works by Proposition~\ref{propLocBord}.\\
It is also possible to give a direct proof based on \carn 2~(c)
in \Dfn~\ref{defDK0}.
\end{proof}

In \clama the local character of the Krull dimension is stated as
\[
\Kdim(\gT)=\sup\sotq{\Kdim(\gT/(\ff=1))}{\ff \hbox{ minimal prime filter}}.
\]  
And it is a direct consequence of the  \dfn of 
dimension in \clama.  Then one deduces Corollary~\ref{corpropLocBord} but the proof of the corollary is not \cov.

\smallskip Now we get with a  \prco the same results than in \thref{thDK1}, but with the \cov \dfn of the dimension.

\begin{theorem}
\label{propDK1}
Let  $\gT$ be a \trdi generated by a subset $S$ and let 
$\ell$ be a nonnegative integer. \Propeq
\begin{enumerate}
\item  The lattice $\gT$ has dimension $\leq \ell.$
\item  For any $x\in S$ the boundary
 $\gT\ul{x}$ has dimension $\leq \ell-1$.
\item  For any $x\in S$ the boundary  $\gT\bal{x}$ has dimension $\leq \ell-1$.
\item 
For any  $x_0,\ldots,x_\ell\in S$
there exist $a_0,\ldots,a_\ell\in \gT$ such that 
\begin{equation}
     a_0 \vi x_0  \leq  0\,,\;\;\;
     a_1 \vi x_1 \leq   a_0 \vu x_0\,,\;\;\; \dots\;\;,\;\;\;
     a_{\ell} \vi x_{\ell} \leq     a_{\ell-1} \vu x_{\ell-1}\,,\;\;\;
     1  \leq  a_{\ell} \vu x_{\ell}.
\end{equation}
\end{enumerate}
When $\gT$ is a \agH the previous conditions are also \eqv to
\begin{enumerate}\setcounter{enumi}{4}
\item   For any  $x_0,\dots,x_{\ell}$  in  $S$  we have the equality
\begin{equation}
     1 = x_{\ell}\vu (x_{\ell}\im( \cdots (x_1 \vu (x_1 \im (x_0\vu
\neg x_0)))\cdots))
\end{equation}
\end{enumerate}
When $\gT$ is a Brouwer \alg the previous conditions are also \eqv to
\begin{enumerate}\setcounter{enumi}{5}
\item   For any  $x_0,\dots,x_{\ell}$  in  $S$  we have the equality
\begin{equation}
     0 = x_{0}\vi (x_{0}-(x_1 \vi (x_1 - ( \cdots (x_\ell\vi
(1- x_\ell))))\cdots))
\end{equation}
\end{enumerate}

\end{theorem}
\begin{proof}
\noindent \textsl{2} $\Leftrightarrow$ \textsl{4}  by induction on $\ell$, (\dfn of the boundary).

\noindent \textsl{3} $\Leftrightarrow$ \textsl{4}:  \dfn of the boundary.

\noindent \textsl{2} $\Leftrightarrow$ \textsl{5}:  by induction on $\ell$ (\dfn of the boundary in a \agH).

\noindent \textsl{3} $\Leftrightarrow$ \textsl{6}:  by induction on $\ell$ (\dfn of the boundary in a Brouwer \alg).

\noindent It remains to see that if \textsl{2} is true for all  $x\in S$, 
then  \textsl{2}
is true for any $x\in\gT$.
This follows from Proposition~\ref{propBordKUnion}  and Corollaries~\ref{corpropTquoBord} and~\ref{corpropLocBord}: for example for all  $x,y\in\gT$, $\gT\ul{x\vu y}$ is a quotient
of $\gT/(\rK_\gT^x\cap\rK_\gT^y)$, hence $\Kdim(\gT\ul{x\vu y})\leq \sup
(\Kdim\,\gT\ul x, \Kdim\,\gT\ul y)$.
\end{proof}

\subsection{Heitmann boundaries and Heitmann dimension of a \trdi}

\subsubsection*{Heitmann J-dimension of a \trdi}
\addcontentsline{toc}{subsubsection}{Heitmann J-dimension of a \trdi}

We now give a \cof \dfn of the dimension $\Jdim$ of a \sps
defined by Heitmann. We call it the \textsl{Heitmann $\rJ$-dimension} of the 
lattice $\gT$ (or of the \sps~$\Spec\,\gT$).

\begin{definition}
\label{defHdimTr}
Let $\gT$ be a \trdi.
The \textsl{Heitmann $\rJ$-dimension of $\gT$}, denoted by $\Jdim\,\gT$, is the 
Krull dimension of the Heitmann lattice  $\He(\gT)$ (cf. \Dfn~\ref{defHeT}).
\end{definition}

From a \cov point of view we only define, for any integer $\ell\geq -1$, the sentence   \gui{$\Jdim\,\gT\leq \ell$} by
\gui{$\Kdim\,\He(\gT)\leq \ell$}.

In \clama one defines the Heitmann $\rJ$-dimension of a spectral space $X$ in the \flw way: let $M_X$ be the set of closed points and $J_X$ the closure of $M_X$ for the patch topology; then
$\Jdim\,X=\Kdim\,J_X$.

\begin{fact}
\label{factKJdim} Let $\gT$ be a \trdi, $\gT'=\gT/(\JT(0)=0)$ and 
$\fa$ an \id of $\gT$.
\begin{enumerate}
\item $\Jdim(\He(\gT))=\Jdim\,\gT'=\Jdim\,\gT\leq 
\Kdim(\gT')
\leq \Kdim\,\gT$.
\item If $\gL=\gT/(\fa=0)$  then $\Jdim\,\gL\leq\Jdim\,\gT$.
\end{enumerate}
\end{fact}
\begin{proof}
Item \textsl{1} is a consequence of Item \textsl{3} in Fact~\ref{factHeHe} 
and Item \textsl{2} is a consequence of Item~\textsl{4}.
\end{proof}

\rem The map $\gT\mapsto \JT(0)$ does not give a functor. Same thing for
$\gT\mapsto \He (\gT)$. In particular Item~\textsl{2} in Fact~\ref{factKJdim} does not work for more general quotients, for example for quotients by 
filters.
Contrarily to $\Kdim$,  $\Jdim$ may increase when passing to a 
 quotient (we have simple examples in commutative \alg).\eoe

\medskip
\exl
Here is an example given by Heitmann of a spectral space such that
$\Jdim(\gT)<\Kdim(\gT/(\JT(0)=0))$. 
\begin{figure}[ht]   
\begin{center}
\includegraphics{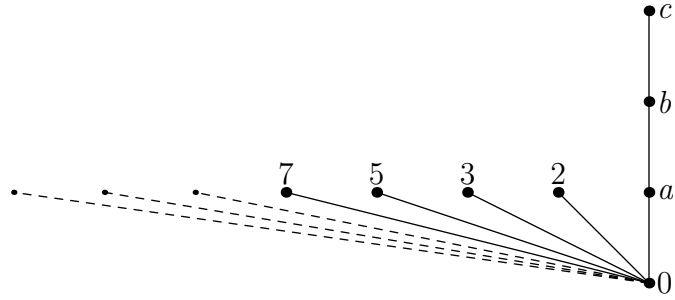}
\end{center}
\caption[]
{\label{figSpec} Heitmann's example}  
\end{figure}  
We consider $X=\Spec\,\ZZ$ and $Y=\bf n$
($n\geq 3$). We glue these \spss by identifying minimal \elts
 (the corresponding singleton  is indeed 
a \ssps for $X$ and $Y$). We get a space $Z=\Spec\,\gT$ with   $M_Z=\Max\,\gT$ as a closed subset. So $M_Z=J_Z=\Jspec\,\gT$ and it has dimension $0$. But the unique minimal
\elt  is the unique lower bound of $M_Z$. Hence $\JT(0)=0$ and
$\Kdim(\gT/(\JT(0)=0))=\Kdim(\gT)=n-2$.
\eoe

\medskip
\rem
Let us describe more explicitly the \dfn of $\Jdim$.

\begin{itemize}
\item $\Jdim\,\gT\leq \ell$ means:
$\forall x_0,\ldots,x_\ell\in \gT\;$
$\Ex a_0,
\ldots,  a_\ell\in \gT$ such that
\[    a_0\vi x_0  \leq_{\He(\gT)} 0,\;
     a_1\vi x_1 \leq_{\He(\gT)}  a_0\vu  x_0,\dots,\;
     a_{\ell}\vi x_{\ell} \leq_{\He(\gT)}    a_{\ell-1}\vu   
x_{\ell-1},\;
     1  \leq_{\He(\gT)} a_{\ell}\vu   x_{\ell}.
\]
\cade
$\forall x_0,\ldots,x_\ell\in \gT\;$
$\Ex a_0,\ldots,  a_\ell\in \gT\;$ $\Tt y\in\gT$
\[\begin{array}{rcl}
(a_0 \vi x_0) \vu y = 1 &  \Rightarrow  &  y=1    \\
(a_1 \vi x_1)   \vu y=1&  \Rightarrow &  a_0 \vu x_0 \vu y=1   \\
\vdots \qquad  &  \vdots &   \qquad \vdots   \\
(a_{\ell}\vi x_{\ell})\vu y=1&  \Rightarrow &  a_{\ell-1}\vu 
x_{\ell-1} \vu y=1
\\
   &   &  a_{\ell}\vu x_{\ell}=1
  \end{array}\]
\item In particular $\Jdim\,\gT\leq 0$ means:\\
  $\forall x_0\in \gT$ $\Ex a_0\in \gT$ $\Tt y\in\gT, 
  (\,(a_0 \vi x_0) \vu y = 1
\Rightarrow    y=1)\; $ and
$\; a_0 \vu x_0=1\,)$.
\item And $\Jdim\,\gT\leq 1$ means:
  $\forall x_0,x_1\in \gT\;$ $\Ex a_0,a_1\in \gT\;$ $\Tt y\in\gT$\,,\\
~~~~~~  $((a_0 \vi x_0) \vu y = 1   \Rightarrow    y=1)\,$ and
$((a_1 \vi x_1) \vu y = 1   \Rightarrow a_0 \vu x_0 \vu   y=1)\,$ and
$ \,a_1 \vu x_1=1$. \eoe
\end{itemize}

\subsubsection*{Heitmann dimension of a \trdi}
\addcontentsline{toc}{subsubsection}{Heitmann dimension of a \trdi}

Although  \citet*{Hei84} defines and uses  the dimension $\Jdim\,X$ where $X$ is the spectrum of a \cori, his proofs are in fact implicitly based on a related, but not equivalent, notion that we will call the \textsl{Heitmann dimension} and that we will denote by $\Hdim$.

We present this notion directly at the level of distributive lattices, where things are explained more simply.

The dimension $\Hdim\,\gT$ is always less than or equal to $\Jdim\,\gT$, which means that the theorems established for  $\Hdim$ will be a fortiori true with $\jdim$ and with $\Kdim$.

\begin{definition}
\label{defBHeit}
Let $\gT$ be a \trdi and $x\in\gT$.
We call \textsl{Heitmann boundary of $x$ in~$\gT$} the quotient lattice 
$\gT_\rH^x\eqdefi\gT/(\rH_\gT^x=0)$, where
\begin{equation} \label{eqboundaryHeit}
\rH_\gT^x\,=\,\dar x \,\vu\, (\JT(0):x)
\end{equation}
We also say that $\rH_\gT^x$ is \textsl{the Heitmann boundary \id of $x$ 
in $\gT$.}
\end{definition}

\begin{lemma}[Krull boundary and Heitmann boundary]
\label{lemTBKBH}~\\
Let $\gT$ be a \trdi, $\gT'=\gT/(\JT(0)=0)$, $\pi\colon \gT\to\gT'$ the 
canonical projection,
$x\in\gT$ and $\ov{x}=\pi(x)$. 
\begin{enumerate}
\item $\rH^{\ov{x}}_{\gT'} =\rK^{\ov{x}}_{\gT'}.$
\item  $\pi^{-1}(\rK^{\ov{x}}_{\gT'}) =\rH^x_\gT$
and ${{\gT'}_\rH^{\ov{x}}}\simeq {{\gT'}\ul{\ov{x}}}\simeq\gT_\rH^x$.
\end{enumerate}
\end{lemma}
\begin{proof}
Clear from the \dfns.
\end{proof}

\begin{lemma}
\label{lemTquoBordH}
Let $\gL=\gT/(\fa=0)$ be a quotient lattice  of a \trdi $\gT$ by an 
\id. By abuse, we denote by $x$ the image of $x\in\gT$ in $\gL$. Then 
$\gL_\rH^x$ is a quotient of  $\gT_\rH^x.$
\end{lemma}
\begin{proof}
Let $\pi\colon \gT\to\gL$ be the canonical projection. We have to show that if
$z\in\rH_\gT^x$ then $\pi(z)\in\rH_\gL^{\pi(x)}$. So let 
$z\leq_\gT x\vu u$ with  $x \vi u  \in\JT(0)$. As 
$\pi(\JT(0))\subseteq\rJ_\gL(0)$, we obtain $\pi(z)\leq_\gL \pi(x)\vu \pi(u)$ with  $\pi(x) \vi \pi(u) 
\in\rJ_\gL(0)$, which  implies  $\pi(z)\in\rH_\gL^{\pi(x)}$.
\end{proof}

\rem The previous lemma would be false for a more general quotient, for example for a quotient by a filter. It remains true whenever 
$\pi(\JT(0))\subseteq \rJ_\gL(0)$. \eoe

\begin{proposition}[comparison of two boundaries à la 
Heitmann]
\label{propBHeitHeyt}~\\
Let $x$ be an \elt of $\gT$ and $\hat{x}$ its image in $\He(\gT)$.
\begin{enumerate}
\item 
$\He(\gT_\rH^x)$ is a quotient of $\He(\gT_\rH^x)$.
\item If $\He(\gT)$ is a \agH, they are equal.
\end{enumerate}
\end{proposition}
\begin{proof}
\textsl{1.} The lattice $(\He(\gT))\ul{\hat{x}}$  is a quotient of $\gT$  whose preorder relation $a \preceq  b$ has the \flw description:
\[
\Ex y\;\;\;\big(x \vi  y \in \JT (0) \;\;\mathrm{and} \;\;\Tt z\;[\;a 
\vu  z = 1
\;\Rightarrow\; x \vu  y \vu  b  \vu  z = 1]\big).\eqno (*)
\]

\noindent
Let us consider the preorder defining $\He(\gT _\rH^x)$:
\[\Tt u\;\;\;(1 \leq_{\gT _\rH^x}  a \vu  u \;\;\Rightarrow\;\;
  1 \leq_{\gT _\rH^x}  b \vu  u).\eqno (**) \]

\noindent
We show that preorder $(*)$ implies preorder $(**)$. 
We have a $y$ with $(*)$.
We consider a $u$ such that $1 \leq  a \vu  u$ in $\gT _\rH^x$ and 
we have to show that
$1 \leq  b \vu  u$  in $\gT _\rH^x$.\\
The relation $1 \leq  a \vu  u$ in $\gT _\rH^x$ is written as $1 \leq  a 
\vu  u \vu x \vu  y'$ for a $y'$ such that $y' \vi  x \in \JT (0)$.\\
We let $z = u \vu  x \vu  y'$, we use $(*) $ and we get
   $x \vu  b \vu  u \vu  (y \vu  y') = 1.$ \\
But $(y \vu  y') \vi  x  \in \JT (0)$,  hence with $y'' = y \vu  
y'$ we get
$1 \leq  b \vu  u \vu  x \vu  y''$    with  $x \vi  y'' \in \JT (0)$
i.e.\  $1 \leq  b \vu  u$  in $\gT _\rH^x$.\\
\textsl{2}. Let $\pi\colon \gT\to\He(\gT)$ be the canonical projection.\\
From Fact~\ref{factHeHe} we have $\pi^{-1}(0)=\JT(0)$ and
$\pi^{-1}(1)=\so{1}$. Assume that $\He(\gT)$ is a \agH. 
Let $\wi{x}$ be an \elt of $\gT$ such that $\pi(\wi{x})\,=\,\pi(x)\im 0\,$ in $\He(\gT)$.
Then we rewrite $(*)$ as
\[
\Tt z\;[\;a \vu  z = 1  \;\Rightarrow\; x \vu  \wi x \vu  b  \vu  z =
1\;].\eqno(*').
\]

\noindent
Similarly  $1 \leq_{\gT _\rH^x}  a \vu  u$, which means
\[\Ex y'\;(x\vi y'\in \JT (0) \;\mathrm{and}\; 1 \leq  a \vu  u \vu  x 
\vu
y'),\]
is rewritten as $1 \leq  a \vu  u \vu  x \vu  \wi x$. So $(**)$ is rewritten
\[
\Tt u\;[\;a \vu  u \vu  x \vu  \wi x= 1  \;\Rightarrow\; b \vu  u 
\vu  x \vu
\wi x\;]\eqno(**')
\]

\noindent
and it is clear that $(*')$ and $(**')$ are \eqv.
\end{proof}

\begin{definition}
\label{defHdim}
The \textsl{Heitmann dimension of a \trdi $\gT$}, denoted by 
$\Hdim\,\gT$, has the \flw inductive \dfn.
\begin{itemize}
\item $\Hdim\,\gT=-1$ \ssi  $\gT\simeq\Un$.
\item If $\ell\geq 0$, $\Hdim\,\gT\leq \ell$ \ssi  for all 
$x\in\gT$,
$\Hdim(\gT_\rH^x)\leq \ell-1$.
\end{itemize}
\end{definition}

From Lemma \ref{lemTquoBordH} we deduce by induction
with the same notation as in \ref{notaKdiminf} the \flw lemma.

\begin{lemma}
\label{lemDimHquo}
If $\gL$ is the quotient of $\gT$ by an \id, then $\Hdim\,\gL\leq
\Hdim\,\gT$.
\end{lemma}

\begin{proposition}[comparison of $\Jdim$ and $\Hdim$]
\label{propJdimHdim}~
%
\begin{enumerate}
\item $\Hdim\,\gT\leq  \Jdim\,\gT$.
\item If $\He(\gT)$ is a \agH,  $\Hdim\,\gT=  
\Jdim\,\gT$.
\end{enumerate}

\noindent In \clama the equality holds when $\He(\gT)$ is \noe.
\end{proposition}
\begin{proof}
Item \textsl{1} is shown by induction on $\Jdim\,\gT$ by using Item \textsl{1} of Proposition~\ref{propBHeitHeyt}.

\noindent \textsl{2}. We assume that  $\He(\gT)$ is a \agH and 
we make an  induction by using Item \textsl{2} of Proposition 
\ref{propBHeitHeyt}.
We have to show that~\hbox{$(\He(\gT))\ul{\hat{x}}=\He(\rH_\gT^x)$} is also a \agH.
This follows from the fact that~\hbox{$(\He(\gT))\ul{\hat{x}}$} is a quotient  of~$\He(\gT)$ by a \idp (indeed $\He(\gT)$ is a \agH) and 
from Fact~\ref{factQuoAgH2}.
\end{proof}

\begin{proposition}
\label{propHdim0} Let $\gT'=\gT/(\JT(0)=0)$.
\begin{enumerate}
\item  $\Hdim\,\gT=\Hdim(\gT')$.
\item  $\Jdim\,\gT\leq 0\;\Longleftrightarrow\; \Hdim\,\gT\leq 
0\;\Longleftrightarrow\;
\Kdim(\gT')\leq 0$, (i.e.\ $\gT'$ is a \agB). This is the case when  $\He(\gT)$
is finite.
\end{enumerate}
\end{proposition}
\begin{proof}
Item \textsl{1} follows from Item \textsl{2} of Lemma \ref{lemTBKBH}.\\
\textsl{2.} We already know that $\Hdim\,\gT\leq 
\Jdim\,\gT\leq
\Kdim\,\gT'$. Item \textsl{1} of Lemma \ref{lemTBKBH} shows that 
$\Hdim\,\gT\leq 0$
implies  $\Kdim\,\gT'\leq 0$.\\
When $\gT$ is finite the topological space $\Jspec\,\gT$ is the set of 
\idemas where  all subsets are open (since points are closed).
Furthermore Item \textsl{1} of Fact \ref{factHeHe} gives also the result when
$\He(\gT)$ is finite.
\end{proof}

\rem
In \clama  $\He(\gT)$ is finite \ssi  the set $M$ of
\idemas is finite (case of semi-local rings in commutative \alg). \eoe

\medskip The \flw proposition is similar to Proposition
\ref{propBordKUnion} when replacing the Krull boundary with the Heitmann boundary.

\begin{proposition}
\label{propBordHUnion}
For all $x,y\in\gT$ we have 
$\rH_\gT^{x}\cap \rH_\gT^{y}= \rH_\gT^{x\vu y}\cap 
\rH_\gT^{x\vi y}$.
\end{proposition}

\begin{proof}
Follows from Proposition \ref{propBordKUnion} and Lemma 
\ref{lemTBKBH}.
\end{proof}

The \flw proposition is similar to Item \textsl{1} of Proposition~\ref{propLocBord}.
\begin{proposition}
\label{propLocBordH}
Let $(\fa_i)_{1\leq i\leq m}$ be a finite family of \ids of $\gT$, with
$\bigcap_{i=1}^m\JT(\fa_i)\subseteq \JT(0)$ (this happens if
$\bigcap_{i=1}^m\fa_i=\{0\}$).
For $x\in\gT$ denote by $x$ its image in~\hbox{$\gT_{\!i}=\gT/(\fa_i=0)$}.
Then the boundary ${\gT_{\!i}}_\rH^x$ can be seen as the quotient 
of 
$\gT_\rH^x$
by an \id~$\fb_i$ and we have $\bigcap_{i=1}^m\fb_i=\so{0}$.
\end{proposition}
\begin{proof}
Follows from Lemma \ref{lemTBKBH} and Proposition 
\ref{propLocBord}
when applied to 
$\gT'=\gT/(\JT(0)=0)$ and to the \ids that are images of $\JT(\fa_i)$ in 
$\gT'$.
\end{proof}

The \flw corollary is similar to Item \textsl{1} of Corollary~\ref{corpropLocBord}.

\begin{corollary}
\label{corpropLocBordH}
Let $(\fa_i)_{1\leq i\leq m}$ be a finite family of \ids of $\gT$ and
$\fa=\bigcap_{i=1}^m\fa_i$. \\
Then
$\Hdim(\gT/(\fa=0))=\sup_i\Hdim(\gT/(\fa_i=0))$.
\end{corollary}
\begin{proof}
We replace $\gT$ with $\gT/(\fa=0)$, so we may assume $\fa=0$.
Things are clear if $\gT=\Un$. Induction works thanks to
Proposition~\ref{propLocBordH}.
\end{proof}

\begin{proposition}
\label{propHdimgen}
Let $S$ be a  system of generators of a \trdi $\gT$ and $\ell\geq 0$.
\Propeq
\begin{enumerate}
\item For all $x\in\gT$,
$\Hdim(\gT_\rH^x)\leq \ell-1$.
\item For all $x\in S$,
$\Hdim(\gT_\rH^x)\leq \ell-1$.
\end{enumerate}
\end{proposition}
\begin{proof}
Follows from Proposition \ref{propBordHUnion}, Lemma
\ref{lemTquoBordH}, and Corollary \ref{corpropLocBordH}. For example with 
$x,y\in \gT$, since $\rH_\gT^{x\vu 
y}\subseteq\rH_\gT^x\cap\rH_\gT^y$, the
lattice $\gT_\rH^{x\vu y}$ is a quotient
of $\gT/(\rH_\gT^x\cap\rH_\gT^y)$ by an \id, hence 
$\Hdim(\gT_\rH^{x\vu
y})\leq \sup (\Hdim\,\gT_\rH^x, \Hdim\,\gT_\rH^y)$.
\end{proof}

  \rem
Let us describe more explicitly the Heitmann dimension.
We intoduce the \gui{iterated Heitmann boundary \ids}.
For $x_0,\dots,x_n\in\gT$ let us note 
\[
\gT_\rH[x_0]=\gT_\rH^{x_0},\,
\gT_\rH[x_0,x_1]=(\gT_\rH^{x_0})_\rH^{x_1},\,
\gT_\rH[x_0,x_1,x_2]=((\gT_\rH^{x_0})_\rH^{x_1})_\rH^{x_2},\,\ldots
\] 
which are successive quotient boundary lattices. Let us denote by 
\[\rH[\gT;x_0,\ldots,x_k]=\rH_\gT[x_0,\ldots ,x_k]\] 
the kernel of the canonical projection
$\gT\to \gT_\rH[x_0,\ldots ,x_k]$.

\noindent Saying $\Hdim\,\gT\leq \ell$ means that for all
$x_0,\ldots ,x_\ell\in\gT$ we have $1\in\rH_\gT[x_0,\ldots ,x_\ell]$.
So we need an explicit description of \ids $\rH_\gT[x_0,\ldots 
,x_\ell]$.
We have to make explicit $\pi^{-1}(\rH[\gT/(\fa=0);\pi(x)])$
(denoted by $\rH[\gT,\fa;x]$) for a canonical projection
$\pi\colon \gT\to\gT/(\fa=0)$.

\noindent By \dfn we have $y\in\rH[\gT,\fa;x]$ \ssi  $y\leq x\vu 
z\;\mod\;\fa$ for a  $z$ such that $\pi(z\vi x)\in \rJ_{\gT/(\fa=0)}(0)$.
This last condition means  
\[\Tt u\in\gT,\; (\pi((z\vi x)\vu
u=\pi(1)\;\Rightarrow \;\pi(u)=\pi(1),
\]
which is also
\[\Tt u\in\gT,\; ((\Ex a\in\fa\;(z\vi x)\vu u\vu a=1)\;\Rightarrow 
\;(\Ex
b\in\fa\; u\vu b=1)).
\]
Furthermore, $y\leq x\vu z\;\mod\;\fa$ means $\Ex a'\in\fa\;y\leq 
x\vu z\vu
a'$ and the condition   $\pi(z\vi x)\in \rJ_{\gT/(\fa=0)}(0)$ remains the same when
replacing $z$ with $z\vu a'$. \\
So we get the \flw condition 
for  $y\in\rH[\gT,\fa;x]$ 
\[ \Ex z\in\gT\; [\,y\leq x\vu z\;\&\;
\Tt u\in\gT,\; ((\Ex a\in\fa\;(z\vi x)\vu u\vu a=1)\;\Rightarrow 
\;(\Ex
b\in\fa\; u\vu b=1))\,].
\]
This formula has a fairly high logical complexity.
Indeed $\Ex a\in\fa$ and $\Ex b\in\fa$ have to be made explicit  with  $\fa=\rH_\gT[x_1,\ldots ,x_k]$ in order to obtain a description of
$y\in\rH_\gT[x_1,\ldots ,x_k,x]$.
Contrarily to the description of $\Jdim\,\gT\leq \ell$ containig only two quantifier alternances in all cases, we see that for $\Hdim\leq \ell$,  expressions become more and more complicate when $\ell$ increases. \\
In fact, for \coris, $\Hdim$ allows us to give proofs by induction
for some \gui{great} classical \thos in  commutative \alg. This is the real reason why we had to introduce this dimension. 
As $\Hdim\,\gT\leq \Jdim\,\gT\leq \Kdim(\gT/(\JT(0)=0))$,
proofs work under the hypothesis of a bound on $\Jdim$.
We lack examples with a better bound than
$\Kdim(\gT/(\JT(0)=0))$. \eoe

\section{Krull and Heitmann dimensions of \coris}
\label{secBPA}
In this section, $\gA$ is always a \cori.
We say that an \id $\fa$ of $\gA$ is \textsl{radical} if $\fa=\sqrt[\gA]{\fa}$.

\subsection{Zariski lattice}

We now recall the main idea of the \cof approach of  \cite{Joy76} for the spectrum of a \cori.

If $J\subseteq \gA$, let us denote by $\cI_\gA(J)$ or $\gen{J}_\gA$ (or
$\gen{J}$ if the context is clear) the \id generated by~$J$; we denote by $\DA(J)$ (or $\rD(J)$ if the context is clear) the nilradical of the \id $\gen{J}$:
\begin{equation} \label{eqZar}
\begin{array}{rclcl}
\DA(J)&  = & \sqrt[\gA]{\gen{J}} &=&\sotq{x\in\gA}{\Ex m\in\NN\;\; 
x^m\in\gen{J}}\,.
\end{array}
\end{equation}
When $J=\so{x_1,\ldots ,x_n}$ we denote $\DA(J)$ by
$\DA(x_1,\ldots ,x_n)$.
If the context is clear, we also use $\wi{x}$ as $\DA(x)$.

By \dfn the {\sl Zariski lattice} of $\gA$, denoted by $\ZarA$, is the set of the  $\DA(x_1,\ldots ,x_n)$'s.  The order relation is inclusion, lower bound and upper bound are given by 
\[
\DA(\fa_1)\vi\DA(\fa_2)=\DA(\fa_1\fa_2)\quad \mathrm{and} \quad
\DA(\fa_1)\vu\DA(\fa_2)=\DA(\fa_1+\fa_2).
\]
The Zariski lattice of $\gA$
is a \trdi, and $\DA(x_1,\ldots ,x_n)=
\wi{x_1}\vu\cdots \vu\wi{x_n}.$
Elements  $\wi{x}$  are a system of generators (stable under $\vi$) of $\ZarA$.

If $J\subseteq \gA$  we define $\wi{J}=\sotq{\wi{x}}{x\in J} \subseteq\ZarA$.

Let $U$ and  $J$  be two finite families in $\gA$; we have the equivalences
\[\Vi\wi{U}\leq_{\ZarA} \Vu\wi{J}
\quad\Longleftrightarrow \quad
\prod\nolimits_{u\in U} u  \in \sqrt{\gen{J}}
\quad\Longleftrightarrow \quad
\cM(U)\cap \gen{J}\neq \emptyset
\]
where $\cM(U)$ is the  multiplicative \mo generated by $U$.

This describes completely the \trdi $\ZarA$. More precisely (\citealt*{CC00,CL2003}) we get the \flw.
\begin{proposition}[\dfn à la Joyal of the spectrum of a \cori]
\label{propZar}~\\
 The lattice $\ZarA$ is
(up to unique \iso) the lattice generated by the symbols $\DA(a)$ for $a\in\gA$
subject to the \flw relations.
\[\begin{array}{cccc}
\DA(0_\gA) =0   ,\; \DA(1_\gA)= 1 ,\;
   \DA(x+y) \leq \DA(x)\vu\DA(y) ,\; \DA(xy) = \DA(x)\vi \DA(y).
\end{array}\]
\end{proposition}

The construction $\gA\mapsto\ZarA$ gives a functor from the category of \coris
to  the category of \trdis.
Via this functor the projection $\gA\to\gA/\DA(0)$ gives an 
\iso
$\ZarA\to \Zar(\gA/\DA(0))$. We have $\ZarA=\Un$ \ssi  $1_\gA=0_\gA$.

\smallskip An important \tho of  \citealt*{Hoc1969} says that any \sps is  homeomorphic to the spectrum of a \cori.
Here is a point-free version of this \tho. 

\smallskip \noindent {\bf Theorem.} \textsl{Any \trdi is 
\isoc to the
Zariski lattice of a \cori}. 

\smallskip \noindent 
For a nonconstructive proof see \citealt*{Ban96}.

\subsection{Ideals, filters and quotients of $\ZarA$}

Recall that in \clama the \textsl{Zariski spectrum} $\Spec\,\gA$ of a \ri is a topological space whose points are the 
\ideps of the ring with the topology defined by the basis of open subsets 
made of $\fD_\gA(a)=\sotq{\fp\in\Spec\,\gA}{a\notin\fp}$.
We denote $\fD_\gA(x_1)\cup\cdots\cup \fD_\gA(x_n)$ by $\fD_\gA(x_1,\ldots ,x_n)$.

\subsubsection*{Ideals of  $\gA$ and of $\ZarA$}

In \clama, any radical \id is the intersection of the \ideps above it.

We use the \flw notation (when $J\subseteq\gA$)
\[
\IZA(J):=\cI_{\ZarA}(\wi{J}).
\]
In particular
\[
\IZA(\so{x_1,\ldots ,x_n})=\dar\DA(x_1,\ldots ,x_n)=
\dar(\wi{x_1}\vu\cdots \vu\wi{x_n})\,.
\]
We have $\IZA(J)=\IZA(\sqrt{\gen{J}})$ and one gets easily the \flw fundamental result.

\begin{fact}
\label{factSpecAzarA} ~  
\begin{itemize}
\item The map $\fa\mapsto \IZA(\fa)$ defines an \iso from the
lattice of radical \ids of $\gA$ to the lattice of \ids of
$\ZarA$.  
\item By restriction to \ideps (resp. \idemas) of the ring $\gA$ and of the \trdi $\ZarA$ we get also a natural bijection.  
\item For any \cori $\gA$, $\Spec\,\gA$ (\coris) is identified to
$\Spec(\ZarA)$ (\trdis).
\end{itemize}
\end{fact}

\rem 
In \clama we have an \iso between the lattice $\ZarA$ and the lattice of
\oqcs of $\Spec\,\gA$.  So one identifies 
\begin{itemize}
\item $\DA(x_1,\ldots ,x_n)$, an \elt of $\ZarA$,
\item  $\fD_{\!\ZarA}(\DA(x_1,\ldots ,x_n))$, a \oqc of $\Spec(\ZarA)$,
\item    and  $\fD_\gA(x_1,\ldots ,x_n)$, a \oqc of $\Spec\,\gA$.
\end{itemize}

\noindent In \coma, we see  $\Spec\,\gA$ as a \gui{point-free topological space}, i.e.\ a space which is defined through a lattice of formal open subsets.
Hence the only identification is  given by the natural \iso between $\ZarA$
 and the \trdi defined formally à la Joyal in Proposition \ref{propZar}. \eoe

\smallskip The following statements are easy.
\begin{fact}[quotients]
\label{factQuoAT}~\\
If $J\subseteq\gA$, then $\Zar(\aqo{\gA}{J}) \simeq \Zar(\gA/\DA(J))
\simeq \Zar(\gA)/(\IZA(J)=0$.
\end{fact}

\begin{fact}[conductors]
\label{factTransporteurs} ~\\ 
Let $\fA$ and $\fB$ be \ids of $\gA$, $\fa=\DA(\fA)$ and
$\fb=\DA(\fB)$.  
Then $(\fa:\fb)=(\fa:\fB)$ is a radical \id of $\gA$
and inside $\ZarA$ we have $\,(\IZA(\fa):\IZA(\fb))=\IZA(\fa:\fb)$.
\end{fact}

\begin{fact}[covering with \ids]
\label{factRecouvI} ~\\
Let $\fa_i$ be a finite family of \ids of $\gA$.  Ideals $\IZA(\fa_i)$
cover $\ZarA$ (i.e.\ their intersection is $0$) \ssi 
$\,\bigcap_i\fa_i\subseteq \DA(0)$.
\end{fact}

In \clama the lattice $\Zar\,\gA$ is \noe (i.e.\ $\Spec\,\gA$ is \noe) \ssi  any radical \id is \gui{radically \tf}, i.e.\  an \elt of~$\Zar\,\gA$.

Furthermore,  $\Zar\,\gA$ is a \agH \ssi 
\hbox{$\Tt \fa,\fb\in\Zar\gA,\;(\fa:\fb)\in\Zar\gA$}.

The \flw result is important in \coma.

\begin{proposition}[\citealt*{CL2003}]
\label{propZarHeyt}
If $\gA$ is a \noe \coh \ri, $\Zar\,\gA$ is a \agH.
If moreover $\gA$ is \fdi, the order relation in  
$\Zar\,\gA$
is decidable. In this case we say that the lattice is \emph{discrete}.
\end{proposition}

\subsubsection*{Filters of $\gA$ and of $\ZarA$}

A \textsl{filter} in a \cori is a \mo $\fF$ such that
$xy\in\fF\Rightarrow x\in\fF$. A \textsl{prime filter} is a filter 
such that $x+y\in\fF\Rightarrow x\in\fF\;\mathrm{or}\;y\in\fF$ (it is the
complement of a \idep).

For $x\in \gA$ the filter $\uar \wi x$ of $\ZarA$ is denoted by 
$\FZ_\gA(x)$.
More \gnlt for  $S\subseteq\gA$ we denote by
$\FZ_\gA(S)$ the following filter of $\ZarA$
\[
\FZ_\gA(S)=
\bigcup\nolimits_{x\in\cM(S)}\uar \wi x.
\]
We have also $\FZ_\gA(S) =\FZ_\gA(\fF)=\bigcup\nolimits_{x\in\fF}\!\uar \wi x$, where $\fF$ is the filter of  $\gA$ generated by $S$.

The following facts are easy.

\begin{fact}
\label{factFZ}
The map $\ff\mapsto \FZ_\gA(\ff)$ gives an injective nondecreasing correspondence
from filters of $\gA$ to filters of 
$\ZarA$, and preserves finite upper bounds (the upper bound  of $\ff_1$ and $\ff_2$  is generated by  
the $f_1f_2$'s where $f_i\in\ff_i$).
This correspondence $\FZ_\gA$  gives a bijection between 
prime filters of $\gA$ and prime filters of $\ZarA$.
\end{fact}

Note, however, that the principal filter of $\ZarA$ generated by
$\wi{a_1}\vu\cdots \vu \wi{a_n}$ (i.e.\ the intersection of filters 
$\uar\wi{a_i}$), does not correspond in general to a filter of $\gA$.

\begin{fact}[localizations]
\label{factLocalises} ~\\
Let $S$ be a \mo of $\gA$, $\fF$ the filter generated by $S$, and
$\ff=\FZ_\gA(S)=\FZ_\gA(\fF)$. \\
Then $S^{-1}\gA=\gA_S=\gA_\fF$ and $\Zar(\gA_S)\simeq\Zar(\gA)/(\ff=1)$.
\end{fact}

\begin{fact}[complement filter]
\label{factComplement} ~\\
Let $x\in\gA$, then the filter $1_\ZarA\setminus \FZ_\gA(x)$
is equal to $\FZ_\gA(1+x\gA)$.
\end{fact}

\begin{fact}[covering with filters]
\label{factRecouvF} ~\\
Let $(S_i)_{1\leq i\leq n}$ be a finite family of \mos of $\gA$. The 
filters
$\FZ_\gA(S_i)$ cover $\ZarA$ (i.e.\ their intersection is 
$\so{1}$) \ssi   the \mos $S_i$ are \com, i.e.\  for all $x_i\in S_i$
we have $\gen{x_1,\ldots ,x_n}=\gen{1}$. \\
More \gnlt   we have $\,\FZ_\gA(S_1)\,\cap \cdots
\cap\,\FZ_\gA(S_n)\subseteq \FZ_\gA(S)$ \ssi  for all $x_i\in S_i$
there exists $x\in S$ such that $x\in\gen{x_1,\ldots ,x_n}$.
\end{fact}

\subsection{Heitmann lattice of a \cori}

In a \cori, the \textsl{Jacobson radical of an \id 
$\fJ$} is (in \clama) the intersection of the \idemas containing $\fJ$. It is denoted by $\JA(\fJ)=\rJ(\gA,\fJ)$, or also $\rJ(\fJ)$.
In \coma we use the following  \dfn (\eqv in \clama):
\begin{equation} \label{eqRadJac}
\JA(\fJ)\eqdefi\sotq{x\in\gA}{\Tt y\in\gA,\;\; 1+xy
\hbox{ is invertible modulo } \fJ}\,.
\end{equation}
We denote by $\JA(x_1,\ldots ,x_n)=\rJ(\gA,x_1,\ldots ,x_n)$ the \id
$\JA(\gen{x_1,\ldots ,x_n})$.  The \id  $\JA(0)$ is called the
\textsl{Jacobson radical  of  $\gA$}.

\begin{definition}
\label{defHeitA}
The \textsl{Heitmann lattice} of a \cori $\gA$ is
the lattice
$\He(\ZarA)$, denoted by $\HeA$.
\end{definition}

In \clama, the \flw lemma is evident.
From a \cov point of view we need  a direct \prco.

\begin{fact}[Jacobson radical]
\label{factRadJac} ~\\
The one-to-one  correspondence  $\IZA$ preserves passing to the  
Jacobson radical, i.e.\
 if $\fJ$ is an \id of $\gA$ and $\fj=\IZA(\fJ)$, then
$\rJ_\ZarA(\fj)=\IZA(\JA(\fJ))$.
\end{fact}
\begin{proof} Let us take an arbitrary $x\in\gA$.
We have to show that $\wi 
x\in\rJ_\ZarA(\fj)$ \ssi  $\wi
x\in\IZA(\JA(\fJ))$. Since $\JA(\fJ)$ is a radical \id, we have to show the \flw \eqvc:
\[  \wi x\in\rJ_\ZarA(\fj)\;\Leftrightarrow\; 
x\in\JA(\fJ). 
\]
By \dfn $\wi x\in\rJ_\ZarA(\fj)$ means
\[\Tt y\in\ZarA\quad (\wi x\vu y=1_\ZarA\;\Rightarrow\;\Ex 
z\in\fj\;\; z\vu y=1_\ZarA)\,,
\]
\cade, since any $y\in\ZarA$ may be written as $\DA(y_1,\ldots ,y_k)$,
\[\Tt y_1,\ldots ,y_k\in\gA\quad (\gen{x,y_1,\ldots ,y_k}=1_\gA\;
\Rightarrow\;\Ex z\in\fj\;\; z\vu \DA(y_1,\ldots ,y_k)=1_\ZarA).
\]
This is \imdt \eqv to
\[\Tt y_1,\ldots ,y_k\in\gA\quad (\gen{x,y_1,\ldots ,y_k}=1_\gA\;
\Rightarrow\;\Ex u\in\fJ \;\;\gen{u,y_1,\ldots ,y_k}=1_\gA),
\]
then to
\[\Tt y\in\gA\;(\gen{x,y}=1_\gA\;
\Rightarrow\;\Ex u\in\fJ \;\;\gen{u,y}=1_\gA),
\]
or also: any  $y\in\gA$ equal to some $1+xa$ is invertible 
modulo $\fJ$.
In other words $x\in\JA(\fJ)$.
\end{proof}

\begin{corollary}
\label{propHeitA}
Let $\fj_1$ and $\fj_2$ be two \itfs of $\gA$. Elements $\DA(\fj_1)$ 
and
$\DA(\fj_2)$ of $\ZarA$ are equal in the quotient $\HeA$ \ssi 
$\JA(\fj_1)=\JA(\fj_2)$. Hence $\HeA$ may be identified to the set of  $\JA(x_1,\ldots ,x_n)$'s, with
$\JA(\fj_1)\vi\JA(\fj_2)=\JA(\fj_1\fj_2)$ and
$\JA(\fj_1)\vu\JA(\fj_2)=\JA(\fj_1+\fj_2)$.
\end{corollary}

\rems ~\\
\textsl{1.} Given the good properties of the $\IZA$ correspondence, we have, with $\gT=\ZarA$, $\Zar(\gA/\JA(0))\simeq\gT/(\JT(0)=0)$. On the other hand, there does not seem to be an \Alg~$\gB$ naturally attached to~$\gA$ for which we have $\Zar\,\gB\simeq\He(\ZarA)$.\\
\textsl{2.} Note that in general $\JA(x_1,\ldots ,x_n)$ is a radical ideal but not the radical of a finitely generated ideal.\\
\textsl{3.} It is also easy to see that $\JA(\fj_1)\vi\JA(\fj_2)=\JA(\fj_1)\cap\JA(\fj_2)=\JA(\fj_1\cap\fj_2)$ (this follows from Lemma \ref{lemJacInter}). It may seem surprising that $\JA(\fj_1)\cap\JA(\fj_2)=\JA(\fj_1\fj_2)$ (it is a priori less clear than for the $\DA$'s). Here is an elementary calculation that (re)proves this fact. We have $x\in\JA(\fj_1)$ \ssi  $\Tt y\;(1+xy)$ is invertible modulo $\fj_1,$ and $x\in\JA(\fj_2)$ if and only if $\Tt y\;(1+xy)$ is invertible modulo $\fj_2$. But if $a=1+xy$ is invertible modulo $\fj_1$ and $\fj_2$, it is invertible modulo their product: indeed $1+aa_1\in\fj_1$ and  $1+aa_2\in\fj_2$ imply that $(1+aa_1)(1+aa_2)$, which is rewritten $1+aa'$, is in 
$\fj_1\fj_2$. \eoe

\subsection{Krull boundaries and Krull dimension of a \cori}

In \coma we give the following \dfn.

\begin{definition}
\label{defKdimA}
The Krull dimension of a \cori
is the Krull dimension of its Zariski lattice.
\end{definition}

As a consequence of Fact \ref{factSpecAzarA} and  \thref{thDK1}, this \dfn is equivalent in \clama to the usual one.

\begin{definition}
\label{defZar2}  Let  $\gA$ be a commutative \ri, $x\in\gA$, and let $\fa$ be a \itf.
\begin{enumerate}
\item [$(1)$] The \textsl{Krull upper boundary} of $\fa$ in $\gA$ is the quotient \ri
\begin{equation}\label{eqBKAC}
\gA_\rK^{\fa}:=\gA/\rK_\gA(\fa)  \quad \hbox{where} \quad
 \rK_\gA(\fa):=\fa+(\sqrt{0}:\fa).
\end{equation}
Write $\rK_\gA(x)$ for $\rK_\gA(x\gA)$ and $\gA_\rK^{x}$ for $\gA_\rK^{x\gA}$. This \ri is called the \emph{upper boundary of $x$ in~$\gA$}.  
We will say that $\rK_\gA(\fa)$ is \emph{the Krull boundary \id of $\fa$ in $\gA$.}%
\item [$(2)$] The \textsl{Krull lower boundary} of $x$ in $\gA$ is the localized commutative \ri
$\gA\bal{x}:=\gA_{\rS\bal{x}}$  where $\rS\bal{x}=x^\NN(1+x\gA)$.
We will say that $x^\NN(1+x\gA)$ is the \textsl{Krull boundary monoid of $x$} in~$\gA$.%
\end{enumerate}
\end{definition}

So an arbitrary \elt  of  $\rK_\gA(y_1,\ldots ,y_n)$ has the form 
$\sum_i
a_iy_i+b$ where all the $by_i$'s are nilpotent.

\begin{proposition}
\label{propZar2} Let   $x\in \gA$ and  $\fj=\gen{j_1,\ldots ,j_n}$ be a
\itf.
Let $\wi x\in\ZarA$ and $\fa=\DA(\fj)=\wi{j_1}\vu\cdots
\vu\wi{j_n}\in\ZarA$. 
\begin{enumerate}
\item The Krull boundary \id of $\fa=\DA(\fj)$ in $\ZarA$,
$\rK_\ZarA^\fa$, is equal to
$\IZA(\rK_\gA(\fj))$. So we can identify  $(\ZarA)\ul{\fa}$ 
with $\Zar(\gA\ul{\fj})$.
\item The boundary Krull filter of $\wi x$ in $\ZarA$, 
$\rK^\ZarA_{\tilde x}$, is equal to $\FZ(\rS\bal{x})$. So we can identify   $(\ZarA)\bal{\tilde x}$ with
$\Zar(\gA\bal{x})$.
\end{enumerate}
\end{proposition}
\begin{proof}
For the boundary \id, we have by \dfn
  \[\rK_\ZarA^\fa=\rK_\ZarA(\DA(\fa))=\DA(\fa) \vu 
(\DA(0):\DA(\fa) ).
\]
By Facts \ref{factSpecAzarA} and \ref{factTransporteurs},
it is equal to  $\IZA(\fa + (\DA(0):\fa ))$ and also to
  \[\IZA(\fj + (\DA(0):\fj ))=\IZA(\rK_\gA^\fj).\] Then we pass to the quotient lattices and we use Fact~\ref{factQuoAT}.\\
For the boundary Krull filter, this works in the same way by using Facts  \ref{factFZ} and \ref{factComplement} and passing to the
quotient lattices  with  Fact~\ref{factLocalises}.
\end{proof}

As corollary of  \thref{propDK1} and Proposition \ref{propZar2}  we get
an analog of \thref{thDK1}, in a \cov version.
Recall that the Krull dimension of a ring equals $-1$ \ssi  
the ring is trivial (i.e.\ $1_\gA=0_\gA$).

\begin{theorem} \label{thDKA} For a \cori  $\gA$ and an 
$\ell\in\NN$  \propeq
\begin{enumerate}
\item The Krull dimension of $\gA$ is $\leq \ell$.
\item For any $x\in \gA$ the Krull dimension of $\gA\ul x$ is 
$\leq \ell-1$.
\item For any \itf $\fj$ of $\gA$ the Krull dimension of 
$\gA\ul\fj$ is $\leq \ell-1$.
\item For any $x\in \gA$ the Krull dimension of $\gA\bal x$ is 
$\leq \ell-1$.
\end{enumerate}
\end{theorem}

This theorem gives us a good intuitive meaning of the Krull dimension.

With Fact \ref{factSpecAzarA} we get the same theorem in classical mathematics.

Given its importance, we will give simple direct proofs of the equivalences between Items \textsl{1},~\textsl{2} and \textsl{4} in classical mathematics.

\begin{Proof}{Direct proof in classical mathematics}
Let us first show the equivalence of Items \textsl{1} and \textsl{2}.
Recall that the prime ideals of $S^{-1}\gA$ are \ids $S^{-1}\fp$ where $\fp$ is a prime ideal of A which does not intersect S.
The equivalence then clearly follows from the following two statements.
\\
(a) Let $x\in\gA$; if $\fm$ is a maximal ideal of $\gA$ it always intersects $\rS\bal{x}$. Indeed  this is clear if $x\in\fm$, and if not, $x$ is invertible modulo $\fm$ which means that $1+x\gA$ intersects $\fm$.
\\
(b) If $\fm$ is a maximal ideal of $\gA$, and if $x\in\fm\setminus\fp$ where $\fp$ is a prime ideal contained in $\fm$, then $\fp\cap \rS\bal{x}=\emptyset$: indeed if $x(1+xy)\in\fp$ then, since $x\notin\fp$ we have $1+xy\in\fp\subset\fm$, which gives the contradiction $1\in\fm$ (since $x\in\fm$).\\
Thus, if $\fp_0\subsetneq \cdots \subsetneq \fp_\ell$ is a chain with  
$\fp_\ell$ maximal, it is shortened by at least its last term when localizing in $\rS\bal{x}$, and it is shortened only by its last term if $x\in\fp_\ell\setminus\fp_{\ell-1}$.
\\
The equivalence of Items \textsl{1} and \textsl{4} is proved in the ``opposite way'', by replacing prime ideals with prime filters.
We first notice that the prime filters of $\gA/\fJ$ are equal to $(S+\fJ)/\fJ$, where $S$ is a prime filter of $\gA$ which does not intersect $\fJ$.
It is then enough to prove the two ``opposite assertions'' of (a) and (b) which are the following:
\\
(a') Let $x\in\gA$; if $S$ is a maximal filter of $\gA$ it always intersects $\rK_\gA^x$. Indeed this is clear if $x\in S$,  and if not, since $S$ is maximal $Sx^\NN$ contains $0$, which means that there is an integer $n$ and an element $s$ of $S$ such that $(sx)^n=0$ and $s\in (\sqrt{0}:x)\subset \rK_\gA^x$.
\\
(b') If $S$ is a maximal filter of $\gA$, and if $x\in S\setminus 
S'$, where $S'\subset S$ is a prime filter, then $S'\cap \rK_\gA^x=\emptyset$.
Indeed if $ax+b\in S'$ with $(bx)^n=0$ then, since $x\notin S'$ we have $ax\notin S'$ and, since $S'$ is prime, $b\in S'\subset S$, but since $x\in S$, $(bx)^n=0\in S$ which is absurd.
\end{Proof}

Moreover the \flw corollary shows that \thref{thDKA} implies the  constructive \elr \carn 
of the dimension through algebraic identities, as described  
in~\citealt*{Lom02,CL2003}.

\begin{corollary}
\label{corthDKA} \Propeq
\begin{itemize}
\item  [$(1)$] The Krull dimension of $\gA$ is $\leq \ell$.
\item  [$(2)$] For all $x_0,\ldots ,x_\ell\in\gA$ there exist  
$b_0,\ldots,b_\ell\in \gA$ such that 
\begin{equation} \label{eqCG}
\left.
\begin{array}{rcl} 
\DA(b_0x_0)& =  &\DA(0)    \\ 
\DA(b_1x_1)& \leq  & \DA(b_0,x_0)  \\
\vdots\quad& \vdots  &\quad  \vdots \\
\DA(b_\ell x_\ell )& \leq  & \DA(b_{\ell -1},x_{\ell -1})  \\
\DA(1)& =  &  \DA(b_\ell,x_\ell )
\end{array}
\right\}
\end{equation}
\item  [$(3)$] For all $x_0,\ldots ,x_\ell\in\gA$ there exist 
$a_0,\ldots,a_\ell\in \gA$ and
$m_0,\ldots,m_\ell\in\NN$ such that
\[ x_0^{m_0}(x_1^{m_1}\cdots(x_\ell^{m_\ell} (1+a_\ell x_\ell) + 
\cdots+a_1x_1)
+ a_0x_0) =0\,.
\]
\end{itemize}
\end{corollary}
\begin{proof}
Let us show the equivalence of (1) and (3). 
Let us use for example for (1) the characterization via localized rings  $\gA\bal x$.
The equivalence for dimension $0$ is clear. Suppose the thing established for dimension $\leq \ell$. We then see that $S^{-1}\gA$ has dimension
  $\leq \ell$ \ssi  for all
$x_0,\ldots ,x_\ell\in \gA$ there exist $a_0,\ldots,a_\ell\in \gA$,
$m_0,\ldots,m_\ell\in\NN$  and $s\in S$ such that
\[ x_0^{m_0}(x_1^{m_1}\cdots(x_\ell^{m_\ell} (s+a_\ell x_\ell) + 
\cdots+a_1x_1) + a_0x_0)=0.
\]
We have $(2) \Rightarrow (1)$  by considering the characterization (4) of the Krull dimension of a distributive lattice given in \thref{propDK1} 
and by applying it to the Zariski lattice $\ZarA$ with  
$S=\sotq{\DA(x)}{x\in\gA}$.
We could also verify by a direct calculation that $(2) \Rightarrow 
(3)$.
\end{proof}

\rem The system of inequalities (\ref{eqCG}) in Item (2) of the previous corollary establishes an interesting and symmetrical relation between the two sequences 
$(b_0,\ldots ,b_\ell)$ and $(x_0,\ldots ,x_\ell)$. When $\ell=0$, 
it means $\DA(b_0)\vi\DA(x_0)=0$ and $\DA(b_0)\vu\DA(x_0)=1$, 
i.e.\  the two elements \elts~$\DA(b_0)$ and $\DA(x_0)$ are complementary.
We therefore introduce the following terminology: when two sequences 
$(b_0,\ldots ,b_\ell)$ and $(x_0,\ldots ,x_\ell)$ verify the inequalities (\ref{eqCG}) we will say that they are \textsl{complementary}.\eoe

\medskip 
Let us also point out that it is easy to establish constructively that 
$\Kdim(\gK[X_1,\ldots,X_n])=n$ when~$\gK$  is a field, or even a zero-dimensional ring (cf.~\cite{CL2003}).  One can also deal constructively with the Krull dimension of geometric rings
(the \pf \Klgs).

\medskip
\rem
We also have the following results (already proved for distributive lattices): 
\begin{itemize}
\item if $\gB$ is a quotient or a localized \ri  of $\gA$, then  
$\Kdim\,\gB\leq\Kdim\,\gA$;
\item if $(\fa_i)_{1\leq i\leq m}$ is a finite family of \ids of $\gA$ 
and
$\fa=\bigcap_{i=1}^m\fa_i$, then
$\Kdim(\gA/\fa)=\sup_i\Kdim(\gA/\fa_i)$;
\item if $(S_i)_{1\leq i\leq m}$ is a finite family of \moco of $\gA$ 
then
$\Kdim(\gA)=\sup_i\Kdim(\gA_{S_i})$;
\item in \clama we have
$\Kdim(\gA)=\sup_\fm\Kdim(\gA_{\fm})$, where $\fm$ runs over all the
\idemas.
\eoe
\end{itemize}

\medskip
\rem
We can illustrate Corollary \ref{corthDKA} 
above by introducing \gui{the iterated Krull boundary \id}.
For $x_0,\ldots ,x_n\in\gA$ consider the successive upper boundaries
\[
(\gA\ul{x_0})\ul{x_1},\, ((\gA\ul{x_0})\ul{x_1})\ul{x_2}, 
\hbox{ etc.},
\] 
and let $\rK_\gA[x_0,\ldots ,x_\ell]$
be the kernel of the canonical projection $\gA\to
(\cdots(\gA\ul{x_0}){\cdots})\ul{x_\ell}$.  Then we have 
$y\in\rK_\gA[x_0,\ldots,x_\ell]$ \ssi  $\Ex a_0, \ldots,  a_\ell\in \gT$ and $m_0,\ldots,m_\ell\in\NN$  verifying
\[  x_0^{m_0}(x_1^{m_1}\cdots(x_\ell^{m_\ell} (y+a_\ell x_\ell) + 
\cdots+a_1x_1)
+ a_0x_0) =0.
\]
And the Krull dimension is $\leq \ell$ \ssi  for all 
$x_0,\ldots,x_\ell\in\gA$ we have $1\in\rK_\gA[x_0,\ldots ,x_\ell]$. 
\eoe

\subsection{Heitmann dimensions of a \cori}
\subsubsection*{Heitmann spectrum}

The spectral space that Heitmann defined to replace the j-spectrum, i.e.\ the closure for the constructible topology of the maximal spectrum in
$\Spec\,\gA$,
corresponds to the following definition.
\begin{definition}
\label{defJspecA}
We call \textsl{Heitmann spectrum} of a \cori $\gA$ the 
subspace $\Jspec(\ZarA)$ of $\Spec\,\gA$. We also denote it by
$\Jspec\,\gA$. We denote  $\jspec(\ZarA)$ by 
$\jspec\,\gA$, i.e.\ the j-spectrum of the ring in the usual meaning.
\end{definition}

In \clama  \thref{thDK3} gives the \flw fact.
\begin{factc}
\label{factHSpecA}
For any \cori $\gA$, the Heitmann spectrum of $\gA$ 
is identified with the \sps  $\Spec(\HeA)$ (in the meaning of \trdis).
\end{factc}

We then have the elementary constructive pointfree definition  of the dimension introduced by Heitmann.
 
\begin{definition}
\label{defJdimA}
The Heitmann $\rJ$-dimension of $\gA$, denoted by $\Jdim\,\gA$, is the Krull dimension of $\Heit(\gA)$, in other words it is the $\Jdim$ of $\ZarA$.
\end{definition}

In \clama $\Jdim\,\gA$ is equal to the dimension of the spectral space 
$\Jspec\,\gA$, defined abstractly \gui{with points}.

We will denote the dimension of $\jspec\,\gA$ by $\jdim\,\gA$.

\medskip \rem
Let us specify the meaning of $\Jdim\,\gA\leq \ell$ in the case of commutative rings. Since this is the Krull dimension of
$\Heit\,\gA$, and the elements of $\Heit\,\gA$  are identified with the Jacobson radicals of \itfs, we obtain the following \carn.\\ 
$\forall x_0,\ldots,x_\ell\in \gA\;$ 
$\Ex \fa_0,\ldots,  \fa_\ell,$ \itfs of $\gA$ such that 
\[\begin{array}{rcl} 
x_0\,\fa_0  &  \subseteq  &  \JA(0)    \\ 
x_1\,\fa_1 &  \subseteq  &  \JA(\gen{x_0}+\fa_0)   \\
\vdots \;  &  \vdots &   \qquad \vdots   \\
x_{\ell}\,\fa_{\ell} &  \subseteq  &  
\JA(\gen{x_{\ell-1}}+\fa_{\ell-1})   \\
\gen{1} &  =  &  \JA(\gen{x_{\ell}}+\fa_{\ell})   
\end{array}\]
We apparently cannot avoid resorting to \itfs and this means that we do not obtain a \gui{first-order} definition.
\\
Note that each membership $x\in\JA(y_1,\ldots ,y_m)$ expresses itself by $\Tt z\in\gA,$ $1+xz$ is invertible modulo
$\gen{y_1,\ldots ,y_m}$, that is to say again 

\medskip  \hspace*{4em}
$\Tt z\in\gA \;\Ex t,u_1,\ldots ,u_m\in\gA, 
\;\;1=(1+xz)t+u_1y_1+\cdots 
+u_my_m.$  
\eoe

\subsubsection*{Heitmann boundaries and Heitmann dimension}

\begin{definition}
\label{defHdimA}
The Heitmann dimension of a commutative ring is the Heitmann dimension of its Zariski lattice.
\end{definition}

\begin{definition}
\label{defHei2} Let $\gA$ be a \cori, $x\in\gA$ and $\fj$ 
a \itf.
The \textsl{Heitmann boundary of $\fj$ in $\gA$} is the quotient  ring
$\gA/\rH_\gA(\fj)$ with
          \[\rH_\gA(\fj):=\fj+(\JA(0):\fj)\]
which is also called the \textsl{Heitmann boundary ideal of $\fj$ in 
$\gA$}.
We will also denote  $\rH_\gA(y_1,\ldots ,y_n)$ for  
$\rH_\gA(\gen{y_1,\ldots
,y_n})$, $\rH_\gA^x$ for $\rH_\gA(x)$ and $\gA_\rH^x$ for 
$\gA/\rH_\gA^x$.
\end{definition}

Thus an arbitrary element of $\rH_\gA(y_1,\ldots ,y_n)$  is written  
$\sum_i
a_iy_i+b$ with all the $by_i$ in~$\JA(0)$.

The following proposition follows from the good properties of the one-to-one correspondence~$\IZA$ (see Facts \ref{factSpecAzarA}, \ref{factQuoAT}, \ref{factTransporteurs}, and \ref{factRadJac}).
\begin{proposition}
\label{propBordH-TA}
For a \itf  $\fj$ the Heitmann boundary of $\fj$ in the meaning of \coris
and the one in the meaning of  \trdis agree. More precisely,
with $j=\DA(\fj)$ and $\gT=\ZarA$, we have: 
\[\IZA(\rH_\gA(\fj))=\rH_\gT(j),\;\;
\emph{and}\;\; \Zar(\gA/\rH_\gA(\fj))\simeq
\gT/(\rH_\gT(j)=0)=\gT_\rH^j.\]
\end{proposition}

As corollary of Propositions \ref{propHdimgen} and 
\ref{propBordH-TA}
we get the \flw result.

\begin{proposition}
\label{lemDHA} Let  $\gA$ be a \cori  and  
$\ell$ a nonnegative integer.
\Propeq
\begin{enumerate}
\item The Heitmann dimension  of $\gA$ is $\leq \ell$.
\item For all $x\in\gA$,
$\Hdim(\gA/\rH_\gA(x))\leq \ell-1$.
\item For all \itfs $\fj$ of $\gA$, $\Hdim(\gA/\rH_\gA(\fj)) \leq 
\ell-1$.
\end{enumerate}
\end{proposition}

\rem So the Heitmann dimension  of $\gA$ can be defined by induction in the \flw way.
\begin{itemize}
\item $\Hdim\,\gA=-1$ \ssi  $1_\gA=0_\gA$.
\item For $\ell\geq 0$, $\Hdim\,\gA\leq \ell$ \ssi  for all 
$x\in\gA$,
$\Hdim(\gA/\rH_\gA(x))\leq \ell-1$.
\end{itemize}
Let us describe more explicitly this \dfn.
We introduce the \textsl{iterated  Heitmann boundary \ids} 
$\rH[\gA;x_0,\ldots,x_k]$.
For $x_1,\ldots ,x_n\in\gA$ let the \flw Heitmann boundary \ris be  
\[
\gA_\rH[x_0]=\gA_\rH^{x_0},\,
\gA_\rH[x_0,x_1]=(\gA_\rH^{x_0})_\rH^{x_1},\,
\gA_\rH[x_0,x_1,x_2]=((\gA_\rH^{x_0})_\rH^{x_1})_\rH^{x_2}, 
\hbox{ etc}.  
\]
The kernel of the canonical projection 
$\gA\to \gA_\rH[x_0,\ldots ,x_k]$ is $\rH[\gA;x_0,\ldots,x_k]=\rH_\gA[x_0,\ldots ,x_k]$.  
To give a good description of these \ids we use the  notation 
\[\bidule{z,x,a,y,b}=1+(1+(z+ax)xy)b.\]
So we have:
\begin{itemize}
\item $z\in\rH_\gA[x_0]$ \ssi  \[\Ex a_0  \ \Tt y_0 \ \Ex
b_0,\;\bidule{z,x_0,a_0,y_0,b_0}=0\,;\]
\item  $z\in\rH_\gA[x_0,x_1]$ \ssi 
 \[\Ex a_1  \ \Tt y_1 \ \Ex b_1\ 
\Ex a_0  \
\Tt y_0 \ \Ex b_0, 
\;\bidule{\bidule{z,x_1,a_1,y_1,b_1},x_0,a_0,y_0,b_0}=0\,;\]
\item  $z\in\rH_\gA[x_0,x_1,x_2]$ \ssi 
\[\Ex a_2  \ \Tt y_2 \ \Ex b_2 \ \Ex a_1  \ \Tt y_1 \ \Ex b_1 \ \Ex 
a_0  \ \Tt
y_0 \ \Ex
b_0,\;\bidule{\bidule{\bidule{z,x_2,a_2,y_2,b_2},x_1,a_1,y_1,b_1},x_0,a_0,y_0,
b_0}=0\,.\]
\end{itemize}
And so on. And the Heitmann dimension  is $\leq \ell$ \ssi  
for all
$x_0,\ldots ,x_\ell\in\gA$ we have $1\in\rH_\gA[x_0,\ldots ,x_\ell]$.
\eoe

\begin{proposition}
\label{propHei2} Let $\fj=\gen{j_1,\ldots ,j_n}$ be a \itf and
$\fJ=\JA(\fj)=\JA(j_1)\vu\cdots \vu\JA(j_n)$.  Then
$\Heit(\gA/\rH_\gA(\fj))$ is identified  with  a quotient \ri
of $(\HeA)\ul{\fJ}$. The two \ris are equal when  $\HeA$ is a \agH.
In \clama this is the case when $\Jspec\,\gA$ is \noe.
\end{proposition}
\begin{proof}
The proof has been given for an arbitrary \trdi replacing $\ZarA$
(Proposition~\ref{propBHeitHeyt}).
\end{proof}

\rem We have already proved (for \trdis) the \flw results:
\begin{itemize}
\item we have always  $\Hdim\,\gA\leq \Jdim\,\gA\leq 
\Kdim(\gA/\JA(0))$;
\item if $(\fa_i)_{1\leq i\leq m}$ is a finite family of \ids of $\gA$ 
and
$\fa=\bigcap_{i=1}^m\fa_i$, then
$\Hdim(\gA/\fa)=\sup_i\Hdim(\gA/\fa_i)$;
\item if $\HeA$ is a \agH\footnote{In particular in \clama if  $\HeA$ is 
\noe.}
we have $\Hdim\,\gA=\Jdim\,\gA$;
\item (in \clama) if $\Max\,\gA$ is \noe, then  
$\jspec\,\gA=\Jspec\,\gA$ and 
$\Hdim\,\gA=\Jdim\,\gA=\jdim\,\gA$;
\item $\Hdim\,\gA\leq 0\Leftrightarrow\Jdim\,\gA\leq 0 \Leftrightarrow\Kdim(\gA/\JA(0))\leq 0$.
\eoe
\end{itemize}

\medskip Note that the lattice  $\HeA$ is a \agH \ssi  we have the \flw \prt:
\[\Tt \fa,\fb\in\Heit\,\gA\;\Ex\fc\in\Heit\,\gA\;(\fc\fb\subseteq\fa\;\mathrm{and}\;
\Tt x\in\gA\;(x\fb\subseteq\fa\Rightarrow x\in\fc))\]
($\fa=\JA(a_1,\ldots ,a_n)$, $\fb=\JA(b_1,\ldots ,b_m)$, 
$\fc=\JA(c_1,\ldots
,c_\ell)$).

\addcontentsline{toc}{section}{References}
\markboth{References}{References}

\small

\bibliographystyle{plainnat}

\normalsize

\normalsize
\stopcontents[english]
\endgroup

%% file: EnglishTheoremsHeitmann.tex


\theoremstyle{plain}
\newtheorem{theorem}{Theorem}[subsection]
\newtheorem{thdef}[theorem]{Theorem and definition}
\newtheorem{lemma}[theorem]{Lemma}
\newtheorem{corollary}[theorem]{Corollary}
\newtheorem{proposition}[theorem]{Proposition}
\newtheorem{propdef}[theorem]{Proposition and definition}
\newtheorem{plcc}[theorem]{Concrete local-global principle}
\newtheorem{fact}[theorem]{Fact}
\newtheorem{valsatz}[theorem]{\vst}

\newtheorem{theoremc}[theorem]{Theorem\etoz}
\newtheorem{lemmac}[theorem]{Lemma\etoz}
\newtheorem{corollaryc}[theorem]{Corollary\etoz}
\newtheorem{propositionc}[theorem]{Proposition\etoz}
\newtheorem{factc}[theorem]{Fact\etoz}
\newtheorem*{Principleofcoveringbyquotients}{Principle of covering by quotients}

\theoremstyle{definition}
\newtheorem{conjecture}[theorem]{Conjecture}
\newtheorem{definition}[theorem]{Definition}
\newtheorem{definitions}[theorem]{Definitions}
\newtheorem{notation}[theorem]{Notation}
\newtheorem{definota}[theorem]{Definition and notation} 
\newtheorem{convention}[theorem]{Convention}
\newtheorem{problem}[theorem]{Problem}
\newtheorem{question}[theorem]{Question}

\theoremstyle{remark}
\newtheorem{remark}[theorem]{Remark}
\newtheorem{remarks}[theorem]{Remarks}
\newtheorem{comment}[theorem]{Comment}
\newtheorem{comments}[theorem]{Comments}
\newtheorem{example}[theorem]{Example}
\newtheorem{examples}[theorem]{Examples}

%% file: EnglishMacrosHeitmann.tex

\newcommand {\rem}{\noindent \textsl{Remark.} }
\newcommand {\rems}{\noindent \textsl{Remarks.}  }
\newcommand {\comm}{\noindent \textsl{Comment.}  }
\newcommand \exl{\noindent \textsl{Example.} }

\newcommand\gui[1]{``{#1}''}

\newcommand \thref[1] {Theorem~\ref{#1}}
\newcommand \paref[1] {page~\pageref{#1}}
\newcommand \pstfref[1] {Positivstellensatz formel~\ref{#1}}
\newcommand \pstref[1] {Positivstellensatz~\ref{#1}}

\newenvironment{Proof}[1]{
\trivlist \item[\hskip \labelsep{\sl #1.}]\hskip 0pt\\}{\hfill 
\mbox{$\Box$}
\endtrivlist}

\newcommand\subsubsec[1] {\subsubsection*{#1}}

\newcommand \recu {induction\xspace} 
\newcommand \hdr {induction hypothesis\xspace}
\newcommand \ssi {if and only if\xspace}
\newcommand \cnes {necessary and sufficient condition\xspace}
\newcommand \spdg {without loss of generality\xspace}
\newcommand \Propeq {The following properties are equivalent.\xspace}
\newcommand \propeq {the following properties are equivalent.\xspace}
\newcommand \disept {17$^{th}$ Hilbert's problem\xspace}

\newcommand \cad  {{i.e.}\xspace}
\newcommand \cade {{i.e.}\ also\xspace}

\newcommand \Vrai {\mathsf{True}}
\newcommand \Faux {\mathsf{False}}


\newcommand \Amo {$\gA$-module\xspace}
\newcommand \Amos {$\gA$-modules\xspace}

\newcommand \Bmo {$\gB$-module\xspace}
\newcommand \Bmos {$\gB$-modules\xspace}

\newcommand \Zmo {$\gZ$-module\xspace}
\newcommand \Zmos {$\gZ$-modules\xspace}

\newcommand \ZZmo {$\ZZ$-module\xspace}
\newcommand \ZZmos {$\ZZ$-modules\xspace}

\newcommand \Ali {$\gA$-\ali}
\newcommand \Alis {$\gA$-\alis}

\newcommand \Alg {$\gA$-\alg}
\newcommand \Algs {$\gA$-\algs}

\newcommand \kev {$\gk$-vector space\xspace}
\newcommand \kevs {$\gk$-vector spaces\xspace}

\newcommand \Kev {$\gK$-vector space\xspace}
\newcommand \Kevs {$\gK$-vector spaces\xspace}

\newcommand \klg {$\gk$-\alg}
\newcommand \klgs {$\gk$-\algs}

\newcommand \Klg {$\gK$-\alg}
\newcommand \Klgs {$\gK$-\algs}

\newcommand \ac {algebraically closed\xspace}
\newcommand \alc {\agq closure\xspace}

\newcommand \adv {valuation domain\xspace}
\newcommand \advs {valuation domains\xspace}

\newcommand \agB {Boolean algebra\xspace} 
\newcommand \agBs {Boolean algebras\xspace} 

\newcommand \agH {Heyting algebra\xspace} 
\newcommand \agHs {Heyting algebras\xspace}

\newcommand \arv {valuation ring\xspace}
\newcommand \arvs {valuation rings\xspace}

\newcommand \agq {algebraic\xspace}

\newcommand \alg {algebra\xspace}
\newcommand \algs {algebras\xspace}

\newcommand \algo{algorithm\xspace}
\newcommand \algos{algorithms\xspace}

\newcommand \algq{algorithmic\xspace}

\newcommand \ali {\lin map\xspace}
\newcommand \alis {\lin maps\xspace}

\newcommand \anar {\ari \ri}
\newcommand \anars {\ari \ris}
\newcommand \Anars {\Ari \ris}

\newcommand \ari{arith\-metic\xspace}

\newcommand \auto {automorphism\xspace}
\newcommand \autos {automorphisms\xspace}


\newcommand \cac {algebraically closed field\xspace}
\newcommand \cacs {algebraically closed fields\xspace}

\newcommand \cara{characteristic\xspace}  

\newcommand \carn{characterization\xspace}  
\newcommand \carns{characterizations\xspace}  

\newcommand \cdi{discrete field\xspace}  
\newcommand \cdis{discrete fields\xspace}  

\newcommand \cdf{fraction field\xspace}
\newcommand \cdfs{fraction fields\xspace}

\newcommand \cdpfv{constructive point-free dual version\xspace}
 
\newcommand \cli{integral closure\xspace}  
\newcommand \clis{integral closures\xspace}  

\newcommand \codi {discrete ordered field\xspace}
\newcommand \codis {discrete ordered fields\xspace}

\newcommand \coe {coefficient\xspace}
\newcommand \coes {coefficients\xspace}

\newcommand \cof {\cov}
\newcommand \cofs {\cov}
\newcommand \covs {\cov}

\newcommand \coh {coherent\xspace}

\newcommand \coli {linear combination\xspace}
\newcommand \colis {linear combinations\xspace}

\newcommand \com {comaximal\xspace}

\newcommand \coo {coordinate\xspace}
\newcommand \coos {coordinates\xspace}

\newcommand \cop {complementary\xspace}

\newcommand \cori {commutative \ri}
\newcommand \coris {commutative \ris}

\newcommand \cosv {conservative\xspace}

\newcommand \cvd {valued discrete field\xspace}
\newcommand \cvds {valued discrete fields\xspace}

\newcommand \cvdsc {separably closed valued discrete field\xspace}
\newcommand \cvdscs {separably closed valued discrete fields\xspace}

\newcommand \cvdac {algebraicalle closed valued discrete field\xspace}
\newcommand \cvdacs {algebraicalle closed valued discrete fields\xspace}


\newcommand \dcd {residually discrete\xspace}

\newcommand \ddp {Pr\"ufer domain\xspace}
\newcommand \ddps {Pr\"ufer domains\xspace}

\newcommand \ddk {Krull dimension\xspace}

\newcommand \demo {proof\xspace}
\newcommand \dems {proofs\xspace}
\newcommand \demos {\dems}

\newcommand \dfn{definition\xspace}  
\newcommand \Dfn{Definition\xspace}  
\newcommand \Dfns{Definitions\xspace}  
\newcommand \dfns{definitions\xspace}  

\newcommand \dij{disjunctive\xspace}
\newcommand \wdij{weakly \dij}

\newcommand \discri{discriminant\xspace}
\newcommand \discris{discriminants\xspace}

\newcommand \dok {Dedekind domain\xspace}
\newcommand \doks {Dedekind domains\xspace}

\newcommand \dve {divisibility\xspace}

\newcommand \dvz {zerodivisor\xspace}
\newcommand \dvzs {zerodivisors\xspace}

\newcommand \eco{\com \elts}  

\newcommand \egmt{also\xspace} 

\newcommand \egt{equality\xspace} 
\newcommand \egts{equalities\xspace} 

\newcommand \elr{elementary\xspace}  

\newcommand \elt{element\xspace}  
\newcommand \elts{elements\xspace}  

\def \endo {endomorphism\xspace}
\def \endos {endomorphisms\xspace}

\newcommand \entrel {entailment relation\xspace}
\newcommand \entrels {entailment relations\xspace}

\newcommand \eqn  {equation\xspace}
\newcommand \eqns  {equations\xspace}

\newcommand \eqv  {equivalent\xspace}

\newcommand \eqvc  {equivalence\xspace}
\newcommand \eqvcs  {equivalences\xspace}

\newcommand \eseq{essentially equivalent\xspace} 
\newcommand \Eseq{Essentially equivalent\xspace} 

\newcommand \esid{essentially identical\xspace} 
\newcommand \Esid{Essentially identical\xspace} 

\newcommand \evc{vector space\xspace} 
\newcommand \evcs{vector spaces\xspace} 


\newcommand \fab {bounded \fcn}
\newcommand \fabs {bounded \fcns}

\newcommand \fac {total \fcn}

\newcommand \facile{\begin{proof}
Left to the reader.
\end{proof}}

\newcommand \fap {partial \fcn}
\newcommand \faps {partial \fcns}

\newcommand \fcn {factorization\xspace}
\newcommand \fcns {factorizations\xspace}

\newcommand \fdi{strongly discrete\xspace} 

\newcommand \flw{following\xspace} 


\newcommand\gmq{geometric\xspace}

\newcommand\gne{generalised\xspace}

\newcommand\gnl{general\xspace}

\newcommand\gnlt{generally\xspace}

\newcommand\gnn{generalization\xspace}
\newcommand\gnns{generalizations\xspace}

\newcommand\gnq{generic\xspace}

\newcommand\grl{$\ell$-group\xspace}
\newcommand\grls{$\ell$-groups\xspace}

\newcommand \gtr{generator\xspace}  
\newcommand \gtrs{generators\xspace}  


\newcommand \homeo {homeomorphism\xspace}
\newcommand \homeos {homeomorphisms\xspace}
\newcommand \homeoc {homeomorphic\xspace}
\newcommand \homeocs {homeomorphic\xspace}

\newcommand \homo {homomorphism\xspace}
\newcommand \homos {homomorphisms\xspace}

\newcommand \id {ideal\xspace}
\newcommand \ids {ideals\xspace}

\newcommand \idd {de\-ter\-mi\-nantal \id}
\newcommand \idds {de\-ter\-mi\-nantal \ids}

\newcommand \idema {maximal \id}
\newcommand \idemas {maximal \ids}

\newcommand \idep {prime \id}
\newcommand \ideps {prime \ids}

\newcommand \idemi {minimal prime\xspace}
\newcommand \idemis {minimal primes\xspace}

\newcommand \idf {Fitting \id}
\newcommand \idfs {Fitting \ids}

\newcommand \idm {idempotent\xspace}
\newcommand \idms {idempotents\xspace}

\newcommand \idp {principal \id}
\newcommand \idps {principal \ids}

\newcommand \idtr {indeterminate\xspace}
\newcommand \idtrs {indeterminates\xspace}

\newcommand \ifr {fractional \id}
\newcommand \ifrs {fractional \ids}

\newcommand \imd {immediate\xspace}
\newcommand \imdt {immediately\xspace}

\newcommand \inteq {intuitively \eqv}

\newcommand \ird {irreducible\xspace}

\newcommand \itf {\tf \id}
\newcommand \itfs {\tf \ids}

\newcommand \iso {isomorphism\xspace}
\newcommand \isos {isomorphisms\xspace}

\newcommand \isoc {isomorphic\xspace}
\newcommand \isocs {isomorphic\xspace}

\newcommand \iv {invertible\xspace}

\newcommand \lec {reader\xspace}

\newcommand \lgb {local global\xspace}

\newcommand \lin {linear\xspace}

\newcommand \llec {the reader\xspace}

\newcommand \lon {localization\xspace}
\newcommand \lons {localizations\xspace}

\newcommand \lop {\lot principal\xspace}

\newcommand \losd {\lot \sdz\xspace}

\def \lot {locally\xspace}

\newcommand \mlp {principal \lon matrix\xspace}
\newcommand \mlps {principal \lon matrices\xspace}

\newcommand \mnp {manipulation\xspace}
\newcommand \mnps {manipulations\xspace}
\newcommand \mnr {\elr \mnp}
\newcommand \mnrs {\elr \mnps}

\newcommand \mo {monoid\xspace}
\newcommand \mos {monoids\xspace}
\newcommand \moco {\com \mos}

\newcommand \mpf {\pf module\xspace}
\newcommand \mpfs {\pf modules\xspace}

\newcommand \mpn {\pn matrix\xspace}
\newcommand \mpns {\pn matrices\xspace}

\newcommand \mpr {\pro module\xspace}
\newcommand \mprs {\pro modules\xspace}

\newcommand \mprn {\prn matrix\xspace}
\newcommand \mprns {\prn matrices\xspace}

\newcommand \mptf {\ptf module\xspace}
\newcommand \mptfs {\ptf modules\xspace}

\newcommand \mrc {projective module of constant rank\xspace}
\newcommand \mrcs {projective modules of constant rank\xspace}


\newcommand \ncr{necessary\xspace}

\newcommand \ncrt{necessarily\xspace}

\newcommand \ndz {regular\xspace}

\newcommand \noe {Noetherian\xspace}
\newcommand \noco {\noe\coh}

\newcommand \nst {Nullstellensatz\xspace}
\newcommand \nsts {Nullstellens\"atze\xspace}

\newcommand \odz {Zariski open set\xspace}

\newcommand \oqc {\qc open subset\xspace}
\newcommand \oqcs {\qc open subsets\xspace}

\newcommand \ort {orthogonal\xspace}


\newcommand \pa {saturated pair\xspace}
\newcommand \pas {saturated pairs\xspace}

\newcommand \pb{problem\xspace}  
\newcommand \pbs{problems\xspace}

\newcommand \peq {purely equational\xspace}

\newcommand \pf {finitely presented\xspace}

\newcommand \plg {\lgb principle\xspace}
\newcommand \plgs {\lgb principles\xspace}

\newcommand \plga {abstract \plg}
\newcommand \plgas {abstract \plgs}

\newcommand \Plgc {Concrete \plg}
\newcommand \plgc {concrete \plg}
\newcommand \plgcs {concrete \plgs}

\newcommand \pn {presentation\xspace}
\newcommand \pns {presentations\xspace}

\newcommand \pol {polynomial\xspace}
\newcommand \pols {polynomials\xspace}

\newcommand \polcar {characteristic \pol}

\newcommand \prc {rank constant \pro}

\newcommand \prmt {precisely\xspace}
\newcommand \Prmt {Precisely\xspace}

\newcommand \prn {projection\xspace}
\newcommand \prns {projections\xspace}

\newcommand \pro {projective\xspace}

\newcommand \proi {potential prime\xspace}
\newcommand \prois {potential primes\xspace}

\newcommand \proc {potential chain\xspace}
\newcommand \procs {potential chains\xspace}

\newcommand \proel {elementary \proc}
\newcommand \proels {elementary \procs}
\newcommand \proelo {\proel of length }
\newcommand \proelos {\proels of length }

\newcommand \prolo {\proc of length }
\newcommand \prolos {\procs of length }

\newcommand \prt {property\xspace}
\newcommand \prts {properties\xspace}

\newcommand \pst {Positivstellensatz\xspace}
\newcommand \psts {Positivstellens\"atze\xspace}

\newcommand \ptf {\tf \pro}


\newcommand \qc {quasi-compact\xspace}

\newcommand \qiri {pp-ring\xspace}
\newcommand \qiris {pp-rings\xspace}

\newcommand \ralg {Horn rule\xspace}
\newcommand \ralgs {Horn rules\xspace}

\newcommand \rcf {real closed field\xspace}
\newcommand \rcfs {real closed fields\xspace}

\newcommand \rdl {linear dependance relation\xspace}
\newcommand \rdls {linear dependance relations\xspace}

\newcommand \rdi {integral dependance relation\xspace}
\newcommand \rdis {integral dependance relations\xspace}

\newcommand \rdij {\dij rule\xspace}
\newcommand \rdijs {\dij rules\xspace}

\newcommand \rdv {valuative divisibility relation\xspace}
\newcommand \rdvs {valuative divisibility relations\xspace}

\newcommand \rdy {dynamical rule\xspace}
\newcommand \rdys {dynamical rules\xspace}

\newcommand \red {direct rule\xspace}
\newcommand \reds {direct rules\xspace}

\newcommand \rex {existential rule\xspace}
\newcommand \rexs {existential rules\xspace}

\newcommand \ri {ring\xspace}
\newcommand \ris {rings\xspace}


\newcommand \sad {dynamical algebraic structure\xspace}
\newcommand \sads {dynamical algebraic structures\xspace}
\newcommand \SAD {Dynamical algebraic structure\xspace}
\newcommand \SADs {Dynamical algebraic structures\xspace}

\newcommand \salg {algebraic structure\xspace}
\newcommand \salgs {algebraic structures\xspace}

\newcommand \sdz {without \dvz}

\newcommand \sfio {fundamental system of orthogonal idempotents\xspace}

\newcommand \sgr {\gtr set\xspace}
\newcommand \sgrs {\gtr sets\xspace}

\newcommand \sli {\lin \sys}
\newcommand \slis {\lin \syss}

\newcommand \spb {separable\xspace}
\newcommand \spl {separable\xspace}

\newcommand \sps {spectral space\xspace}
\newcommand \ssps {spectral subspace\xspace}
\newcommand \spss {spectral spaces\xspace}
\newcommand \sspss {spectral subspaces\xspace}

\newcommand \sys {system\xspace}
\newcommand \syss {systems\xspace}

\newcommand \talg {Horn theory\xspace}
\newcommand \talgs {Horn theories\xspace}

\newcommand \tco {coherent theory\xspace}
\newcommand \tcos {coherent theories\xspace}

\newcommand \twdij {\wdij theory\xspace}
\newcommand \twdijs {\wdij theories\xspace}

\newcommand \tdy {dynamical theory\xspace}
\newcommand \tdys {dynamical theories\xspace}

\newcommand \tel {regular theory\xspace}
\newcommand \tels {regular theories\xspace}

\newcommand \telri {cartesian theory\xspace}
\newcommand \telris {cartesian theories\xspace}

\newcommand \tf {finitely generated\xspace}

\newcommand \tfo {formal theory\xspace}
\newcommand \tfos {theory formelles\xspace}

\newcommand \tgm {\gmq theory\xspace}
\newcommand \tgms {\gmq theories\xspace}

\newcommand \Tho {Theorem\xspace}
\newcommand \tho {theorem\xspace}
\newcommand \thos {theorems\xspace}

\newcommand \tpe {purely equational theory\xspace}
\newcommand \tpes {purely equational theories\xspace}

\newcommand \trdi {distributive lattice\xspace} 
\newcommand \trdis {distributive lattices\xspace}

\newcommand \uvl {universal\xspace}

\newcommand \vfn {verification\xspace}
\newcommand \vfns {verifications\xspace}

\newcommand \vst {Valuativstellensatz\xspace}
\newcommand \vsts {Valuativstellensätze\xspace}

\newcommand \zed {zero-dimensional\xspace}
\newcommand \zedr {zero-dimensional reduced\xspace}


\newcommand \cov {constructive\xspace}

\newcommand \coma {\cov \maths}
\newcommand \clama {classical \maths}

\renewcommand \cot {constructively\xspace}

\newcommand \mathe {mathematical\xspace}
\newcommand \maths {mathematics\xspace}

\newcommand \matn {mathematician\xspace}

\newcommand \pte {excluded middle principle\xspace}

\newcommand \prco {\cov proof\xspace}
\newcommand \prcos {constructive proofs\xspace}

\newcommand \tcg {compactness theorem\xspace}
\newcommand \Tcgi {The \tcg implies the following result. }

%% file: frenchHeitmann.tex

\clearpage
\setcounter{section}{0}
\setcounter{equation}{0}

\selectlanguage{french}
\def\frenchproofname{\textsl{Démonstration}}
\FrenchFootnotes

\setlistindentFB
\let\list\listFB
\let\itemize\itemizeFB
\setlabelitemsFB
\let\enumerate\enumerateFB
\let\description\descriptionFB

\input{FrenchTheoremsHeitmann.tex}

\input{FrenchMacrosHeitmann.tex}

\thickmuskip = 7mu plus 2mu

\pagestyle{headings}
\patchcmd{\sectionmark}{\MakeUppercase}{}{}{}

\stMF
\startcontents[french]

\title{Dimension de Heitmann des treillis distributifs\\ et des anneaux commutatifs}
\author{Thierry Coquand, Henri Lombardi , Claude Quitté}

\date{
 \today \\[1em] Version corrigée des 4 premières sections de l'article\\ \citealt*{fCLQ2006}\\[.3em]
\normalsize  
}
\maketitle


\rdb
 \label{beginfrench}

\begin{abstract}
Cet article est une version corrigée des 4 premières sections de l'article
\citealt*{fCLQ2006}. 

Les sections 5 à 7 de l'article original sont traitées de manière un peu plus simple dans \citealt*{fACMC} (et \citealt*{fCACM} pour la version anglaise).

 Nous étudions la notion de dimension introduite par Heitmann dans
son article remarquable \citealt*{fHei84}, ainsi qu'une notion voisine, seulement implicite dans ses démonstrations.  Nous développons ceci d'abord dans le cadre général de la théorie des treillis distributifs et des espaces spectraux.   Nous appliquons ensuite cette problématique dans le
cadre de l'algèbre commutative.
\end{abstract}
\mni MSC 2000: 13C15, 03F65, 13A15, 13E05

\sni Mots clés :  Mathématiques constructives, \trdi, \agH, espace spectral, treillis de Zariski, spectre de Zariski, dimension de Krull, spectre maximal, treillis de Heitmann, spectre de Heitmann, dimensions de Heitmann.

\begin{center}
{\bf \large Avertissement} 
\end{center}

L'article original est paru aux \textsl{Publications mathématiques de Besançon. Algèbre et Théorie des Nombres.}
(2006), pages 57--100.

Nous corrigeons ici un certain nombre d'erreurs, répertoriées en post-scriptum page \pageref{postscriptum} à la fin du texte.
Nous donnons aussi des références bibliographiques supplémentaires.
 
Nous nous conformons à l'orthographe nouvelle recommandée (par exemple: à priori, corolaire, connaitre), et les \lecs et \lecs subissent l'alternance des sexes.

\small

\setcounter{tocdepth}{4}
\markboth{Table des matières}{Table des matières}

\printcontents[french]{}{1}{}
\normalsize

\markboth{Introduction}{Introduction}
\section*{Introduction} \label{sec Introduction}
\addcontentsline{toc}{section}{Introduction}

Nous étudions la notion de dimension introduite par Heitmann dans
son article \citealt*{fHei84}, ainsi qu'une notion voisine,
seulement implicite dans ses preuves.  Nous développons ceci d'abord
dans le cadre général de la théorie des \trdis et des espaces
spectraux.  Nous appliquons ensuite cette problématique dans le
cadre de l'\alg commutative.

Dans la dualité entre \trdis et espaces spectraux, le spectre de
Zariski d'un anneau commutatif correspond (comme l'a indiqué André Joyal dans
\citealt*{fJoy76}) au treillis des \ids qui sont radicaux d'\itfs.   Nous
montrons que l'espace spectral défini par Heitmann pour sa notion de
dimension correspond au treillis formé par les idéaux qui sont
radicaux de Jacobson d'\itfs.   Ceci nous permet d'obtenir une \dfn
\cov \elr de la dimension définie par Heitmann (que nous notons
$\Jdim$).  Nous introduisons une autre dimension, que nous appelons
dimension de Heitmann (et que nous notons $\Hdim$), qui est
\gui{meilleure} en ce sens que $\Hdim\leq \Jdim$ et qu'elle permet des
preuves par \recu naturelles.

Comme conséquences, on trouve dans \citealt*{fACMC} des versions \covs de certains \thos classiques
importants, dans leur version non noethérienne (souvent due
à Heitmann).

Les versions \covs de ce ces \thos s'avèrent en fin de compte plus
simples, et parfois plus générales, que les versions classiques
abstraites correspondantes.

En particulier il y a les versions non noethériennes des
\thos de Swan et de Serre (splitting off) obtenues pour la
première fois dans \citealt*{fCLQ2004} et dans \citealt*{fDuc2006}.

\smallskip
Naturellement, le principal avantage que nous voyons dans notre
traitement est son caractère tout à fait \elr.   En particulier
nous n'utilisons pas d'hypothèses \gui{non nécessaires} comme
l'axiome du choix et le principe du tiers exclu, inévitables pour
faire fonctionner les preuves classiques antérieures.

Enfin, le fait de s'être débarrassé de toute hypothèse
noethérienne est aussi non négligeable et permet de mieux voir
l'essence des choses.

\medskip En conclusion cet article peut être vu pour l'essentiel
comme une mise au point \cov de la théorie des espaces spectraux via
celle des \trdis,  avec une insistance particulière sur la dimension
de Heitmann, mal connue, qui a pourtant quelques applications marquantes en algèbre commutative.

Dans le texte qui suit les \thos, propositions et lemmes démontrés en \clama sont affectés d'une étoile. On signale de cette manière que la démonstration utilise des principes non \cofs. En général, une démonstration \cov est alors impossible car le résultat sous la forme indiquée implique un principe non \cof (presque toujours une utilisation du tiers exclu). Par exemple en \clama on peut toujours récupérer les points d'un espace spectral à partir du \trdi formé par ses \oqcs, mais ce n'est pas toujours possible d'un point de vue \cof.  

\medskip \rem Nous avons résolu de la manière suivante un
problème de terminologie qui se pose en rédigeant cet article.  Le
mot \gui{dualité} apparait à priori dans le contexte des \trdis
avec deux significations différentes.  Il y a d'une part la
dualité qui correspond au renversement de la relation d'ordre dans
un treillis.  D'autre part il y a la dualité entre treillis
distributifs et espaces spectraux, qui correspond à une
antiéquivalence de catégories.  Nous avons décidé de réserver
\gui{dualité} pour ce dernier usage.  Le terme \gui{treillis dual} a
donc été systématiquement remplacé par \gui{treillis
opposé}.  De même on a remplacé \gui{la notion duale} par
\gui{la notion renversée} ou par \gui{la notion opposée}, et
\gui{par dualité} par \gui{par renversement de l'ordre}.\eoe

\section{Treillis distributifs}
\label{fsecTRDI}

Les axiomes des \trdis peuvent être formulés avec des \egts
universelles concernant uniquement les deux lois $\vi$ et $\vu$ et les
deux constantes $0_\gT$ (l'\elt minimum du \trdi $\gT$) et $1_\gT$ (le
maximum).  La relation d'ordre est alors définie par $a\leq_\gT
b\;\Leftrightarrow\;a\vi b=a$.  On obtient ainsi une théorie
purement équationnelle, avec toutes les facilités afférentes. 
Par exemple on peut définir un \trdi par générateurs et
relations, la catégorie comporte des limites inductives (qu'on peut
définir par générateurs et relations) et des limites projectives
(qui ont pour ensembles sous-jacents les limites projectives
ensemblistes correspondantes).

Un ensemble totalement ordonné est un \trdi s'il possède un
maximum et un minimum.  On note ${\bf n}$ un ensemble totalement
ordonné à $n$ éléments, c'est un \trdi si $n\neq 0$.  Le
treillis $\Deux$ est le \trdi libre à 0 \gtr,  et $\Trois$ celui à
un \gtr. 

Pour tout \trdi $\gT$, si l'on remplace la relation d'ordre $x\leq_\gT
y$ par la relation symétrique $y\leq_\gT x$ on obtient le
\textsl{treillis opposé} $\gT\cir$ avec échange de $\vi$ et $\vu$
(on dit parfois \textsl{treillis dual}).

\subsection{Idéaux, filtres}

Si $\varphi :\gT\rightarrow \gT'$ est un morphisme de \trdis,  
$\varphi^{-1}(0)$
est appelé un \textsl{\id de $\gT$}. Un \id $\fII $ de $\gT$ est une 
partie de
$\gT$ soumise
aux  contraintes suivantes:
\begin{equation} \label{feqIdeal}
\left.
\begin{array}{rcl}
   & &  0 \in \fII    \\
x,y\in \fII & \Longrightarrow   &  x\vu y \in \fII    \\
x\in \fII ,\; z\in \gT& \Longrightarrow   &  x\vi z \in \fII    \\
\end{array}
\right\}
\end{equation}
(la dernière peut se réécrire $(x\in \fII ,\;y\leq x)\Rightarrow y\in 
\fII $).
Un \textsl{\id principal} est un \id engendré par un seul \elt $a$:
il est égal à
\begin{equation} \label{feqda}
\,\dar a=\sotq{x\in \gT}{x\leq a}.
\end{equation}
L'\id $\,\dar a$, muni des lois $\vi$ et $\vu$ de $\gT$, est un \trdi
dans lequel l'\elt maximum est~$a$.  L'injection canonique $\,\dar
a\rightarrow \gT$ \textsl{n'est pas} un morphisme de \trdis parce que
l'image de $a$ n'est pas égale à $1_\gT$.  Par contre
l'application surjective $\gT\rightarrow \,\dar a,\;x\mapsto x\vi a$
est un morphisme surjectif, qui munit $\dar a$ d'une
structure quotient.

\smallskip La notion opposée à celle d'\id est la notion de {\sl filtre}.  Le filtre principal engendré par $a$ est noté $\uar a$.

\smallskip L'\textsl{\id engendré} par une partie $J$ de $\gT$ est
$\cI_\gT(J)=\sotq{x\in\gT}{\Ex J_0\in \Pf(J),\,x\leq \Vu J_0}$. 
En conséquence \textsl{tout \itf est principal}.

Si $A$ et $B$ sont deux parties de $\gT$ on note
\begin{equation} \label{feqvuvi}
A\vu B=\sotq{a\vu b}{a\in A,\,b\in B}  \quad \mathrm{et}\quad  A\vi
B=\sotq{a\vi b}{a\in A,\,b\in B}.
\end{equation}

Alors l'\id engendré par deux \ids $\fa$ et $\fb$ est égal à
\begin{equation} \label{feqSupId}
\cI_\gT(\fa\cup \fb) = \fa\vu\fb =\sotq{z}{\exists 
x\in\fa,\,\exists
y\in\fb,\,z\leq x\vu y}\,.
\end{equation}

L'ensemble $\Idl(\gT)$ des \ids de $\gT$\footnote{En fait,
il faut introduire une restriction pour obtenir vraiment un ensemble, de façon à avoir un procédé bien défini de construction des \ids concernés.
Par exemple on peut considérer l'ensemble des idéaux obtenus à partir des idéaux principaux par itération certaines opérations prédéfinies, comme les réunions et intersections dénombrables.} forme lui même un \trdi pour
l'inclusion, avec pour borne inrérieure de $\fa$ et $\fb$ l'\id:
\begin{equation} \label{feqInfId}
\fa\cap \fb=\fa\vi\fb.
\end{equation}

Ainsi les opérations $\vu$ et $\vi$ définies en (\ref{feqvuvi})
correspondent au sup et au inf dans le treillis des \ids. 

On notera $\cF_\gT(S)$ le filtre de $\gT$ engendré par le sous
ensemble $S$.  
Quand on considère le treillis des filtres il faut
faire attention à ce que produit le renversement de la relation
d'ordre: $\ff\cap\ffg=\ff\vu\ffg$ est le inf de $\ff$ et $\ffg$,
tandis que leur sup est égal à $\cF_\gT(\ff\cup \ffg)=\ff\vi
\ffg$.

\medskip Le \textsl{treillis quotient de $\gT$ par l'\id $\fJ$}, noté
$\gT/(\fJ=0)$ est défini comme le \trdi engendré par les \elts de
$\gT$ avec pour relations: les relations vraies dans $\gT$ d'une part,
et les relations $x=0$ pour les $x\in \fJ$ d'autre part.  Il peut
aussi être défini par la relation de préordre
\[ a\preceq b\quad\Longleftrightarrow\quad a\leq_{\gT/(\fJ=0)}b 
\equidef \quad
\exists x\in \fJ \;\;a\leq  x\vu b
\]
Ceci donne
\[ a\equiv b\;\;\mod\;(\fJ=0)\quad \Longleftrightarrow  \quad \exists 
x\in \fJ
\;\;a\vu x=b\vu x
\]
et dans le cas du quotient par un \id principal $\,\dar a$ on obtient
$\gT/(a=0)\simeq\,\uar a$ avec le morphisme $y\mapsto y\vu a$ de $\gT$ 
vers
$\,\uar a$.

\subsubsection*{Transporteur, différence}
\addcontentsline{toc}{subsubsection}{Transporteur, différence}

Par analogie avec l'\alg commutative, si $\fb$  est un \id et  $A$ 
une partie de
$\gT$ on  notera
\begin{equation} \label{feqTrans}
\fb:A\eqdefi\sotq{x\in\gT}{\Tt a\in A\quad a\vi x\in  \fb}
\end{equation}
Si $\fa$ est l'\id engendré par $A$ on a $\fb:A=\fb:\fa$,
on l'appelle le \textsl{transporteur de $\fa$ dans $\fb$.}

On note aussi
$b:a$ l'\id
$ (\dar b):(\dar a)=\sotq{x\in\gT}{x\in\gT\;|\;x\vi a\leq  b}$.

La notion opposée est celle de \textsl{filtre différence de deux 
filtres}
\begin{equation} \label{feqDiff}
\ff\setminus\ff'\eqdefi\sotq{x\in\gT}{\Tt a\in \ff'\quad a\vu 
x\in  \ff}
\end{equation}
On note aussi
$b\setminus a$ le filtre
$ (\uar b)\setminus(\uar a)=\sotq{x\in\gT}{b\leq  x\vu a}$.

\subsubsection*{Radical de Jacobson}
\addcontentsline{toc}{subsubsection}{Radical de Jacobson}

Un \id $\fm$ d'un \trdi $\gT$ non trivial (i.e. distinct de $\Un$) est
dit \textsl{maximal} \hbox{si $\gT/(\fm=0)\,=\,\Deux$}, \cad si $1\notin\fm$ et
$\Tt x\in\gT\;(x\in\fm$ ou $\Ex y\in \fm \;x\vu y=1)$.

Il revient au même de dire qu'il s'agit d'un \id \gui{maximal
parmi les \ids stricts}.

En \clama on a le lemme suivant.
\begin{flemmac}
\label{flemHspec1}
Dans un \trdi $\gT\neq\Un$ l'intersection des \idemas est égale à 
l'\id
\[ \sotq{a\in\gT}{\forall x\in\gT \;(a\vu x = 1 \Rightarrow  
x=1)}.
\]
On l'appelle le \emph{radical de Jacobson de $\gT$}. On le note 
$\JT(0)$. \\
Plus généralement l'intersection des \idemas contenant un \id 
strict $\fJ$
est égale à l'\id
\begin{equation} \label{feqRJJ}
\JT(\fJ)\,=\, \sotq{a\in\gT}{\forall x\in\gT \;(a\vu x  = 1 
\Rightarrow \Ex
z\in \fJ \;\;z\vu x=1)}
\end{equation}
On l'appelle le \emph{radical de Jacobson de l'\id $\fJ$}. En 
particulier:
\begin{equation} \label{feqRJb}
\JT(\dar b)=\sotq{a\in\gT}{\forall x\in\gT \;(\,a\vu x = 1\; 
\Rightarrow
\;\;b\vu x=1\,)}
\end{equation}
\end{flemmac}
\begin{proof}
La deuxième affirmation résulte de la première en passant au
treillis quotient $\gT/(\fJ=0)$.  Voyons la première.  On montre que
$a$ est en dehors d'au moins un \idema \ssi $\Ex x\neq 1$ tel que
$a\vu x=1$.  Si c'est le cas, un \idema qui contient $x$ (il en existe
puisque $x\neq 1$) ne peut pas contenir $a$ car il contiendrait $a\vu
x$.  Inversement, si $\fm$ est un \idema ne contenant pas $a$, l'\id
engendré par $\fm$ et $a$ contient $1$.  Or cet \id est l'ensemble
des \elts majorés par au moins un $a\vu x$ où $x$ parcourt $\fm$.
\end{proof}

En \clama un \trdi est appelé \textsl{treillis de Jacobson} si tout
\idep est égal à son radical de Jacobson.  Comme tout \id est
l'intersection des \ideps qui le contiennent, cela implique que tout \id
est égal à son radical de Jacobson.

En \coma on  adopte les \dfns suivantes.

\medskip 
\begin{fdefinitions}[radical de Jacobson, treillis faiblement 
Jacobson]
\label{fdefJac} ~
\begin{enumerate}
\item Si $\fJ$ est un \id de $\gT$  son \textsl{radical de Jacobson} 
est défini
par l'\egt $(\ref{feqRJJ})$
(on ne fait pas l'hypothèse que  $\gT\neq\Un$).
On notera $\JT(a)$ pour  $\JT(\dar a)$.
\item Un \trdi est appelé un \textsl{treillis faiblement Jacobson} si 
tout \id
principal est égal à son radical de Jacobson, \cade si
\begin{equation} \label{feqTJac}
\Tt a,b\in\gT\;\;[\,(\forall x\in\gT\, (a\vu 
x=1 \;\Rightarrow\;b\vu 
x=1))\;\Rightarrow
\;a\leq b\,]
\end{equation}
\end{enumerate}
\end{fdefinitions}

On vérifie sans difficulté que $\JT(\fJ)$ est un \id et que
$1\in\JT(\fJ)\Leftrightarrow 1\in\fJ$.

\hum{Cela serait bien d'avoir une formulation \cov
  sans quantification sur \gui{l'ensemble} des \ids
  pour la notion de treillis
  de Jacobson. Encore que la notion la plus pertiente
  semble être celle de treillis faiblement Jacobson
  puisque le $\jspec$ doit être remplacé par le $\Jspec$}

\subsection{Quotients}

Un \textsl{\trdi quotient $\gT'$ de $\gT$} est donné par une relation 
binaire
$\preceq$ sur $\gT$ vérifiant les propriétés suivantes:
\begin{equation} \label{feqPreceq}
\left.
\begin{array}{rcl}
a\leq b&  \Longrightarrow  & a\preceq b   \\
a\preceq b,\,b\preceq c&  \Longrightarrow  & a\preceq c   \\
a\preceq b,\,a\preceq c&  \Longrightarrow  & a\preceq b\vi c   \\
b\preceq a,\,c\preceq a&  \Longrightarrow  & b\vu c\preceq a
\end{array}
\right\}
\end{equation}
L'\egt $a=_{\gT'}b$ (ou $a\approx b$) est définie par $a\approx b \iff a\preceq b \hbox{ et } b\preceq a$.

\begin{fproposition}[quotients d'un type particulier]
\label{fpropIdealFiltre} Soit $\gT$ un \trdi et
$(J,U)$ un couple de parties de $\gT$.
On considère le quotient $\gT'$ de $\gT$ défini par les
relations $x=0$ pour les $x\in J$ et $y=1$ pour les $y\in U$. Alors
  on a $a\leq_{\gT'}b$ \ssi
il existe une partie finie $J_0$ de $J$ et une partie finie $U_0$ de
$U$ telles que:
\begin{equation} \label{feqpropIdealFiltre}
a \vi \Vi U_0 \; \leq_\gT\; b \vu \Vu J_0
\end{equation}
Nous noterons $\gT/(J=0,U=1)$ ce treillis quotient $\gT'.$
\end{fproposition}

\subsubsection*{Idéaux dans un quotient}
\addcontentsline{toc}{subsubsection}{Idéaux dans un quotient}

Le fait suivant résulte des \egts (\ref{feqIdeal}), (\ref{feqvuvi}),
(\ref{feqSupId}) et (\ref{feqInfId}).
\begin{ffact}
\label{ffactIdDansQuo}
Soit  $\pi\colon \gT\to\gL$ un treillis quotient. Alors $\fa\mt\pi(\fa)$ définit une fonction de $\Idl(\gT)$ vers~$\Idl(\gL)$ et $\fb\mt\pi^{-1}(\fb)$ définit une fonction de $\Idl(\gL)$ vers~$\Idl(\gT)$.
\begin{itemize}
\item La fonction $\fb\mt\pi^{-1}(\fb)$ est un 
morphisme  pour $\vu$ et $\vi$ (mais pas nécessairement pour~$\so0$).
\item La fonction $\fa\mt\pi(\fa)$ est un \homo surjectif de treillis.
\item Un \id de $\gT$ est de la forme $\pi^{-1}(\fb)$ \ssi il est 
saturé pour
la relation~$\approx$.
\item Résultats analogues pour les filtres.
\end{itemize}
\end{ffact}

Notez qu'en algèbre commutative, le morphisme de passage au 
quotient par un
idéal ne se comporte pas aussi bien pour les \ids 
puisqu'on peut très bien avoir 
$\fa+(\fb\cap\fc)\varsubsetneq
(\fa+\fb)\cap(\fa+\fc)$.

  Le lemme suivant donne quelques renseignements complémentaires 
pour les
quotients par un \id et par un filtre.

\begin{flemma}\label{flemIQT}
Soit $\fa$ un \id et $\ff$ un filtre de $\gT$.
\begin{enumerate}
\item Si $\gL=\gT/(\fa=0)$ alors la projection canonique 
$\pi\colon \gT\to\gL$ établit une bijection croissante entre les \ids  de $\gT$ contenant 
$\fa$ et les \ids de $\gL$. La bijection réciproque est fournie par 
$\fj\mapsto\pi^{-1}(\fj)$. En outre si $\fj$ est un \id de $\gL$, on obtient~\hbox{$\pi^{-1}(\rJ_\gL(\fj)) = \JT(\pi^{-1}(\fj))$}.
\item  Si $\gL=\gT/(\ff=1)$ alors la projection canonique 
$\pi\colon \gT\to\gL$
établit une bijection croissante entre les \ids $\fJ$ de $\gT$ 
vérifiant
\gui{$\Tt f \in \ff,\;\;\fJ:f=\fJ$} et les \ids de~$\gL$.
\end{enumerate}
\end{flemma}

\rem On notera que $\gT\mapsto\JT(0)$ n'est pas une opération 
fonctorielle. La deuxième affirmation du point \textsl{1} du lemme précédent, qui admet 
une preuve \cov directe,  s'explique facilement en \clama par le fait que, dans 
le cas très particulier du quotient par un \id,  les \idemas de $\gL$  contenant $\fj$
correspondent par $\pi^{-1}$ aux \idemas de~$\gT$ contenant~$\pi^{-1}(\fj)$.
\eoe

\subsubsection*{Recollement de treillis quotients}
\addcontentsline{toc}{subsubsection}{Recollement de treillis quotients}

En \alg commutative, si $\fa$ et $\fb$ sont deux \ids d'un anneau 
$\gA$
on a une \gui{suite exacte} de \Amos (avec $j$ et $p$ des \homos 
d'anneaux)
\[0\to\gA/(\fa\cap\fb)\vers{j}(\gA/\fa) \times 
(\gA/\fb)\vers{p}\gA/(\fa+\fb)\to 0
\]
qu'on peut lire en langage courant: le système de congruences  
$x\equiv a\;\mod\;\fa$, $x\equiv b\;\mod\;\fb$ admet une solution \ssi $a\equiv
b\;\mod\;\fa+\fb$ et dans ce cas la solution est unique modulo $\fa\cap\fb$.
Il est remarquable que ce \gui{\tho des restes chinois} se 
généralise à un système \textsl{quelconque} de congruences \ssi l'anneau est
\textsl{arithmétique} \cite[Théorème XII-1.6]{fACMC}, \cad si le treillis des \ids est distributif.
Le \tho des restes chinois \gui{contemporain} concerne le cas particulier d'une
famille d'\ids deux à deux comaximaux, et il fonctionne sans hypothèse sur l'anneau de base.

D'autres épimorphismes de la catégorie des anneaux commutatifs sont les
localisations. Et il y a un principe de recollement analogue au \tho des restes
chinois pour les localisations, extrêmement fécond (le principe 
local-global).

\smallskip De la même manière on peut récupérer un \trdi à partir
d'un nombre fini de ses quotients,
si l'information qu'ils contiennent est \gui{suffisante}. On peut 
voir ceci au choix comme une procédure de recollement (de passage du local au 
global), ou comme une version du \tho des restes chinois pour les \trdis.  Voyons 
les choses plus précisément.

\begin{fdefinition}
\label{fdefRecolTD}
Soit $\gT$ un \trdi,  $(\fa_i)_{i=1,\ldots n}$ (resp. 
$(\ff_i)_{i=1,\ldots n}$)
une famille finie d'\ids (resp. de filtres)  de $\gT$.  On dit que 
les \ids
$\fa_i$ \textsl{recouvrent $\gT$} si $\bigcap_i\fa_i=\so{0}$. De 
même on dit
que les filtres $\ff_i$ \textsl{recouvrent $\gT$} si 
$\bigcap_i\ff_i=\so{1}$.
\end{fdefinition}

Pour un \id $\fb$ nous écrivons $x\equiv y\;\mod\;\fb$ comme 
abréviation
pour  $x\equiv y\;\mod$ \hbox{$(\fb=0)$}.
\begin{ffact}
\label{ffactRecolTD}
Soit $\gT$ un \trdi,  $(\fa_i)_{i=1,\ldots, n}$ une famille finie 
d'\ids principaux \hbox{($\fa_i=\dar s_i$)}  de
$\gT$ et $\fa=\bigcap_i\fa_i$.
\begin{enumerate}
\item Si $(x_i)$ est une famille d'\elts de $\gT$ telle que pour 
chaque $i,j$ on
a $x_i\equiv x_j\;\mod\;\fa_i\vu\fa_j$, alors il existe un unique $x$ 
modulo
$\fa$ vérifiant:  $x\equiv x_i\;\mod\;\fa_i\;(i=1,\ldots ,n)$.
\item Notons $\gT_i=\gT/(\fa_i=0)$, 
$\gT_{ij}=\gT_{ji}=\gT/(\fa_i\vu\fa_j=0)$,
$\pi_i\colon \gT\to\gT_i$ et $\pi_{ij}:\gT_i\to\gT_{ij}$ les projections 
canoniques.
Si les $\fa_i$ recouvrent $\gT$, alors $(\gT,(\pi_i)_{i=1,\ldots, n})$  est 
la limite projective du diagramme 
\[((\gT_i)_{1\leq i\leq n},(\gT_{ij})_{1\leq  i<j\leq n};(\pi_{ij})_{1\leq i\neq j\leq n})
\] 
(voir la figure ci-après).
\item Soit maintenant $(\ff_i)_{i=1,\ldots, n}$ une famille finie de 
filtres principaux,
notons $\gT_i=\gT/(\ff_i=1)$, 
$\gT_{ij}=\gT_{ji}=\gT/(\ff_i\vi\ff_j=1)$,
$\pi_i\colon \gT\to\gT_i$ et $\pi_{ij}:\gT_i\to\gT_{ij}$ les projections 
canoniques.
Si les $\ff_i$ recouvrent $\gT$, alors $(\gT,(\pi_i)_{i=1,\ldots, n})$  est 
la limite
projective du diagramme \[((\gT_i)_{1\leq i\leq n},(\gT_{ij})_{1\leq 
i<j\leq
n};(\pi_{ij})_{1\leq i\neq j\leq n}).\]
\end{enumerate}
\end{ffact}
 
 {\hspace*{10em}{
\xymatrix @R=2em @C=7em{
          &  \gT \ar[rd]^{\pi _{k}}\ar[d]^{\pi _{j}}\ar[ld]_{\pi _{i}}\\
 \gT _i\ar[d]_{\pi _{ij}}\ar@/-0.75cm/[dr]^{\pi _{ik}} &
     \gT _j\ar@/-.8cm/[dl]_{\pi _{ji}}\ar@/-.8cm/[dr]^{\pi _{jk}} &
        \gT _k\ar@/-0.75cm/[dl]_{\pi _{ki}}\ar[d]^{\pi _{kj}} &
\\
 \gT _{ij}  & 
    \gT _{ik}   & 
      \gT _{jk}   
}
}}

\begin{proof}
\textsl{1}. Il suffit de le démontrer avec $\fa=0$, ce qui est le point \textsl{2}.

\smallskip \noindent \textsl{2}. Soit $(\gH,(\psi_i)_{i\in I})$ la limite projective du diagramme.
On a un unique morphisme \[\varphi\colon \gT\to \gH\] tel que $\varphi\circ \psi_i=\pi_i$
pour chaque~$i$. Et~$\varphi$ est injectif par hypothèse: $\varphi(x)=\varphi(y)$ implique $\varphi(x)\equiv\varphi(y) \mod\, (\fa_i=0)$ pour chaque $i$, et on a $\bigcap_i\fa_i=0$. 
On doit montrer qu'il est surjectif. \hbox{Soit $x=(x_i)_{i\in I}$} un \elt de $\gH$: \hbox{on a $x_i\in \gT_i$} pour chaque~$i$ et~\hbox{$\pi_{ij}(x_i)=\pi_{ji}(x_j)$} \hbox{pour $i\neq j$}. \hbox{Si $x_i=\pi_i(y_i)$}  on a donc dans $\gT$ la congruence 
\[
y_i\equiv y_j \mod \dar (s_i\vu s_j).
\]
 L'injectivité de $\varphi$ signifie que $\Vi_{i=1}^ns_i=0$. On a $\pi_i(y_i)=\pi_i(y_i\vu s_i)$ donc on peut supposer \hbox{que $y_i\geq s_i$}.
Les \egts \hbox{$\pi_{ij}(x_i)=\pi_{ji}(x_j)$} s'écrivent   
\[
y_i\equiv y_j \mod \dar (s_i\vu s_j).
\]
\cad $y_i\vu s_j=y_j\vu s_i$. Posons $y=\Vi_{i=1}^ny_i$.
\\
Alors, avec par exemple $j=1$, on obtient
\[ 
y\vu s_1=y_1\vu\Vi\nolimits_{i=2}^n(y_i\vu s_1)=y_1\vi\Vi\nolimits_{i=2}^n(y_1\vu s_i)=y_1
\]
(car $a\vi(a\vu b)=a$). Ainsi $\pi_j(y)=x_j$ pour chaque $j$.
Et $\varphi$ est bien surjective.
\end{proof}

Il y a aussi une procédure de recollement proprement dit.
Pour l'établir nous avons besoin du lemme suivant.

Rappelons que pour $s\in\gT$ le quotient $\gT/(s=0)$ est isomorphe au filtre principal $\uar s$ que l'on voit comme un \trdi dont l'\elt $0$ est $s$.
\begin{flemma}[dans un \trdi,  les quotients principaux sont \gui{scindés}] \label{flemquoprinctrdi} ~ 
\\
Soit $\pi\colon \gT  \to \gT'$ un morphisme de \trdis et $s\in \gT$.
\Propeq
\begin{enumerate}
\item $\pi$ est un morphisme de passage au   quotient de $\gT$ par l'idéal principal $\fa=\dar s$.
\item Il existe un morphisme $\varphi\colon\gT'\to\,\uar s$ tel que
$\pi\circ \varphi=\Id_{\gT'}$.
\end{enumerate}
Dans ce cas $\varphi$ est uniquement déterminé par $\pi$ et $s$.\\
Naturellement, l'énoncé renversé est valable pour un quotient par un filtre principal. 
\end{flemma}
%
\begin{proof} 
\textsl{1} $\Rightarrow$ \textsl{2.} Soit $y\in \gT'$. On a $y=\pi(x)$ pour un $x\in \gT $. 
\\
On veut définir  $\varphi\colon \gT'\to\uar s$ par l'\egt $\varphi(y)=x\vu s$. Tout d'abord c'est bien défini: si $\pi(x)=\pi(x')$, alors $x\vu s=x'\vu s$ d'après le rappel précédent. Ensuite il est immédiat que $\varphi$ est un morphisme de \trdis et que $\pi\circ \varphi=\Id_{\gT'}$.

\smallskip \noindent 
\textsl{2} $\Rightarrow$ \textsl{1.} L'\egt $\pi\circ \varphi=\Id_{\gT'}$ implique
que $\pi$ est surjectif et que $\varphi$ est un \iso de $\gT'$ sur $\uar s$
avec la restriction de $\pi$ pour \iso réciproque. 
Ceci montre que $\varphi$ est uniquement déterminé par $\pi$ et $s$. On doit montrer l'\eqvc 
\[
{\pi(x_1)=\pi(x_2)\; \Leftrightarrow \;x_1\vu s=x_2\vu s.}
\]
Comme $\varphi(0)=s$, on a $\pi(s)=0$, et $x_1\vu s=x_2\vu s$ implique $\pi(x_1)=\pi(x_2).$\\
Si $\pi(x_1)=\pi(x_2)$ alors  $\pi(x_1\vu s)=\pi(x_2\vu s)$, et puisque la restriction de $\pi$ à $\uar s$ est injective, cela implique
$x_1\vu s=x_2\vu s$. 
\end{proof}
\rem Nous avons utilisé en titre du lemme l'expression \gui{les quotients principaux sont scindés} par analogie avec les surjections scindées 
\hbox{entre \Amos}, 
vue l'\egt $\pi\circ \varphi=\Id_{\gT'}$, mais l'analogie est limitée. Ici la \gui{section} $\varphi$ de $\pi$ est unique
(une différence importante),
et ce n'est \gui{pas tout à fait} un morphisme de~$\gT'$ dans~$\gT $ (une autre différence importante).
\eoe

\begin{fproposition}[recollement de \trdis]
\label{fpropRecolTD} 
  Supposons donnés un ensemble fini totalement ordonné~$I$ et dans la catégorie des \trdis  un diagramme
\[
\big((\gT_i)_{i\in I},(\gT_{ij})_{i<j\in I},(\gT_{ijk})_{i<j<k\in I} ;
(\pi_{ij})_{i\neq j},(\pi_{ijk})_{i< j, j\neq k\neq i}\big)
\]
comme dans la figure ci-après, 
ainsi qu'une famille d'\elts 
\[
(s_{ij})_{i\neq j\in I}\in \prod\nolimits_{i\neq j\in I}\gT_{i}
\]
satisfaisant les conditions suivantes:
\begin{itemize}
\item le diagrammme est commutatif ($\pi_{ijk}\circ \pi_{ij}=\pi_{ikj}\circ \pi_{ik}$ pour tous $i$, $j$, $k$ distincts), 
\item pour $i\neq j$, $\pi_{ij}$ est un morphisme de passage au quotient par l'\id $\dar s_{ij}$,
\item pour $i$, $j$, $k$ distincts, $\pi_{ij}(s_{ik})=\pi_{ji}(s_{jk})$ et  $\pi_{ijk}$ est un morphisme de passage au quotient par \hbox{l'\id $\dar\pi_{ij}(s_{ik})$}.   
\end{itemize}

\smallskip {\hspace*{10em}
\xymatrix @R=2em @C=7em{
 \gT_i\ar[d]_{\pi _{ij}}\ar@/-0.75cm/[dr]^{\pi _{ik}} &
     \gT_j\ar@/-.8cm/[dl]_{\pi _{ji}}\ar@/-.8cm/[dr]^{\pi _{jk}} &
        \gT_k\ar@/-0.75cm/[dl]_{\pi _{ki}}\ar[d]^{\pi _{kj}} &
\\
 ~\gT_{ij}~ \ar[rd]_{\pi _{ijk}} & 
    ~\gT_{ik}~  \ar[d]^{\pi _{ikj}} & 
      ~\gT_{jk}~  \ar[ld]^{\pi _{jki}} 
\\
   &  ~\gT_{ijk}~ 
\\
}
}

\smallskip \noindent Alors si $\big(\gT\,;\,(\pi_i)_{i\in I}\big)$ est la limite projective du diagramme, les~\hbox{$\pi_i\colon \gT\to \gT_i$} forment un recouvrement par quotients principaux de $\gT$, et le diagramme est isomorphe à celui obtenu
dans le fait~\ref{ffactRecolTD}.
Plus précisément, il existe des $s_i\in\gT$ tels que chaque~$\pi_i$ est un morphisme de passage au quotient par l'\id $\dar s_i$ et $\pi_i(s_j)=s_{ij}$ pour tous $i\neq j$.

\noindent Le résultat analoque est valable pour les quotients par des filtres principaux.
\end{fproposition}

\begin{proof}
Nous posons $s_{ii}=0$, $\gT_{ii}=\gT_i$, $\varphi_{ii}=\pi_{ii}=\Id_{\gT_i}$. Le lemme \ref{flemquoprinctrdi} nous donne des \gui{sections}  $\varphi_{ij}\colon \gT_{ij}\to \gT_i$ et~$\varphi_{ijk}\colon \gT_{ijk}\to \gT_{ij}$. 
\\
Les conditions imposées impliquent que les  
\ids~\hbox{$\dar\pi_{jk}(s_{ji})$} \hbox{et $\dar\pi_{kj}(s_{ki})$} sont égaux, i.e.~\hbox{$\pi_{jk}(s_{ji})=\pi_{kj}(s_{ki})$}. 
\\
Pour $i\in I$, on définit $s_i\in \prod_k\gT_k$ par 
${s_i=(s_{ji})_{j\in I}}$, de sorte que $\pi_j(s_i)=s_{ji}$. Les coordonnées de $s_i$ sont compatibles (i.e. $s\in \gT$) car
$\pi_{jk}(s_{ji})=\pi_{kj}(s_{ki})$.
\\
Nous définissons ensuite une application $\varphi_i\colon \gT_i\to \prod_k\gT_k$ par
\[
{\varphi_i(x)=y=(y_j)_{j\in I} \hbox{ avec } y_j=\varphi_{ji}(x_j)=\varphi_{ji}\big(\pi_{ij}(x )\big).}
\]
 Montrons que les coordonnées de $y$ sont compatibles (i.e. $y\in \gT$). En effet
\[{y_j=s_{ji}\vu y_j, \hbox{ so }
\pi_{jk}(y_j)= \pi_{jk}(s_{ji}\vu y_j)=\pi_{jk}(s_{ji})\vu\pi_{jk}(y_j),
}
\]
 de même $\pi_{kj}(y_k)=\pi_{kj}(s_{ki})\vu\pi_{kj}(y_k)$.
Et puisque $\pi_{jki}$ est un morphisme de passage au quotient par l'\id $\dar\pi_{jk}(s_{ji})=\dar\pi_{kj}(s_{ki})$, l'\egt 
\[
{\pi_{jk}(y_j)=\pi_{kj}(y_k)}
\]
 peut être testée en prenant les images par $\pi_{jki}$. \\
Or, puisque $\pi_{ji}(y_j)=\pi_{ji}\big(\varphi_{ji}(x_j\big)=x_j=\pi_{ij}(x)$, on obtient en utilisant  la commutativité du diagramme
\[\pi_{jki}\big(\pi_{jk}(y_j)\big)=\pi_{ijk}\big(\pi_{ji}(y_j)\big)=\pi_{ijk}\big(\pi_{ij}(x)\big).
\]
De même $\pi_{kji}\big(\pi_{kj}(y_k)\big)=\pi_{ikj}\big(\pi_{ik}(x)\big)$. Et nous concluons en utilisant une deuxième fois la commutativité du diagramme.\\
Une fois établi que $\varphi_i$ est bien une application $\gT_i\to \gT$, nous constatons facilement \hum{écrire les détails?} que $\pi_i\circ \varphi_i=\Id_{\gT_i}$, que l'image de
$\varphi_i$ est le filtre $\uar s_i$ de $\gT$ et que~$\varphi_i$ est un morphisme de
\trdis de $\gT_i$ sur le filtre $\uar s_i$. Donc, par le lemme \ref{flemquoprinctrdi}, $\pi_i$ est un morphisme de passage au quotient
par $\dar s_i$.
\end{proof}

\subsubsection*{Treillis de Heitmann}
\addcontentsline{toc}{subsubsection}{Treillis de Heitmann}

Un quotient intéressant, qui n'est ni un quotient par un \id ni un 
quotient
par un filtre, est le treillis de Heitmann.

\begin{flemma}
\label{flemHeT}
Sur un \trdi arbitraire $\gT$ la relation $\JT(a)\subseteq\JT(b)$
est une relation de préordre $a\preceq b$ qui définit un quotient 
de $\gT$.
On a aussi:
\begin{equation} \label{feqJaJb}
a\preceq b \quad \Longleftrightarrow\quad  a\in \JT(b)
\quad \Longleftrightarrow\quad \forall x\in\gT \; (a\vu x = 1 \Rightarrow 
b\vu x=1)
\end{equation}
\end{flemma}
\begin{proof}
Les équivalences $a\in \JT(b) \;\Leftrightarrow\;
\JT(a)\subseteq\JT(b)\;\Leftrightarrow\;\forall x\in\gT \;(a\vu x = 1
\Rightarrow b\vu x=1)$ résultent de ce qui a été dit page 
\pageref{feqRJb}
concernant le radical de Jacobson d'un \id (voir 
l'\egt~(\ref{feqRJb})).\\
Par ailleurs on vérifie sans difficulté les relations 
(\ref{feqPreceq})
nécessaires pour qu'un préordre définisse un quotient.
\end{proof}

\begin{fdefinition}
\label{fdefHeT}
On appelle \textsl{treillis de Heitmann de $\gT$} et on note $\He(\gT)$ 
le
treillis quotient de~$\gT$ obtenu en remplaçant sur $\gT$ la 
relation
d'ordre $\leq_\gT $  par la relation de préordre 
$\preceq_{\He(\gT)}$
définie comme suit

\vspace{-1em}
\begin{equation} \label{feqdefHeT}
\begin{array}{rcl}\qquad 
a\preceq_{\He(\gT)} b & \equidef  &   \JT(a)\subseteq\JT(b)  \quad \hbox{(cf. \dfn \ref{fdefJac})}
  \end{array}
  \end{equation}
Ce treillis quotient peut être identifié à l'ensemble des 
\ids $\JT(a)$,
avec la projection canonique
\[ \gT\longrightarrow \He(\gT),\quad a\longmapsto \JT(a)\]

\end{fdefinition}

Notez qu'avec l'identification précédente on a les \egts:
\begin{equation} \label{feqHeT2}
\JT(a\vi b)=\JT(a)\vi_{\He(\gT)}\JT(b),\quad
\JT(a\vu b)=\JT(a)\vu_{\He(\gT)}\JT(b)
\end{equation}

Dire que le treillis $\gT$ est faiblement Jacobson revient à dire 
que
$\gT=\He(\gT)$.

\smallskip Le lemme suivant est une précision (et une généralisation) de la
première \egt ci-dessus. Il nous sera utile dans la suite.

\begin{flemma}
\label{flemJacInter}
Si $\fa$ et $\fb$ sont deux \ids de $\gT$, on a
$\JT(\fa\cap\fb)=\JT(\fa)\cap\JT(\fb)$.
\end{flemma}
\begin{proof}
Il suffit de montrer que si $z\in\JT(\fa)\cap\JT(\fb)$ alors
$z\in\JT(\fa\cap\fb)$. Soit $t\in\gT$ tel que $z\vu t=1$, nous 
cherchons
$c\in\fa\cap\fb$ tel que $c\vu t=1$.
Or nous avons un  $a\in\fa$ tel que $a\vu t=1$ et  un  $b\in\fb$ tel 
que $b\vu
t=1$. Il suffit donc de prendre $c=a\vi b$.
\end{proof}

On notera que la preuve ne marcherait pas pour une intersection 
infinie d'\ids. 

\hum{Il est probable que les \ids de $\He(\gT)$ sont exactement les
$\pi(\JT(\fa))$. }

\begin{ffact}
\label{ffactHeHe} Soit $\gT$ un \trdi,  $\gT'=\gT/(\JT(0)=0)$, 
$x\in\gT$ et $\fa$
un idéal.
\begin{enumerate}
\item $x=_{\He(\gT)}1\;\Longleftrightarrow\; x=1$.
\item $x=_{\He(\gT)}0\;\Longleftrightarrow\; x\in\JT(0)$.
\item $\He(\He(\gT))\;=\;\He(\gT')\;=\;\He(\gT)$.
\item Si $\gL=\gT/(\fa=0)$,
$\He(\gL)$ s'identifie à $\He(\gT)/(\JT(\fa)=0)$.
\end{enumerate}
\end{ffact}

\rem On notera cependant que $\He$ ne définit pas un foncteur.
\eoe

\begin{proof}
Les points \textsl{1} et \textsl{2} sont immédiats.  \\
Le point \textsl{4} est laissé \alec.  Il implique
$\He(\gT')\;=\;\He(\gT)$.\\
Dans le point \textsl{3} les treillis $\He(\He(\gT))$ et $\He(\gT')$ sont identifiés à
des quotients de $\gT$.  Montrons l'égalité $\He(\He(\gT))=\He(\gT)$,
\cad que pour $a,b\in\gT$, $a\preceq_{\He(\He(\gT))}b\Rightarrow
a\preceq_{\He(\gT)}b$.  Par \dfn l'hypothèse signifie: $\Tt x \in
\gT\;(a\vu x=_{\He(\gT)}1\;\Rightarrow \;b\vu x=_{\He(\gT)}1)$.  Or
d'après le point~\textsl{1} cela veut dire $\Tt x \in \gT\;(a\vu
x=1\;\Rightarrow \;b\vu x=1)$, \cad $a\preceq_{\He(\gT)}b$.
\end{proof}

\subsection{Algèbres de Heyting, de Brouwer, de Boole}\label{fsubsecAgHagB}

\subsubsection*{Algèbres de Heyting}
\addcontentsline{toc}{subsubsection}{Algèbres de Heyting}

Un \trdi $\gT$ est appelé un {\sl treillis implicatif} (\citealt*{fCur63}) 
ou une {\sl \agH} (\citealt*{fJoh1986}) lorsqu'il existe
une opération binaire  $\im$ vérifiant pour tous $a,\,b,\,c$:
\begin{equation} \label{feqAgHey}
a\vi b \leq c \;\;\Longleftrightarrow \;\; a \leq  (b\im c).
\end{equation}
Ceci signifie que pour tous $b, c\in\gT$, l'\id transporteur $(c:b)$  est principal, 
son \gtr
étant noté $b\im c$.
Donc si elle existe, l'opération $\im$ est déterminée de 
manière unique par la structure du treillis.
On définit alors   $\neg x := x\im 0$.
La structure d'\agH peut être définie comme purement 
équationnelle en donnant de axiomes adéquats. Précisément un treillis~$\gT$ (non 
supposé
distributif) muni d'une loi $\im$ est une \agH \ssi les axiomes 
suivants sont
vérifiés (cf. \citealt*{fJoh1986}):
\[\begin{array}{rcl}
a\im a&=   &1    \\
a\vi(a\im b)&=   &a\vi b    \\
b\vi(a\im b)&=   & b   \\
a\im(b\vi c)&=   &(a\im b)\vi(a\im c)
\end{array}\]
Notons aussi les faits importants suivants:

\[\begin{array}{rcl}
(a\vu b)\im c &=& (a\im c)\vi(b\im c)    \\
\neg(a\vu b)&=   & \neg a\vi \lnot b   \\
 a&\leq    &\neg\neg a   \\
\neg a\vu b&\leq    & a\im b   \\
a\leq b&\Leftrightarrow& a\im b =1
\end{array}\]

  Tout \trdi fini est une \agH,  car tout \itf est principal.

\smallskip Un cas particulier important d'\agH est une \textsl{\agB}:
c'est un \trdi dans lequel tout \elt $x$ possède \textsl{un
complément}, \cad un \elt $y$ vérifiant $y\vi x=0$ et $y\vu x=1$
($y$ est noté $\lnot x$ et l'on a $a\im b=\lnot a\vu b$).

\smallskip Un \textsl{\homo d'\agHs} est un \homo $\varphi :\gT\to\gT'$
de \trdis qui vérifie $\varphi(a\im b)=\varphi(a)\im\varphi(b)$ pour
tous $a,b\in\gT$.

\smallskip Le fait suivant est immédiat.

\begin{ffact}
\label{ffactQuoAgH}
Soit $\pi\colon \gT\to\gT'$ un \homo de \trdis.  Supposons que $\gT$ et 
$\gT'$ sont
deux \agHs et notons $a\preceq b$ pour 
$\varphi(a)\leq_{\gT'}\varphi(b)$. Alors
$\pi$ est un \homo d'\agHs \ssi on a pour tous $a,a',b,b'\in\gT$:
\[
a\preceq a'\Rightarrow (a'\im b)\preceq(a\im b)  \qquad 
\mathrm{et}\qquad
b\preceq b'\Rightarrow (a\im b)\preceq(a\im b')
\]
\end{ffact}

On a aussi:

\begin{ffact}
\label{ffactQuoAgH2}
Si $\gT$ est une \agH tout quotient $\gT/(y=0)$ (\cad tout quotient 
par un \id
principal) est aussi une \agH. 
\end{ffact}
\begin{proof}
Soit $\pi\colon \gT\to\gT'=\gT/(y=0)$ la projection canonique. On a
$\pi(x)\vi\pi(a)\,\leq_{\gT'}\, \pi(b)\;\Leftrightarrow\; \pi(x \vi
a)\,\leq_{\gT'}\, \pi(b)\;\Leftrightarrow\; x\vi a \,\leq\, b\vu
y\;\Leftrightarrow\; x\,\leq\, a\im(b\vu y)$. Or $y\,\leq\, b\vu 
y\,\leq\,
a\im(b\vu y)$, donc $\pi(x)\vi\pi(a)\,\leq_{\gT'}\, 
\pi(b)\;\Leftrightarrow\;
x\,\leq\, (a\im(b\vu y))\vu y$, \cad $\pi(x)\leq_{\gT'}\pi(a\im(b\vu 
y))$, ce
qui montre que $\pi(a\im(b\vu y))$ vaut pour $\pi(a)\im\pi(b)$ dans 
$\gT'$.
\end{proof}

\hum{

1.  Cependant il ne semble pas que $\pi$ soit en général un \homo
d'\agHs. 

2.  Il serait bon d'avoir un exemple d'un treillis quotient d'une \agH
qui ne serait pas une \agH. 

3.  De manière générale, il serait bon d'avoir des exemples
variés d'\agH non noethériennes à notre disposition.

}

\medskip \rem 
La notion d'\agH est reminiscente de la notion d'anneau
cohérent en \alg commutative.  En effet un anneau cohérent peut
être caractérisé comme suit: l'intersection de deux \itfs est un
\itf et le transporteur d'un \itf dans un \itf est un \itf.   Si l'on
\gui{relit} ceci pour un \trdi en se rappelant que tout \itf est
principal on obtient une \agH. 
\eoe

\medskip \rem 
Tout \trdi $\gT$ engendre une \agH de façon
naturelle.  Autrement dit on peut rajouter formellement un \gtr pour
tout \id $(b:c)$.  Mais si on part d'un \trdi qui se trouve être une
\agH,  l'\agH qu'il engendre est strictement plus grande.  Prenons par
exemple le treillis $\Trois$ (fini, donc une \agH), qui est le \trdi libre à un \gtr.  
L'\agH qu'il engendre est donc l'\agH libre à un générateur.  Or
celle-ci est infinie (cf.  \cite[section 4.11]{fJoh1986}).  A contrario le treillis
booléen engendré par $\gT$ (cf.  \cite{fCC00}, \cite[Théorème XI-1.8]{fACMC}) reste égal à $\gT$
lorsque celui-ci est booléen.
\eoe

\subsubsection*{Treillis avec négation}
\addcontentsline{toc}{subsubsection}{Treillis avec négation}

Un \trdi \textsl{possède une négation} si pour tout $x$ l'idéal 
$(0:x)$ est
principal, engendré par un \elt que l'on note  $\lnot x$.
Les règles suivantes sont immédiates.

\[\begin{array}{rclcrcl}
x\vi y =0&\Leftrightarrow& y\leq \lnot x&,&a\leq b&  \Rightarrow  & \lnot b   \leq     \lnot a \\
a& \leq   & \lnot\lnot a  &,   & \lnot a  &=& \lnot\lnot\lnot a \\
\lnot(a\vu b)& =   &  \lnot a\vi\lnot b &  , & \lnot a\vu\lnot b 
&\leq &
\lnot(a\vi b)\\
\lnot(x\vu\lnot x)&=&0&,&\lnot\lnot(x\vu\lnot x)&=&1
\end{array}\]

Si pour tout $a$, $\lnot\lnot a=a$, le treillis est une \agB parce 
qu'alors
$x\vu\lnot x=1$.

\begin{ffact}
\label{ffactSpecMin}
Si $\gT$ possède une négation,   notons $\Fmin(\gT)=\ff$ le filtre
engendré par tous les $x\vu\lnot x$. Alors 
$\He(\gT\cir)=(\gT/(\Fmin(\gT)=1))\cir$, et ce
treillis est une \agB. 
\end{ffact}
\begin{proof}
Il est clair que $\lnot x$ est un complément de $x$ dans 
$\gT/(\ff=1)$; ce
treillis est donc une \agB.  En présence de la négation, la 
relation
$a\leq_{\He(\gT\cir)}b$ est équivalente à $\lnot a\leq \lnot b$
et ceci est facilement équivalent à $b\leq a\;\mod\;(\ff=1)$.
\end{proof}

\begin{ffact}
\label{ffactWJavecneg}
Si $\gT$ est un treillis avec négation, le treillis $\gT\cir$ est 
faiblement
Jacobson \ssi $\gT$ est une \agB. 
\end{ffact}
\begin{proof}
En présence de négation, les équations (\ref{feqJaJb}) et 
(\ref{feqdefHeT})
donnent pour $a\leq_{\He(\gT\cir)}b$ la condition équivalente 
$\lnot b\leq
\lnot a$.
Le treillis  $\gT\cir$ est donc faiblement Jacobson \ssi $\lnot b\leq \lnot a$ implique  $a\leq b$. 
En particulier on obtient $b=\lnot\lnot b$
en prenant $a=\lnot\lnot b$.
\end{proof}

\subsubsection*{Algèbres de Brouwer}
\addcontentsline{toc}{subsubsection}{Algèbres de Brouwer}

  Un \trdi dont le treillis opposé est une \agH est appelé une
  \textsl{\alg de Brouwer}.  C'est un \trdi dans lequel tous les filtres
 différence  $c\setminus b$ sont principaux (voir \pref{feqDiff}).  On note alors $c-b$ le \gtr de
  $c\setminus b$.

En passant au treillis opposé le fait suivant dit la même chose
que le fait~\ref{ffactSpecMin}.

\begin{ffact}
\label{ffactSpecMax}
On dit que \emph{le treillis $\gT$ possède un complément de Brouwer} lorsque pour tout $x$ 
le filtre
$(1\setminus x)$ est principal. Il est alors engendré par un unique \elt que l'on note  
$1- x$. \\
Dans ce cas,
notons $\Imax(\gT)$ l'\id engendré par tous les $x\vi(1-x)$. Alors 
$\He(\gT)=\gT/(\Imax(\gT)=0)$ et ce treillis est une \agB.
\end{ffact}

Nous laissons \alec le soin de traduire le fait 
\ref{ffactWJavecneg} lorsque l'on renverse la relation d'ordre.

\subsection{Treillis distributifs noethériens}

En \clama,  pour un \trdi $\gT$ \propeq
\begin{itemize}
\item [$(1)$] Tout \id de $\gT$ est principal.
\item [$(2)$] Toute suite croissante d'\elts de $\gT$ est stationnaire.
\item [$(3)$] Toute suite croissante d'\ids  de $\gT$ est stationnaire.
\end{itemize}

Un tel treillis est appelé \textsl{noethérien} (par analogie avec
l'\alg commutative, on pourrait aussi l'appeler \textsl{principal}).  C'est
clairement est une \agH en \clama.

Tout sous-treillis et tout treillis quotient d'un treillis
noethérien est noethérien.

En \coma la notion est plus délicate.  Aucun treillis non trivial ne
vérifie le point (2) (qui est à priori la formulation la plus faible
des trois).  On pourrait définir un \trdi noethérien comme un
treillis vérifiant une condition \gui{ACC \cov} du style: toute
suite croissante admet deux termes consécutifs égaux.  Cette
condition est équivalente à (2) en \clama.   Mais il y a à priori
plusieurs variantes intéressantes.

En pratique, on est en général intéressé par le
fait que certains \ids bien précis sont principaux, comme dans le
cas des \agHs.   Or le fait qu'un treillis est une \agH ne résulte
pas \cot de la condition ACC \cov (de la même manière, en \alg
commutative, la cohérence, qui est souvent plus importante que la
noethérianité, ne résulte d'aucune variante \cov connue de la
noethérianité).  Voir à ce sujet la proposition~\ref{fpropZarHeyt}.

\medskip \rem
Montrons en \clama que si $\gT$ et $\gT\cir$ sont noethériens 
alors~$\gT$ est fini.  Les \idemas sont des $\dar x$ où $x$ est un
prédécesseur immédiat de $1$.  Et le spectre maximal est fini,
parce que si $(\fm_n)=(\dar x_n)$ est une suite infinie d'\idemas,  la
suite $(\Vi_{i\leq n}x_i)$ est strictement décroissante.  On peut
ensuite appliquer le résultat à chacun des treillis quotients par
les \idemas.   On termine par le lemme de K\"onig.  Rendre cette
preuve \cov,  avec une \dfn \cov suffisamment forte de la
noethérianité est un défi intéressant.
\eoe

\section{Espaces spectraux}
\label{fsecESSP}

\subsection{Généralités}

\subsubsection*{En \clama} 
\addcontentsline{toc}{subsubsection}{En \clama}

Un {\sl \id premier} $\fp$ d'un treillis $\gT\neq \Un$ est un \id dont
le complé\-mentaire $\ff$ est un filtre (qui est alors un {\sl
filtre premier}).  On a alors $\gT/(\fp=0,\ff=1)\simeq\Deux$.  Il
revient au même de se donner un \idep de $\gT$ ou un morphisme de
\trdis $\gT\rightarrow \Deux$.

Dans cette section, nous noterons $\theta_\fp:\gT\to\Deux$ l'\homo
associé à l'\idep $\fp$.

On vérifie facilement que si $S$ est une partie génératrice du
\trdi $\gT$, un \idep~$\fp$ de $\gT$ est complètement
caractérisé par sa trace sur $S$ (cf.  \cite{fCC00}).

Un \textsl{\idema} (resp.  \textsl{premier minimal}) est un \id maximal
parmi les \ids stricts (resp.  minimal parmi les \ideps).  Il revient
au même de dire que $\fm$ est maximal ou que
$\gT/(\fm=0)\simeq\Deux$, les \idemas sont donc premiers.  Il revient
au même de dire que $\fp$ est un \idep minimal ou que son
complémentaire est un filtre maximal.

En \clama tout \id strict est contenu dans un \idema et (par
renversement) tout filtre strict est contenu dans un filtre maximal.

\smallskip Le \textsl{spectre d'un \trdi $\gT$} est l'ensemble $\Spec
\,\gT$ de ses \ideps,  muni de la topologie suivante: une base
d'ouverts est donnée par les 
\[
\DT(a)\eqdefi\sotq{\fp\in\Spec
\,\gT}{a\notin\fp},\quad a\in \gT.
\]
On vérifie que
\begin{equation} \label{feqDa}
\left.\begin{array}{rclcrcl}
  \DT(a\vi b)   & =  & \DT(a)\cap \DT(b) ,&\quad & \DT(0)  & =  & 
\emptyset  ,\\
  \DT(a\vu b)   & =  & \DT(a)\cup \DT(b) ,&&  \DT(1) & =  &  
\Spec\,\gT.
  \end{array}
\right\}
\end{equation}

Le complémentaire de $\DT(a)$ est un fermé qu'on note $\VT(a)$.

On étend la notation $\VT(a)$ comme suit: si $I\subseteq\gT$, on
pose $\VT(I)\eqdefi\bigcap_{x\in I}\VT(x)$.  Si $\cI_\gT(I)=\fII$, on
a $\VT(I)=\VT(\fII)$.  On dit parfois que $\VT(I)$ est \textsl{la
variété associée à $I$}.

\medskip\noindent
{\bf Définition.} 
Un espace topologique homéomorphe à un espace $\Spec(\gT)$
est appelé un \textsl{espace spectral}. Les espaces spectraux proviennent de \citealt*{fSto37}.

\medskip Johnstone les appelle des \textsl{espaces cohérents} (\cite{fJoh1986}). C'est Hochster qui les a baptisés dans \citealt*{fHoc1969}.

Avec la logique classique et l'axiome du choix, l'espace $\Spec \,\gT$
a \gui{suffisamment de points}: on peut retrouver le treillis $\gT$
à partir de son spectre.  Voici comment.

Tout d'abord on a le fameux \tho de Krull.


\medskip\noindent
{\bf Théorème$\etl$ de Krul}\label{fThKrull} (en \clama)\\ 
{\sl Supposons que $\fJ$ est un \id,  $\fF$ un filtre et 
$\fJ\cap\fF=\emptyset$.
Alors il existe un \idep $\fP$ tel \hbox{que $\fJ\subseteq\fP$} et
$\fP\cap\fF=\emptyset$.
  }

\medskip
On en déduit que:
\begin{itemize}
\item L'application $a\in\gT\,\mapsto\,\DT(a)\in\cP(\Spec\,\gT)$
est injective: elle identifie $\gT$ à un treillis d'ensembles 
(\textsl{\tho de représentation de Birkhoff}).
\item Si $\varphi : \gT\to\gT'$ est un \homo injectif l'application
$\varphi^\star:\Spec\,\gT'\to\Spec\,\gT$ obtenue par dualité est 
surjective.
\item Tout \id de $\gT$ est intersection des \ideps qui le 
contiennent.
\item L'application $\fII\mapsto \VT(\fII)$, des \ids de $\gT$ vers 
les fermés
de $\Spec\,\gT$ est un \iso d'ensembles ordonnés (pour l'inclusion 
et
l'inclusion renversée).
\end{itemize}

On montre aussi que les \oqcs de $\Spec \,\gT$ sont exactement les
$\DT(a)$.  D'après les \egts (\ref{feqDa}) les \oqcs de $\Spec \,\gT$
forment un \trdi de parties de~$\Spec \,\gT$, isomorphe à $\gT$.

\`A partir d'un espace spectral $X$ on peut considérer le \trdi
$\OQC(X)$ formé par ses \oqcs.   Puisque pour tout \trdi $\gT$,
$\OQC(\Spec(\gT))$ est canoniquement isomorphe à $\gT$, pour tout
espace spectral $X$, $\Spec(\OQC(X))$ est canoniquement 
homéo\-morphe~à~$X$.

\smallskip Tout \homo $\varphi :\gT\rightarrow \gT'$ de \trdis fournit
par dualité une application continue $\varphi^\star:\Spec
\,\gT'\rightarrow \Spec \,\gT$, qui est appelée une
\textsl{application spectrale}.  Pour qu'une application continue entre
espaces spectraux soit spectrale il faut et il suffit que l'image
réciproque de tout \oqc soit un \oqc. 

L'article fondateur  \citealt*{fSto37} démontre pour l'essentiel que la catégorie spectrale ainsi définie est antiéquivalente à
celle des \trdis \cite[{II-3.3}, coherent locales]{fJoh1986}.
Plus précisément, cet énoncé qui semble ici tautologique devient non trivial lorsque l'on donne une \dfn des espaces spectraux en termes purement d'espaces topologiques, comme dans la remarque qui suit.
Pour plus de détails sur cette antiéquivalence, on peut se reporter au théorème de Krull, à \citealt*[\hbox{section V-8}]{fBW74},
à \citealt*{fCL2001-2018} et à l'article de synthèse \citealt*{fLom2020}.

\medskip 
\rem
Une \dfn purement topologique des espaces spectraux 
est la sui\-vante~\cite{fSto37}.
\begin{itemize}
\item L'espace est de Kolmogorov (i.e., de type $\mathrm{T}_0$): 
étant donnés deux points il existe un voisinage de l'un des deux qui ne contient pas l'autre.
\item L'espace est \qc. 
\item L'intersection de deux \oqcs est un \oqc. 
\item Tout ouvert est réunion d'\oqcs. 
\item Pour tout fermé $F$ et pour tout ensemble $S$ d'\oqcs tels 
que 
\[\textstyle F\cap
\bigcap_{U\in S'} U\neq \emptyset\,\hbox{   pour toute partie finie  }\,S'
\,\hbox{  de  }\,S
\] 
on a aussi
$F\cap \bigcap_{U\in S} U\neq \emptyset$.
\end{itemize}
En présence des quatre premières propriétés la dernière 
peut se
reformuler comme suit (\cite{fHoc1969}). 
\begin{itemize}
\item Tout fermé irréductible\footnote{Un fermé qui n'est pas réunion de deux fermés strictement plus petits} admet un point
générique.\eoe

\end{itemize}


\subsubsection*{Points génériques, relation d'ordre}
\addcontentsline{toc}{subsubsection}{Points génériques, relation d'ordre}

On dit qu'un point $x \in X$ d'un espace spectral est le
\textsl{point générique du fermé $F$} \hbox{si $F=\ovs{x }$}.  Ce point
(quand il existe) est nécessairement unique car les espaces spectraux
sont des espaces de Kolmogorov.  Les fermés $\ovs{x }$
sont exactement tous les fermés irréductibles de $X$.  La relation
d'ordre $y\in\ovs{x}$ sera notée $x\leq_X y$.

Lorsque $X=\Spec\,\gT$ la relation $\fp\leq_X \fq$ est simplement la
relation d'inclusion usuelle entre \ideps du \trdi $\gT$.

Les points fermés de $\Spec\,\gT$ sont les \idemas de $\gT$.

\medskip 
On appelle \textsl{espace de Stone}\footnote{La terminologie ne semble pas entièrement fixée.  \cite{fBW74} appellent espace de Stone un espace topologique qui est à très peu près un espace spectral. Leur but est une catégorie d'espaces topologiques antiéquivalente à celle des \trdis \gui{non bornés}, i.e., sans $0$ et $1$.}  un espace spectral dont le treillis des \oqcs est une \agB\footnote{Il est homéomorphe à un espace $\Spec\,\gB$ pour une \agB \(\gB\)}.   Il est bien connu que les espaces de
Stone peuvent être caractérisés comme les espaces compacts
totalement discontinus.

\subsubsection*{En \coma}
\addcontentsline{toc}{subsubsection}{En \coma}

 D'un point de vue \cof,  $\gT$ est une version \gui{sans points} de
 $\Spec\,\gT$.  En d'autres termes, à défaut d'avoir accès aux
 points de $\Spec\,\gT$, on peut se contenter de l'ensemble de ses
 \oqcs,  qui sont directement visibles (sans recours à l'axiome du
 choix ni au principe du tiers exclu).  La version sans points est
 plus facile à appréhender.  Au contraire les points de
 $\Spec\,\gT$ ne sont pas en général des objets accessibles sans
 recours à des principes non constructifs.

En \coma on a à priori plusieurs possibilités pour définir le
spectre d'un \trdi (toutes équivalentes en \clama).
Le plus raisonnable semble de définir $\Spec\,\gT$ comme l'ensemble
des filtres premiers de $\gT$, \cad les filtres pour lesquels on a
\[x\vi y\in\fF\quad \Longrightarrow \quad x\in\fF \;\;\mathrm{ou}\;\;
y\in\fF \] avec un \gui{ou} explicite.  Mais de tels espaces
$\Spec\,\gT$ n'ont pas toujours suffisamment de points{\footnote{~On
peut par exemple définir un \trdi infini dénombrable explicite qui
ne possède pas d'\ideps récursifs.  Pour un tel \trdi,  il ne peut
pas y avoir de \prco que $\Spec\,\gT$ est non vide}} et on ne peut pas
affirmer \cot que les deux catégories sont antiéquivalentes, du
moins si l'on définit les morphismes entre espaces spectraux comme des
applications, car les applications nécessitent des points.

Une solution alternative satisfaisante (mais un peu troublante au
premier abord) est de considérer $\Spec\,\gT$ comme un \gui{espace
topologique sans points}, \cad un espace topologique défini
uniquement à travers sa base d'ouverts $\DT(a)$ (où $a$ parcourt
$\gT$).  Les morphismes sont alors définis de manière purement
formelle comme donnés par les morphismes des treillis
correspondants, en renversant le sens des flèches.  De ce point de
vue l'antiéquivalence de la catégorie spectrale et de la
catégorie des \trdis devient une pure tautologie définitionnelle.

En tout état de cause, bien que la catégorie spectrale reste utile
pour l'intuition, tout le travail se fait dans la catégorie des
\trdis.   L'avantage est naturellement que l'on obtient des \thos
\cofs. 

Dans cet article les spectres seront étudiés uniquement du point
de vue des \clama,  comme source d'inspiration importante pour de
bonnes notions concernant les \trdis.

\subsubsection*{Espaces spectraux noethériens}
\addcontentsline{toc}{subsubsection}{Espaces spectraux noethériens}
Un espace topologique $X$ est dit \textsl{noethérien} si toute suite
croissante d'ouverts est stationnaire.  Il revient au même de dire
que tout ouvert est \qc.  Pour un espace spectral, il est équivalent de dire que le treillis $\OQC(X)$ est noethérien.  Dans un espace spectral noethérien tout ouvert est un $\DT(a)$ et tout fermé un $\VT(b)$.

\hum{Je serais curieux de connaitre une formulation sans point de
la propriété suivante, plus faible que la noethérianité: toute
suite croissante pour $\leq _X$ est stationnaire.}

\subsubsection*{Deux autres topologies intéressantes sur 
$\Spec\,\gT$}
\addcontentsline{toc}{subsubsection}{Deux autres topologies intéressantes sur 
$\Spec\,\gT$}

\hum{Par rapport à l'article original, l'explication ci-dessous est allongée.}
En \clama on a une bijection canonique entre les ensembles
sous-jacents aux espaces $\Spec\,\gT$ et $\Spec\,\gT\cir$: à un
\idep de $\gT$ on associe le filtre premier complémentaire, qui est un \idep de $\gT\cir$. Cela permet d'identifier ces
deux ensembles, même si parfois l'effet n'est pas très heureux.  
Une fois les ensembles sous-jacents identifiés, la topologie n'est pas la même.  Les ouverts de base de $\Spec\,\gT\cir$ sont les $\DTo(a)=\VT(a)$.  
Modulo cette identification, pour $X=\Spec\,\gT$ et
$X'=\Spec\,\gT\cir$, la relation d'ordre $\leq_{X'}$ est la relation
opposée~à~$\leq_X$ (l'ordre est renversé), mais ce qui se
passe pour la topologie est plus compliqué.

\smallskip 
On doit également considérer la \textsl{topologie constructible}
(en anglais: patch topology) dont les ouverts de base sont les $\DT(a) \cap
\VT(b)$.  Cela donne un espace compact naturellement homémorphe à
$\Spec\,\gT^{\rm bool}$ où $\gT^{\rm bool}$ est le treillis
booléen engendré par $\gT$.  En \clama on obtient $\gT^{\rm bool}$
comme la sous-\agB de l'ensemble des parties de $\Spec\,\gT$
engendrée par les $\DT(a)$.  Ce treillis peut aussi être décrit
\cot comme suit (cf.  \cite{fCC00}).  On considère une copie disjointe
de $\gT$, que l'on note $\dot{\gT}$.  Alors $\gT^{\rm bool}$ est un \trdi
défini par \gtrs et relations.  Les \gtrs sont les \elts de
l'ensemble $T_1=\gT\cup\dot{\gT}$ et les relations sont obtenues comme
suit~: si $A,F,B,E$ sont quatre parties finies de $\gT$ on a
\[
  \Vi A \vi \Vi E\leq_\gT\Vu B\vu \Vu F
\quad \Longrightarrow \quad
\Vi A \vi \Vi \dot{F} \leq_{T_1} \Vu B \vu \Vu \dot{E}
\]
On montre que
  $\gT$ et $\dot{\gT}$ s'injectent naturellement dans $\gT^{\rm bool}$
et que l'implication ci-dessus est en fait une équivalence.
On obtient par dualité deux applications spectrales bijectives
$\Spec\,\gT^{\rm bool}\to\Spec\,\gT$ et $\Spec\,\gT^{\rm
bool}\to\Spec\,\gT\cir$.

\subsubsection*{Espaces spectraux finis}
\addcontentsline{toc}{subsubsection}{Espaces spectraux finis}
 
 En \clama les espaces duaux des \trdis \textsl{finis} sont les espaces spectraux
 finis, qui ne sont rien d'autre que les ensembles ordonnés finis,
 (car il suffit de connaitre l'adhérence des points pour
 connaitre la topologie) avec pour base d'ouverts les $\dar a$. 
 Les ouverts sont tous \qcs,  ce sont les parties initiales, et les
 fermés sont les parties finales.  Enfin, une application entre
 espaces spectraux finis est spectrale \ssi elle est croissante (pour
 les relations d'ordre associées).

La notion d'espace spectral apparait ainsi comme une
généralisation pertinente de la notion d'ensemble ordonné fini
au cas infini. Voir \citet*[Théorème XI-5.6, dualité entre ensembles ordonnés finis et \trdis finis]{fACMC}.

Dans le cas fini, si l'on identifie les ensembles sous-jacents à $\Spec\,\gT$ et
$\Spec\,\gT\cir$ les deux spectres sont presque les mêmes: c'est le
même ensemble ordonné au renversement près de la relation
d'ordre.  En outre les ouverts et les fermés sont simplement échangés.

\subsection{Treillis quotients et sous-espaces spectraux}\label{fsecSESP}

\subsubsection*{Caractérisation des sous-espaces spectraux}
\addcontentsline{toc}{subsubsection}{Caractérisation des sous-espaces spectraux}

En utilisant l'antiéquivalence des catégories, on  pourrait définir
directement la notion de \textsl{sous-espace spectral} comme la notion duale de la notion de treillis quotient.  Le \tho \ref{fpropSESP} explique cela en détail. 

Nous commençons par un lemme facile, qui caractérise les points
de $\Spec\,\gT$ qui \gui{sont des \elts de $\Spec\,\gT'$} lorsque
$\gT'$ est un quotient de $\gT$.

\begin{flemmac}
\label{flemSESP}
Soit $\gT'$ un treillis quotient de $\gT$ et $\pi\colon \gT\to\gT'$ la
projection canonique.  Notons $X=\Spec\,\gT'$, $Y=\Spec\,\gT$ et
$\pi^\star\colon X\to Y$ l'injection duale de $\pi$.  Rappelons que pour un
\idep $\fp$ de $\gT$ nous notons $\theta_\fp:\gT\to\Deux$ l'\homo
correspondant de noyau $\fp$.
\Propeq
\begin{itemize}
\item $\fp\in\pi^\star(\Spec\,\gT').$
\item $\theta_\fp$ se factorise par $\gT'.$
\item $\Tt a,b\in\gT\;((a\preceq b,\,b\in\fp)\Rightarrow a\in\fp).$
\end{itemize}
Cela peut se reformuler comme suit. Si le treillis quotient $\gT'$ 
est défini
par un système $R$ de relations $x_i=y_i$, \propeq
\begin{itemize}
\item $\fp\in\pi^\star(\Spec\,\gT').$
\item $\theta_\fp$ \gui{réalise un modèle de $R$}, \cad $\Tt i\;\;
\theta_\fp(x_i)=\theta_\fp(y_i).$
\item  $\Tt i\;\; (x_i\in\fp\;\Leftrightarrow\; y_i\in\fp).$
\end{itemize}
\end{flemmac}

Dans le \tho suivant  nous identifions $\Spec\,\gT'$ à une partie de
$\Spec\,\gT$ au moyen de l'injection $\pi^\star$.
Des résultats analogues énoncés dans un langage un peu différent se trouvent dans \citet*[section~3]{fEsc2001}\footnote{Escard{\'o} écrit son article dans le langage des locales. Il parle de patch topology plutôt que de topologie constructible. Si $Y=\Spec\,\gT$, il note $\Patch\, Y$ pour l'espace de Stone $\Spec\,\gT^{\rm bool}$.}.

\begin{ftheoremc}[\dfn et caractérisations des sous-espaces 
spectraux]
\label{fpropSESP} ~
\begin{enumerate}
\item Avec les notations du lemme \ref{flemSESP},  $X$ est un 
\textsl{sous-espace
topologique} de $Y$. En outre $\OQC(X)=\sotq{U\cap X}{U\in\OQC(Y)}$. 
On dit que
\textsl{$X$ est un sous-espace spectral de~$Y$.}
\item Pour qu'une partie $X$ d'un espace spectral $Y$ soit un sous-espace
spectral il faut et suffit que les conditions suivantes soient 
vérifiées: \\
-- La topologie induite par $Y$  fait de $X$ un espace spectral, et\\
--  $\OQC(X)=\sotq{U\cap X}{U\in\OQC(Y)}$.
\item Une partie $X$ d'un espace spectral $Y$ est un sous-espace 
spectral \ssi
elle est fermée pour la topologie constructible.
\item Si $Z$ est une partie arbitraire d'un espace spectral 
$Y=\Spec\,\gT$ son
adhérence pour la topologie constructible est égale à 
$X=\Spec\,\gT'$ où
$\gT'$ est le treillis quotient de~$\gT$ défini par la relation de 
préordre
$\preceq$ suivante:
\begin{equation} \label{feqSSES}
a\preceq b\quad \Longleftrightarrow\quad (\DT(a)\cap Z)\subseteq 
(\DT(b)\cap Z).
\end{equation}
En outre, $X$ est le plus petit sous-espace spectral de $Y$ contenant 
$Z$.
\end{enumerate}
\end{ftheoremc}
\begin{proof}
Le point \textsl{1} est facile, et définit la notion de sous-espace 
spectral. Le point \textsl{2} en résulte. Le point~\textsl{3} résulte  des points \textsl{2} et \textsl{4}. Montrons le 
point \textsl{4}. \\
Remarquons tout d'abord que la relation (\ref{feqSSES}) définit bien 
un treillis quotient $\gT'$ car les relations~(\ref{feqPreceq}) sont 
trivialement vérifiées si on tient compte des relations (\ref{feqDa}).\\
Montrons que $X=\Spec\,\gT'$ est le plus petit sous-espace spectral 
de $Y$ contenant $Z$.\\
Tout d'abord $Z\subseteq X$: soit $\fp\in Z$, nous voulons montrer 
que si
$b\in\fp$ et $a\preceq b$ alors $a\in\fp$. Si $\DT(a)\cap Z \subseteq 
\DT(b)$ et
$b\in\fp$ alors $\fp\notin \DT(b)$ donc  $\fp\notin \DT(a)\cap Z$, donc
$\fp\notin \DT(a)$, \cad $a\in\fp$.

\noindent Par ailleurs $X$ est minimal. En effet effet si 
$X_1=\Spec\,\gT_1$ est
un sous-espace spectral de $Y$ contenant~$Z$, on a
$a\leq_{\gT_1}b\;\Leftrightarrow\;(\DT(a)\cap X_1)\subseteq 
(\DT(b)\cap X_1)$ ce
qui implique $(\DT(a)\cap Z)\subseteq (\DT(b)\cap Z)$ et donc   
$a\leq_{\gT'}b$, 
d'où $X\subseteq X_1$.\\
  Il reste à montrer que $X$ est l'adhérence de $Z$ pour la 
topologie
constructible. Notons $\wi Z$ cette adhérence.
Nous voulons donc démontrer pour tout $\fp\in\Spec\,\gT$ 
l'équivalence des
deux propriétés suivantes:\\
(1) $\fp\in\wi Z$, \cad:  $\Tt a,b\in\gT ,\;(\fp\in \DT(a)\cap
\VT(b)\,\Rightarrow\, \DT(a)\cap \VT(b)\cap Z\neq \emptyset) $,\\
(2)  $\fp\in\Spec\,\gT'$.\\
Or (2) équivaut successivement à 

\vspace{-1.3em}
\[\begin{array}{lcr}
\Tt a,b\in\gT\;\;((a\preceq b,\,b\in\fp)\Rightarrow a\in\fp) &\quad    
& (3)
\\[1mm]
\Tt a,b\in\gT\;\;(a\preceq b,\,b\in\fp,\, a\notin\fp) \;\mathrm{sont\;
incompatibles} &\quad    & (4)  
\\[1mm]
\Tt a,b\in\gT\;\;\DT(a)\cap Z\subseteq  \DT(b) 
\;\,\mathrm{et}\,\;\fp\in
\DT(a)\cap \VT(b) \,\;\mathrm{sont\; incompatibles} &\quad    & (5) 
\\[1mm]
\Tt a,b\in\gT\;\;\DT(a)\cap \VT(b)\cap Z=\emptyset 
\;\,\mathrm{et}\,\;\fp\in
\DT(a)\cap \VT(b) \,\;\mathrm{sont\; incompatibles} &\quad    & (6)
\end{array}\]

\vspace{-.4em}
\noindent et (6)  est clairement équivalent à (1).
\end{proof}

\begin{fcorollaryc}
\label{fcorpropSESP}
Toute réunion finie et toute intersection de sous-espaces spectraux 
de
$X=\Spec\,\gT$ est un sous-espace spectral.
\begin{itemize}
\item Si $X_i=\Spec\,\gT_i$ pour un quotient $\pi_i\colon \gT\to\gT_i$ alors 
$\bigcap_iX_i$
correspond au quotient engendré par toutes les relations 
$\pi_i(x)=\pi_i(y)$.
\item  Si la famille est finie alors $\bigcup_iX_i$ correspond au quotient par la 
relation
$\&_i\big(\pi_i(x)=\pi_i(y)\big)$.
%
\end{itemize}

\end{fcorollaryc}

\begin{fpropositionc}[ouverts et fermés de base]
\label{fpropositionOFBSES}  Soit $\gT$ un \trdi et $X=\Spec\,\gT$.
\begin{enumerate}
\item $\DT(a)$ est un sous-espace spectral de $X$ canoniquement 
homéomorphe
à $
\Spec(\gT/(a=1))$.
\item  $\VT(b)$ est un sous-espace spectral de $X$ canoniquement 
homéomorphe
à $
\Spec(\gT/(b=0))$.
\end{enumerate}
\end{fpropositionc}
\begin{proof}
Soit $x\preceq y$ l'ordre partiel correspondant au sous-espace 
spectral
$\DT(a)$. On a donc:
\[x\preceq y\;\Leftrightarrow\;\DT(x) \cap \DT(a)\subseteq \DT(y) \cap
\DT(a)\;\Leftrightarrow\;\DT(x\vi a)\subseteq \DT(y\vi 
a)\;\Leftrightarrow\;x\vi
a\leq y\vi a
\]
et ceci est bien la relation de préordre correspondant au quotient
$\Spec(\gT/(a=1)).$\\
Soit maintenant $x\preceq' y$ l'ordre partiel correspondant au 
sous-espace spectral $\VT(b)$. On a:
\[\begin{array}{rcccl}
x\preceq' y& \Longleftrightarrow  &\DT(x) \cap \VT(b)\subseteq \DT(y) 
\cap
\VT(b)   & \Longleftrightarrow   &  \DT(x) \cup \DT(b)\subseteq 
\DT(y) \cup
\DT(b) \\[1mm]
& \Longleftrightarrow  &  \DT(x\vu b)\subseteq \DT(y\vu b) & 
\Longleftrightarrow
&   x\vu b\leq y\vu b
\end{array}\]
et ceci est bien la relation de préordre correspondant au quotient
$\Spec(\gT/(b=0)).$
\end{proof}
\subsubsection*{Fermés de $\Spec\,\gT$}
\addcontentsline{toc}{subsubsection}{Fermés de $\Spec\,\gT$}

Dans ce paragraphe $\gT$ est un \trdi fixé et $X=\Spec\,\gT$.
Si $Z\subseteq X$ on notera $\ov{Z}$ l'ahérence de $Z$ pour la 
topologie
usuelle de $X$.

\begin{fpropositionc}[sous-ensembles fermés de $\Spec\,\gT$]
\label{fpropositionFSES} ~
\begin{enumerate}
\item Un fermé arbitraire de $\Spec\,\gT$ est de la forme
$\VT(\fJ)=\bigcap_{x\in \fJ}\VT(x)$ où $\fJ$ est un \id arbitraire 
de $\gT$.
C'est un sous-espace spectral et il correspond au quotient 
$\gT/(\fJ=0)$.
\item
L'intersection d'une famille de fermés correspond au sup de la 
famille
d'idéaux. La réunion de deux fermés correspond à 
l'intersection des deux
idéaux.
\item\label{fenumtra}
Le treillis $\gT/((a:b)=0)$ est le quotient correspondant à 
$\ov{\VT(a)\cap
\DT(b)}$.
\item
  Donc $\gT$ est une \agH \ssi $X$ vérifie la propriété 
suivante: pour tous
\oqcs $U_1$ et $U_2$, l'adhérence de $U_1\setminus U_2$ est 
le complémentaire d'un 
\oqc. 
\item \`A l'adhérence de $\DT(x)$ correspond le quotient 
$\gT/((0:x)=0)$.
\item Donc à la frontière de $\DT(x)$ correspond le quotient
$\gT\ul x=\gT/(\rK_\gT^x=0)$, où
\begin{equation} \label{feqbordsup}
\rK_\gT^x\,=\,\dar x \,\vu\, (0:x).
\end{equation}
Le treillis $\gT\ul x$ sera appelé \emph{le bord supérieur (de 
Krull) de $x$
dans $\gT$}. On dira aussi que $\rK_\gT^x$ est \emph{l'\id bord de 
Krull de $x$
dans $\gT$.}\\
Lorsque $\gT$ est une \agH,  $\rK_\gT^x\,=\,\dar (x \,\vu\, \lnot x)$ 
et
$\gT\ul x\simeq \uar (x \,\vu\, \lnot x)$ avec l'\homo surjectif
$\pi\ul x:
\left|
\begin{array}{rcl}
\gT& \to  & \uar (x \,\vu\, \lnot x)  \\
y&  \mapsto  & y \,\vu\, x \,\vu\, \lnot x
\end{array}
  \right.
$.
\end{enumerate}
\end{fpropositionc}
\begin{proof}
Pour le seul point délicat (le point \textsl{\ref{fenumtra}}), on dit: puisque
$(a:b)=\sotq{x}{x\vi b\leq a}$, la variété associée
$\VT(a:b)$ est l'intersection des $\VT(x)$ tels que $\VT(a)\subseteq 
\VT(x)\cup
\VT(b)$, \cade tels que  $\VT(a)\cap \DT(b)\subseteq \VT(x)$. Or tout 
fermé de
$\Spec\, \gT$ est une intersection de fermés de base $\VT(x)$, donc 
on obtient
bien l'adhérence de $\VT(a)\cap \DT(b)$.
\end{proof}

\rems

\noindent 1) On notera qu'un ouvert arbitraire de $X$ n'est pas en 
général
un sous-espace spectral.

\noindent 2) La \dfn que nous avons donnée pour le treillis bord $\gT\ul x$
est clairement \cov. 
Notre traduction dans le point \textsl{6} du bord d'un \oqc en termes de treillis bord quotient est correcte en \clama. La démonstration nécessite les \clama 
car elle utilise la fait que l'\sps $\Spec\,\gT$ a suffisamment de points.   
\eoe

\medskip Le lemme suivant permet de mieux cerner l'\id 
bord de Krull de $x$ dans~$\gT$.

\begin{flemma}
\label{flemBKReg}
Pour tout $x\in\gT$ l'\id $\fj$ bord de Krull de $x$ dans $\gT$ est
\emph{régulier}, \cad \gui{son annulateur est réduit à $0$}. I.e.
$0:\fj=0$.
\end{flemma}
\begin{proof}
Soit $u\in(0:\fj)$. Puisque $\fj=\dar x\vu (0:x)$ on a $u\vi x=0$ et, 
pour tout
$z\in(0:x)$, $u\vi z=0$. En particulier $u\vi u=0$.
\end{proof}

Notez que dans le cas d'une \agH il ne s'agit de rien d'autre que de 
la loi
découverte par Brouwer: $\lnot(x \,\vu\, \lnot x)=0$.

\smallskip La proposition qui suit est la version duale, \cov,  
\gui{sans
points} du fait topologique suivant: si $A$ et $B$  sont fermés, la 
réunion
des bords de $A\cup B$  et de  $A\cap B$  est égale à celle des 
bords de $A$
et $B$.
\begin{fproposition}
\label{fpropBordKUnion}
Pour tous $x,y\in\gT$ on a
$\;\rK_\gT^{x}\cap \rK_\gT^{y}= \rK_\gT^{x\vu y}\cap 
\rK_\gT^{x\vi
y}.$
\end{fproposition}
\begin{proof}
Soit $z\in\rK_\gT^{x}\cap \rK_\gT^{y}$, autrement dit il existe $u$ 
et $v$ tels
que $z\leq x\vu u$ et $u\vi x=0$, $z\leq y\vu v$ et $v\vi y=0$.
Alors $z\leq (x\vu (u\vu v))\vi (y\vu (u\vu v)) = (x\vi y)\vu(u\vu 
v)$ avec
$(u\vu v)\vi(x\vi y)=(u\vi(x\vi y))\vu (v\vi(x\vi y))=0\vu 0=0$ donc
$z\in\rK_\gT^{x\vi y}$. \\
De même $z\leq (x\vu y)\vu(u\vi v)$ avec $(u\vi v)\vi(x\vu y)=0$ 
donc
$z\in\rK_\gT^{x\vu y}$.\\
Enfin supposons  $z\in\rK_\gT^{x\vu y}\cap \rK_\gT^{x\vi y}$, 
autrement dit il
existe $u$ et $v$ tels que $z\leq x\vu y\vu u$ et $u\vi (x\vu y)=0$,
$z\leq (x\vi y)\vu v$ et $v\vi x\vi y=0$.
Soit $u_1=(y\vu u)\vi v$. On a  $z\leq (x\vi y)\vu v\leq x\vu v$ et 
$z\leq x\vu
(y\vu u)$ donc $z\leq x\vu u_1$. Par ailleurs $x\vi u_1=x\vi (y\vu u) 
\vi v\leq
x\vi y\vi v =0$ et donc $z\in\rK_\gT^{x}$.
\end{proof}

\subsubsection*{Fermés de $\Spec\,\gT\cir$}
\addcontentsline{toc}{subsubsection}{Fermés de $\Spec\,\gT\cir$}

\hum{Dans l'article original, on avait partout $\DT(\cdot)$ à la place de $\VTo(\cdot)$, mais cette notation semble très peu naturelle pour $\DT(F)$.}

Notons qu'il est naturel de noter $\VTo(a)$ pour $\DT(a)$.
Introduisons alors la notation suivante, pour $F\subseteq \gT$:
$\VTo(F)=\bigcap_{a\in\fF} \VTo(a)$. Si $\fF$ est le filtre engendré 
par $F$, on
a $ \VTo(F)=\VTo(\fF)$.

La proposition suivante découle de la proposition 
\ref{fpropositionFSES} par
renversement de l'ordre (on n'a réécrit que les points \textsl{1} et \textsl{6}) 
modulo
l'identification des ensembles sous-jacents à $\Spec\,\gT$ et
$\Spec\,\gT\cir$.

  Rappelons que la notion opposée à l'idéal $a:b$ est le filtre 
$a\setminus
b\eqdefi \sotq{z}{z\vu b\geq a}$. 

\begin{fpropositionc}[sous-ensembles fermés de $\Spec\,\gT\cir$]
\label{fFdSES} ~
\begin{enumerate}
\item Un fermé arbitraire de $\Spec\,\gT\cir$ est de la forme
$\bigcap_{x\in\fF} \VTo(x)$ où $\fF$ est un filtre arbitraire 
de~$\gT$.
C'est le sous-espace spectral qui correspond au quotient 
$\gT/(\fF=1)$.

\item On définit le quotient
$\gT\bal x=\gT/(\rK^\gT_x=1)$, où $\rK^\gT_x$ est le filtre
\begin{equation} \label{feqbordinf}
\rK^\gT_x\,=\,\uar x \,\vi\, (1\setminus x)
\end{equation}
Le treillis $\gT\bal x$ sera appelé \emph{le bord inférieur (de 
Krull) de
$x$ dans $\gT$}.
On dira aussi que $\rK^\gT_x$ est \emph{le filtre bord de Krull de 
$x$}.\\
Lorsque $\gT$ est une \alg de Brouwer, $\rK^\gT_x\,=\,
\uar (x \,\vi\, (1- x))$ et
$\gT\bal x\simeq \dar (x \,\vi\, (1- x))$ avec l'\homo surjectif
$\pi\bal x:
\left|
\begin{array}{rcl}
\gT&\to&\dar (x \,\vi\, (1- x))\\
  y&\mapsto& y \,\vi\, x \,\vi\, (1- x)
\end{array}
  \right.
$.\\
Le quotient $\gT\bal x$ correspond à la 
frontière de $\VTo(x)$ pour la topologie de 
$\Spec\,\gT\cir$\footnote{Il s'agit de la notion \gui{opposée} à celle de frontière. L'intersection des adhérences de $\VTo(x)=\DT(x)$ et $\DTo(x)=\VT(x)$ est remplacée par la réunion de leurs intérieurs. Dans $\Spec\,\gT$, c'est donc le complémentaire de la frontière de~$\DT(x)$.}.
\end{enumerate}
\end{fpropositionc}

\subsubsection*{Recollement d'espaces spectraux}
\addcontentsline{toc}{subsubsection}{Recollement d'espaces spectraux}

\hum{Ajout de la phrase suivante.}
Comme la théorie des \trdis est purement équationnelle, la catégorie possède des limites inductives et projectives arbitraires. Les limites projectives et les limites inductives filtrantes sont conservées par le foncteur d'oubli dans la catégorie des ensembles. Les propriétés duales sont donc satisfaites dans la catégorie antiéquivalente des espaces spectraux et morphismes spectraux. Dans certains cas
ces limites correspondent à celles obtenues dans la catégorie des espaces topologiques et applications continues.

Voici ce que donne par dualité la proposition~\ref{fpropRecolTD}
(on laisse \alec la traduction du fait~\ref{ffactRecolTD}).

\begin{fpropositionc}[recollement d'une famille finie d'espaces spectraux le long  d'\oqcs]
\label{fpropRecolSpec}~
\begin{enumerate}
\item Soit $(X_i)_{1\leq i\leq n}$ une famille finie d'espaces 
spectraux, et
pour chaque $i\neq  j$ un \oqc $X_{ij}$ de $X_i$ avec 
un \iso
$\varphi_{ij}\colon X_{ij}\to X_{ji}$. On suppose que $\varphi_{ij}= \varphi_{ji}^{-1}$ pour tous $i,j$ et que les relations de compatibilité 
naturelles \gui{trois par trois}
sont vérifiées: si $x=\varphi_{ji}(y)=\varphi_{ki}(z)$ alors $\varphi_{jk}(y)=z$. Alors la limite inductive du diagramme dans la 
catégorie des espaces spectraux est un espace $X$ pour lequel chacun des $X_i$ s'identifie à 
un \oqc via le morphisme $X_i\to X$. 
\item En remplaçant \gui{ouvert quasi-compact} par \gui{fermé de base} le résultat analogue est \egmt valable.
\end{enumerate}

\hum{Dans le point 1, $X_i$ s'identifie à un \oqc de $X$ via le morphisme $X_i\to X$. Dans l'article original il y avait seulement écrit \gui{$X_i$ est un sous-espace spectral}. Le commentaire qui suit est 
en outre plus étoffé que dans l'article original.}

\end{fpropositionc}

\comm Notez que $X$ est aussi la limite 
inductive du diagramme formé par les $X_i$, les inclusions $f_{ij}:X_{ij}\to X_i$ et les isomorphismes $\varphi_{ij}$,
dans la catégorie des ensembles et dans celle des espaces topologiques.
En langage plus imagé: l'espace topologique~$X$ s'obtient comme recollement des espaces $X_i$ le long des $X_{ij}$ en identifiant $x\in X_{ij}$ à $\varphi_{ij}(x)\in X_{ji}$.

En \clama on aurait plutôt tendance à déduire la proposition~\ref{fpropRecolTD}  de la proposition~\ref{fpropRecolSpec}
car cette dernière a une démonstration directe facile. Cependant ce raccourci élégant ne permet pas d'obtenir la \prco de la proposition~\ref{fpropRecolTD}. 

Dans le point \textsl{2} de la proposition~\ref{fpropRecolSpec}, si on prend pour $X_{ij}$ des fermés arbitraires, (qui sont des sous-espaces spectraux) au lieu de
 fermés de base, le recollement aura lieu en tant qu'espaces topologiques mais ne fournirait pas nécessairement un espace spectral.
 Dans le point \textsl{1} si on prend une infinité d'ouverts de base, le recollement aura lieu en tant qu'espaces topologiques mais ne fournirait pas nécessairement un espace spectral. \eoe

\subsection{Spectre maximal et spectre de Heitmann}

Dans l'article remarquable \cite[\textsl{Generating non-Noetherian modules efficiently}]{fHei84} Raymond Heitmann explique que la notion
usuelle de j-spectrum pour un anneau commutatif n'est pas la bonne 
dans le cas
non noethérien car elle ne correspond pas à un espace spectral au sens de
Stone. Il introduit la modification suivante de la \dfn usuelle: au lieu
de considérer l'ensemble des \ideps qui sont intersections 
d'\idemas il propose de considérer l'adhérence du spectre maximal dans le 
spectre premier, adhérence à prendre au sens de la topologie 
constructible (la patch topology).

\begin{fdefinitions}
\label{fdefHspec1}
Soit $\gT$ un \trdi. 
\begin{enumerate}
\item  On note $\Max\,\gT$ le sous-espace topologique de $\Spec\,\gT$ formé par les \idemas de~$\gT$. 
On l'appelle le \textsl{spectre maximal de~$\gT$}.
\item  On note $\jspec\,\gT$ le sous-espace topologique de 
$\Spec\,\gT$ formé
par les $\fp$ qui vérifient l'égalité \hbox{$\JT(\fp)=\fp$}, \cad les  \ideps $\fp$ 
qui sont
intersections d'\idemas (c'est le j-spectrum \gui{usuel}).
\item On appelle \textsl{$\rJ$-spectre de Heitmann de $\gT$} et on note
$\Jspec\,\gT$ l'adhérence du spectre maximal dans $\Spec\,\gT$, 
adhérence à prendre au sens de la topologie constructible. Cet ensemble est muni de la topologie induite par $\Spec\,\gT$.
\item  On note $\Min\,\gT$ le sous-espace topologique de $\Spec\,\gT$ 
formé par les \ideps minimaux de $\gT$. On l'appelle le \textsl{spectre  minimal de $\gT$}.
\end{enumerate}
\end{fdefinitions}

Notez que malgré leurs dénominations, les espaces topologiques $\Max\,\gT$,  $\jspec\,\gT$ 
et
$\Min\,\gT$ ne sont pas en général des espaces spectraux.

\begin{ftheoremc}
\label{fthDK3}
Soit $\gT$ un \trdi.  L'espace  $\Jspec\,\gT$ est un sous-espace 
spectral de
$\Spec\,\gT$ canoniquement homéomorphe à $\Spec\,\He(\gT)$. Plus
précisément, si $M=\Max\,\gT$,  on a pour
$a,b\in\gT$:
\begin{equation} \label{feqthDK3}
\DT(a)\cap M\,\subseteq\, \DT(b)\cap M \quad \Longleftrightarrow\quad
a\preceq_{\He(\gT)}b.
\end{equation}
\end{ftheoremc}
\begin{proof}
Le treillis $\He(\gT)$ a été défini page
\pageref{defHeT}.\\
La deuxième affirmation implique la première. En effet 
$\Jspec\,\gT$,
d'après le \tho \ref{fpropSESP}\,(4), est le spectre du treillis 
quotient
$\gT'$ correspondant à la relation de préordre $a\leq _{\gT'}b$ 
définie
par $\DT(a)\cap M\,\subseteq\, \DT(b)\cap M$.\\
Pour la deuxième affirmation on remarque que \propeq
\[\begin{array}{lcl}
\DT(a)\cap M\subseteq \DT(b)\cap M& \;  &  (1)  \\ [1mm]
\DT(a)\cap \VT(b)\cap M =\emptyset  &&  (2)  \\[1mm]
\Tt \fm\in M\; (b\notin\fm \mathrm{\;ou\;} a\in\fm)  &&  (3)  \\[1mm]
  \Tt \fm\in M\; (b\in\fm\Rightarrow  a\in\fm) &&  (4)
\end{array}\]
Et l'assertion (4) revient à dire que, vu dans le treillis quotient 
$\gT/(b=0)$, $a$
appartient au radical de Jacobson. Cela signifie $a\in \JT(b)$,
\cad $a\preceq_{\He(\gT)}b$.
\end{proof}

Quelques points de comparaison.

\begin{ffactc}
\label{ffactSpec=Jspec} ~
\begin{enumerate}
\item $\Spec\,\gT=\Jspec\,\gT$ \ssi $\gT=\He(\gT)$, \cad si $\gT$ est 
faiblement
Jacobson.
\item $\Max\,\gT=\Jspec\,\gT$ \ssi $\He(\gT)$ est une \agB. 
\item Si $\gT$ possède un complément de Brouwer, $\Max\,\gT$ est 
un fermé
de $\Spec\,\gT$, correspondant à l'\id $\Imax(\gT)$. C'est un 
espace de Stone,
il est égal à $\Jspec\,\gT$.
\item $\Min\,\gT=\Jspec\,\gT\cir$ \ssi $\He(\gT\cir)$ est une \agB. 
\item Si $\gT$ possède une négation, $\Min\,\gT$  est un fermé 
de
$\Spec\,\gT\cir$, correspondant au filtre $\Fmin(\gT)$. C'est un 
espace de
Stone, il est égal à $\Jspec\,\gT\cir$.
\end{enumerate}
\end{ffactc}
\begin{proof}
Le point \textsl{1} résulte du \tho \ref{fthDK3}. Pour le point \textsl{2}, (on est en 
\clama)
on remarque qu'un treillis distributif est une \agB \ssi ses \ideps 
sont tous
maximaux. Le point \textsl{3} résulte du point \textsl{2} et du fait 
\ref{ffactSpecMax}. Les
points \textsl{4} et \textsl{5} s'obtiennent à partir des points \textsl{2} et \textsl{3} en passant au 
treillis
opposé.
\end{proof}

La proposition suivante est signalée par Heitmann dans 
\cite{fHei84}. L'hypothèse dans le point \textsl{2} est que l'espace $M=\Max\,\gT$ est noethérien, \cad que tout ouvert est \qc.  Comme la toplogie de $M$ est induite par celle de $\Spec\,\gT$, les ouverts $\fD_\gT(a)\cap M$ forment une base de la topologie. Par ailleurs on a  $\fD_\gT(a_1 \vu\dots\vu a_n)=\fD_\gT(a_1)\cup\dots\cup\fD_\gT(a_n)$. Donc, lorsque~$M$ est noethérien, tout ouvert de~$M$
est de la forme $\fD_\gT(a)\cap M$ et tout fermé est un fermé de base $\fV_\gT(a)\cap M$.
\begin{fpropositionc}[comparaison of $\Jspec$ and $\jspec$]
\label{fpropJspecjspec} ~
\begin{enumerate}
\item On a toujours $\jspec\,\gT\subseteq \Jspec\,\gT$.
\item Si $M=\Max\,\gT$ est noethérien, a fortiori si $\Spec\,\gT$ est noethérien, on a 
$\jspec\,\gT=\Jspec\,\gT$.
\end{enumerate}
\end{fpropositionc}
\begin{proof}
On considère un \elt $\fp$ fixé de $\Spec\,\gT$.
Tout d'abord les propriétés suivantes sont successivement 
équivalentes
\[\begin{array}{rcl}
\fp\in\jspec\,\gT&   &     \\[1mm]
\Tt a\in \gT\quad [\;\fp\in \DT(a)&  \Rightarrow  &
\Ex \fm\in (M\,\cap \DT(a)),\, \fp\subseteq \fm \;]  \\[1mm]
\Tt a\in \gT\quad [\;\fp\in \DT(a)&  \Rightarrow  &
\Ex \fm\in (M\,\cap  \DT(a)),\, \Tt b\in\gT\,(\fm\in \DT(b)
\Rightarrow \fp\in \DT(b))\;] \\[1mm]
\Tt a\in \gT\quad [\;\fp\in \DT(a)&  \Rightarrow  &
\Ex \fm\in (M\,\cap \DT(a)),\, \Tt b\in\gT\,(\fp\in \VT(b)
\Rightarrow \fm\in \VT(b))\;]\\[1mm]
\Tt a\in \gT\quad [\;\fp\in \DT(a)&  \Rightarrow  &
\Ex \fm\in M,\, \Tt b\in\gT\,(\fp\in \VT(b)
\Rightarrow \fm\in  M\cap \DT(a) \cap \VT(b) )\;]\qquad (*)
\end{array}\]
De même sont équivalentes les propriétés
\[\begin{array}{rcl}
\fp\in\Jspec\,\gT&   &     \\[1mm]
\Tt a,b\in \gT\quad [\;\fp\in (\DT(a)\cap \VT(b))&  \Rightarrow  &
\Ex \fm\in M\cap \DT(a) \cap \VT(b)\;]
\qquad\qquad  (**)
\end{array}\]
\textsl{1}. Ainsi on voit que $(*)$ est plus fort que $(**)$, puisque dans $(**)$,
$\fm$ peut dépendre de $a$ et $b$ tandis que dans $(*)$ il ne doit 
dépendre
que de $a$.\\
\textsl{2}. Si $M=\Max\,\gT$ est noethérien le fermé 
$\bigcap\limits_{\VT(b)\ni\fp}
(M\cap \VT(b))$ est égal
à un fermé de base $M\cap \VT(b_0)$ et  $(**)$ avec ce $b_0$ donne $(*)$.
\end{proof}

\comm
Comme le fait remarquer Heitmann,  les \thos connus utilisant le 
$\jspec$ le
font toujours sous l'hypothèse \gui{$\Max$ noethérien}. Il est 
donc
probable que  $\Jspec$ soit la seule notion vraiment intéressante.
Notez que $\jspec\,\gT$ n'est un sous-espace spectral de $\Spec\,\gT$ 
que lorsqu'il est égal à $\Jspec\,\gT$.\eoe

\hum{

\textsl{1}. La notion opposée du $\Jspec$ n'est pas sans intérêt en 
\alg
commutative:
par renversement, le radical de Jacobson est remplacé par le filtre
complémentaire de la réunion des \ideps minimaux.

\textsl{2}. Pour le point \textsl{2}, l'hypothèse fonctionnerait-elle avec $\Jspec$ noethérien ou $\Max$  noethérien~? Quel est le rapport entre les trois conditions de noethérianité ($\jspec\,\gT$, $\Jspec\,\gT$ et $\Max\,\gT$)~? En fait Heitmann affirme au début de l'article comme une évidence que l'on a $\Jspec=\jspec$ quand le spectre maximal est noethérien.
Il discute un peu plus cette question à la fin de l'article.

\textsl{3}. Il serait intéressant de formuler l'hypothèse \gui{$\Max\,\gT$ est noethérien} sans utiliser les points de $\Spec\,\gT$, \cad comme une propriété du \trdi $\gT$ formulée \cot.   
}

\junk {Ajout non retenu dans la version finale: cette section 2.4 

\subsection{Quelques points de l'anti\eqvc de catégories}
\label{fsubsecAntiEquiv}

Signalons sans démonstration les résultats généraux suivants concernant les morphismes dans la catégorie des \trdis et dans celle des espaces spectraux (voir le théorème de Krull page~\pageref{fThKrull}, \cite[Theorem~IV-2.6]{fBW74} et \cite{fLom2020}). 

Le contexte est le suivant: soit $f:\gT\to\gT'$ un morphisme de treillis distributifs et $\Spec(f)$, noté $f^\star$, le morphisme dual, de $X=\Spec(\gT')$ vers $Y=\Spec(\gT)$, dans la catégorie des espaces spectraux.

Rappelons quelques \dfns usuelles en \clama. 
\begin{itemize}
\item Le morphisme $f$ est dit \textsl{lying over} (en français, il possède la propriété de relèvement) lorsque~$f^\star$ est surjectif: tout \idep de $\gT$ est image réciproque d'un \idep de~$\gT'$.
\item  Le morphisme $f$ est dit \textsl{going up} (en français, il possède la propriété de montée pour les cha\^{\i}nes d'\ideps) lorsque l'on a: \textsl{si $\fq\in\Spec(\gT')$, $f^\star(\fq)=\fp$, et $\fp\subseteq\fp_2$ dans $\Spec(\gT)$, il existe  $\fq_2\in\Spec(\gT')$ tel que
$\fq\subseteq\fq_2$ et $f^\star(\fq_2)=\fp_2$}.
\item  De même $f$ est dit \textsl{going down} (en français, il possède la propriété de descente pour les cha\^{\i}nes d'\ideps) lorsque l'on a: \textsl{si $\fq\in\Spec(\gT')$, $f^\star(\fq)=\fp$, et $\fp\supseteq\fp_2$ dans $\Spec(\gT)$, il existe  $\fq_2\in\Spec(\gT')$ tel que $\fq\supseteq\fq_2$ et  $f^\star(\fq_2)=\fp_2$}.
\item  On dit que le morphisme $f$ \textsl{possède la propriété d'incomparabilité} lorsque ses  \gui{fibres} sont formées d'\ideps deux à deux incomparables: si $\fq_1\subseteq \fq_2\in X$ et $f^\star(\fq_1)=f^\star(\fq_2)$ dans~$Y$ alors $\fq_1= \fq_2$.
\item  L'espace spectral $\Spec(\gT)$ est dit \textsl{normal} si tout point est majoré par un unique point fermé (tout \idep de $\gT$ est contenu dans un unique \idema).
%
\end{itemize}

\begin{ftheorem} \label{th-dico-trdi-spec-mor}
 On a les équivalences suivantes.   
\begin{enumerate}
\item $f$ est injectif $\Leftrightarrow$ $f$ est un monomorphisme $\Leftrightarrow$ $f^\star$ est un épimorphisme $\Leftrightarrow$ $f^\star$ est surjectif (lying over).
\item $f$ est un épimorphisme $\Leftrightarrow$ $f^\star$ est un monomorphisme $\Leftrightarrow$ $f^\star$ est injectif.
\item $f$ est surjectif\footnote{Autrement dit, puisque c'est une structure équationnelle, $f$  est un morphisme de passage au quotient.} $\Leftrightarrow$ $f^\star$ est un isomorphisme sur son image, qui est un
sous-espace spectral de $Y$ (voir la section \ref{fsecSESP}, en particulier le lemme \ref{flemSESP}, le théorème~\ref{fpropSESP} et la proposition~\ref{fpropositionFSES}).
\item $f$ est going up $\Leftrightarrow$  pour
tous $a,c\in\gT$ et $y\in\gT'$ on a
\[
f(a)\leq f(c)\vu y \;\Rightarrow\;\exists x\in\gT\; (a\leq c \vu x \hbox{ et } f(x)\leq y).
\] 
\item $f$ est going down $\Leftrightarrow$  pour
tous $a,c\in\gT$ et $y\in\gT'$ on a
\[
f(a)\geq f(c)\vi y \;\Rightarrow\;\exists x\in\gT\; (a\geq c \vi x \hbox{ et } f(x)\geq y).
\]
\item $f$ possède la propriété d'incomparabilité $\Leftrightarrow$ $f$ est zéro-dimensionnel\footnote{Voir ci-dessous \dots.}.   
\end{enumerate}

\end{ftheorem}

\smallskip Il peut cependant y avoir des épimorphismes de \trdis non surjectifs (voir \cite[\hbox{section V-8}]{fBW74}).
Cela correspond à la possibilité d'un morphisme bijectif entre espaces spectraux qui ne soit pas un isomorphisme. Par exemple le morphisme spectral bijectif  $\Spec\,\gT^{\rm
bool}\to\Spec\,\gT$ n'est pas (en général) un \iso et le morphisme de treillis $\gT\to\gT^{\rm bool}$ est un \gui{épimono} qui n'est pas (en général) surjectif.
}

\section[Dimensions de Krull et de Heitmann: \trdis]{Dimensions de 
Krull et de
Heitmann pour un \trdi}
\label{fsecHtrdi}

  Nous arrivons dans cette section au coeur de l'article.
Nous reprenons le point de vue des \coma. 
Les seules preuves non \covs sont celles qui font le lien entre une 
notion
classique et sa reformulation \cov. 
En \clama la \textsl{dimension de Krull} d'un \trdi est définie comme 
en \alg
commutative: c'est la borne supérieure des longueurs des 
chaines
strictement croissantes d'idéaux premiers.

\subsection{Dimension et bords de Krull}

Nous rappelons maintenant une version constructive \elr de la 
dimension de Krull (\citet*{fCL2003,fCLR05}) en nous appuyant sur l'intuition suivante: une  variété est de dimension $\leq k$ \ssi le bord de toute sous-variété est de dimension $\leq k-1$. Des approches \covs sensiblement équivalentes sont dans \cite{fEsp78,fEsp82,fEsp83}.

Le \tho suivant en \clama nous donne une bonne signification 
intuitive de la
dimension de Krull d'un \trdi. 

\begin{ftheoremc}
\label{fthDK1} Soit un \trdi $\gT$ engendré par une partie $S$ et 
$\ell$
un entier positif ou nul.
Les conditions suivantes sont équivalentes.
\begin{enumerate}
\item  Le treillis $\gT$ est de dimension $\leq \ell.$
\item  Pour tout $x\in S$ le
bord  $\gT\ul{x}$ est de dimension $\leq \ell-1$.
\item  Pour tout $x\in S$ le
bord  $\gT\bal{x}$ est de dimension $\leq \ell-1$.
\item 
Pour tous $x_0,\ldots,x_\ell\in S$
il existe $a_0$,\ldots,  $a_\ell\in \gT$ vérifiant:
\begin{equation}
     a_0 \vi x_0  \leq  0\,,\;\;\;
     a_1 \vi x_1 \leq   a_0 \vu x_0\,,\;\;\; \dots\;\;,\;\;\;
     a_{\ell} \vi x_{\ell} \leq     a_{\ell-1} \vu x_{\ell-1}\,,\;\;\;
     1  \leq  a_{\ell} \vu x_{\ell}.
\end{equation}
\end{enumerate}
Lorsque $\gT$ est une \agH les conditions précédentes sont 
aussi
équivalentes à
\begin{enumerate}\setcounter{enumi}{4}
\item   Pour toute
suite  $x_0,\dots,x_{\ell}$  dans  $S$  on a l'\egt
\begin{equation}
     1 = x_{\ell}\vu (x_{\ell}\im( \cdots (x_1 \vu (x_1 \im (x_0\vu
\neg x_0)))\cdots))
\end{equation}
\end{enumerate}
Lorsque $\gT$ est une \alg de Brouwer les conditions 
précédentes sont aussi équivalentes à
\begin{enumerate}\setcounter{enumi}{5}
\item   Pour toute
suite  $x_0,\dots,x_{\ell}$  dans  $S$  on a l'\egt
\begin{equation}
     0 = x_{0}\vi (x_{0}-(x_1 \vi (x_1 - ( \cdots (x_\ell\vi
(1- x_\ell))))\cdots))
\end{equation}
\end{enumerate}
\end{ftheoremc}

En particulier un treillis est de dimension $\leq 0$
\ssi  c'est une \agB. 

L'équivalence entre les points \textsl{1}, \textsl{4} et \textsl{5} a été établie, sans utiliser la notion de bord, dans  \cite{fCL2003}, en poursuivant la problématique
d'André Joyal \cite{fJoy71,fJoy76} et de Luis Espa\~nol \cite{fEsp78,fEsp82,fEsp83,fEsp86,fEsp88}.
Citons aussi l'article récent \cite{fEsp2010} sur le sujet.

On notera aussi que la théorie des \agHs de dimension $\leq k$ est une théorie équationnelle.

\begin{proof}[Démonstration
du \tho \ref{fthDK1}]
On commence par noter que le quotient $\gT\ul{x}=\gT/\rK_\gT^x$  peut 
aussi
être vu comme l'ensemble ordonné obtenu à partir de la 
relation
de préordre $\leq ^x$ définie sur~$\gT$ comme suit:
\begin{equation} \label{feqBordSup}
a\leq ^x b \qquad \Longleftrightarrow\qquad \exists y\in 
\gT\;\;(\,x\vi
y=0\;\;\& \;\;a\leq  x\vu y \vu b\,)
\end{equation}
\textsl{1} $\Leftrightarrow$ \textsl{2}:
Nous montrons tout d'abord que tout filtre maximal $\ff $ de $\gT$ 
devient
trivial dans~$\gT\ul{x}$, \cad qu'il contient $0$.
Autrement dit on doit trouver $a$ dans $\ff $ tel \hbox{que $a\leq ^x0$}.
Si $x\in \ff $ alors on prend $a=x$ et $y=0$ dans (\ref{feqBordSup}).
Si $x\notin \ff $ il existe $z\in \ff $ tel \hbox{que $x\vi z= 0$} (puisque 
le
filtre engendré par $\ff $ et $x$ est trivial dans $\gT$) et on 
prend~\hbox{$a=y=z$} dans (\ref{feqBordSup}).
Ceci montre que la dimension de  $\gT\ul{x}$ chute de au moins une
unité par rapport à celle de~$\gT$ (supposée finie).\\
Ensuite nous montrons que si on a deux filtres premiers $\ff '\subset
\ff $, $\ff $ maximal et $x\in \ff \setminus \ff '$ alors~$\ff '$ ne 
devient pas
trivial dans  $\gT\ul{x}$ (ceci montre que la dimension de  
$\gT\ul{x}$
chute de seulement une unité si $x$ est bien choisi). En effet,
dans le cas contraire, on aurait \hbox{un $z\in \ff '$} tel \hbox{que $z\vi x=0$},
mais comme $z$ et $x\in \ff $ cela ferait $0\in \ff $, ce qui est
absurde.\\
Enfin nous remarquons que si $\ff '\subset \ff $ sont des filtres 
premiers
distincts et si $S$ engendre $\gT$ on peut trouver $x\in S$ tel que
$x\in \ff \setminus \ff '$.

\noindent \textsl{1} $\Leftrightarrow$ \textsl{3}: conséquence de \textsl{1} $\Leftrightarrow$ \textsl{2} par renversement de l'ordre.

\noindent \textsl{2} $\Leftrightarrow$ \textsl{4}:  par récurrence sur $\ell$, vue 
la
\dfn du bord.

\noindent \textsl{2} $\Leftrightarrow$ \textsl{5}:  par récurrence sur $\ell$, vue 
la
\dfn du bord
dans le cas d'une \agH. 

\noindent \textsl{3} $\Leftrightarrow$ \textsl{6}:  par récurrence sur $\ell$, vue la \dfn du bord dans le cas d'une \alg de Brouwer.
\end{proof}

Le \tho \ref{fthDK2} qui suit caractérise, en \clama et  en termes \cofs, la \ddk de $\Spec\, \gT$, \cad la longueur maximale des chaines de fermés irréductibles.
Par ailleurs on a montré (point \textsl{6} du \thref{fpropositionFSES}) que si~$X$ est le spectre d'un \trdi $\gT$ 
et $x\in\gT$, alors la  frontière  de l'\oqc~$\DT(x)$ de $X$ est 
canoniquement isomorphe à $\Spec(\gT\ul{x})$.
On obtient donc comme corolaire des \thos \ref{fpropositionFSES} et  \ref{fthDK1} une \carn (en \clama et en termes \cofs) de la dimension des espaces spectraux. Rappelons que l'unique espace spectral
de dimension~$-1$ est le l'espace vide, qui correspond au \trdi trivial $\Un$.

\begin{ftheoremc}
\label{fthDK2}
Soit $k$ un entier $\geq 0$. Un espace spectral $X$ est de dimension 
$\leq  k$ \ssi tout \oqc de $X$ a une frontière (un bord) de dimension~\hbox{$\leq k-1$}.
\end{ftheoremc}

Concernant la dimension de Krull, on choisit en \coma la \dfn 
suivante:

\begin{fdefinition}[définition \cov de la dimension de Krull]
\label{fdefDK0} ~\\
La dimension de Krull (notée $\Kdim$) des \trdis est définie 
comme suit.
\begin{enumerate}
\item $\Kdim(\gT)=-1$ \ssi $1=_\gT0$ (i.e. le treillis est réduit 
à un
point).
\item Pour $\ell\geq 0$ on définit $\Kdim(\gT)\leq \ell$ par les 
conditions
équivalentes suivantes:
\begin{enumerate}
\item \label{fl1}$\forall x\in \gT,\; \Kdim(\gT\ul x)\leq \ell-1$
\item \label{fl2}$\forall x\in \gT,\;\Kdim(\gT\bal{x})\leq \ell-1$
\item \label{fl3}$\forall x_0,\ldots,x_\ell\in \gT$
$\Ex a_0,\ldots,  a_\ell\in \gT$ vérifiant:
\[    a_0 \vi x_0  \leq  0\,,\;\;\;
     a_1 \vi x_1 \leq   a_0 \vu x_0\,,\dots\,,\;\;\;
     a_{\ell} \vi x_{\ell} \leq     a_{\ell-1} \vu x_{\ell-1}\,,\;\;\;
     1  \leq  a_{\ell} \vu x_{\ell}.
\]
\end{enumerate}
\end{enumerate}
\end{fdefinition}

Notez qu'il y a en fait trois \dfns possibles ci-dessus pour 
$\Kdim(\gT)\leq\ell$. Les \dfns basées sur \textsl{\ref{fl1}} et \textsl{\ref{fl2}} sont inductives, tandis que la \dfn basée sur \textsl{\ref{fl3}} est globale. 
L'équivalence des \dfns basées sur
\textsl{\ref{fl1}} et \textsl{\ref{fl3}} est immédiate par induction (même chose pour \textsl{\ref{fl2}} et \textsl{\ref{fl3}}).

Par exemple, pour $\ell=2$ les inégalités dans le point \textsl{2c} correspondent au dessin suivant dans~$\gT$.
\[\SCO{x_0}{x_1}{x_2}{a_0}{a_1}{a_2}\]

Le fait que la \dfn fonctionne aussi bien avec le bord supérieur 
qu'avec le bord inférieur donne la constatation suivante.
\begin{ffact}
\label{fcorTTO}
Un \trdi et le treillis opposé ont même dimension.
\end{ffact}

En \clama on peut s'en rendre compte directement en considérant les
chaines d'\ideps qui ont pour complémentaires des chaines 
de filtres premiers (et vice versa): si on identifie les ensembles sous-jacents à $X=\Spec\,\gT$ et  $X'=\Spec\,\gT\cir$ les relations d'ordre $\leq_X$ et $\leq_{X'}$ sont opposées.

\medskip \rem On peut illustrer le point \textsl{2c} dans la \dfn \ref{fdefDK0}.
Nous introduisons \gui{l'idéal bord de Krull itéré}.
Pour $x_1,\ldots ,x_n\in\gT$ nous notons 
\[
\gT_\rK[x_0]=\gT\ul{x_0},\,\gT_\rK[x_0,x_1]=(\gT\ul{x_0})\ul{x_1}
,\,\gT_\rK[x_0,x_1,x_2]=((\gT\ul{x_0})\ul{x_1})\ul{x_2}, 
\hbox{ etc}.\,\,
\]  
les treillis bords quotients successifs, et $\rK[\gT;x_0,\ldots
,x_k]=\rK_\gT[x_0,\ldots ,x_k]$ désigne le noyau de la projection 
canonique
$\gT\to \gT_\rK[x_0,\ldots ,x_k]$.  Alors on a 
$y\in\rK_\gT[x_0,\ldots ,x_\ell]$
\ssi il existe $ a_0,
\ldots  a_\ell\in \gT$ vérifiant:
\[    a_0 \vi x_0  \leq  0\,,\;\;\;
     a_1 \vi x_1 \leq   a_0 \vu x_0\,,\dots\,,\;\;\;
     a_{\ell} \vi x_{\ell} \leq     a_{\ell-1} \vu x_{\ell-1}\,,\;\;\;
     y  \leq  a_{\ell} \vu x_{\ell}.
\]
Si $\gT$ est une \agH on a:
\[
\rK_\gT[x_0,\ldots ,x_\ell]=\dar(x_{\ell}\vu (x_{\ell}\im( \cdots (x_1 
\vu (x_1 \im (x_0\vu
\neg x_0)))\cdots)))
\]
La dimension de Krull du treillis est $\leq \ell$ \ssi 
$1\in\rK_\gT[x_0,\ldots ,x_\ell]$
pour tous $x_0,\ldots ,x_\ell$. \eoe

\medskip En \coma la dimension de Krull de $\gT$ n'est pas à priori 
un \elt bien
défini de $\NN\cup\so{-1}\cup\so{\infty}$. En \clama cet \elt est défini 
comme la borne
inférieure des entiers $\ell$ tels que  $\Kdim(\gT)\leq \ell$.
On utilise en \coma les \textsl{notations} suivantes{\footnote{~En fait 
si on
regarde $\Kdim(\gT)$ comme l'ensemble des $\ell$ pour lesquels  
$\Kdim(\gT)\leq
\ell$, on raisonne avec des parties finales (éventuellement vides) 
de
$\NN\cup\so{-1}$, la relation d'ordre est alors l'inclusion 
renversée, la
borne supérieure l'intersection et la borne inférieure la 
réunion.}}, pour
se rapprocher du langage classique:

\begin{fnotation}
\label{fnotaKdiminf}
Soient $\gT$, $\gL$, $\gT_i$ des \trdis. 
\begin{itemize}
\item  $\Kdim\,\gL\leq \Kdim\,\gT$ signifie: $\Tt\ell\geq -1\; 
(\Kdim\,\gT\leq
\ell\;\Rightarrow \Kdim\,\gL\leq \ell)$.
\item  $\Kdim\,\gL= \Kdim\,\gT$ signifie:  $\Kdim\,\gL\leq  
\Kdim\,\gT$ et
$\Kdim\,\gL\geq  \Kdim\,\gT$.
\item  $\Kdim\,\gT\leq  \sup_i\Kdim\,\gT_i$ signifie: $\Tt\ell\geq 
-1\; (\&_i
(\Kdim\,\gT_i\leq \ell)\;\Rightarrow\Kdim\,\gT\leq \ell) $.
\item  $\Kdim\,\gT=  \sup_i\Kdim\,\gT_i$ signifie: $\Tt\ell\geq -1\; 
(\&_i
(\Kdim\,\gT_i\leq \ell)\;\Leftrightarrow\Kdim\,\gT\leq \ell) $.
\end{itemize}
\end{fnotation}

 Notons $\Bd(V,X)$ le bord de $V$ dans $X$ ($X$ est un 
espace
topologique et  $V$  est une partie de~$X$). 
Alors si $Y$ est un sous-espace de
$X$ on a $\Bd(V\cap Y,Y)\subseteq \Bd(V,X)\cap Y$, avec égalité 
si $Y$ est
un ouvert. La proposition suivante donne une version duale, \cov,  
sans points,
de cette affirmation.
\begin{fproposition}[bord de Krull d'un treillis quotient]
\label{fpropTquoBord} ~\\
Soit $\gL$ un treillis quotient d'un \trdi $\gT$. Par abus, nous 
notons $x$ l'image de $x\in\gT$ dans $\gL$. 
Alors  $\gL\ul x$ est un quotient de  $\gT\ul x$ et  $\gL\bal x$ est un quotient de  $\gT\bal x$. 
En outre si $\gL$ est le quotient de $\gT$ par un filtre $\ff$, $\gL\ul x$ est le quotient de  $\gT\ul x$  par le filtre image de $\ff$ dans $\gT\ul x$.
\end{fproposition}
\begin{proof}
Soient $a,b,x\in\gT$, si $a\leq_{\gT\ul x}b$ il existe $z\in\gT$ tel 
que
$x\vi z\leq_{\gT} 0$ et $a\leq _{\gT}x\vu z\vu b$. Puisque  $\gL$ est 
un
quotient de $\gT$, on a fortiori $x\vi z\leq_{\gL}0$ et 
$a\leq_{\gL}x\vu z\vu b$
et donc
  $a\leq_{\gL\ul x}b$.
Voyons le deuxième point. Notons  $\pi\colon \gT\to\gL$, $\pi\ul 
x:\gT\to\gT\ul x$
et $\theta\colon \gT\ul x\to\gL\ul x$ les projections. Il est clair que 
$\theta(\pi\ul
x(\ff))=\so{1}$  de sorte l'on a une factorisation de $\theta$ via 
$\gT\ul
x/(\pi\ul x(\ff)=1)$. Inversement soient $a,b\in\gT$
tels que $a\leq _{\gL\ul x} b$. Nous voulons montrer que $a\leq 
_{\gT\ul x
/(\pi\ul x(\ff)=1)} b$. Par hypothèse il existe $z\in\gT$ tel que
$a \leq_{\gL}x\vu z\vu b$ et $x\vi z\leq_{\gL}0$. Cela signifie qu'il 
existe
$f_1$ et $f_2\in\ff$  tels que $a\vi f_1\leq_{\gT} b\vu  x\vu z$ et 
$x\vi z\vi
f_2\leq_{\gT} 0$.
En prenant $f=f_1\vi f_2$ et $z'=z\vi f_2$ on obtient $a\vi 
f\leq_{\gT} b\vu
x\vu z'$ et $x\vi z'\leq_{\gT}0$, \cad  $a\vi f\leq_{\gT\ul x}  b$.
\end{proof}

Le corolaire suivant donne une version duale, \cov,  sans points, du 
fait
suivant: la dimension d'un sous-espace spectral est toujours 
inférieure ou
égale à celle de l'espace entier.
\begin{fcorollary}
\label{fcorpropTquoBord}
Si $\gL$ est un treillis quotient de $\gT$ on a $\Kdim\,\gL\leq 
\Kdim\,\gT$.
\end{fcorollary}

Dans la proposition suivante le point \textsl{2} est une version duale, \cov,  
sans
points, du fait topologique suivant: la notion de bord est locale. 
C'est surtout
le point \textsl{1} qui nous sera utile dans la suite. 

Par ailleurs, en renversant la relation d'ordre on 
aurait un
énoncé analogue pour l'autre bord.

\begin{fproposition}[caractère local du bord de Krull]
\label{fpropLocBord} ~
\begin{enumerate}
\item Soit $(\fa_i)_{1\leq i\leq m}$ une famille finie d'\ids de 
$\gT$, avec
$\bigcap_{i=1}^m\fa_i=\so{0}$.
Pour $x\in\gT$ notons encore $x$ son image dans $\gT_i=\gT/(\fa_i=0)$.
Le bord ${\gT\!_i}\ul x$ peut être vu comme le quotient de  $\gT 
\ul x$
par un \id $\fb_i$ et on a: $\bigcap_{i=1}^m\fb_i=\so{0}$.
\item Soit $(\ff_i)_{1\leq i\leq m}$ une famille finie de filtres de 
$\gT$, avec
$\bigcap_{i=1}^m\ff_i=\so{1}$.
Pour $x\in\gT$ notons encore $x$ son image dans $\gT_i=\gT/(\ff_i=1)$.
Le bord ${\gT\!_i}\ul x$ peut être vu comme le quotient de  $\gT 
\ul x$
par un filtre $\ffg_i$ et on a: $\bigcap_{i=1}^m\ffg_i=\so{1}$.
\end{enumerate}
\end{fproposition}
\begin{proof}
Voyons le point \textsl{1}. Considérons la projection $\pi_i\colon \gT\to\gT_i$  
puis la
projection $\pi'_i\colon \gT_i\to{\gT\!_i} \ul x$.
La composée $\gT\to{\gT\!_i} \ul x$ montre que ${\gT\!_i}\ul x\simeq
\gT/(\pi_i^{-1}(\rK_{\gT_i}^x)=0)$, et l'\id
$\pi_i^{-1}(\rK_{\gT_i}^x)$ contient
$\rK_\gT^x$. Ceci prouve la première affirmation.
Soit maintenant $y\in\gT$ tel que pour chaque $i$, 
$\pi_i(y)\in\rK_{\gT_i}^x$.
Cela revient à dire qu'il existe $b_i$ tel que
$\pi_i(y)\leq \pi_i(x)\vu\pi_i(b_i)$ et 
$\pi_i(b_i)\vi\pi_i(x)=\pi_i(0)$,
\cad pour un certain $a_i\in\fa_i$: $y\leq x\vu a_i\vu b_i$ et $x\vi 
b_i\in
\fa_i$. En prenant $c_i=a_i\vu b_i$ cela fait $y\leq x\vu c_i$ et 
$x\vi
c_i\in\fa_i$. Enfin avec $c=c_1\vi\cdots \vi c_m$ on obtient $y\leq 
x\vu c$ avec
$c\vi x=0$. Donc $y\in\rK_\gT^x$, et cela montre qu'un $z\in\gT \ul x$
qui est dans tous les $\fb_i$ est nul (car il est la classe 
d'un~$y$).\\
Pour le point \textsl{2} c'est une conséquence immédiate de la proposition
\ref{fpropTquoBord} et du fait \ref{ffactIdDansQuo}, qui affirme que 
tout passage
au quotient est un homomorphisme pour les treillis des filtres (en 
particulier
si une intersection finie de filtres est égale à $1$ cela reste 
vrai après
passage au quotient).
\end{proof}

\begin{fcorollary}[caractère local de la dimension de Krull]
\label{fcorpropLocBord} ~
\begin{enumerate}
\item Soit $(\fa_i)_{1\leq i\leq m}$ une famille finie d'\ids de 
$\gT$ et
$\fa=\bigcap_{i=1}^m\fa_i$. \\ Alors
$\Kdim(\gT/(\fa=0))=\sup_i\Kdim(\gT/(\fa_i=0))$.
\item Soit $(\ff_i)_{1\leq i\leq m}$ une famille finie de filtres de 
$\gT$ et
$\ff=\bigcap_{i=1}^m\ff_i$. \\ Alors
$\Kdim(\gT/(\ff=1))=\sup_i\Kdim(\gT/(\ff_i=1))$.
\end{enumerate}
\end{fcorollary}
\begin{proof}
Il suffit de prouver le point \textsl{1}. En remplaçant $\gT$ par 
$\gT/(\fa=0)$ on se
ramène au cas \hbox{où $\fa=0$}. Le résultat est clair pour $\Kdim=-1$. 
Et la preuve par \recu fonctionne grâce à la
proposition~\ref{fpropLocBord}.\\
On peut aussi donner une preuve directe basée sur la 
caractérisation 2~(c)
dans la \dfn~\ref{fdefDK0}.
\end{proof}

Notez qu'en \clama la caractère local de la dimension de Krull est
en général énoncé sous la forme
\[
\Kdim(\gT)=\sup\sotq{\Kdim(\gT/(\ff=1))}{\ff\hbox{ filtre premier 
minimal}}.
\]  
Il s'agit d'une conséquence directe de la \dfn de la
dimension en \clama.   On peut ensuite en déduire le corolaire
\ref{fcorpropLocBord} mais la preuve qu'on obtient ainsi n'est pas
\cov. 

\smallskip On obtient maintenant \cot le  \thref{fthDK1} dans lequel on remplace la \dfn  classique abstraite de la dimension en \clama par notre \dfn \cov.

\begin{ftheorem}
\label{fpropDK1}
Soit un \trdi $\gT$ engendré par une partie $S$ et $\ell$
un entier positif ou nul.
Les conditions suivantes sont équivalentes.
\begin{enumerate}
\item  Le treillis $\gT$ est de dimension $\leq \ell$
\item  Pour tout $x\in S$ le
bord  $\gT\ul{x}$ est de dimension $\leq \ell-1$.
\item  Pour tout $x\in S$ le
bord  $\gT\bal{x}$ est de dimension $\leq \ell-1$.
\item 
Pour tous $x_0,\ldots,x_\ell\in S$
il existe $a_0$,\ldots,  $a_\ell\in \gT$ vérifiant:
\begin{equation}
     a_0\vi x_0  \leq  0\,,\;\;\;
     a_1\vi x_1 \leq   a_0\vu x_0\,,\dots\,,\;\;\;
     a_{\ell}\vi x_{\ell} \leq     a_{\ell-1}\vu x_{\ell-1}\,,\;\;\;
     1  \leq  a_{\ell}\vu x_{\ell}.
\end{equation}
\end{enumerate}
Lorsque $\gT$ est une \agH les conditions précédentes 
sont aussi
équivalentes~à
\begin{enumerate}\setcounter{enumi}{4}
\item   Pour toute
suite  $x_0,\dots,x_{\ell}$  dans  $S$  on a l'\egt
\begin{equation}
     1 = x_{\ell}\vu (x_{\ell}\im( \cdots (x_1 \vu (x_1 \im (x_0\vu
\neg x_0)))\cdots))
\end{equation}
\end{enumerate}
Lorsque $\gT$ une \alg de Brouwer les conditions 
précédentes sont
aussi
équivalentes à
\begin{enumerate}\setcounter{enumi}{5}
\item   Pour toute
suite  $x_0,\dots,x_{\ell}$  dans  $S$  on a l'\egt
\begin{equation}
     0 = x_{0}\vi (x_{0}-(x_1 \vi (x_1 - ( \cdots (x_\ell\vi
(1- x_\ell))))\cdots))
\end{equation}
\end{enumerate}
\end{ftheorem}
\begin{proof}
\noindent \textsl{2} $\Leftrightarrow$ \textsl{4}  par récurrence sur $\ell$, vue 
la
\dfn du bord.

\noindent \textsl{3} $\Leftrightarrow$ \textsl{4}:  par récurrence sur $\ell$, vue 
la
\dfn du bord.

\noindent \textsl{2} $\Leftrightarrow$ \textsl{5}:  par récurrence sur $\ell$, vue 
la
\dfn du bord
dans le cas d'une \agH. 

\noindent \textsl{3} $\Leftrightarrow$ \textsl{6}:  par récurrence sur $\ell$, vue 
la
\dfn du bord
dans le cas d'une \alg de Brouwer.

\noindent Il reste donc à voir que si \textsl{2} est vrai pour  $x\in S$, 
alors  \textsl{2}
est vrai pour tout $x\in\gT$.
Cela résulte de la proposition \ref{fpropBordKUnion}, et des 
corolaires
\ref{fcorpropTquoBord} et  \ref{fcorpropLocBord}: par exemple pour tous 
$x,y\in
\gT$, $\gT\ul{x\vu y}$ est un quotient
de $\gT/(\rK_\gT^x\cap\rK_\gT^y)$ donc $\Kdim(\gT\ul{x\vu y})\leq \sup
(\Kdim\,\gT\ul x, \Kdim\,\gT\ul y)$.
\end{proof}

\subsection{Dimensions et bord de Heitmann}

\subsubsection*{J-dimension de Heitmann pour un \trdi}
\addcontentsline{toc}{subsubsection}{$\rJ$-dimension de Heitmann}

Nous donnons maintenant la \dfn \cov de la $\Jdim$
de Heitmann, que nous appelons \textsl{$\rJ$-dimension de Heitmann} du 
treillis
$\gT$ (ou de l'espace spectral $\Spec\,\gT$).

\begin{fdefinition}
\label{fdefHdimTr}
Soit $\gT$ un \trdi. 
La \textsl{$\rJ$-dimension de Heitmann de $\gT$}, notée $\Jdim\,\gT$, est la 
dimension de Krull du treillis de Heitmann $\He(\gT)$ (cf. \dfn\ref{defHeT}).
\end{fdefinition}

En fait d'un point de vue \cof on se contente de définir, pour tout
entier $\ell\geq -1$, la phrase   \gui{$\Jdim\,\gT\leq \ell$} par
\gui{$\Kdim\,\He(\gT)\leq \ell$}.

Du point de vue classique on peut donner la \dfn directement, à la 
Heitmann,
pour un espace spectral $X$ comme suit: si $M_X$ est l'ensemble des 
points
fermés et $J_X$ l'adhérence de $M_X$ pour la topologie 
constructible, alors
$\Jdim\,X=\Kdim\,J_X$.

\begin{ffact}
\label{ffactKJdim} Soit $\gT$ un \trdi,  $\gT'=\gT/(\JT(0)=0)$ et 
$\fa$ un
idéal de $\gT$.
\begin{enumerate}
\item $\Jdim(\He(\gT))=\Jdim\,\gT'=\Jdim\,\gT\leq 
\Kdim(\gT')\leq \Kdim\,\gT$.
\item Soit $\gL=\gT/(\fa=0)$.
Alors $\Jdim\,\gL\leq\Jdim\,\gT$.
\end{enumerate}
\end{ffact}
\begin{proof}
Le point \textsl{1} est conséquence du point \textsl{3} dans le fait~\ref{ffactHeHe} 
et le point \textsl{2} conséquence du point~\textsl{4}.
\end{proof}

\rem L'opération $\gT\mapsto \JT(0)$ n'est pas fonctorielle. Même 
chose pour
$\gT\mapsto \He\,\gT$. En particulier l'item \textsl{2} dans le fait \ref{ffactKJdim} ne fonctionne plus à priori pour
un quotient plus général, par exemple pour un quotient par un 
filtre.
Contrairement à la $\Kdim$, la $\Jdim$ peut augmenter par passage 
à un quotient (on a des exemples simples en \alg commutative).

\medskip
\exl
Voici un exemple d\^{u} à Heitmann d'un espace spectral pour lequel
$\Jdim(\gT)<\Kdim(\gT/(\JT(0)=0))$. 
\begin{figure}[ht]   
\begin{center}
\includegraphics{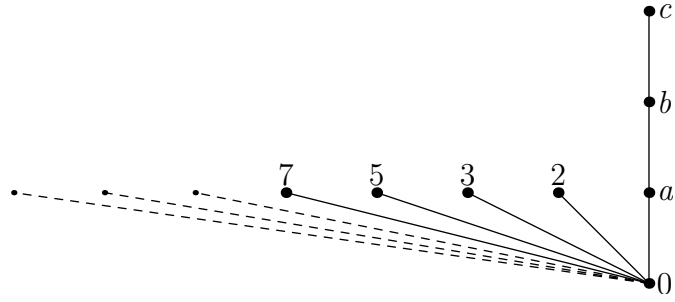}
\end{center}
\caption[]
{\label{ffigSpec} Exemple de Heitmann }  
\end{figure}  
On considère $X=\Spec\,\ZZ$ et $Y=\bf n$
($n\geq 3$) on recolle ces deux espaces spectraux en identifiant les 
deux \elts
minimaux (le singleton correspondant est bien dans 
les deux cas un sous-espace
spectral). On obtient un espace $Z=\Spec\,\gT$ avec $M_Z=\Max\,\gT$ 
fermé donc égal à $J_Z=\Jspec\,\gT$ et zéro dimensionnel, tandis que 
l'unique élément minimal est le seul minorant de $M_Z$, donc $\JT(0)=0$ et
$\Kdim(\gT/(\JT(0)=0))=\Kdim(\gT)=n-2$.
\eoe

\medskip
\rem
Donnons la \dfn de la $\Jdim$ complètement mise à plat.

\begin{itemize}
\item $\Jdim\,\gT\leq \ell$ signifie:
$\forall x_0,\ldots,x_\ell\in \gT\;$
$\Ex a_0,
\ldots,  a_\ell\in \gT$ vérifiant:
\[    a_0\vi x_0  \leq_{\He(\gT)} 0,\;
     a_1\vi x_1 \leq_{\He(\gT)}  a_0\vu  x_0,\dots,\;
     a_{\ell}\vi x_{\ell} \leq_{\He(\gT)}    a_{\ell-1}\vu   
x_{\ell-1},\;
     1  \leq_{\He(\gT)} a_{\ell}\vu   x_{\ell}.
\]
\cade
$\forall x_0,\ldots,x_\ell\in \gT\;$
$\Ex a_0,\ldots,  a_\ell\in \gT\;$ $\Tt y\in\gT$
\[\begin{array}{rcl}
(a_0 \vi x_0) \vu y = 1 &  \Rightarrow  &  y=1    \\
(a_1 \vi x_1)   \vu y=1&  \Rightarrow &  a_0 \vu x_0 \vu y=1   \\
\vdots \qquad  &  \vdots &   \qquad \vdots   \\
(a_{\ell}\vi x_{\ell})\vu y=1&  \Rightarrow &  a_{\ell-1}\vu 
x_{\ell-1} \vu y=1
\\
   &   &  a_{\ell}\vu x_{\ell}=1
  \end{array}\]
\item En particulier $\Jdim\,\gT\leq 0$ signifie:\\
  $\forall x_0\in \gT$ $\Ex a_0\in \gT$ $\Tt y\in\gT, 
  (\,(a_0 \vi x_0) \vu y = 1
\Rightarrow    y=1)\; $ et
$\; a_0 \vu x_0=1\,)$.
\item Et $\Jdim\,\gT\leq 1$ signifie:
  $\forall x_0,x_1\in \gT\;$ $\Ex a_0,a_1\in \gT\;$ $\Tt y\in\gT$\,,\\
~~~~~~  $((a_0 \vi x_0) \vu y = 1   \Rightarrow    y=1)\,$ et
$((a_1 \vi x_1) \vu y = 1   \Rightarrow a_0 \vu x_0 \vu   y=1)\,$ et
$ \,a_1 \vu x_1=1$.\eoe
\end{itemize}

\subsubsection*{La dimension de Heitmann pour un \trdi}
\addcontentsline{toc}{subsubsection}{Dimension de Heitmann}

Bien que \citet*{fHei84} définisse et utilise  la 
dimension $\Jdim\,X$  où $X$ est le spectre d'un anneau commutatif, ses 
preuves sont en fait implicitement basées sur une notion voisine, mais non 
équivalente, que nous appellerons la \textsl{dimension de Heitmann} et que nous noterons 
$\Hdim$.

\smallskip Nous présentons cette notion directement au niveau des 
\trdis 
où les choses s'expliquent plus simplement.

La dimension $\Hdim\,\gT$ est toujours inférieure ou égale à
$\Jdim\,\gT$, ce qui fait que les \thos établis pour la $\Hdim$ 
seront a
fortiori vrais avec la $\Jdim$ et avec la $\Kdim$.

\begin{fdefinition}
\label{fdefBHeit}
Soit $\gT$ un \trdi et $x\in\gT$.
On appelle \textsl{bord de Heitmann de $x$ dans~$\gT$} le treillis quotient
$\gT_\rH^x\eqdefi\gT/(\rH_\gT^x=0)$, où
\begin{equation} \label{feqbordHeit}
\rH_\gT^x\,=\,\dar x \,\vu\, (\JT(0):x)
\end{equation}
On dira aussi que $\rH_\gT^x$ est \textsl{l'\id bord de Heitmann de $x$ 
dans
$\gT$.}
\end{fdefinition}

\begin{flemma}[bord de Krull et bord de Heitmann]
\label{flemTBKBH} ~\\
Soit $\gT$ un \trdi,  $\gT'=\gT/(\JT(0)=0)$, $\pi\colon \gT\to\gT'$ la 
projection
canonique,
$x\in\gT$ et $\ov{x}=\pi(x)$. Alors on a:
\begin{enumerate}
\item $\rH^{\ov{x}}_{\gT'} =\rK^{\ov{x}}_{\gT'}.$
\item  $\pi^{-1}(\rK^{\ov{x}}_{\gT'}) =\rH^x_\gT$
et ${{\gT'}_\rH^{\ov{x}}}\simeq {{\gT'}\ul{\ov{x}}}\simeq\gT_\rH^x$.
\end{enumerate}
\end{flemma}
\begin{proof}
Clair d'après les \dfns. 
\end{proof}

\begin{flemma}
\label{flemTquoBordH}
Soit $\gL=\gT/(\fa=0)$ un treillis quotient d'un \trdi $\gT$ par un 
\id.  Par
abus, nous notons $x$ l'image de $x\in\gT$ dans $\gL$. Alors 
$\gL_\rH^x$ est un
quotient de  $\gT_\rH^x.$
\end{flemma}
\begin{proof}
Notons $\pi\colon \gT\to\gL$ la projection canonique. On veut montrer que si
$z\in\rH_\gT^x$ alors $\pi(z)\in\rH_\gL^{\pi(x)}$. On suppose donc 
$z\leq_\gT
x\vu u$ avec $x \vi u  \in\JT(0)$. Comme 
$\pi(\JT(0))\subseteq\rJ_\gL(0)$, cela
donne $\pi(z)\leq_\gL \pi(x)\vu \pi(u)$ avec $\pi(x) \vi \pi(u) 
\in\rJ_\gL(0)$,
ce qui implique  $\pi(z)\in\rH_\gL^{\pi(x)}$.
\end{proof}

\rem Le lemme précédent serait faux pour un quotient plus 
général, par
exemple pour un quotient par un filtre. Il reste vrai chaque fois que
$\pi(\JT(0))\subseteq \rJ_\gL(0)$.\eoe

\begin{fproposition}[comparaison de deux bords à la 
Heitmann]
\label{fpropBHeitHeyt}~\\
On considère pour $x\in\gT$ son image $\hat{x}\in\He(\gT)$.
\begin{enumerate}
\item On compare les deux quotients de $\gT$ que sont 
$\He(\gT_\rH^x)$ et
$(\He(\gT))\ul{\hat{x}}$: le premier est un quotient du second.
\item Lorsque $\He(\gT)$ est une \agH,  l'inclusion est une \egt. 
\end{enumerate}
\end{fproposition}
\begin{proof}
Le treillis $(\He(\gT))\ul{\hat{x}}$  est un quotient de $\gT$  dont 
la relation
de préordre $a \preceq  b$
peut etre décrite de la manière suivante:
\[\Ex y\;\;\;(x \vi  y \in \JT (0) \;\;\mathrm{et} \;\;\Tt z\;[\;a 
\vu  z = 1
\;\Rightarrow\; x \vu  y \vu  b  \vu  z = 1\;]).\eqno (*)
\]
Considérons
  le préordre qui définit $\He(\gT _\rH^x)$:
\[\Tt u\;\;\;(1 \leq_{\gT _\rH^x}  a \vu  u \;\;\Rightarrow\;\;
  1 \leq_{\gT _\rH^x}  b \vu  u).\eqno (**)\] 
Prouvons que la relation de préordre $\preceq$ entraine la relation 
de
préordre $(**)$.\\
On a un $y$ vérifiant $(*)$.
On considère un $u$ tel que $1 \leq  a \vu  u$ dans $\gT _\rH^x$ et 
on cherche
à montrer que
$1 \leq  b \vu  u$  dans $\gT _\rH^x$.\\
La relation $1 \leq  a \vu  u$ dans $\gT _\rH^x$ s'ecrit $1 \leq  a 
\vu  u \vu
x \vu  y'$
pour un $y'$ tel que $y' \vi  x \in \JT (0)$.\\
On pose $z = u \vu  x \vu  y'$, on applique $(*) $ et on obtient
   $x \vu  b \vu  u \vu  (y \vu  y') = 1.$ \\
Mais $(y \vu  y') \vi  x  \in \JT (0)$  donc en posant $y'' = y \vu  
y'$
on a
$1 \leq  b \vu  u \vu  x \vu  y''$    avec $x \vi  y'' \in \JT (0)$
\cad  $1 \leq  b \vu  u$  dans $\gT _\rH^x$.\\
Voyons le point \textsl{2}. Nous notons $\pi\colon \gT\to\He(\gT)$ la projection 
canonique.\\
Rappelons (fait~\ref{ffactHeHe}) que $\pi^{-1}(0)=\JT(0)$ et
$\pi^{-1}(1)=\so{1}$. Nous supposons que $\He(\gT)$ est une \agH.  
Notons
$\wi{x}$ un \elt de $\gT$ tel que $\pi(\wi{x})\,=\,\pi(x)\im 0\,$ dans
$\He(\gT)$.
Alors on peut réécrire $(*)$ sous la forme
\[
\Tt z\;[\;a \vu  z = 1  \;\Rightarrow\; x \vu  \wi x \vu  b  \vu  z =
1\;]\eqno (*')
\]
De même  $1 \leq_{\gT _\rH^x}  a \vu  u$, qui signifie
\[\Ex y'\;(x\vi y'\in \JT (0) \;\mathrm{et}\; 1 \leq  a \vu  u \vu  x 
\vu
y'),\]
se réécrit $1 \leq  a \vu  u \vu  x \vu  \wi x$. En 
conséquence $(**)$ se
réécrit
\[
\Tt u\;[\;a \vu  u \vu  x \vu  \wi x= 1  \;\Rightarrow\; b \vu  u 
\vu  x \vu
\wi x\eqno (**')
\]
et il est clair que $(*')$ et $(**')$ sont équivalents.
\end{proof}

\begin{fdefinition}
\label{fdefHdim}
La \textsl{dimension de Heitmann d'un \trdi $\gT$}, notée 
$\Hdim\,\gT$, est
définie de manière inductive comme suit:
\begin{itemize}
\item $\Hdim\,\gT=-1$ \ssi $\gT=\Un$.
\item Pour $\ell\geq 0$, $\Hdim\,\gT\leq \ell$ \ssi pour tout 
$x\in\gT$,
$\Hdim(\gT_\rH^x)\leq \ell-1$.
\end{itemize}
\end{fdefinition}

Du lemme \ref{flemTquoBordH} on déduit par \recu sur $\ell$,
avec la même convention de notation qu'en \ref{notaKdiminf} le 
lemme
suivant.

\begin{flemma}
\label{flemDimHquo}
Si $\gL$ est le quotient de $\gT$ par un \id,  alors $\Hdim\,\gL\leq
\Hdim\,\gT$.
\end{flemma}

\begin{fproposition}[comparaison des dimensions $\Jdim$ et $\Hdim$]
\label{fpropJdimHdim}~
%
\begin{enumerate}
\item On a toujours $\Hdim\,\gT\leq  \Jdim\,\gT$.
\item Lorsque $\He(\gT)$ est une \agH,  on a \egt: $\Hdim\,\gT=  
\Jdim\,\gT$.
\end{enumerate}

\noindent En \clama l'\egt a donc lieu lorsque $\He(\gT)$ est noethérien.
\end{fproposition}
\begin{proof}
Le point \textsl{1} se démontre par \recu sur $\Jdim\,\gT$ à partir du 
point \textsl{1} de
la proposition~\ref{fpropBHeitHeyt}.

\noindent Pour le point \textsl{2}, on suppose que  $\He(\gT)$ est une \agH et 
l'on fait
une \recu en utilisant le point \textsl{2} de la proposition 
\ref{fpropBHeitHeyt}.
Pour que la \recu fonctionne il faut montrer 
que~\hbox{$(\He(\gT))\ul{\hat{x}}=\He(\rH_\gT^x)$} est également une \agH. 
Or cela résulte de ce que~\hbox{$(\He(\gT))\ul{\hat{x}}$} est un quotient 
de~$\He(\gT)$ par un \id principal (puisque $\He(\gT)$ est une \agH) et 
du
fait~\ref{ffactQuoAgH2}.
\end{proof}

\begin{fproposition}
\label{fpropHdim0} Notons $\gT'$ pour le quotient $\gT/(\JT(0)=0)$.
\begin{enumerate}
\item On a toujours $\Hdim\,\gT=\Hdim(\gT')$.
\item On a les équivalences $\Jdim\,\gT\leq 0\;\Longleftrightarrow\; \Hdim\,\gT\leq 
0\;\Longleftrightarrow\;
\Kdim(\gT')\leq 0$, (autrement dit $\gT'$ est une \agB). C'est le cas 
lorsque  $\He(\gT)$
est fini.
\end{enumerate}
\end{fproposition}
\begin{proof}
Le premier point résulte du lemme \ref{flemTBKBH}, point \textsl{2}.\\
Pour le deuxième point on sait déjà que $\Hdim\,\gT\leq 
\Jdim\,\gT\leq
\Kdim\,\gT'$. Le lemme \ref{flemTBKBH}, point \textsl{1}, prouve que 
$\Hdim\,\gT\leq 0$
implique  $\Kdim\,\gT'\leq 0$.\\
Lorsque $\gT$ est fini l'espace $\Jspec\,\gT$ est simplement 
l'ensemble des
\idemas avec pour topologie toutes les parties (puisque les points 
sont
fermés).
Par ailleurs le fait \ref{ffactHeHe}.\textsl{1} permet de conclure aussi lorsque
$\He(\gT)$ est fini.
\end{proof}

\hum{l'ancienne preuve était bizarrement nettement plus 
compliquée}

\rem
En \clama  $\He(\gT)$ est fini
  \ssi l'ensemble $M$ des
  \idemas est fini (le cas semi-local en \alg
  commutative). \eoe

\medskip La proposition suivante est l'analogue de la proposition
\ref{fpropBordKUnion}, en remplaçant le bord de Krull par le bord 
de
Heitmann.

\begin{fproposition}
\label{fpropBordHUnion}
Pour tous $x,y\in\gT$  on a  $\rH_\gT^{x}\cap \rH_\gT^{y}= \rH_\gT^{x\vu y}\cap 
\rH_\gT^{x\vi y}$.
\end{fproposition}

\begin{proof}
Résulte de la proposition \ref{fpropBordKUnion} et du lemme 
\ref{flemTBKBH}.
\end{proof}

La proposition suivante est l'analogue du point \textsl{1} de la
proposition~\ref{fpropLocBord}.
\begin{fproposition}
\label{fpropLocBordH}
Soit $(\fa_i)_{1\leq i\leq m}$ une famille finie d'\ids de $\gT$, avec
$\bigcap_{i=1}^m\JT(\fa_i)\subseteq \JT(0)$ (c'est le cas en 
particulier si
$\bigcap_{i=1}^m\fa_i=\{0\}$).
Pour $x\in\gT$ notons encore $x$ son image dans 
le quotient~\hbox{$\gT_{\!i}=\gT/(\fa_i=0)$}.
Alors le bord ${\gT_{\!i}}_\rH^x$ peut être vu comme le quotient 
de 
$\gT_\rH^x$
par un \id~$\fb_i$ et l'on a: $\bigcap_{i=1}^m\fb_i=\so{0}$.
\end{fproposition}
\begin{proof}
Résulte du lemme \ref{flemTBKBH} et de la proposition 
\ref{fpropLocBord},
appliquée au treillis
\hbox{$\gT'=\gT/(\JT(0)=0)$} et aux idéaux images des $\JT(\fa_i)$ dans 
$\gT'$.
\end{proof}

Le corolaire suivant est l'analogue du point \textsl{1} du 
corolaire~\ref{fcorpropLocBord} (notations comme en~\ref{notaKdiminf}).

\begin{fcorollary}
\label{fcorpropLocBordH}
Soit $(\fa_i)_{1\leq i\leq m}$ une famille finie d'\ids de $\gT$ et
$\fa=\bigcap_{i=1}^m\fa_i$. \\
Alors
$\Hdim(\gT/(\fa=0))=\sup_i\Hdim(\gT/(\fa_i=0))$.
\end{fcorollary}
\begin{proof}
  En remplaçant $\gT$ par $\gT/(\fa=0)$ on se ramène au cas 
où $\fa=0$.
La chose est claire pour $\gT=\Un$. Et la preuve par \recu fonctionne
grâce à la proposition~\ref{fpropLocBordH}.
\end{proof}

\begin{fproposition}
\label{fpropHdimgen}
Soit $S$ un système \gtr du \trdi $\gT$ et $\ell\geq 0$.
\Propeq
\begin{enumerate}
\item Pour tout $x\in\gT$,
$\Hdim(\gT_\rH^x)\leq \ell-1$.
\item Pour tout $x\in S$,
$\Hdim(\gT_\rH^x)\leq \ell-1$.
\end{enumerate}
\end{fproposition}
\begin{proof}
Cela résulte de la proposition \ref{fpropBordHUnion}, du lemme
\ref{flemTquoBordH} et du corolaire \ref{fcorpropLocBordH}: par 
exemple pour \hbox{tous
$x,y\in \gT$}, puisque $\rH_\gT^{x\vu 
y}\subseteq\rH_\gT^x\cap\rH_\gT^y$, le
treillis $\gT_\rH^{x\vu y}$ est un quotient
de $\gT/(\rH_\gT^x\cap\rH_\gT^y)$ par un idéal, donc 
$\Hdim(\gT_\rH^{x\vu
y})\leq \sup (\Hdim\,\gT_\rH^x, \Hdim\,\gT_\rH^y)$.
\end{proof}

  \rem
Explicitons encore un peu plus la dimension de Heitmann.
Pour ceci nous introduisons \gui{l'idéal bord de Heitmann 
itéré}.
Pour $x_0$, \ldots, $x_n\in\gT$ nous notons 
\[
\gT_\rH[x_0]=\gT_\rH^{x_0},\,
\gT_\rH[x_0,x_1]=(\gT_\rH^{x_0})_\rH^{x_1},\,
\gT_\rH[x_0,x_1,x_2]=((\gT_\rH^{x_0})_\rH^{x_1})_\rH^{x_2},\,\ldots
\] 
les treillis bords quotients successifs. Et $\rH[\gT;x_0,\ldots
,x_k]=\rH_\gT[x_1,\ldots ,x_k]$ désigne le noyau de la projection 
canonique
$\gT\to \gT_\rH[x_0,\ldots ,x_k]$.

\noindent Dire que $\Hdim\,\gT\leq \ell$ revient à dire que pour 
tous
$x_0,\ldots ,x_\ell\in\gT$ on a $1\in\rH_\gT[x_0,\ldots ,x_\ell]$.
Il nous faut donc expliciter les \ids $\rH_\gT[x_0,\ldots 
,x_\ell]$.
Pour ceci nous devons expliciter $\pi^{-1}(\rH[\gT/(\fa=0);\pi(x)])$
(que nous noterons $\rH[\gT,\fa;x]$) lorsqu'on a une projection
$\pi\colon \gT\to\gT/(\fa=0)$.

\noindent Par \dfn on a $y\in\rH[\gT,\fa;x]$ \ssi $y\leq x\vu 
z\;\mod\;\fa$ pour
un $z$ qui vérifie $\pi(z\vi x)\in \rJ_{\gT/(\fa=0)}(0)$.
Cette dernière condition signifie: $\Tt u\in\gT,\; (\pi((z\vi x)\vu
u=\pi(1)\;\Rightarrow \;\pi(u)=\pi(1)$.
Et ceci s'écrit encore
\[\Tt u\in\gT,\; ((\Ex a\in\fa\;(z\vi x)\vu u\vu a=1)\;\Rightarrow 
\;(\Ex
b\in\fa\; u\vu b=1))\]
Par ailleurs, $y\leq x\vu z\;\mod\;\fa$ signifie $\Ex a'\in\fa\;y\leq x\vu z\vu
a'$ et la condition   $\pi(z\vi x)\in \rJ_{\gT/(\fa=0)}(0)$ n'est pas changée
si on remplace $z$ par $z\vu a'$. En bref nous obtenons la condition suivante
\hbox{pour  $y\in\rH[\gT,\fa;x]$}:
\[ \Ex z\in\gT\; [\,y\leq x\vu z\;\mathrm{et}\;
\Tt u\in\gT,\; ((\Ex a\in\fa\;(z\vi x)\vu u\vu a=1)\;\Rightarrow 
\;(\Ex
b\in\fa\; u\vu b=1))\,]\]
On voit donc apparaitre une formule d'une complexité logique 
redoutable.
Surtout si on songe que $\Ex a\in\fa$ et $\Ex b\in\fa$ devront 
être explicités avec $\fa=\rH_\gT[x_1,\ldots ,x_k]$ si on veut obtenir une expression pour $y\in\rH_\gT[x_1,\ldots ,x_k,x]$.
Contrairement à l'expression pour $\Jdim\,\gT\leq \ell$ qui ne 
comportait que deux alternances de quantificateurs quel que soit l'entier $\ell$, on voit pour $\Hdim\leq \ell$ des expressions de plus en plus imbriquées au fur et à mesure que $\ell$ augmente. En fait il se trouve que pour les anneaux commutatifs, c'est la $\Hdim$ qui fait fonctionner les preuves par \recu pour les \gui{grands} \thos classiques d'\alg commutative, et c'est la vraie raison pour laquelle nous avons été amenés à introduire cette dimension. Comme elle vérifie
$\Hdim\,\gT\leq \Jdim\,\gT\leq \Kdim(\gT/(\JT(0)=0))$
on a quand même des moyens raisonnables pour la majorer.
Mais on manque d'exemples avec une majoration meilleure que celle par
$\Kdim(\gT/(\JT(0)=0))$.\eoe

\hum{~

1. On peut se poser la question de comparer $\Hdim\,\gT$ et  
$\Hdim\,\gT\cir$.

\smallskip
2. On pourrait peut être parler des treillis locaux\,? 
Constructivement on peut
distinguer entre $\Tt x, y\in\gT\;(x\vu y=1)\Rightarrow 
(x=1\;\mathrm{ou}\;y=1)$
(treillis local) et
$\gT/(\JT(0)=0)\simeq\Deux$ (treillis local résiduellement discret)}
\section[Dimensions de Krull et Heitmann: anneaux 
commutatifs]{Dimensions de
Krull et Heitmann pour les anneaux commutatifs}
\label{fsecBPA}
Dans cette section, $\gA$ désigne toujours un anneau commutatif. Nous disons qu'un \id $\fa$ de~$\gA$ est \textsl{radical} si $\fa=\sqrt[\gA]{\fa}.$

\subsection{Le treillis de Zariski}

Nous rappelons ici l'approche \cov de  \cite{fJoy76} pour le spectre
d'un anneau commutatif.

Si $J\subseteq \gA$, nous notons $\cI_\gA(J)$ ou $\gen{J}_\gA$ (ou
$\gen{J}$ si le contexte est clair) l'\id engendré par~$J$; nous
notons $\DA(J)$ (ou $\rD(J)$ si le contexte est clair) le
nilradical de l'\id $\gen{J}$:
\begin{equation} \label{feqZar}
\begin{array}{rclcl}
\DA(J)&  = & \sqrt[\gA]{\gen{J}} &=&\sotq{x\in\gA}{\Ex m\in\NN\;\; 
x^m\in\gen{J}}
\end{array}
\end{equation}
Lorsque $J=\so{x_1,\ldots ,x_n}$ nous notons
$\DA(x_1,\ldots ,x_n)$ pour $\DA(J)$.
Si le contexte est clair, nous abrégeons $\DA(x)$ en $\wi{x}$.

Par \dfn le {\sl treillis de Zariski} de $\gA$, noté $\ZarA$, a pour
\elts les radicaux d'\itfs: ce sont donc les $\DA(x_1,\ldots ,x_n)$,
\cad les $\DA(\fa)$ pour les \itfs~$\fa$.  La relation d'ordre est
l'inclusion, le inf et le sup sont donnés par
\[
\DA(\fa_1)\vi\DA(\fa_2)=\DA(\fa_1\fa_2)\quad \mathrm{et} \quad
\DA(\fa_1)\vu\DA(\fa_2)=\DA(\fa_1+\fa_2).
\]
Le treillis de Zariski de $\gA$
est un \trdi,  et $\DA(x_1,\ldots ,x_n)=
\wi{x_1}\vu\cdots \vu\wi{x_n}$. Les  $\wi{x}$  forment un système \gtr de $\ZarA$, stable par $\vi$.

Si $J\subseteq \gA$  nous notons $\wi{J}=\sotq{\wi{x}}{x\in J}
\subseteq\ZarA$.

Soient $U$ et  $J$  deux familles finies dans $\gA$, on a les équivalences
\[ \Vi\wi{U}\leq_{\ZarA} \Vu\wi{J}
\quad\Longleftrightarrow \quad
\prod\nolimits_{u\in U} u  \in \sqrt{\gen{J}}
\quad\Longleftrightarrow \quad
\cM(U)\cap \gen{J}\neq \emptyset
\]
où $\cM(U)$ est le \mo multiplicatif engendré par $U$.

Cela suffit à décrire le treillis $\ZarA$. Plus précisément 
on a
(cf. \cite{fCC00,fCL2003}):
\begin{fproposition}
[\dfn à la Joyal du spectre d'un anneau commutatif]\label{fpropZar}~
\\
 Le treillis $\ZarA$  est (à \iso unique près) le treillis engendré par des symboles $\DA(a)$ soumis aux relations suivantes
\[\begin{array}{cccc}
\DA(0_\gA) =0   \;,\; \DA(1_\gA)= 1 \;,\;
   \DA(x+y) \leq \DA(x)\vu\DA(y) \;,\; \DA(xy) = \DA(x)\vi \DA(y)
\end{array}\]
\end{fproposition}

L'opération $\Zar$ est un foncteur de la catégorie des anneaux 
commutatifs
vers celle des \trdis. 
Notez que via ce foncteur la projection $\gA\to\gA/\DA(0)$ donne un 
\iso
$\ZarA\simeq \Zar(\gA/\DA(0))$. On a $\ZarA=\Un$ \ssi $1_\gA=0_\gA$.

\smallskip Un \tho important de  \cite{fHoc1969} affirme que tout 
espace spectral est homéomorphe au spectre d'un anneau commutatif.
Voici une version sans point du \tho de Hochster: 

\smallskip \noindent {\bf \Tho.} \textsl{Tout \trdi est isomorphe au treillis de Zariski d'un anneau commutatif.} 

\smallskip \noindent 
Pour une preuve non \cov  voir \cite{fBan96}.

\subsection{Idéaux, filtres et quotients de $\ZarA$}

Rappelons qu'en \clama le \textsl{spectre de Zariski} $\Spec\,\gA$ d'un anneau
commutatif est un espace topologique dont les points sont les 
\ideps de l'anneau et dont la topologie est définie par la base 
d'ouverts formée par les $\fD_\gA(a)=\sotq{\fp\in\Spec\,\gA}{a\notin\fp}$.
On note aussi $\fD_\gA(x_1,\ldots ,x_n)$ pour 
$\fD_\gA(x_1)\cup\cdots\cup \fD_\gA(x_n)$.

\subsubsection*{Idéaux de $\gA$ et de $\ZarA$}

En \clama,  tout \id radical est l'intersection des \ideps qui  le 
contiennent.

Introduisons la notation (lorsque $J\subseteq\gA$)
\[
\IZA(J)=\cI_{\ZarA}(\wi{J})
\]
pour l'\id de $\ZarA$ engendré par $\wi{J}$.  En particulier
\[
\IZA(\so{x_1,\ldots ,x_n})=\dar\DA(x_1,\ldots ,x_n)=
\dar(\wi{x_1}\vu\cdots \vu\wi{x_n})\,.
\]

On a $\IZA(J)=\IZA(\sqrt{\gen{J}})$, et on établit facilement le
fait fondamental suivant.

\begin{ffact}
\label{ffactSpecAzarA} ~  
\begin{itemize}
\item L'application $\fa\mapsto \IZA(\fa)$ définit un \iso du
treillis des \ids radicaux de $\gA$ vers le treillis des \ids de
$\ZarA$.  
\item Par restriction les \ideps (resp.  les \idemas) de
l'anneau $\gA$ et ceux du \trdi $\ZarA$ sont également en
correspondance naturelle bijective.  
\item Pour tout anneau commutatif
$\gA$, $\Spec\,\gA$ (au sens des anneaux commutatifs) s'identifie à
$\Spec(\ZarA)$ (au sens des \trdis).
\end{itemize}
\end{ffact}

\rem En \clama on a un \iso entre le treillis $\ZarA$ et le treillis des
\oqcs de $\Spec\,\gA$.  On peut alors identifier 
\begin{itemize}
\item $\DA(x_1,\ldots ,x_n)$, qui est un \elt de $\ZarA$,
\item  $\fD_{\!\ZarA}(\DA(x_1,\ldots ,x_n))$, qui est un \oqc de $\Spec(\ZarA)$,
\item    et  $\fD_\gA(x_1,\ldots ,x_n)$, qui est un \oqc de $\Spec\,\gA$.
\end{itemize}

\noindent Du point de vue \cof,  on considère $\Spec\,\gA$ comme un \gui{espace topologique sans point}, \cad un espace défini uniquement à
travers une base d'ouverts. 
Des identifications ci-dessus il ne reste alors que celle donnée entre $\ZarA$ et le \trdi défini formellement à la Joyal dans la proposition 
\ref{fpropZar}.\eoe

\medskip  On a aussi facilement:
\begin{ffact}[quotients]
\label{ffactQuoAT} ~\\
Si $J\subseteq\gA$, alors $\Zar(\aqo{\gA}{J}) \simeq \Zar(\gA/\DA(J))
\simeq \Zar(\gA)/(\IZA(J)=0)$.
\end{ffact}

\begin{ffact}[transporteurs]
\label{ffactTransporteurs} ~\\
Soient $\fA$ et $\fB$ des \ids de $\gA$, $\fa=\DA(\fA)$ et
$\fb=\DA(\fB)$.  
Alors $\fa:\fb=\fa:\fB$ est un \id radical de $\gA$
et dans $\ZarA$ on a $\IZA(\fa):\IZA(\fb)=\IZA(\fa:\fb)$.
\end{ffact}

\begin{ffact}[recouvrement par des \ids]
\label{ffactRecouvI} ~\\
Soit $\fa_i$ une famille finie d'\ids de $\gA$.  Les $\IZA(\fa_i)$
recouvrent $\ZarA$ (\cad leur intersection est réduite à $0$) \ssi
$\bigcap_i\fa_i\subseteq \DA(0)$.
\end{ffact}

En \clama le treillis $\Zar\,\gA$ est noethérien (ce qui 
revient à
dire que~$\Spec\,\gA$ est noethérien) \ssi tout \id radical est
\gui{radicalement de type fini}, \cad est un \elt de~$\Zar\,\gA$.

Par ailleurs,  $\Zar\,\gA$ est une \agH \ssi
\hbox{$\Tt \fa,\fb\in\Zar\,\gA\;(\fa:\fb)\in\Zar\,\gA$}.

Le résultat suivant est important en \coma. 

\begin{fproposition}
\label{fpropZarHeyt} (cf. \cite{fCL2003})
Si $\gA$ est un anneau noethérien cohérent $\Zar\,\gA$ est une 
\agH. 
Si en outre $\gA$ est fortement discret, la relation d'ordre dans  
$\Zar\,\gA$
est décidable. On dit alors que le treillis est \emph{discret}.
\end{fproposition}

\subsubsection*{Filtres de $\gA$ et  de $\ZarA$}

Un \textsl{filtre} dans un anneau commutatif est un \mo $\fF$ qui 
vérifie
$xy\in\fF\Rightarrow x\in\fF$. Un \textsl{filtre premier} est un filtre 
qui
vérifie $x+y\in\fF\Rightarrow x\in\fF\;\mathrm{ou}\;y\in\fF$ (c'est 
le
complémentaire d'un \idep).

Pour $x\in \gA$ le filtre $\uar \wi x$ de $\ZarA$ est noté 
$\FZ_\gA(x)$.
Plus généralement pour  $S\subseteq\gA$ on note
$\FZ_\gA(S)$ le filtre de $\ZarA$:
\[
\FZ_\gA(S)=
\bigcup\nolimits_{x\in\cM(S)}\uar \wi x.
\]
On a aussi $\FZ_\gA(S) =\FZ_\gA(\fF)=\bigcup\nolimits_{x\in\fF}\!\uar \wi x$, où $\fF$ est le filtre de  $\gA$ engendré par $S$.

Les faits suivants sont faciles.

\begin{ffact}
\label{ffactFZ}
L'application $\ff\mapsto \FZ_\gA(\ff)$ établit une correspondance
\emph{injective} croissante des filtres de $\gA$ vers les filtres de 
$\ZarA$, et
un  sup fini (le sup de $\ff_1$ et $\ff_2$  est engendré par les 
$f_1f_2$ où
$f_i\in\ff_i$) donne pour image le sup fini des filtres images.
Cette correspondance $\FZ_\gA$  se restreint en une bijection entre 
les filtres
premiers de $\gA$ et ceux de $\ZarA$.
\end{ffact}

Notez cependant que le filtre principal de $\ZarA$ engendré par
$\wi{a_1}\vu\cdots \vu \wi{a_n}$ (\cad l'intersection des filtres 
$\uar\wi
{a_i}$), ne correspond en général à aucun filtre de $\gA$.

\begin{ffact}[localisés]
\label{ffactLocalises} ~ \\
Soit $S$ un \mo de $\gA$, $\fF$ le filtre engendré par $S$, et
$\ff=\FZ_\gA(S)=\FZ_\gA(\fF)$. \\
Alors $S^{-1}\gA=\gA_S=\gA_\fF$ et $\Zar(\gA_S)\simeq\Zar(\gA)/(\ff=1)$.
\end{ffact}

\begin{ffact}[filtre complémentaire]
\label{ffactComplement} ~\\
Soit $x\in\gA$, le filtre $1_\ZarA\setminus \FZ_\gA(x)$
est égal à $\FZ_\gA(1+x\gA)$.
\end{ffact}

\begin{ffact}[recouvrement par des filtres]
\label{ffactRecouvF} ~\\
Soit $(S_i)_{1\leq i\leq n}$ une famille finie de \mos de $\gA$. Les 
filtres
$\FZ_\gA(S_i)$ recouvrent $\ZarA$ (\cad leur intersection est 
réduite à
$\so{1}$) \ssi les \mos $S_i$ sont \com \cad que pour tous $x_i\in 
S_i$
on a $\gen{x_1,\ldots ,x_n}=\gen{1}$. \\
Plus généralement on a $\,\FZ_\gA(S_1)\,\cap \cdots
\cap\,\FZ_\gA(S_n)\subseteq \FZ_\gA(S)$ \ssi pour  tous $x_i\in S_i$
il existe $x\in S$ tel que $x\in\gen{x_1,\ldots ,x_n}$.
\end{ffact}

\subsection{Le treillis de Heitmann}

Dans un anneau commutatif, le \textsl{radical de Jacobson d'un \id 
$\fJ$} est (du point de vue des \clama) l'intersection des \idemas qui contiennent $\fJ$. On
le note $\JA(\fJ)=\rJ(\gA,\fJ)$, ou encore $\rJ(\fJ)$ si le contexte est 
clair.
En \coma on utilise la \dfn suivante, classiquement équivalente:
\begin{equation} \label{feqRadJac}
\JA(\fJ)\eqdefi\sotq{x\in\gA}{\Tt y\in\gA,\;\; 1+xy
\hbox{ est inversible modulo } \fJ}.
\end{equation}
On notera $\JA(x_1,\ldots ,x_n)=\rJ(\gA,x_1,\ldots ,x_n)$ pour
$\JA(\gen{x_1,\ldots ,x_n})$.  L'\id  $\JA(0)$ est appelé le
\textsl{radical de Jacobson de l'anneau $\gA$}.

\begin{fdefinition}
\label{fdefHeitA}
On appelle \textsl{treillis de Heitmann} d'un anneau commutatif $\gA$ 
le treillis
$\He(\ZarA)$, on le note $\HeA$.
\end{fdefinition}

En \clama,  vu le fait \ref{ffactSpecAzarA} et vue la \dfn du radical 
de Jacobson
via les intersections d'\ids maximaux, le fait suivant,
qui conduit à une interprétation simple du treillis $\HeA$, est 
évident.
Nous sommes néanmoins intéressés par une \prco directe.

\begin{ffact}[radical de Jacobson]
\label{ffactRadJac} ~\\
La correspondance bijective $\IZA$ préserve le passage au radical 
de Jacobson.
Autrement dit si $\fJ$ est un \id de $\gA$ et $\fj=\IZA(\fJ)$, alors
$\rJ_\ZarA(\fj)=\IZA(\JA(\fJ))$.
\end{ffact}
\begin{proof} Soit un $x$ arbitraire dans $\gA$.
Nous devons montrer que $\wi 
x\in\rJ_\ZarA(\fj)$ \ssi $\wi
x\in\IZA(\JA(\fJ))$. Puisque $\JA(\fJ)$ est un \id radical, on 
doit
démontrer l'équivalence:
\[  \wi x\in\rJ_\ZarA(\fj)\;\Leftrightarrow\; 
x\in\JA(\fJ). 
\]
Par \dfn $\wi x\in\rJ_\ZarA(\fj)$ signifie
\[\Tt y\in\ZarA\quad (\wi x\vu y=1_\ZarA\;\Rightarrow\;\Ex 
z\in\fj\;\; z\vu
y=1_\ZarA),
\]
\cade puisque tout $y\in\ZarA$ est de la forme $\DA(y_1,\ldots ,y_k)$,
\[\Tt y_1,\ldots ,y_k\in\gA\quad (\gen{x,y_1,\ldots ,y_k}=1_\gA\;
\Rightarrow\;\Ex z\in\fj\;\; z\vu \DA(y_1,\ldots ,y_k)=1_\ZarA),\]
ceci est immédiatement équivalent à
\[\Tt y_1,\ldots ,y_k\in\gA\quad (\gen{x,y_1,\ldots ,y_k}=1_\gA\;
\Rightarrow\;\Ex u\in\fJ \;\;\gen{u,y_1,\ldots ,y_k}=1_\gA),
\]
puis à
\[\Tt y\in\gA\;(\gen{x,y}=1_\gA\;
\Rightarrow\;\Ex u\in\fJ \;\;\gen{u,y}=1_\gA),
\]
ou encore à: tout  $y\in\gA$ de la forme $1+xa$ est inversible 
modulo $\fJ$.
\Cad $x\in\JA(\fJ)$.
\end{proof}

\begin{fcorollary} \label{fpropHeitA}
Soient $\fj_1$ et $\fj_2$ deux \itfs de $\gA$. Les \elts $\DA(\fj_1)$ 
et $\DA(\fj_2)$ de $\ZarA$ sont égaux dans le quotient $\HeA$ \ssi
$\JA(\fj_1)=\JA(\fj_2)$. 
En conséquence $\HeA$ s'identifie à l'ensemble des $\JA(x_1,\ldots ,x_n)$, avec
$\JA(\fj_1)\vi\JA(\fj_2)=\JA(\fj_1\fj_2)$ et
$\JA(\fj_1)\vu\JA(\fj_2)=\JA(\fj_1+\fj_2)$.
\end{fcorollary}

\rems ~\\
1. Vu les bonnes propriétés de la correspondance $\IZA$ on
a, avec $\gT=\ZarA$, $\Zar(\gA/\JA(0))\simeq\gT/(\JT(0)=0)$. Par 
contre il ne
semble pas qu'il y ait une $\gA$-algèbre $\gB$ naturellement 
attachée à~$\gA$ pour laquelle on ait $\Zar\,\gB\simeq\He(\ZarA)$.\\
2. Notez que, en général $\JA(x_1,\ldots ,x_n)$ est un \id 
radical mais pas
le radical d'un \itf.  \\
3. On voit aussi facilement que
$\JA(\fj_1)\vi\JA(\fj_2)=\JA(\fj_1)\cap\JA(\fj_2)=\JA(\fj_1\cap\fj_2)$ 
(cela
résulte d'ailleurs du lemme \ref{flemJacInter}). Il peut sembler 
surprenant que
$\JA(\fj_1)\cap\JA(\fj_2)=\JA(\fj_1\fj_2)$ (c'est à priori moins 
clair que pour
les $\DA$). Voici le calcul élémentaire qui (re)démontre ce 
fait.
On a $x\in\JA(\fj_1)$ \ssi $\Tt y\;(1+xy)$ est inversible modulo 
$\fj_1,$ et
$x\in\JA(\fj_2)$ \ssi $\Tt y\;(1+xy)$ est inversible modulo $\fj_2$.
Mais si $a=1+xy$ est inversible modulo $\fj_1$ et $\fj_2$, il est 
inversible
modulo leur produit: en effet $1+aa_1\in\fj_1$ et  $1+aa_2\in\fj_2$  
impliquent
que $(1+aa_1)(1+aa_2)$, qui se réécrit $1+aa'$, est dans 
$\fj_1\fj_2$.

\subsection{Dimension et  bords de Krull}

En \coma on donne la \dfn suivante.

\begin{fdefinition}
\label{fdefKdimA}
La dimension de Krull d'un anneau commutatif
est la dimension de Krull de son treillis de Zariski.
\end{fdefinition}

Vus le fait \ref{ffactSpecAzarA} et le \tho \ref{fthDK1}, il s'agit 
d'une \dfn
équivalente à la \dfn usuelle en \clama. 

\begin{fdefinition}
\label{fdefZar2} Soit  $\gA$  un anneau commutatif, $x\in\gA$ et $\fj$ 
un \itf. 
\begin{itemize}
\item [$(1)$] Le \textsl{bord supérieur de Krull de $\fj$ dans $\gA$} 
est
l'anneau quotient
$\gA\ul{\fj}:=\gA/\rK_\gA(\fj)$ où
\begin{equation} \label{feqBKAC}
\rK_\gA(\fj):=\fj+(\DA(0):\fj)
\end{equation}
  On note aussi $\gA\ul{x}=\gA\ul{x\gA}$ et on l'appelle le \textsl{bord
supérieur de $x$ dans $\gA$}. On dira aussi que $\rK_\gA(\fj)$ est 
\textsl{l'\id
bord de Krull de $\fj$.}  On notera aussi  $\rK_\gA(y_1,\ldots ,y_n)$ 
pour
$\rK_\gA(\gen{y_1,\ldots ,y_n})$ et $\rK_\gA^x$ pour $\rK_\gA(x)$.
\item [$(2)$] Le \textsl{bord inférieur de Krull de $x$ dans $\gA$} 
est l'anneau
localisé
$\gA\bal{x}:=\gA_{\rS\bal{x}}$ où $\rS\bal{x}=x^\NN(1+x\gA)$.
On dira aussi que le \mo $x^\NN(1+x\gA)$ est le \textsl{\mo bord de 
Krull de~$x$.}
\end{itemize}
\end{fdefinition}

Ainsi un \elt arbitraire de  $\rK_\gA(y_1,\ldots ,y_n)$ s'écrit 
$\sum_i
a_iy_i+b$ avec tous les $by_i$ nilpotents.

\begin{fproposition}
\label{fpropZar2} Soit   $x\in \gA$ et  $\fj=\gen{j_1,\ldots ,j_n}$ un 
\itf. 
Considérons $\wi x\in\ZarA$ et $\fa=\DA(\fj)=\wi{j_1}\vu\cdots
\vu\wi{j_n}\in\ZarA$. Alors
\begin{enumerate}
\item L'idéal bord de Krull de $\fa=\DA(\fj)$ dans $\ZarA$,
$\rK_\ZarA^\fa$, est égal à
$\IZA(\rK_\gA(\fj))$. En conséquence  $(\ZarA)\ul{\fa}$ 
s'identifie
naturellement avec $\Zar(\gA\ul{\fj})$.
\item Le filtre  bord de Krull de $\wi x$ dans $\ZarA$, 
$\rK^\ZarA_{\tilde x}$, est
égal à
$\FZ(\rS\bal{x})$. En conséquence   $(\ZarA)\bal{\tilde x}$ s'identifie
naturellement avec
$\Zar(\gA\bal{x})$.
\end{enumerate}
\end{fproposition}
\begin{proof}
Pour l'idéal bord, on a par \dfn
  \[\rK_\ZarA^\fa=\rK_\ZarA(\DA(\fa))=\DA(\fa) \vu 
(\DA(0):\DA(\fa) ).
\]
D'après les faits \ref{ffactSpecAzarA} et \ref{ffactTransporteurs},
il est égal à $\IZA(\fa + (\DA(0):\fa ))$ et aussi à
  $\IZA(\fj + (\DA(0):\fj ))=\IZA(\rK_\gA^\fj)$. Enfin pour les 
passages au
quotient on applique le fait~\ref{ffactQuoAT}.\\
Pour le filtre bord de Krull, cela fonctionne de la même 
manière en
utilisant les faits \ref{ffactFZ} et \ref{ffactComplement} puis en 
passant au
treillis quotient avec le fait~\ref{ffactLocalises}.
\end{proof}

\hum{pour le bord inférieur avec un \itf à la place d'un \elt il 
semble qu'on
obtient non pas un anneau, mais un schéma de Grothendieck obtenu en 
recollant
un nombre fini d'anneaux le long de localisations convenables}

Comme corolaire du \thref{fpropDK1} et de la proposition \ref{fpropZar2}  on 
obtient
l'analogue suivant du \thref{fthDK1}, dans une version 
entièrement \cov. 
Rappelons que la dimension de Krull d'un anneau \hbox{égale $-1$} \ssi 
l'anneau est
trivial (i.e., $1_\gA=0_\gA$).

\begin{ftheorem}
\label{fthDKA} Pour un anneau commutatif  $\gA$ et un entier $\ell\geq 
0$ \propeq
\begin{enumerate}
\item La dimension de Krull de $\gA$ est $\leq \ell$.
\item Pour tout $x\in \gA$ la dimension de Krull de $\gA\ul x$ est 
$\leq \ell-
1$.
\item Pour tout \itf $\fj$ de $\gA$ la dimension de Krull de 
$\gA\ul\fj$ est
$\leq \ell-1$.
\item Pour tout $x\in \gA$ la dimension de Krull de $\gA\bal x$ est 
$\leq \ell-
1$.
\end{enumerate}
\end{ftheorem}

Ce \tho nous donne une bonne signification intuitive de la dimension 
de Krull.

Avec le fait \ref{ffactSpecAzarA} on obtient le même \tho en 
\clama. 

Vu son importance, nous allons donner des preuves directes simples des équivalences entre les points \textsl{1}, \textsl{2} et \textsl{4} en \clama. 

\begin{proof}[Démonstration directe en mathématiques classiques] 
Montrons d'abord l'équivalence des points \textsl{1} et \textsl{2}.
Rappelons que les \ideps de $S^{-1}\gA$ sont de la forme $S^{-1}\fp$ 
où $\fp$
est un \idep
de $\gA$ qui ne coupe pas $S$.
L'équivalence résulte alors clairement des deux affirmations 
suivantes.\\
(a) Soit $x\in\gA$, si $\fm$ est un idéal maximal de $\gA$ il coupe 
toujours
$\rS\bal{x}$. En effet si $x\in\fm$ c'est clair et sinon, $x$ est 
inversible
modulo $\fm$ ce qui signifie que $1+x\gA$ coupe $\fm$.\\
(b) Si $\fm$ est un idéal maximal de $\gA,$  et si 
$x\in\fm\setminus\fp$ où
$\fp$ est un \idep contenu dans $\fm$,
alors $\fp\cap \rS\bal{x}=\emptyset$: en effet si $x(1+xy)\in\fp$ 
alors,
puisque $x\notin\fp$ on a $1+xy\in\fp\subset\fm$, ce qui donne
la contradiction $1\in\fm$ (puisque $x\in\fm$).\\
Ainsi, si $\fp_0\subsetneq \cdots \subsetneq \fp_\ell$
est une chaine avec
$\fp_\ell$ maximal, elle est raccourcie d'au moins son dernier terme 
lorsqu'on
localise en $\rS\bal{x}$, et elle n'est raccourcie
que de son dernier terme si
$x\in\fp_\ell\setminus\fp_{\ell-1}$.\\
L'équivalence des points \textsl{1} et \textsl{4} se démontre de manière 
\gui{opposée},
en remplaçant les idéaux premiers par les filtres premiers.
On remarque d'abord que les filtres premiers de $\gA/\fJ$ sont de la 
forme
$(S+\fJ)/\fJ$, où $S$ est un filtre premier de $\gA$ qui ne coupe 
pas $\fJ$.
Il suffit alors de démontrer les deux affirmations \gui{opposées} 
de (a) et
(b) qui sont les suivantes:\\
(a') Soit $x\in\gA$, si $S$ est un filtre maximal de $\gA$ il coupe 
toujours
$\rK_\gA^x$. En effet si $x\in S$ c'est clair et sinon,
puisque $S$ est maximal $Sx^\NN$ contient $0$, ce qui signifie qu'il
y a un entier $n$ et un \elt $s$ de $S$ tels que $sx^n=0$.
Alors $(sx)^n=0$ et $s\in (\sqrt{0}:x)\subset \rK_\gA^x$.\\
(b') Si $S$ est un filtre maximal de $\gA,$  et si $x\in S\setminus 
S'$ où
$S'\subset S$ est un filtre premier,
alors $S'\cap \rK_\gA^x=\emptyset$. En effet si $ax+b\in S'$ avec 
$(bx)^n=0$
alors, puisque $x\notin S'$ on a $ax\notin S'$ et, vu que $S'$ est 
premier,
$b\in S'\subset S$, mais comme $x\in S$, $(bx)^n=0\in S$ ce qui est 
absurde.
\end{proof}

En outre le \tho \ref{fthDKA} implique la caractérisation \cov \elr 
de cette
dimension  en terme d'identités algébriques, décrite 
dans~\cite{fLom02,fCL2003}, comme suit.

\begin{fcorollary}
\label{fcorthDKA} \Propeq
\begin{itemize}
\item  [$(1)$] La dimension de Krull de $\gA$ est $\leq \ell$
\item  [$(2)$] Pour tous $x_0,\ldots ,x_\ell\in\gA$ il existe 
$b_0,\ldots,b_\ell\in \gA$ tels que 
\begin{equation} \label{feqCG}
\left.
\begin{array}{rcl} 
\DA(b_0x_0)& =  &\DA(0)    \\ 
\DA(b_1x_1)& \leq  & \DA(b_0,x_0)  \\
\vdots\quad& \vdots  &\quad  \vdots \\
\DA(b_\ell x_\ell )& \leq  & \DA(b_{\ell -1},x_{\ell -1})  \\
\DA(1)& =  &  \DA(b_\ell,x_\ell )
\end{array}
\right\}
\end{equation}
\item  [$(3)$] Pour tous $x_0,\ldots ,x_\ell\in\gA$ il existe 
$a_0,\ldots,a_\ell\in \gA$ et
$m_0,\ldots,m_\ell\in\NN$ tels que
\[ x_0^{m_0}(x_1^{m_1}\cdots(x_\ell^{m_\ell} (1+a_\ell x_\ell) + 
\cdots+a_1x_1)
+ a_0x_0) =0
\]
\end{itemize}
\end{fcorollary}
\begin{proof}
Montrons l'equivalence de (1) et (3). 
Utilisons par exemple pour (1) la caractérisation via les localisés  $\gA\bal x$.
L'équivalence pour la dimension $0$ est claire. Supposons la chose 
établie
pour la dimension $\leq \ell$. On voit alors que $S^{-1}\gA$ est de 
dimension
  $\leq \ell$ \ssi pour tous
$x_0,\ldots ,x_\ell\in \gA$ il existe $a_0,\ldots,a_\ell\in \gA$,
$m_0,\ldots,m_\ell\in\NN$  et $s\in S$ tels que
\[ x_0^{m_0}(x_1^{m_1}\cdots(x_\ell^{m_\ell} (s+a_\ell x_\ell) + 
\cdots+a_1x_1) +
a_0x_0)=0.\]
Il reste donc à remplacer $s$ par un \elt arbitraire de la forme
$x_{\ell+1}^{m_{\ell+1}} (1+a_{\ell+1} x_{\ell+1})$. \\
On a  $(3) \Rightarrow (2)$ en prenant:
$b_\ell=1+a_\ell x_\ell$, et 
$b_{k -1} = x_k^{m_k} b_k+ a_{k -1}x_{k -1}$, pour $k=\ell,$
\ldots, $1$. \\
On a  $(2) \Rightarrow (1)$ en considérant la caractérisation (4) 
de la dimension de Krull
d'un \trdi donnée dans le \thref{fpropDK1} 
et en l'appliquant au treillis de Zariski $\ZarA$ avec 
$S=\sotq{\DA(x)}{x\in\gA}$.
On pourrait aussi vérifier par un calcul direct que $(2) \Rightarrow 
(3)$.
\end{proof}

\rem Le système d'inégalités (\ref{feqCG}) dans le point (2) du 
corolaire précédent établit une relation intéressante et
symétrique entre les deux suites 
$(b_0,\ldots ,b_\ell)$ et $(x_0,\ldots ,x_\ell)$. Lorsque $\ell=0$, 
cela signifie $\DA(b_0)\vi\DA(x_0)=0$ et $\DA(b_0)\vu\DA(x_0)=1$, 
\cad 
que les deux \elts~$\DA(b_0)$ et $\DA(x_0)$ sont compléments l'un de l'autre
dans $\ZarA$.
Dans $\Spec\,\gA$ cela signifie que les ouverts de base 
correspondants sont complémentaires. Nous introduisons donc  
la terminologie suivante: lorsque deux suites 
$(b_0,\ldots ,b_\ell)$ et $(x_0,\ldots ,x_\ell)$ vérifient les 
inégalités (\ref{feqCG}) nous dirons qu'elles sont 
\textsl{complémentaires}.\eoe

\medskip 
Signalons aussi qu'il est facile d'établir \cot que
$\Kdim(\gK[X_1,\ldots,X_n])=n$ lorsque~$\gK$ est un corps, ou même
un anneau zéro dimensionnel (cf.~\cite{fCL2003}).  On peut aussi traiter
de façon \cov la dimension de Krull des anneaux géométriques
(les $\gK$-algèbres de présentation finie).

\medskip
\rem
On a aussi (déjà démontré pour les \trdis) les résultats 
suivants:
\begin{itemize}
\item

si $\gB$ est un quotient ou un localisé de $\gA$,  
$\Kdim\,\gB\leq
\Kdim\,\gA$,

\item si $(\fa_i)_{1\leq i\leq m}$ une famille finie d'\ids de $\gA$ 
et
$\fa=\bigcap_{i=1}^m\fa_i$, alors
$\Kdim(\gA/\fa)=\sup_i\Kdim(\gA/\fa_i)$.

\item si $(S_i)_{1\leq i\leq m}$ une famille finie de \moco de $\gA$ 
alors
$\Kdim(\gA)=\sup_i\Kdim(\gA_{S_i})$.

\item en \clama on a
$\Kdim(\gA)=\sup_\fm\Kdim(\gA_{\fm})$, où $\fm$ parcourt tous les 
\idemas.
\eoe 
\end{itemize}

\medskip
\rem
On peut illustrer le corolaire \ref{fcorthDKA} ci-dessus
en introduisant \gui{l'idéal bord de Krull itéré}.
Pour $x_1,\ldots ,x_n\in\gA$ considérons
\[(\gA\ul{x_0})\ul{x_1},\, ((\gA\ul{x_0})\ul{x_1})\ul{x_2},\, 
etc\ldots\,,
\] 
  les anneaux bords supérieurs successifs, et notons $\rK_\gA[x_0,\ldots ,x_\ell]$
le noyau de la projection canonique 
$\gA\to
(\cdots(\gA\ul{x_0}){\cdots})\ul{x_\ell}$.  
Alors on a 
$y\in\rK_\gA[x_0,\ldots,x_\ell]$ \ssi $\Ex a_0,\ldots,  a_\ell\in \gT$ 
et $m_0,\ldots,m_\ell\in\NN$  vérifiant:
\[  x_0^{m_0}(x_1^{m_1}\cdots(x_\ell^{m_\ell} (y+a_\ell x_\ell) + 
\cdots+a_1x_1)
+ a_0x_0) =0.
\]
Et la dimension de Krull est $\leq \ell$ \ssi pour tous $x_0,\ldots
,x_\ell\in\gA$ on a $1\in\rK_\gA[x_0,\ldots ,x_\ell]$.\eoe

\subsection{Dimensions de Heitmann}
\subsubsection*{Le spectre de Heitmann}

L'espace spectral que Heitmann a défini pour remplacer le 
j-spectrum,
\cad l'adhérence pour la topologie constructible du spectre maximal 
dans
$\Spec\,\gA$,
  correspond à la \dfn suivante.
\begin{fdefinition}
\label{fdefJspecA}
On appelle \textsl{spectre de Heitmann} d'un anneau commutatif $\gA$ le 
sous-espace $\Jspec(\ZarA)$ de $\Spec\,\gA$. On le note aussi 
$\Jspec\,\gA$. On note
$\jspec\,\gA$ pour $\jspec(\ZarA)$, \cad le j-spectrum de l'anneau au 
sens
usuel.
\end{fdefinition}

En \clama le \tho \ref{fthDK3} donne:
\begin{ffact}
\label{ffactHSpecA}
Pour tout anneau commutatif $\gA$, le spectre de Heitmann de $\gA$ 
s'identifie
à l'espace spectral  $\Spec(\HeA)$ (au sens des \trdis).
\end{ffact}

On a alors la \dfn \cov \elr sans points de la dimension introduite 
par
Heitmann.
\begin{fdefinition}
\label{fdefJdimA}
La $\rJ$-dimension de Heitmann de $\gA$, notée $\Jdim\,\gA$, est la 
dimension
de Krull de $\Heit(\gA)$, autrement dit c'est la $\Jdim$ de $\ZarA$.
\end{fdefinition}

En \clama $\Jdim\,\gA$ est égal à la dimension de l'espace 
spectral
$\Jspec\,\gA$, définie de manière abstraite \gui{avec points}.

On peut aussi noter $\jdim\,\gA$ pour la dimension de $\jspec\,\gA$, 
qui n'est
pas un espace spectral (et nous ne proposons pas de \dfn \cov sans 
point pour
cette dimension).

\medskip \rem
Précisons la signification de $\Jdim\,\gA\leq \ell$ dans le cas des 
anneaux commutatifs. Comme il s'agit de la dimension de Krull de 
$\Heit\,\gA$ et que les \elts de $\Heit\,\gA$  s'identifient aux 
radicaux de Jacobson d'\itfs on obtient la caractérisation 
suivante:\\ 
$\forall x_0,\ldots,x_\ell\in \gA\;$ 
$\Ex \fa_0,\ldots,  \fa_\ell,$ \itfs de $\gA$ tels que: 
\[\begin{array}{rcl} 
x_0\,\fa_0  &  \subseteq  &  \JA(0)    \\ 
x_1\,\fa_1 &  \subseteq  &  \JA(\gen{x_0}+\fa_0)   \\
\vdots \;  &  \vdots &   \qquad \vdots   \\
x_{\ell}\,\fa_{\ell} &  \subseteq  &  
\JA(\gen{x_{\ell-1}}+\fa_{\ell-1})   \\
\gen{1} &  =  &  \JA(\gen{x_{\ell}}+\fa_{\ell})   
\end{array}\]
On ne peut apparemment pas éviter le recours aux \itfs et cela fait 
que l'on n'obtient pas une \dfn \gui{au premier ordre}.\\
Notez que chaque appartenance $x\in\JA(y_1,\ldots ,y_m)$ s'exprime 
elle-même par: $\Tt z\in\gA,$ $1+xz$ est inversible modulo
$\gen{y_1,\ldots ,y_m}$, \cade 

\medskip  \hspace*{4em}
$\Tt z\in\gA \;\Ex t,u_1,\ldots ,u_m\in\gA, 
\;\;1=(1+xz)t+u_1y_1+\cdots 
+u_my_m.$  
\eoe

\subsubsection*{Dimension  et  bord de Heitmann}

\begin{fdefinition}
\label{fdefHdimA}
La dimension de Heitmann d'un anneau commutatif
est la dimension de Heitmann de son treillis de Zariski.
\end{fdefinition}

\begin{fdefinition}
\label{fdefHei2} Soit  $\gA$  un anneau commutatif, $x\in\gA$ et $\fj$ 
un \itf. 
Le \textsl{bord de Heitmann de $\fj$ dans $\gA$} est l'anneau quotient
$\gA/\rH_\gA(\fj)$ avec
          \[\rH_\gA(\fj):=\fj+(\JA(0):\fj)\]
qui est aussi appelé \textsl{l'\id bord de Heitmann de $\fj$ dans 
$\gA$}.
On notera aussi  $\rH_\gA(y_1,\ldots ,y_n)$ pour  
$\rH_\gA(\gen{y_1,\ldots
,y_n})$, $\rH_\gA^x$ pour $\rH_\gA(x)$ et $\gA_\rH^x$ pour 
$\gA/\rH_\gA^x$.
\end{fdefinition}

Ainsi un \elt arbitraire de  $\rH_\gA(y_1,\ldots ,y_n)$ s'écrit 
$\sum_i
a_iy_i+b$ avec tous les $by_i$ dans $\JA(0)$.

La proposition suivante résulte des bonnes propriétés de la 
correspondance
bijective $\IZA$ (voir les faits \ref{ffactSpecAzarA}, \ref{ffactQuoAT}, 
\ref{ffactTransporteurs}
et \ref{ffactRadJac}).
\begin{fproposition}
\label{fpropBordH-TA}
Pour un \itf  $\fj$ le bord de Heitmann de $\fj$ au sens des anneaux 
commutatifs
et celui au sens  des \trdis se correspondent. Plus précisément,
avec $j=\DA(\fj)$ et $\gT=\ZarA$, on a: 
\[\IZA(\rH_\gA(\fj))=\rH_\gT(j),\;\;
\emph{et}\;\; \Zar(\gA/\rH_\gA(\fj))\simeq
\gT/(\rH_\gT(j)=0)=\gT_\rH^j.\]
\end{fproposition}

Comme corolaire des propositions \ref{fpropHdimgen} et 
\ref{fpropBordH-TA}
on obtient le résultat suivant.

\begin{fproposition}
\label{flemDHA} Pour un anneau commutatif  $\gA$ et un entier 
$\ell\geq 0$
\propeq
\begin{enumerate}
\item La dimension de Heitmann de $\gA$ est $\leq \ell$.
\item Pour tout $x\in\gA$,
$\Hdim(\gA/\rH_\gA(x))\leq \ell-1$.
\item Pour tout \itf $\fj$ de $\gA$, $\Hdim(\gA/\rH_\gA(\fj)) \leq 
\ell-1$.
\end{enumerate}
\end{fproposition}

\rem La dimension de Heitmann de $\gA$ peut donc être définie 
de manière
inductive comme suit:
\begin{itemize}
\item $\Hdim\,\gA=-1$ \ssi $1_\gA=0_\gA$.
\item Pour $\ell\geq 0$, $\Hdim\,\gA\leq \ell$ \ssi pour tout 
$x\in\gA$,
$\Hdim(\gA/\rH_\gA(x))\leq \ell-1$.
\end{itemize}
Pour illustrer cette \dfn avec précision, 
nous introduisons \gui{l'idéal bord de Heitmann itéré}.
Pour $x_0,\ldots ,x_n\in\gA$ nous notons 
\[
\gA_\rH[x_0]=\gA_\rH^{x_0},\,
\gA_\rH[x_0,x_1]=(\gA_\rH^{x_0})_\rH^{x_1},\,
\gA_\rH[x_0,x_1,x_2]=((\gA_\rH^{x_0})_\rH^{x_1})_\rH^{x_2},\, 
\hbox{ etc}\ldots\,\,  
\]
les
anneaux bords de Heitmann successifs, et $\rH[\gA;x_0,\ldots
,x_k]=\rH_\gA[x_0,\ldots ,x_k]$ désigne le noyau de la projection 
canonique
$\gA\to \gA_\rH[x_0,\ldots ,x_k]$.  Pour décrire ces \ids nous 
avons besoin de
la notation \[\bidule{z,x,a,y,b}=1+(1+(z+ax)xy)b.\]
Alors on a:
\begin{itemize}
\item $z\in\rH_\gA[x_0]$ \ssi \[\Ex a_0  \ \Tt y_0 \ \Ex
b_0,\;\bidule{z,x_0,a_0,y_0,b_0}=0\]
\item  $z\in\rH_\gA[x_0,x_1]$ \ssi \[\Ex a_1  \ \Tt y_1 \ \Ex b_1\ 
\Ex a_0  \ \Tt y_0 \ \Ex b_0,  \;\bidule{\bidule{z,x_1,a_1,y_1,b_1},x_0,a_0,y_0,b_0}=0\]
\item  $z\in\rH_\gA[x_0,x_1,x_2]$ \ssi
\[\Ex a_2  \ \Tt y_2 \ \Ex b_2 \ \Ex a_1  \ \Tt y_1 \ \Ex b_1 \ \Ex 
a_0  \ \Tt y_0 \ \Ex b_0,\;\bidule{\bidule{\bidule{z,x_2,a_2,y_2,b_2},x_1,a_1,y_1,b_1},x_0,a_0,y_0,
b_0}=0\]
\end{itemize}
Et ainsi de suite. Et la dimension de Heitmann est $\leq \ell$ \ssi 
pour tous
$x_0,\ldots ,x_\ell\in\gA$ on a $1\in\rH_\gA[x_0,\ldots ,x_\ell]$.\eoe

\begin{fproposition}
\label{fpropHei2} Soit $\fj=\gen{j_1,\ldots ,j_n}$ un \itf.   Notons
$\fJ=\JA(\fj)=\JA(j_1)\vu\cdots \vu\JA(j_n)$.  Alors
$\Heit(\gA/\rH_\gA(\fj))$ s'identifie naturellement avec un quotient
de $(\HeA)\ul{\fJ}$.  Il y a \egt lorsque $\HeA$ est une \agH. 
En \clama c'est le cas lorsque $\Jspec\,\gA$ est noethérien.
\end{fproposition}
\begin{proof}
Déjà démontré pour un \trdi arbitraire à la place de $\ZarA$
(proposition~\ref{fpropBHeitHeyt}).
\end{proof}

\rem On a aussi (déjà démontré pour les \trdis) les 
résultats
suivants:
\begin{itemize}
\item on a toujours  $\Hdim\,\gA\leq \Jdim\,\gA\leq 
\Kdim(\gA/\JA(0))$,
\item si $(\fa_i)_{1\leq i\leq m}$ est une famille finie d'\ids de $\gA$ 
et
$\fa=\bigcap_{i=1}^m\fa_i$, alors
$\Hdim(\gA/\fa)=\sup_i\Hdim(\gA/\fa_i)$.
\item si $\HeA$ est une \agH\footnote{En particulier en \clama si  $\HeA$ est noethérien.}
on a $\Hdim\,\gA=\Jdim\,\gA$,
\item si $\Max\,\gA$ est noethérien, alors  
$\jspec\,\gA=\Jspec\,\gA$  et
$\Hdim\,\gA=\Jdim\,\gA=\jdim\,\gA$.
\item $\Hdim\,\gA\leq 0\;\Leftrightarrow\;\Jdim\,\gA\leq
0\;\Leftrightarrow\;\Kdim(\gA/\JA(0))\leq 0$. \eoe
\end{itemize}

\medskip Notez que le treillis  $\HeA$ est une \agH \ssi est
vérifiée la propriété suivante:
\[\Tt \fa,\fb\in\Heit\,\gA\;\Ex\fc\in\Heit\,\gA\;(\fc\fb\subseteq\fa\;\mathrm{et}\;
\Tt x\in\gA\;(x\fb\subseteq\fa\Rightarrow x\in\fc))\]
($\fa=\JA(a_1,\ldots ,a_n)$, $\fb=\JA(b_1,\ldots ,b_m)$, 
$\fc=\JA(c_1,\ldots
,c_\ell)$).

\addcontentsline{toc}{section}{Références}

\markboth{Références}{Références}

\small
\bibliographystyle{plainnat-fr}

\normalsize
\newpage
\rdb
\markboth{Post-Scriptum}{Post-Scriptum}

\subsection*{Post-Scriptum: Errata dans l'article original.}
\addcontentsline{toc}{section}{Post-Scriptum}\label{postscriptum}

\medskip $\bullet$ Le fait \ref{ffactRecolTD} et la proposition 
\ref{fpropRecolTD} ne sont pas corrects dans la formulation générale qui était proposée. Un commentaire à ce sujet est fait après la proposition \ref{fpropRecolSpec}. On donne maintenant des énoncés corrects (moins forts) avec leurs démonstrations. Un lemme est ajouté, ce qui décale les numéros qui suivent d'une unité.

\medskip $\bullet$ 
Dans la section \ref{fsubsecAgHagB} \textsl{Algèbres de Heyting, de Brouwer, de Boole,}
il faut lire $a\leq \neg\neg a$ et non pas l'inégalité contraire.

\medskip $\bullet$ 
Dans la section \ref{fsecESSP} on a remplacé $\rD_\gT(\cdot)$ et $\rV_\gT(\cdot)$ par 
$\DT(\cdot)$ et $\VT(\cdot)$ pour les ouverts et fermés de $\Spec\,\gT$. Concernant $\Spec\,\gT\cir$ on a introduit les notations $\DTo(\cdot)$ et $\VTo(\cdot)$ pour que le propos soit plus clair (proposition \ref{fFdSES}).

\medskip $\bullet$ 
Dans le  point \textsl{4} de la proposition \ref{fpropositionFSES}, on a rectifié comme suit:

\smallskip \noindent  pour tous
ouverts quasi-compacts $U_1$ et $U_2$, l'adhérence de $U_1\setminus U_2$ est le complémentaire d'un 
ouvert quasi-compact.

\medskip $\bullet$ 
La proposition 3.13 de l'article original a été ramenée en \ref{fpropJspecjspec}, où elle a mieux sa place.
Cette proposition et sa démonstration sont clarifiées. Dans le point \textsl{2} l'hypothèse \gui{$\jspec\,\gT$ noethérien} a été remplacée par \gui{$\Max\,\gT$ noethérien}. 

\medskip $\bullet$ 
Juste avant la proposition \ref{fpropZar}
l'équivalence 
\[
\widetilde{U}\leq_{\mathrm{Zar}\mathbf{A}} \widetilde{J}
\quad\Longleftrightarrow \quad
\prod\nolimits_{u\in U} u  \in \sqrt{\langle{J}\rangle}
\quad\Longleftrightarrow \quad
{\cal M}(U)\cap \langle{J}\rangle \neq \emptyset
\]

\noindent a été remplacée par la suivante, dans laquelle
le premier terme est changé
\[\bigwedge\widetilde{U}\leq_{\mathrm{Zar}\mathbf{A}}\bigvee \widetilde{J}
\quad\Longleftrightarrow \quad
\prod\nolimits_{u\in U} u  \in \sqrt{\langle{J}\rangle}
\quad\Longleftrightarrow \quad
{\cal M}(U)\cap \langle{J}\rangle \neq \emptyset
\]

\medskip Par ailleurs, signalons que nous avons en plusieurs endroits ajouté des références nouvelles et des commentaires variés.


%% file: FrenchTheoremsHeitmann.tex


\theoremstyle{plain}
\newtheorem{ftheorem}{Théorème}[subsection]
\newtheorem{fthdef}[ftheorem]{Théorème et définition}
\newtheorem{fpstf}[ftheorem]{Positivstellensatz formel}
\newtheorem{fpst}[ftheorem]{Positivstellensatz}
\newtheorem{flemma}[ftheorem]{Lemme}
\newtheorem{fcorollary}[ftheorem]{Corolaire}
\newtheorem{fconjecture}[ftheorem]{Conjecture}
\newtheorem{fproposition}[ftheorem]{Proposition}
\newtheorem{fpbu}[ftheorem]{Problème universel}
\newtheorem{fprpta}[ftheorem]{Propriétés attendues}
\newtheorem{fpropdef}[ftheorem]{Proposition et définition}
\newtheorem{ffact}[ftheorem]{Fait}
\newtheorem{fplcc}[ftheorem]{Principe local-global concret}

\newtheorem{ftheoremc}[ftheorem]{Th\'{e}or\`{e}me\etoz}
\newtheorem{flemmac}[ftheorem]{Lemme\etoz}
\newtheorem{fcorollaryc}[ftheorem]{Corolaire\etoz}
\newtheorem{fproprietec}[ftheorem]{Propri\'{e}t\'{e}\etoz}
\newtheorem{fpropositionc}[ftheorem]{Proposition\etoz}
\newtheorem{ffactc}[ftheorem]{Fait\etoz}
\newtheorem{fvalsatz}[ftheorem]{\vst}

\theoremstyle{definition}
\newtheorem{fconvention}[ftheorem]{Convention}
\newtheorem{fdefinition}[ftheorem]{Définition}
\newtheorem{fdfni}[ftheorem]{Définition informelle}
\newtheorem{fdefinitions}[ftheorem]{Définitions}
\newtheorem{fnotation}[ftheorem]{Notation}
\newtheorem{fproblem}[ftheorem]{Problème}
\newtheorem{fquestion}[ftheorem]{Question}
\newtheorem{fquestions}[ftheorem]{Questions}
\newtheorem{fcontext}[ftheorem]{Contexte}
\newtheorem{fdefinitionc}[ftheorem]{Définition\etoz}
\newtheorem{fdefinota}[ftheorem]{Définition et notation}

\theoremstyle{remark}
\newtheorem{fexample}[ftheorem]{Exemple}
\newtheorem{fexamples}[ftheorem]{Exemples}
\newtheorem{fnotes}[ftheorem]{Notes}
\newtheorem{fremark}[ftheorem]{Remarque}
\newtheorem{fremarks}[ftheorem]{Remarques}
\newtheorem{fcomment}[ftheorem]{Commentaire}


%% file: FrenchMacrosHeitmann.tex

\newcommand \Vrai {\mathsf{Vrai}}
\newcommand \Faux {\mathsf{Faux}}

\newcommand  \rem {\noindent \textsl{Remarque.} }
\newcommand  \rems {\noindent \textsl{Remarques.}  }
\newcommand  \comm {\noindent \textsl{Commentaire.}  }
\newcommand  \exl  {\noindent \textsl{Exemple.}}

\newcommand\oge{\leavevmode\raise.3ex\hbox{$\scriptscriptstyle\langle\!\langle\,$}}
\newcommand\feg{\leavevmode\raise.3ex\hbox{$\scriptscriptstyle\,\rangle\!\rangle$}}
\newcommand\gui[1]{\oge{#1}\feg}

\renewcommand \grave{\mathaccent18}    

\newcommand \recu {récurrence\xspace} 
\newcommand \hdr {hypothèse de \recu}
\newcommand \cad {c'est-à-dire\xspace} 
\newcommand \Cad {C'est-à-dire\xspace} 
\newcommand \cade {c'est-à-dire encore\xspace} 
\newcommand \ssi {si, et seulement si, }
\newcommand \cnes {condition nécessaire et
suffisante\xspace} 
\newcommand \spdg {sans perte de généralité\xspace} 
\newcommand \Propeq {Les propriétés suivantes sont
équivalentes:}
\newcommand \propeq {les propriétés suivantes sont
équivalentes:}
\newcommand \disept {17$^{\mathrm{\grave{e}me}}$ problème de 
Hilbert\xspace}


\newcommand \Amo {$\gA$-module\xspace} 
\newcommand \Amos {$\gA$-modules\xspace} 

\newcommand \agB {\alg de Boole\xspace} 
\newcommand \agBs {\algs de Boole\xspace} 

\newcommand \agH {\alg de Heyting\xspace} 
\newcommand \agHs {\algs de Heyting\xspace}

\newcommand \alg {algèbre\xspace} 
\newcommand \algs {algèbres\xspace} 

\newcommand \auto {automorphisme\xspace} 
\newcommand \autos {automorphismes\xspace}

\newcommand \cac {corps algébriquement clos\xspace} 

\newcommand \carn {caractérisation\xspace} 

\newcommand \coli {combinaison linéaire\xspace} 
\newcommand \colis {combinaisons linéaires\xspace} 

\newcommand \com {comaximaux\xspace} 
\newcommand \comz {comaximaux}

\newcommand \ddk{dimension de Krull\xspace} 

\newcommand \dfn{définition\xspace} 
\newcommand \dfns{définitions\xspace} 

\newcommand \egmt {également\xspace} 

\newcommand \egt {égalité\xspace} 
\newcommand \egts {égalités\xspace} 

\newcommand \elr{élémentaire\xspace} 
\newcommand \elrs{élémentaires\xspace} 

\newcommand \elt{élément\xspace} 
\newcommand \elts{éléments\xspace} 

\def \endo {endomorphisme\xspace} 
\def \endos {endomorphismes\xspace} 

\newcommand \entrel {relation implicative\xspace} 
\newcommand \entrels {relations implicatives\xspace} 

\newcommand \eqvc {équivalence\xspace}

\newcommand \gtr{générateur\xspace} 
\newcommand \gtrs{générateurs\xspace}

\newcommand \homo {homomorphisme\xspace} 
\newcommand \homos {homomorphismes\xspace} 

\newcommand \id {idéal\xspace} 
\newcommand \ids {idéaux\xspace} 

\newcommand \idema {\id maximal\xspace} 
\newcommand \idemas {\ids maximaux\xspace} 

\newcommand \idf {idéal de Fitting\xspace} 
\newcommand \idfs {idéaux de Fitting\xspace} 

\newcommand \idep {\id premier\xspace} 
\newcommand \ideps {\ids premiers\xspace} 

\newcommand \idemi {\idep minimal\xspace} 
\newcommand \idemis {ideps minimaux\xspace} 

\newcommand \iso {isomorphisme\xspace} 
\newcommand \isos {isomorphismes\xspace} 

\newcommand \itf {\id \tf}
\newcommand \itfs {\ids \tf}

\newcommand \thref[1] {\tho~\ref{#1}}
\newcommand \paref[1] {page~\pageref{#1}}

\newcommand \mo {mono\"{\i}de\xspace} 
\newcommand \moco {\mos \com}
\newcommand \mos {mono\"{\i}des\xspace} 

\newcommand \ndz {non diviseur de zéro\xspace} 
\newcommand \ndzs {non diviseurs de zéro\xspace} 

\newcommand \nst {Nullstellensatz\xspace} 
\newcommand \nsts {Nullstellens\"atze\xspace} 

\newcommand \odz {ouvert de Zariski\xspace} 

\newcommand \oqc {ouvert \qc}
\newcommand \oqcs {ouverts \qcs}

\newcommand \pf {de présentation finie\xspace} 
\newcommand \pfz {de présentation finie}

\newcommand \pol {polynôme\xspace} 
\newcommand \pols {polynômes\xspace} 

\newcommand \polcar
{\pol caractéristique\xspace} 

\newcommand \qc {quasi-compact\xspace} 
\newcommand \qcs {quasi-compacts\xspace}

\newcommand \rdl {relation de dépendance linéaire\xspace} 
\newcommand \rdls {relations de dépendance linéaire\xspace} 

\newcommand \rdi {relation de dépendance intégrale\xspace} 
\newcommand \rdis {relations de dépendance intégrale\xspace} 

\newcommand \sad {structure algébrique dynamique\xspace} 
\newcommand \sads {structures algébriques dynamiques\xspace} 

\newcommand \sps {espace spectral\xspace} 
\newcommand \ssps {sous-\sps} 
\newcommand \spss {espace spectraux\xspace} 
\newcommand \sspss {sous-\spss} 

\newcommand \tf {de type fini\xspace} 
\newcommand \trdi {treillis distributif\xspace} 
\newcommand \trdis {treillis distributifs\xspace}

\newcommand \Tho {Théorème\xspace} 
\newcommand \tho {théorème\xspace} 
\newcommand \thos {théorèmes\xspace}

\newcommand \pst {Positivstellensatz\xspace} 


\newcommand \cof {constructif\xspace} 
\newcommand \cofs {constructifs\xspace} 

\newcommand \cov {constructive\xspace} 
\newcommand \covs {constructives\xspace} 

\newcommand \coma {\maths\covs}
\newcommand \clama {\maths classiques\xspace} 

\renewcommand \cot {constructivement\xspace} 

\newcommand \LLPO{{\bf LLPO}}

\newcommand \maths {mathématiques\xspace} 
\renewcommand \math {mathématique\xspace} 

\newcommand \prco {preuve \cov}
\newcommand \prcos {preuves \covs}

\newcommand \pte {principe du tiers exclu\xspace} 

\newcommand \tcg {théorème de complétude de G\"odel 
}
\newcommand \tcgz {théorème de complétude de 
G\"odel}
\newcommand \acgz {axiome de complétude de G\"odel}
\newcommand \Tcgi {Le \tcg implique le résultat
suivant.\xspace}

\newcounter{MF}
\newcommand\stMF{\stepcounter{MF}}

\newcommand{\lec}{\stMF\ifodd\value{MF}lecteur \else 
lectrice \fi}
\newcommand{\lecz}{\stMF\ifodd\value{MF}lecteur\else lectrice\fi}

\newcommand{\lecs}{\stMF\ifodd\value{MF}lecteurs \else 
lectrices \fi}
\newcommand{\lecsz}{\stMF\ifodd\value{MF}lecteurs\else 
lectrices\fi}

\newcommand{\alec}{\stMF\ifodd\value{MF}au lecteur \else%
à la lectrice \fi}
\newcommand{\alecz}{\stMF\ifodd\value{MF}au lecteur\else%
à la lectrice\fi}

\newcommand{\dlec}{\stMF\ifodd\value{MF}du lecteur \else%
de la lectrice \fi}
\newcommand{\dlecz}{\stMF\ifodd\value{MF}du lecteur\else%
de la lectrice\fi}

\newcommand{\llec}{\stMF\ifodd\value{MF}le lecteur \else la lectrice \fi}
\newcommand{\llecz}{\stMF\ifodd\value{MF}le lecteur\else la lectrice\fi}

\newcommand{\Llec}{\stMF\ifodd\value{MF}Le lecteur \else La lectrice \fi}

\newcommand{\lui}{\ifodd\value{MF}lui \else
elle \fi}
\newcommand{\luiz}{\ifodd\value{MF}lui\else
elle\fi}

\newcommand{\celui}{\ifodd\value{MF}celui \else
celle \fi}

\newcommand{\ceux}{\ifodd\value{MF}ceux \else
celles \fi}

\newcommand{\er}{\ifodd\value{MF}er \else
ère \fi}

\newcommand{\eux}{\ifodd\value{MF}eux \else
elles \fi}

\newcommand{\eUx}{\ifodd\value{MF}eux \else
euse \fi}

\newcommand{\leux}{\ifodd\value{MF}leux \else
leuse \fi}

\newcommand{\il}{\ifodd\value{MF}il \else
elle \fi}

\newcommand{\ien}{\ifodd\value{MF}ien \else
ienne \fi}

\newcommand{\e}{\ifodd\value{MF} \else e \fi}
\newcommand{\ez}{\ifodd\value{MF}\else e\fi}

\newcommand{\n}{\ifodd\value{MF}n \else nne \fi}
\newcommand{\nz}{\ifodd\value{MF}n\else nne\fi}

\makeatletter
\newcommand{\la}{\@ifstar{\ifodd\value{MF}le\else
la\fi}{\stMF\ifodd\value{MF}le \else la \fi}}
\makeatother